\documentclass{article}

\usepackage[utf8]{inputenc}
\usepackage{enumitem}
\usepackage{multirow}
\usepackage{bbm}
\usepackage{comment} 
\usepackage{color} 
\usepackage[dvipsnames]{xcolor}
\usepackage{hyperref} 
\usepackage{cite}
\usepackage{adjustbox}
\usepackage{tabularx}
\usepackage{bbm}
\usepackage{mathtools}

\usepackage{amsmath, amsthm}
\usepackage{amsthm}
\usepackage{amsfonts} 
\usepackage{amssymb}
\usepackage{pifont}
\usepackage[margin=1in]{geometry}
\usepackage{bm}
\usepackage{graphicx}
\usepackage[graphicx]{realboxes}
\usepackage{epstopdf} 
\usepackage{subfigure}
\usepackage{soul}
\usepackage[normalem]{ulem}
\usepackage{xr}
\usepackage{titlesec}
\usepackage{url}
\usepackage{mathtools}
\usepackage{amsmath}
\usepackage{algorithm}
\usepackage[noend]{algpseudocode}
\usepackage{lineno} % package for line numbers on paragraphs (uses command: \linenumbers)
\usepackage{adjustbox}
\usepackage{stackengine,graphicx}
\stackMath

\newcommand\frightarrow{\scalebox{1}[.4]{$\rightarrow$}}
\newcommand\darrow[1][]{\mathrel{\stackon[1pt]{\stackanchor[1pt]{\frightarrow}{\frightarrow}}{\scriptstyle#1}}}
\newcommand{\probP}{\text{I\kern-0.15em P}}

\makeatletter
\renewcommand{\@seccntformat}[1]{\csname the#1\endcsname.\quad}
\makeatother

\newcounter{dummy} \numberwithin{dummy}{section}
\newtheorem{theo}[dummy]{Theorem}
\newtheorem{defi}[dummy]{Definition}

\newtheorem{prop}[dummy]{Proposition}
\newtheorem{lemm}[dummy]{Lemma}
\newtheorem{rema}[dummy]{Remark}
\newtheorem{coro}[dummy]{Corollary}

\usepackage{todo}

\usepackage{amssymb}
\usepackage{pifont}
\usepackage{authblk}

\DeclarePairedDelimiter{\abs}{\lvert}{\rvert}

\newcommand{\bLozenge}{\mathbin{\blacklozenge}}
\newcommand{\colb}[1]{#1}

\definecolor{mycolor1}{rgb}{0.97255,0.97255,0.97255}%

\newcommand{\distas}[1]{\mathbin{\overset{#1}{\kern\z@\sim}}}%
\def \P {{\mathbf P}}

\def \V {{\mathbf V}}

\def \C {{\mathbf C}}
\def \R {{\mathbf R}}

\def \D {{\mathbf D}}

\def \C {{\mathbf C}}
\def \e {{\mathbf e}}

\def \M {{\mathbf M}}
\def \a {{\mathbf a}}
\def \p {{\mathbf p}}
\def \q {{\mathbf q}}
\def \v {{\mathbf v}}
\def \w {{\mathbf w}}
\def \u {{\mathbf u}}
\def \x {{\mathbf x}}
\def \y {{\mathbf y}}
\def \z {{\mathbf z}}

\def \A {{\mathbf A}}
\def \B {{\mathbf B}}
\def \r {{\mathbf r}}

\def \X {{\mathbf X}}

\def \Z {{\mathbf Z}}

\def \B {{\mathbf B}}

\def \I {{\mathbf I}}

\newcommand{\norm}[1]{\ensuremath{\left\|#1\right\|}}	

\definecolor{blue-violet}{rgb}{0.54, 0.17, 0.89}

\definecolor{darkgreen}{rgb}{0.2,0.8,0.1}

\definecolor{princetonorange}{rgb}{1.0, 0.56, 0.0}

\newbox\keywbox
\newcommand\keywords{%
\noindent\rule{\wd\keywbox}{0.25pt}\\\textbf{Keywords:}\ }

\newbox\keywbox
\newcommand\MSC{%
\noindent\rule{\wd\keywbox}{0.25pt}\\\textbf{Math Subject Classification:}\ }

\newcommand{\cmark}{\ding{51}}%
\newcommand{\xmark}{\ding{55}}

\allowdisplaybreaks
\makeatletter
\renewcommand{\thanks}[1]{
  \footnotetext{#1}
}
\makeatother

\title{Accelerated Gradient Methods for Nonconvex Optimization: Escape Trajectories From Strict Saddle Points and Convergence to Local Minima}

\date{}

\author{Rishabh Dixit, Mert G\"urb\"uzbalaban, and Waheed~U.~Bajwa 
\thanks{R.\ Dixit (Department of Electrical and Computer Engineering), M.\ G\"urb\"uzbalaban (Departments of Electrical \& Computer Engineering, Management Science and Information Systems, and Statistics), and W.\ U.\ Bajwa (Departments of Electrical \& Computer Engineering and Statistics) are at Rutgers University--New Brunswick, NJ 08854 USA (Emails: {\tt rd762@scarletmail.rutgers.edu, \{mg1366,~waheed.bajwa\}@rutgers.edu}; M.\ G\"urb\"uzbalaban and W.\ U.\ Bajwa are the corresponding authors).}

\thanks{This work was supported in part by the National Science Foundation under grants CCF-1814888, CCF-1907658, CCF-1910110, and DMS-2053485, by the Army Research Office under grant W911NF-21-1-0301, and by the Office of Naval Research under grants N00014-21-1-2244 and N00014-24-1-2628.}}

\begin{document}

\maketitle

\begin{abstract}
This paper considers the problem of understanding the behavior of a {general} class of accelerated gradient methods on smooth nonconvex functions. Motivated by some recent works that have proposed effective algorithms, based on Polyak's heavy ball method and the Nesterov accelerated gradient method, to achieve convergence to a local minimum of nonconvex functions, this work proposes a broad class of Nesterov-type accelerated methods and puts forth a rigorous study of these methods encompassing {the escape from saddle points and convergence to local minima through both an asymptotic and a non-asymptotic analysis}. In the asymptotic regime, this paper answers an open question of whether Nesterov{'s} accelerated {gradient} method {(NAG)} with variable momentum parameter avoids strict saddle points almost surely. This work also develops two metrics of asymptotic {rates of} convergence and divergence, and evaluates {these two metrics for} {{several popular} standard accelerated methods} such as the NAG and Nesterov's accelerated gradient with constant momentum (NCM) near strict saddle points. In the non-asymptotic regime, this work provides an analysis that leads to the ``linear'' exit time estimates from strict saddle neighborhoods for trajectories of these accelerated methods as well the necessary conditions for the existence of such trajectories. Finally, this work studies a sub-class of accelerated methods that can converge in convex neighborhoods of nonconvex functions with {a near optimal rate to a local minimum} and at the same time this sub-class offers superior saddle-escape behavior compared to {that of NAG}. 
\end{abstract}

\keywords Accelerated gradient methods; Asymptotic analysis; Exit-time estimates; Local convergence guarantees; Nonconvex optimization; Strict-saddle property.

\MSC 90C26; 65K05; 65K10; 37N40; 34D20

\vspace{\baselineskip}\noindent
Communicated by Dima Drusvyatskiy.

\section{Introduction}\label{sec:intro}
Gradient-based first-order methods have been the focal point of theoretical optimization for many decades. The low computational complexity coupled with provable convergence guarantees have made these methods very popular among practitioners and theorists. In particular, the first-order momentum methods, which have roots in the classical Hamiltonian mechanics and over-relaxation methods in linear algebra, have proved to be very efficient in solving convex optimization problems. The seminal work of Polyak \cite{polyak1964some}, which proposed a momentum-based method and the later well-celebrated Nesterov accelerated gradient method \cite{nesterov1983method}, 
showcased the effectiveness of using acceleration/momentum step in order to speed up convergence {if the momentum parameter is appropriately chosen}. {While for strongly convex objectives, momentum parameter can be chosen as a constant, for convex problems a particular time-varying momentum parameter choice for Nesterov's method leads to optimal complexity \cite{nesterov2003introductory}}. Introduction of {these momentum-based methods} 
resulted in faster convergence when compared to the standard gradient descent method for the class of {strongly convex / convex functions}. From there onward many variants of accelerated methods have come into existence; see  \cite{sutskever2013importance,dozat2016incorporating,lucas2018aggregated,zou2018weighted} as just a small list of such methods. While such methods have proven efficacy in tackling convex problems  {in terms of the convergence rate,} less is known about them in the nonconvex regime.
    
Due to the deluge of learning problems in the last decade, much of the focus has been towards developing efficient algorithms (first-order / higher-order)  {with provable second-order convergence guarantees \cite{jin2017accelerated, carmon2018accelerated, fang2019sharp} for} highly nonconvex functions such as those arising in low-rank matrix factorization, phase retrieval, matrix completion, blind deconvolution, dictionary learning, etc. Such problems at the very least are nonconvex (possibly non-smooth) with a very large number of saddle points in their function landscape,\footnote{ {The set of saddle points for some simple non-smooth nonconvex functions can even be uncountable~\cite{davis2021subgradient}.}}  {and it is not desirable to converge to these stationary points.} 

{The saddle points for any twice continuously differentiable, i.e., any $\mathcal{C}^2$ function by definition are those critical points of the function where the Hessian is not a definite matrix, and it can have both positive eigenvalues (imparting contractive dynamics to gradient flow) and negative eigenvalues (imparting expansive dynamics to gradient flow). Though most of the first-order (gradient-based) methods almost surely avoid strict saddle points~\cite{lee2016gradient,lee2019first,o2019behavior},\footnote{ {Strict saddle points are those critical points where the function's Hessian has at least one negative eigenvalue.}} these methods can possibly spend exponentially large amount of time in small neighborhoods of such saddle points if their trajectories do not quickly pick up the expansive dynamics. In that case a trajectory can pass large amount of time in such small neighborhoods, thereby delaying its convergence to a local minimum,} something which is not desirable while developing convergence guarantees. It therefore becomes imperative to understand the local behavior of accelerated methods on such ill-structured geometries of nonconvex functions.

 {Even the question that ``\emph{Does the Nesterov accelerated gradient method, with time-varying momentum, almost surely avoids strict saddle points?}", where ``almost surely" is with respect to the initialization, has not been answered yet, to the best of our knowledge. This motivates us to put forth our first major question: ``\emph{Is there some class of general accelerated gradient methods with time-dependent momentum terms that almost surely avoids strict saddle points? And if such a class exists then what is the convergence/divergence behavior of trajectories generated by the algorithms in this class asymptotically close to the strict saddle points?\emph}". {We will argue that} the first part of the question can be answered by using the standard machinery of the {Stable Center Manifold theorem} \cite{shub2013global}, while the second part can be answered by explicitly computing the eigenvalues of the Jacobian map for the algorithms in the given class. We know that the Jacobian of any algorithmic map of the form $N: \x_k \mapsto \x_{k+1}$, where $\{\x_k\}$ is the iterate sequence for the algorithm, when evaluated at any fixed point of $N$, gives information of the asymptotic rate of convergence/divergence from this fixed point (see Theorem~4.1 of \cite{chicone_ordinary_1999}). Hence evaluating the Jacobian of the algorithmic map in the vicinity of saddle points can shed some light on the asymptotic escape rate from saddle neighborhoods and its dependency on the momentum parameters for this class of algorithms.}

The second major question is:  {``\emph{How much time (number of iterations) does an algorithm, from this class of general accelerated gradient methods, spends in some small neighborhood of the strict saddle points before exiting the neighborhood definitively?}", where this time is referred to as the ``local exit time" or simply the ``exit time". The question of the ``exit time" can be resolved when one has the closed-form expression for the trajectory of the first-order method around the saddle point as a function of its expansive and contractive dynamics \cite{dixit2020exit}. Since such closed-form expressions for trajectories 
are hard to compute, answering the question of local exit time becomes {non-trivial.}
However, one can compute approximate expressions for these trajectories and then estimate their exit time provided the approximation error remains bounded. But even when one somehow develops an estimate of the exit time using the trajectory approximations, our second question, which is the exit time from saddle neighborhoods of the general accelerated methods is still not answered completely due to the non-monotonic nature of these methods in terms of their iterates around stationary points.} {More specifically,} while the simple gradient descent exhibits monotonicity property in radial vectors $\x_k -\x^*$ around a saddle point $\x^*$ of Morse functions \cite{dixit2021boundary}, i.e., the trajectory of gradient descent keeps on monotonically expanding away from the strict saddle points once it starts `escaping,' such conclusion may not necessarily hold for accelerated methods.  {Therefore the concept of `first exit time' loses its value in the case of accelerated methods as the trajectory can possibly re-enter the saddle neighborhood after escape.} A natural approach would be to look for an alternative metric beyond the Euclidean distance such as a weighted Euclidean distance under which the accelerated gradient method remains monotonic after escape. Also a related question to the exit-time problem is whether asymptotic convergence/divergence rate from a strict saddle point can provide an intuitive explanation for the exit time computed with respect to the alternative metric.

 {The third major question asks: ``\emph{Is it possible to increase/decrease the momentum of the Nesterov acceleration \cite{nesterov1983method} so as to escape saddle points faster (with respect to some weighted Euclidean distance) and at the same time recover fast convergence guarantees to a local minimum?}".} It is well known that the Ordinary Differential Equation (ODE) limit of such accelerated methods contains a damping/friction parameter that comes from the momentum term of the accelerated method \cite{su2014differential}. Hence in the continuous time regime, researchers have established both theoretically and empirically \cite{su2014differential} that changing the friction parameter can possibly improve/worsen the behavior of accelerated methods around {the local minima of convex functions}. It therefore begs the question of what will be the impact of tweaking the friction/damping parameter in the ODE on the behavior of the momentum methods around strict saddle neighborhoods of nonconvex functions. In particular, could there be a significant improvement in the escape behavior of the general accelerated method (which is the discretization of the ODE) over the Nesterov accelerated gradient method (NAG) and,  {at the same time, could one achieve convergence rates close to the optimal rate of $\mathcal{O}(1/k^2)$ in a convex neighborhood of any local minimum?}

 {The fourth and final major question in this work asks: ``\emph{For the {proposed} class of general accelerated gradient methods, can one obtain rates of convergence to any {second-order} stationary point of any smooth nonconvex function?}". Note that this question is of significance since the final goal of any algorithm is to converge to a stationary point (preferably second order\footnote{By second-order guarantees, we mean that the algorithm converges to a local minimum {of a function that has no higher-order saddle points}.}) of any smooth function. While fast saddle escape and near-optimal local convergence rates to some minimum are important, the overall convergence rate of any algorithm depends on both local and global analysis. In particular, the global rate analysis of any algorithm depends on two crucial facts: ($i$) There exists some Lyapunov function that decreases monotonically on the iterate sequence generated by the algorithm, and ($ii$) The critical points of the Lyapunov function are the same as the critical points of the original function. Our goal in this regard is to identify a sub-class within the general class of accelerated gradient methods for which there exists a suitable Lyapunov function so as to derive the rates of convergence.}
    
To answer all these questions effectively, we put forth in this work a rigorous analysis of a class of accelerated algorithms that generalizes the Nesterov accelerated gradient method for smooth nonconvex functions. In particular, we provide {asymptotic, local and global analysis} of a class of accelerated methods for nonconvex strict saddle functions, and then develop an acceleration {scheme} that escapes $\epsilon$-neighborhood of any strict saddle point (where $\epsilon$ is sufficiently small) at a linear rate ($\mathcal{O}(\log(\epsilon^{-1}) )$) and, at the same time, does not trade off much in terms of the convergence rate to a local minimum. The proposed class of accelerated methods is then tested for its efficacy on a phase retrieval problem, {a low-rank matrix factorization problem,} and a positive-definite quadratic program.

\subsection{Our contributions}\label{subsec-contributions}
Our first set of contributions deals with the asymptotic analysis of a class of accelerated methods. {We will refer to this class} as the general accelerated gradient methods \eqref{generalds} {obeying the update rules}
\begin{align}
\tag{\textbf{G-AGM}}
   \begin{aligned}
\y_k &= \x_k + \beta_k(\x_k-\x_{k-1}),  \\
\x_{k+1} &= \y_k - h \nabla f(\y_k), \label{generalds}  
\end{aligned}  
\end{align}
{starting from initial points $\x_0,\x_{-1}\in\mathbb{R}^n$} where {the momentum sequence} $\beta_k \to \beta$ for some non-negative $\beta \in \mathbb{R}$, {the stepsize} $h \in (0, \frac{1}{L}]$, and $L$ denotes the gradient Lipschitz constant for the function $f(\cdot)$. Note that this class {is general in the sense that it} {recovers} the Nesterov accelerated gradient method \eqref{originalnesterov} {\cite{nesterov1983method}} for {the choice of} $ \beta_k = \frac{k}{k+3}$, the Nesterov's constant momentum method \eqref{generaldsconst} for $\beta_k = \beta\leq 1 $ {\cite{nesterov2003introductory}}, and the gradient descent method with $\beta_k = 0 $. {Some of our results will also apply to slightly more general methods with updates of the form \begin{align}
\begin{aligned}
    \y_k &= p_k \x_k - q_k \x_{k-1}, \\
    \x_{k+1} &= \y_k - h \nabla f(\y_k), \label{generalds_adv}
\end{aligned}
\end{align}
where $p_k \to p\neq 0$, and $q_k \to q $ for some $p,q\in\mathbb{R}$ as $k \to \infty$ with $p-q=1$.\footnote{{Here, the condition $p-q=1$ ensures that critical points are fixed points, i.e. if $\x_0=\x_{-1}$ is a crticial point with $\nabla f(\x_0)=\mathbf{0}$, then $\x_k=\x_{0}$ for every $k\geq 0$}.}  Clearly, in the special case of $p_k=1+\beta_k$ and $q_k=\beta_k$, the updates \eqref{generalds_adv} will be equivalent to \eqref{generalds}. The algorithmic update \eqref{generalds_adv} offers two degrees of freedom in the form of sequences $\{p_k\}, \{q_k\}$ as opposed to \eqref{generalds} where only one degree of freedom exists in the form of momentum sequence $\{\beta_k\}$.} \looseness=-1

Our first novel result in this regard is Theorem \ref{diffeomorphthm} that allows us to treat the algorithmic maps $N_k : \x_k \mapsto \x_{k+1}$, from the update \eqref{generalds_adv} for any analytic $f(\cdot)$, as almost sure local diffeomorphisms on the Euclidean space $\mathbb{R}^{n}$ for any $k$ (almost sure with respect to choice of step-size $h$). {A key property of our result is that it} allows the forward map $N_k$ to vary with $k$ and therefore can be directly applied to algorithms with varying step size or momentum such as the \eqref{originalnesterov}. Next, in Lemma \ref{lemma_pk} we show that for the algorithmic update \eqref{generalds_adv}, the forward map $P_k$ corresponding to the augmented iterate vector $[\x_k;\x_{k-1}]$, where $ P_k : [\x_k;\x_{k-1}] \mapsto [\x_{k+1};\x_{k}] $, is an almost sure diffeomorphism on $\mathbb{R}^{2n}$ for $\mathcal{C}^2$ Hessian Lipschitz functions. We then present Theorem \ref{measuretheorem1} using the standard theory of dynamical systems \cite{shub2013global}, and use it to develop Theorem \ref{measuretheorem2}, which establishes almost-sure non-convergence to the unstable fixed points (see definition within Theorem \ref{measuretheorem2}) of the map $P_k$ when this map $P_k$ is $k$-invariant. Then using Theorem \ref{measuretheorem2} and tools from Banach space theory, we develop Theorem \ref{measuretheorem3}, which establishes that for the update $\w_{k+1} = P_k(\w_k)$, the sequence $\{\w_k\}$ almost surely does not converge to the unstable fixed points of the map $P$ where the map $P$ is the uniform limit of the sequence of maps $\{P_k\}$. 

The next contribution of this work is Theorem \ref{generalacclimiteigen}, which provides the asymptotic eigenvalues of the Jacobian $\lim_k D P_k([\x_k;\x_{k-1}])$ where the sequence $\{\x_k\}$ is generated from \eqref{generalds} and converges to $[\x^*;\x^*]$ such that $\x^*$ is any critical point of the $\mathcal{C}^2$ function. Next, when $P_k$ is the map corresponding to the general accelerated update \eqref{generalds} and $P_k $ converges uniformly to $P$ on compact sets, Section~\ref{convergingtrajsec} discusses necessary conditions for the existence of trajectories of $\{[\x_k;\x_{k-1}]\}$ which converge to any fixed point $[\x^*;\x^*]$ of the map $P$. Then, using Theorem \ref{measuretheorem3} and the eigenvalues of the map $DP$ derived in Theorem \ref{generalacclimiteigen}, we prove Theorem \ref{measurethm7} which establishes that for a class of accelerated gradient methods \eqref{generalds}, the sequence $\{[\x_k;\x_{k-1}]\}$ almost surely does not converge to $[\x^*;\x^*]$ where $\x^*$ is any strict saddle point of the function $f$ under some mild assumptions on $f$ (almost sure with respect to choice of step-size $h$ and the initialization $[\x_0;\x_{-1}]$). The algorithmic class \eqref{generalds} from Theorem \ref{measurethm7} subsumes the following algorithms namely, the Nesterov accelerated method \eqref{originalnesterov}, general accelerated method \eqref{generalds} with $\beta_k \to \beta$ for $\beta_k \leq 1$, or the Nesterov constant momentum method (\ref{generaldsconst}). To the best of our knowledge, this is the first non-convergence result to strict saddle points for accelerated methods with variable momentum step size. {Prior to our work, although \cite{o2019behavior} establishes an almost sure non-convergence result to strict saddle points for the heavy ball momentum method, their analysis is limited to the constant momentum setting inherent to that method. We refer the reader to Section~\ref{ssec:prior work} for further discussion of the relationship between our work and \cite{o2019behavior}.}

 {In our next set of contributions, Section \ref{metricsection} proposes two metrics that can measure a discrete dynamical system's asymptotic speed of convergence to and divergence from any critical point of a smooth function. Although metrics associated with the asymptotic convergence have been studied for many decades in the form of asymptotic stability of continuous time dynamical systems in control theory literature (see \cite{lyapunov1992general, letov1955stability, hahn1967stability, braun2021stability}), and recently in \cite{lessard2022analysis} that covers accelerated gradient methods, less is known when it comes to measuring the asymptotic speed of divergence.} These metrics from Section \ref{metricsection} are represented as $\mathcal{M}^{\star}(f)$ and $\mathcal{M}_{\star}(f)$, and are associated with the asymptotic rate of divergence from any critical point (cf.~\eqref{asymptot1}) and convergence to any critical point (cf.~\eqref{asymptot2}), respectively. Furthermore, we evaluate these metrics from any critical point $\x^*$ of some smooth function $f(\cdot)$ for trajectories generated by the update \eqref{generalds}. Then Theorem \ref{metricconvergethm2} provides an upper bound on the asymptotic rate of convergence to any strict saddle point or any local minimum $\x^*$, whereas Theorem \ref{metricedivergethm} provides a lower bound on the asymptotic rate of divergence from any strict saddle point $\x^*$ for the general accelerated methods \eqref{generalds}. We afterwards use these bounds to evaluate asymptotic rates of convergence and divergence for {some standard optimization algorithms such as the gradient descent method (GD), \eqref{originalnesterov} and \eqref{generaldsconst}.}  {Again, to the best of our knowledge, this is the first asymptotic convergence and divergence {rate} result for accelerated gradient methods from critical points of nonconvex functions.} 

 {The novelty of the results from Theorems \ref{metricconvergethm2}, \ref{metricedivergethm} lies in the fact that unlike \cite{o2019behavior} and other works, where the analysis relies on the iteration map $[\x_{k}; \x_{k-1}] \mapsto [\x_{k+1}; \x_{k}] $ in the $2n$-dimensional vector space, Theorems \ref{metricconvergethm2}, \ref{metricedivergethm} {develop guarantees in the $n$-dimensional vector space by analyzing the iteration map} $\x_k \mapsto \x_{k+1}$. This dimensional reduction technique is 
 new to our knowledge and helps in better understanding the asymptotic properties of the trajectories of the iterate sequence $\{\x_k\}$ in the ambient $n$-dimensional space as opposed to the properties of the iterate pair $[\x_{k}; \x_{k-1}] $ in the $2n$-dimensional vector space. Since the behavior of trajectories of the coupled iterate pair given by $ \{[\x_{k}; \x_{k-1}]\}$ may not be the same as the behavior of trajectories of $\{\x_k\}$ ({Section~\ref{counterexsecrev} and} Section~\ref{monotnewds1}), it becomes imperative to understand the dynamics of the trajectory of $\{\x_k\}$ close to critical points so as to better understand the behavior of accelerated gradient methods in the ambient $n$-dimensional space. Moreover, one cannot draw parallels between the dynamics of these accelerated methods in $n$-dimensional space and $2n$-dimensional space {(see Section~\ref{counterexsecrev})}. For instance, it is possible that for some $k$ and some critical point $\x^*$, the trajectory of $ \{[\x_{k}; \x_{k-1}]\}$ has expansive dynamics from the pair $[\x^*, \x^*]$ whereas the trajectory of $\{\x_k\}$ has contractive dynamics towards the point $\x^*$ (see Section \ref{monotnewds1}). Hence, Theorems~\ref{metricconvergethm2}, \ref{metricedivergethm} play a major role in developing an asymptotic understanding of the dynamics of the accelerated methods \eqref{generalds} in the $n$-dimensional space.} 

Our next set of contributions is provided in Section \ref{exittimesection}, where we study some  
{key} local properties such as the exit time, which is the number of iterations required to escape some $\epsilon$-neighborhood of any strict saddle point, monotonicity of the iterate distance from the strict saddle point after escape in a weighted Euclidean metric, and more, for a sub-family \eqref{ds1} of the general accelerated methods \eqref{generalds}. This sub-family of accelerated methods \eqref{ds1} satisfies the condition that $\beta_k \to \beta$ at a rate of $\mathcal{O}(1/k)$ and is of sufficient interest since the \eqref{originalnesterov} and \eqref{generaldsconst} are part of the sub-family \eqref{ds1}. {Since our focus in Section \ref{exittimesection} is the study of non-asymptotic properties such as local exit time estimates, we need expressions for the trajectories of \eqref{ds1} locally around strict saddle points, which is not straightforward due to the coupled nature of iterate pair $(\x_k, \x_{k-1})$ for any $k$. Next, observe that if the trajectory of the iterate pair $[\x_k; \x_{k-1}]$ exits some $\epsilon$-neighborhood of $[\x^*;\x^*]$ in $2n$-dimensional vector space for any strict saddle point $\x^*$ of $f$, then it must be that the trajectory of the iterate $\x_k$ also exits some $\mathcal{O}(\epsilon)$ neighborhood of $\x^*$ in $n$-dimensional vector space. Hence the exit time estimates for the $n$-dimensional system can be bounded by those in the $2n$-dimensional system. Therefore, we first transform our dynamics for the iterate pair $[\x_k; \x_{k-1}]$ from a $2n$-dimensional real vector space to a $2n$-dimensional complex vector space by diagonalization, where the dynamics in complex space are linear and hence {easier} to work with.} {In this regard our first contribution is Theorem \ref{exittimethm1}, which provides a linear exit time bound, i.e., $\mathcal{O}(\log(\epsilon^{-1}))$ exit time for a complex linear dynamical system from some $\epsilon$-neighborhood of a fixed point of the dynamics, where the fixed point is weakly hyperbolic\footnote{{Here, with a slight abuse of terminology, we define a fixed point to be weakly hyperbolic if it has a center unstable manifold. The definition of center unstable manifold comes in Section \ref{exittimesection} in Definition \ref{defweakhyp}.}} (for hyperbolic fixed points and its properties see \cite{ott2002chaos}).} Theorem \ref{exittimethm1} also provides the conditions under which this exit time bound holds. 

Next, Section \ref{monotonicsection} establishes that the trajectories of this complex dynamical system in some $\xi$-neighbor\-hood of the weakly hyperbolic fixed point, where $\xi \gg \epsilon$, are indeed monotonic after escape, i.e., once the escape phase starts, the radial distance (distance of the complex iterate from the weakly hyperbolic fixed point) of the trajectory increases continuously with the iterations as long as the trajectory stays within some large $\xi$-radius ball around the weakly hyperbolic fixed point.  {This is a crucial result since it makes sure that any trajectory that exits a sufficiently small $\epsilon$-neighborhood of the weakly hyperbolic fixed point will keep on escaping, and will not return to this neighborhood anytime soon.} 

Thereafter in Section \ref{sectiondynamicformula}, we represent our general accelerated method \eqref{ds1} in the form of a complex dynamical system and show that in a weighted Euclidean metric, the trajectories of \eqref{ds1} can achieve a linear exit time ($\mathcal{O}(\log(\epsilon^{-1}))$) from any sufficiently small $\epsilon$-neighborhood of a strict saddle point $\x^*$.  {Further, the set of trajectories with this linear escape rate is non-trivial and has indeed a positive Lebesgue measure with respect to the initialization set.} More importantly, this linear exit time bound from \eqref{exitimeformal} decreases with increasing the limiting momentum parameter $\beta$ in \eqref{ds1} and, therefore, \emph{larger momentum can result in an even faster exit from the strict saddle neighborhood!} To the best of our knowledge, this is the first work that quantifies the saddle exit time as a function of the limiting momentum parameter $\beta$ and highlights the effect of using larger momentum while escaping strict saddle points.\footnote{{These results are obtained for fixed values of the function parameters, step size $h$, and the limiting momentum parameter $\beta$. In particular, if one were to ask for results at a fixed $\epsilon$, the analysis requires $\beta$ to be chosen below a maximum that remains $\mathcal{O}(1)$ as $\epsilon \to 0$. This, however, is not limiting, since as a practical matter, $\beta$ is never chosen too large in implementations.}}

In Section \ref{convexsectionintro}, {we study an accelerated scheme \eqref{familyof momentum} (derived in \cite{apidopoulos2020convergence}) within the class of general accelerated methods \eqref{generalds} where this scheme of \eqref{familyof momentum} corresponds to the momentum parameter $\beta_k = k/(k+3-r)$ for $r \in [0,3)$. We extend the convergence result from \cite{apidopoulos2020convergence} for convex functions (Theorem \ref{thmconvexrate}) to nonconvex functions and show using Lemma \ref{convexextensionlemma} {that} \eqref{familyof momentum} achieves local convergence in strictly convex neighborhoods of nonconvex functions with a rate close to that of Nesterov accelerated method \eqref{originalnesterov}, i.e., the rate is of order $\mathcal{O}(k^{-(2- \frac{2r}{3})})$ in strictly convex neighborhoods. The relatively slower rate of convergence of \eqref{familyof momentum} in strictly convex neighborhoods for $r>0$ is traded-off with the superior escape behavior of \eqref{familyof momentum} over \eqref{originalnesterov} from strict saddle neighborhoods, as shown using Corollary \ref{commentseccorr}. Moreover, in Section \ref{ODEsection} of Appendix \ref{local minima analysis appendix}, we also provide an intuitive explanation behind the working of this novel momentum scheme from \cite{apidopoulos2020convergence} by analyzing its ODE limit and showing that a lower damping or friction term in the ODE results in higher momentum in the discretized method.}

\begin{table}[H]
{
\centering
\renewcommand\thempfootnote{\arabic{mpfootnote}}
\begin{minipage}{1\textwidth}
\caption{\small Summary of the similarities and differences between this work and some related prior works.}
\label{table:2a}
\hspace{-0.7cm}
\resizebox{1.1\columnwidth}{!}
{
\renewcommand{\arraystretch}{1.3}
\large
\begin{adjustbox}{angle=0}
\begin{tabular}{|| c c c c c||}
 \hline
 \multirow{2}*{\textbf{References}} & 
 \multirow{2}*{\textbf{Base algorithm}} & \textbf{Exit time estimate from}  &
 \multirow{2}*{\textbf{Asymptotic rates}} & \multirow{2}*{\textbf{Convergence guarantees}}
 \\
  & & \multirow{2}*{\textbf{strict saddle neighborhood}}  & \multirow{2}*{\textbf{at strict saddle point}} &  \multirow{2}*{\textbf{to local minimum}}
 \\ [1.5ex]
 \hline
 \hline
 \cite{jin2017accelerated}  & Accelerated gradient method & \xmark & \xmark  & \cmark;
   \\
   & with constant momentum  & &  & probabilistic\\[1.5ex]
 \hline
  \cite{carmon2018accelerated} &    Accelerated gradient method & \xmark & \xmark & \cmark; \\ &  with constant momentum & &  & probabilistic\\[1.5ex]
 \hline
  \cite{zhang2021escape} &   Accelerated gradient method & \xmark & \xmark & \cmark; \\ &  with constant momentum & &  & probabilistic \\[1.5ex]
  \hline
  \cite{barakat2019convergence} &   Polyak's momentum with & \xmark & \xmark & \xmark; \\   & adaptive step size & &  & first order guarantees \\[1.5ex]
  \hline
  \cite{chen2018convergence}  & AdaGrad with
  & \xmark & \xmark & \xmark; \\   & First Order Momentum & &  & first order guarantees \\[1.5ex]
  \hline
  \cite{reddi2019convergence}   & AMSGrad
  & \xmark & \xmark & \xmark; \\   & & &  & first order guarantees \\[1.5ex]
  \hline
  \cite{de2018convergence}   & RMSProp; ADAM
  & \xmark & \xmark & \xmark; \\   & & &  & first order guarantees \\[1.5ex]
  \hline
  \cite{zou2019sufficient}  & RMSProp; ADAM;
  & \xmark & \xmark & \xmark; \\   &  AdaGrad with EMA; AdamNC & &  & first order guarantees \\[1.5ex]
  \hline
\cite{o2019behavior} &  Polyak's heavy ball method; & \cmark; & \xmark  & \xmark \\  & Nesterov accelerated method &  for quadratic functions ($\langle\x, \A\x \rangle$) &  & \\[1.5ex]
 \hline
\textbf{This work}  &  General accelerated methods  & \cmark; & \cmark;  & \cmark; \\  & \eqref{generalds}; \eqref{ds1}; \eqref{familyof momentum}  & for locally Hessian Lipschitz,   & for $\mathcal{C}^{\omega}$ functions  & almost sure guarantees for \\
&  & strict saddle functions  &   & locally Hessian Lipschitz,   \\ 
 & & &  & coercive Morse functions \\[1.5ex]
 \hline
\end{tabular}
\end{adjustbox}
\renewcommand{\arraystretch}{1}
}
\end{minipage}}
\end{table}

Last, but not the least, we derive in Section \ref{globalsection} second-order convergence guarantees for {a sub-family of} the accelerated gradient methods \eqref{generalds} on a class of nonconvex functions (Theorem \ref{supptheoremnew}). In particular, this sub-family for $\beta_k \leq \frac{1}{\sqrt{2}}$ is also shown to have an almost-sure convergence (``almost sure" with respect to initialization and choice of step size) rate of $K=\mathcal{O}(\epsilon^{-2})$ to some $\xi$-neighborhood of a local minimum, where $\xi = \Omega(\epsilon)$ {and} $\epsilon = \Omega(\inf_{0 \leq k \leq K} \norm{\nabla f(\x_k)})$, from Theorem \ref{thmlipschitzrate}. {We should note here that there are constructions of worst-case functions in the literature where \( \epsilon \) can scale exponentially with the problem dimension~\cite{du2017gradient}. Our results in Theorem~\ref{thmlipschitzrate} apply to coercive Morse objective functions, and whether the construction in~\cite{du2017gradient}---or suitable modifications thereof---can be adapted to satisfy the coercive Morse condition remains an interesting and nontrivial open question. Accordingly, in this work, we do not claim an explicit dimension-dependent bound for the quantity \( \inf_{0 \leq k \leq K} \|\nabla f(\x_k)\| \), and we leave this question for future work.}

Next, using Kirszbraun's theorem for extending locally Lipschitz maps to globally Lipschitz maps (Theorem~\ref{kirszbraunthm}), we extend the almost-sure convergence guarantee from Theorem \ref{supptheoremnew} for $\beta_k \leq \frac{1}{\sqrt{2}}$ and rates of convergence to local minimum from Theorem \ref{thmlipschitzrate} to nonconvex functions that are not globally gradient Lipschitz continuous (Theorem \ref{kirszthmadapted}). 
Finally, in Section \ref{numericalsection}, we present numerical simulations for the sub-classes of accelerated methods \eqref{ds1} and \eqref{familyof momentum} on the phase retrieval problem{, the low-rank matrix factorization problem,} and a simple positive-definite quadratic program, in order to showcase the efficacy of the larger momentum parameters.

Table \ref{table:2a} summarises the similarities and differences between this work and some related prior works pertaining to accelerated methods for nonconvex optimization. Since this work does not deal with the stochastic nonconvex optimization problem, those references have been omitted from the table. The first three references \cite{jin2017accelerated, carmon2018accelerated, zhang2021escape} in the table focus on acceleration-based first-order algorithms with provable convergence guarantees to a local minimum. These algorithms are designed to escape saddle points with a high probability. Similarly, the references \cite{barakat2019convergence, chen2018convergence, reddi2019convergence, de2018convergence, zou2019sufficient} derive guarantees for {adaptive gradient} methods {such as} ADAM, AMSGrad, AdaGrad, etc., {which aim to accelerate gradient descent by adjusting the stepsize dynamically based on the history of the gradients}; however, these guarantees are only {for} first order {stationarity} and the saddle escape aspect is not discussed in these works. The only work {we are aware of} that provides exit-time bounds from saddle neighborhoods is \cite{o2019behavior}, but those bounds are for the class of quadratic functions. A more detailed discussion of the relationship of this paper to prior works is provided next.

\subsection{Relation to prior works}\label{ssec:prior work}
{Since this work deals with the asymptotic analysis (in neighborhoods of critical points as neighborhood diameter vanishes) for twice continuously differentiable and analytic functions, and the non-asymptotic analysis that subsumes local analysis (in small neighborhoods of critical points) and global convergence analysis of accelerated methods for twice continuously differentiable functions, we refrain from discussing non-accelerated methods or even higher-order methods in the optimization literature.} Such detailed discussion can be found in our works \cite{dixit2020exit,dixit2021boundary}. 

We start our discussion with measure-theoretic contributions like the work \cite{o2019behavior}, which establishes that Polyak's heavy ball method almost-surely avoids strict saddle points. The work \cite{o2019behavior} also computes exit time of Nesterov acceleration for quadratic functions. In our work, we prove that the class of general accelerated methods \eqref{generalds} with time-varying momentum $\beta_k$, which subsumes the Nesterov acceleration method \eqref{originalnesterov}, almost-surely avoids strict saddle points under some mild assumptions on the function $f$. Though such an extension may appear trivial, yet one cannot directly extend the proof technique from \cite{o2019behavior} to the case of time-varying momentum in order to arrive at the same conclusion.  {The primary reason behind this limitation is that for the almost-sure non-convergence guarantee in \cite{o2019behavior} to hold, the augmented iteration map $P :[\x_k;\x_{k-1}] \mapsto [\x_{k+1};\x_k]$ must be independent of $k$, whereas in the presence of time-varying momentum $\beta_k$, this map $P_k$ is $\beta_k$ dependent and thus evolves with $k$.} Then using tools from Banach space theory \cite{dunford1988linear, megginson2012introduction}, we develop a more general proof technique and hence are able to show that the almost-sure non-convergence of trajectories will hold for a more general class of algorithms \eqref{generalds} in which the algorithmic map is time-varying. Moreover, unlike \cite{o2019behavior} that computes exit times only for quadratic functions, we provide a more general expression of the exit time bound for the class of twice continuously differentiable locally Hessian Lipschitz functions using Theorem \ref{exittimethm1} and the bound~\eqref{exittimetradeoffbound}.
\footnote{{While \cite{o2019behavior} provides non-asymptotic 
exit time bounds in $\mathbb{R}^{2n}$ for 
quadratic functions, these results do not extend directly to general $\mathcal{C}^2$ functions. Even within small $\epsilon$-neighborhoods of strict saddle points of $\mathcal{C}^2$ functions, approximating the dynamics via a quadratic yields a zeroth-order Hessian approximation, $\nabla^2 f(\mathbf{x}_k) \approx \nabla^2 f(\mathbf{x}^*)$. Due to Hessian Lipschitz continuity, this introduces a local error of $\mathcal{O}(\epsilon)$, which can accumulate to $\mathcal{O}(K\epsilon)$ over $K$ steps---potentially large when $K = \mathcal{O}(1/\epsilon)$. Hence, for non-asymptotic analysis in $\mathbb{R}^{2n}$, sharper estimates beyond quadratic approximations are needed, as developed in our work.}}

Our next set of results deals with the analysis of the Jacobian of the algorithmic map asymptotically close to the strict saddle points and local minima. In doing so we propose two novel metrics of asymptotic comparison that capture the asymptotic convergence and divergence rates of trajectories in arbitrary small critical-point neighborhoods, and provide bounds on these metrics for the family of general accelerated methods \eqref{generalds}. The bounds evaluated for these metrics  {help in understanding the relationship between} the superior escape behavior of \eqref{generalds} algorithms from strict saddle neighborhoods and a large asymptotic momentum $\beta$, as evident from Figure \ref{fig10} and the numerical results in Section \ref{numericalsection}. To the best of our knowledge, no other existing work has looked into the asymptotic behavior of trajectories for accelerated methods on nonconvex functions, at least not quantitatively. {Since conducting a local analysis of these methods around strict saddle points can be quite challenging (see Section \ref{exittimesection}), an asymptotic analysis can be done \emph{a priori} to get some useful insights on the local behavior of these algorithms. Moreover, an asymptotic analysis can be easily conducted without using the complex machinery of trajectory approximations (as in Section \ref{exittimesection}), since we are only interested in the limiting behavior.}

Next, there have been recent works in the convex regime such as \cite{gitman2019understanding} that provide an asymptotic understanding of a larger class of quasi-hyperbolic momentum (QHM) methods first proposed in \cite{ma2018quasi}. The authors in \cite{gitman2019understanding} also provide a stability analysis and local convergence rate for convex quadratic functions. Similarly, a recent work \cite{can2022entropic} proposes generalized momentum methods that can be reparameterized to yield the QHM update. Another recent work \cite{lessard2022analysis} defines a metric for asymptotic contraction rate for some standard algorithms on the class of smooth convex functions. Next, in addition to asymptotic convergence analysis, the authors in \cite{can2022entropic} also provide rate of convergence over strongly convex objectives for their generalized momentum scheme using a particular Lyapunov function. But without investigating the asymptotic behavior of accelerated methods in the vicinity of saddle points in nonconvex geometries, one cannot build a concrete understanding of these algorithms over a general class of {optimization} problems, {many} of which are nonconvex and have very large number of saddle points in their function landscape.  {For any algorithm operating on this general nonconvex function class, a fast convergence rate to some local minimum in a convex region can be marred by an exponentially large passage time around saddle points. Having a prior understanding of these algorithms and their trajectories asymptotically close to saddle points can then help in unravelling the relationships between algorithmic parameters and the passage/exit times from very small saddle neighborhoods. Therefore, asymptotic analysis around saddle points for accelerated gradient methods can provide a deeper understanding of these algorithms for many learning problems, which could be leveraged to improve these methods so as to obtain superior convergence rates.} 

The next set of acceleration-based methods includes papers {that provide} convergence guarantees to second or first-order stationary points for nonconvex functions. For convergence guarantees to first-order stationary points, we have a plethora of works that include \cite{ghadimi2016accelerated,de2018convergence,chen2018convergence,zou2019sufficient,reddi2019convergence,barakat2019convergence}.  {Moreover for certain specialized nonconvex geometries like phase retrieval, dictionary learning, etc., there are works that provide the second-order guarantees for accelerated gradient-based methods (see \cite{zhou2016geometrical, vial2021phase, pauwels2017fienup, koppel2018parallel}).} For second-order guarantees in the general nonconvex setting, many interesting works have emerged in the last few years. For instance, the work in \cite{reddi2017generic} provides an extension of the Stochastic Gradient Descent (SGD) method to methods like the Stochastic Variance Reduced Gradient (SVRG) algorithm for escaping saddles. Recently, methods approximating the second-order information of the function landscape while preserving their first-order nature have also been employed to escape saddle points. Examples include \cite{jin2017accelerated}, where the authors prove that an acceleration step is able to utilize the negative curvature information better than the gradient descent step while escaping saddle points, or the work \cite{xu2018first}, where the acceleration step combined with a stochastic perturbation results in a more efficient negative curvature search in a saddle neighborhood. Moreover, both \cite{allen2018natasha,allen2018neon2} build on the idea of utilizing acceleration as a source of finding the negative curvature direction.

In the class of first-order algorithms, there also exist trust region-based methods that utilize momentum. The work in \cite{fang2019sharp} is one such method that presents a novel stopping criterion with a heavy ball-controlled mechanism for escaping saddles using the SGD method. If the SGD iterate escapes some neighborhood in a certain number of iterations, the algorithm is restarted with the next round of SGD, else the ergodic average of the iterate sequence is designated to be a second-order stationary solution. More recently there have been works like \cite{carmon2018accelerated,zhang2021escape} that utilize multi-step noise to facilitate negative curvature search and then revert back to the standard acceleration techniques in order to converge to a second-order stationary point with high probability. In addition, stochastic momentum methods have been very popular while solving stochastic nonconvex optimization problems such as \cite{liu2020improved,wang2021escaping,das2020faster}. 

However, none of the above works have explored a more general class of accelerated methods from the lens of dynamical system around saddle points. Though they provide convergence rates to second-order stationary points, yet they completely skip the question of how much time these methods spend in some open neighborhood of a strict saddle point, and is it possible to escape sufficiently small saddle regions in linear time. It should be noted that proving the existence of a non-zero measure set of trajectories with linear escape rate is important since this fact can assist in linear speedup for the convergence of the \eqref{generalds} family to some second-order stationary point. {This is because the convergence rate to a second-order stationary point implicitly depends on how much time an algorithm spends in small neighborhoods of first-order stationary points.}

\subsection{Notational convention}
All vectors in the paper are in bold lower-case letters, all matrices are in bold upper-case letters, $\mathbf{0}$ is the null vector of appropriate dimension, $\mathbf{I}$ represents the identity matrix of appropriate dimension, $\mathbbm{1}$ is the indicator function and $\langle\cdot, \cdot\rangle $ represents the inner product of two vectors. In addition, unless otherwise stated, all vector norms $\norm{\cdot}$ are $\ell_2$ norms, while the matrix norm $\|\cdot\| $ or equivalently $\|\cdot\|_2$ denotes the operator norm. Also, the symbol $ (\cdot)^T$ is the transpose operator, $ (\cdot)^H$ is the Hermitian operator, $\lambda(\cdot)$ is the general eigenvalue operator unless otherwise stated and the operator ${det}(\cdot)$ returns the determinant of any square matrix. The spaces $\ell^1$ and $\ell^{\infty}$ are the standard scalar sequence spaces on the real field equipped with the absolute sum norm and the sup norm respectively, $x_{\bullet}$ represents an element of the $\ell^1$ or $\ell^{\infty}$ Banach space where $x_{\bullet}= \{x_k\}_{k=0}^{\infty}$ and the symbol `$\rightharpoonup$' implies weak convergence in some Banach space.   

Next, $\mathcal{C}^m$ represents the class of $m$-continuously differentiable functions, $\mathcal{C}^{\omega}$ represents the class of analytic functions, $D$ is the differential operator acting on a smooth function that maps from one differentiable manifold to another, the closure of any set $S$ is $\bar{S}$, $\emptyset$ is the empty set and the complement of any set $S$ is denoted by $S^c$. Additionally, $\mathbb{Z}^*$ is the set of non-negative integers, $\mathrm{id}$ represents the identity map, $\bigoplus$ is the direct sum operator on spaces, $\bigotimes$ is the product operator on measures, and $W(\cdot)$ is the Lambert $W$ function~\cite{corless1996lambertw}. Next, $\textbf{\textit{i}}$ is the square root of $-1$, the operator $dim(\cdot)$ gives the dimension of a vector space, and the operators $\frac{\partial \x_{k+1}}{\partial \x_{k-1} }$, $\frac{\partial \x_{k+1}}{\partial \x_{k} }$ represent Jacobian matrices. Throughout the paper, $k$ and $K$ are used for the discrete time, `$\darrow$' implies uniform convergence where $f_k \darrow f$ on some compact set $X$ implies $\sup_{\x \in X}\norm{f_k(\x)-f(\x)}_2 \to 0$ as $k \to \infty$, the operations `$\dot{\x}$' and `$\ddot{\x}$' represent $\frac{\partial}{\partial t}(\x)$ and $\frac{\partial^2}{\partial t^2}(\x)$, respectively and   {for any pair of smooth maps $T_1, T_2$, the operations $ D T_1 (T_2(\x))$ or $ \nabla^2 T_1 (T_2(\x))$ imply that the operators $D$ or $\nabla^2$ first act on the map $T_1$ and the resultant maps are then evaluated at the point $T_2(\x)$.}

Next, the symbol $\mathcal{O}$ represents the Big-O notation and sometimes we use $a \ll b \iff a = \mathcal{O}(b)$, the symbol $\Omega$ is the Big-Omega notation, $o(\cdot)$ represents the little-o notation, $\probP, \probP_1, \probP_2$ are probability measures, ``$\text{Unif }$" is uniform distribution on a compact set, $\mathcal{B}(\x)$ is an open ball around $\x$ and $a.s.$ {abbreviates}
almost surely. Further, for any matrix expressed as $\Z+\mathcal{O}(c)$ with $c$ being a scalar, the matrix-valued perturbation term $\mathcal{O}(c)$ is with respect to the Frobenius norm. Finally, $\gtrapprox$ and $\lessapprox$ represent the `approximately greater than' and `approximately less than' symbols, respectively, where $a \lessapprox b $ implies $a \leq  b + \tilde{g}(\epsilon) $, $a \gtrapprox b $ implies $a + \tilde{g}(\epsilon)\geq  b  $ for some absolutely continuous function $\tilde{g}(\cdot)$ of $\epsilon$ where $\tilde{g}(\cdot) \geq 0$ and $\tilde{g}(\epsilon) \to 0$ as $\epsilon \to 0$ and the symbols $\bLozenge, \clubsuit, \spadesuit$ are used to mark parts of certain proofs in the appendices.

\section{Problem formulation}
Consider a nonconvex (twice continuously differentiable) smooth function $f(\cdot)$ that has only first-order strict saddle points in its geometry. By first-order strict saddle points, we mean that the Hessian of the function $f(\cdot)$ at these points has at least one negative eigenvalue, i.e., the function has negative curvature. 

{Consider} 
the class of general accelerated gradient methods \eqref{generalds} {we introduced in Section \ref{subsec-contributions}}. 
We are interested in analyzing the asymptotic and local behavior of the trajectories generated by the class of accelerated gradient methods \eqref{generalds} around the strict saddle points of $f(\cdot)$. 
The asymptotic analysis includes investigating properties such as almost sure non-convergence to strict saddle points, as well as the asymptotic rates of convergence to and divergence from these strict saddle points. The local analysis answers the question of exit time $K_{exit}$ of these methods from some $\epsilon$-strict saddle neighborhood and also characterizes the conditions under which this family of accelerated algorithms can achieve linear exit time, i.e., $K_{exit} = \mathcal{O}(\log(\epsilon^{-1}))$, from such $\epsilon$-strict saddle neighborhoods. Formally, for any iterate sequence $\{\x_k\}$ that is initialized in some $\epsilon$-neighborhood of a strict saddle point $\x^*$, i.e., $\norm{\x_0-\x^*} = \epsilon$, the exit time $K_{exit} $ of this sequence from the ball $\mathcal{B}_{\epsilon}(\x^*)$ is {defined as}: 
\begin{align}
    K_{exit} = \inf_{K>0} \bigg\{ K \hspace{0.1cm} \bigg\vert \hspace{0.1cm} \norm{\x_K-\x^*} > \epsilon \bigg\}.
\end{align}
{It is important to note that deriving an upper bound on the exit time does not by itself complete the local escape analysis, since the notion of exit time is only meaningful if the trajectory does not return to the strict saddle neighborhood after escaping it. To address this, Section~\ref{monotonicsection} examines the monotonicity of the radial distances of the escaping iterates, defined as \( \| \x_k - \x^* \| \) for \( k \geq K \), which provides a ``no-return condition'' for trajectories with respect to the strict saddle neighborhood. Specifically, Section~\ref{monotonicsection} explicitly derives an upper bound on the radius \( \xi \gg \epsilon \) of a ball around the strict saddle point such that, within this region, the radial distance of any escaping trajectory increases monotonically. Therefore, for any \( \epsilon > 0 \) with \( \epsilon \ll \xi \), if a trajectory escapes the \( \epsilon \)-radius ball around the strict saddle point, it will continue moving outward until it exits the larger \( \xi \)-radius ball. This yields a precise no-return guarantee for the smaller \( \epsilon \)-radius neighborhood.}

Next suppose that the iterate $\x_k$ generated from \eqref{generalds} after some $K$ iterations lies in some convex or strongly convex neighborhood of any local minimum of the smooth nonconvex $f(\cdot)$, {i.e., the restriction of $f$ to this neighborhood is convex or strongly convex}. Then, {a natural question we will investigate} is {whether} it is possible for a sub-family of \eqref{generalds} to achieve convergence to the local minimum such that the rate of convergence is close to the standard asymptotic rate of $\mathcal{O}(1/k^2)$ from Nesterov accelerated gradient method \eqref{originalnesterov} (close in the sense that a sub-family of \eqref{generalds} achieves rate of $\mathcal{O}(1/k^{2-\delta})$ for small $\delta$) 
{while we} get superior escape behavior for this sub-family of \eqref{generalds} from strict saddle neighborhoods of nonconvex functions. Finally, we {will} show that {\eqref{generalds} with certain parameter choices} can converge to a local minimum almost surely for a class of smooth nonconvex functions and also derive its rate of {local} convergence.

\section{Asymptotic analysis for a class of accelerated methods}
\subsection{Preliminaries for almost sure non-convergence to strict saddle points and asymptotic rates}\label{ssec:prelim.asymptotics}
 {We start with developing a general framework that can be used to prove almost sure non-convergence to strict saddle points for a more general class than \eqref{generalds} and also to derive asymptotic convergence/ divergence rates for the \eqref{generalds} class.
{More specifically, our} 
subsequent results starting from Theorem \ref{diffeomorphthm} to Theorem \ref{measuretheorem3} not only hold for the class of accelerated methods of the form \eqref{generalds} but also hold for the larger class of methods given by the updates \eqref{generalds_adv}.

Before presenting the first result in this work we introduce the measure $\probP_1$, which is a probability measure defined on the reals where $\probP_1$ is absolutely continuous with respect to the Lebesgue measure on $\mathbb{R}$; {our results will hold for any such measure.} Throughout the paper, probability $\probP_1$ is defined with respect to the step size $h$ from the update \eqref{generalds_adv} and $\probP_1$-almost surely implies that the statement holds for almost every choice of $h$ in $\mathbb{R}$. { Following the literature \cite{nesterov2003introductory}}, we define the following class of functions
$$ \mathcal{C}^{2,1}_{L}(\mathbb{R}^n) = \bigg\{ f : \mathbb{R}^n \rightarrow \mathbb{R}; \hspace{0.1cm} f \in \mathcal{C}^2 \hspace{0.1cm} \bigg\vert \hspace{0.1cm} \sup_{\substack{\x, \y \in \mathbb{R}^n \\ {\x \neq \y} }}\frac{\norm{\nabla f(\x) - \nabla f(\y)}}{\norm{\x -\y}} \leq L \bigg\}, $$
which is the class of twice continuously differentiable gradient Lipschitz functions on $\mathbb{R}^n$ whose gradient Lipschitz constant is bounded by $L$. } {We note that for gradient descent dynamics in continuous time, there are constructions of functions $f$ that are infinitely many differentiable where the dynamics stay in a compact set without being convergent \cite{dereich2021convergence}. Therefore, to guarantee convergence to a stationary point, one needs to impose additional assumptions, such as analyticity.
}
  {{Similarly, a subset of our results that relates to finer convergence properties of the iterates will require analyticity of the objective. For this purpose, we introduce the following class of analytic functions:}
$$\mathcal{C}^{\omega}_{L}(\mathbb{R}^n) = \bigg\{ f : \mathbb{R}^n \rightarrow \mathbb{R}; \hspace{0.1cm} f \in \mathcal{C}^{\omega} \hspace{0.1cm} \bigg\vert \hspace{0.1cm} \sup_{\substack{\x, \y \in \mathbb{R}^n \\ {\x \neq \y} }}\frac{\norm{\nabla f(\x) - \nabla f(\y)}}{\norm{\x -\y}} \leq L \bigg\}, $$
which is the class of analytic gradient Lipschitz functions on $\mathbb{R}^n$ whose gradient Lipschitz constant is bounded by $L$. 
} {We also consider the maps $G:\mathbb{R}^n \to \mathbb{R}^n$, $\bar{R}_k: \mathbb{R}^{2n} \to \mathbb{R}^n$ defined as 
\begin{equation}G(\x):=\x - h \nabla f(\x), \quad \bar{R}_k([\x;\y]) := p_k \x - q_k \y,
\label{def-G}
\end{equation}
for $\x,\y\in\mathbb{R}^n$ so that 
$ \x_{k+1} = G \circ \bar{R}_k ([\x_k;\x_{k-1}])$.
With slight abuse of notation, we also define $R_k:\mathbb{R}^n \to \mathbb{R}^n$ 
\begin{equation} R_k(\x) :=  \bar{R}_k ([\x;\x_{k-1}]), 
\label{def-Rk}
\end{equation}
to be the restriction of $\bar{R}_k(\x,\y)$ to $\y=\x_{k-1}$ and consider the composition \begin{equation} N_k \equiv G \circ R_k \quad
\mbox{so that} \quad 
 \x_{k+1} = N_k (\x_k),  
 \label{def-Nk}
 \end{equation}
for $k\geq -1$ where $\{\x_k\}$ are the iterates obeying \eqref{generalds_adv} {and we also use the convention $N_{-1}(\x):= \x + \x_0 - \x_{-1}$ so that $N_{-1}(\x_{-1})= \x_0$}.} 
The first theorem in this work allows us to treat the maps $N_k$  
for any $k$ as local diffeomorphisms on $\mathbb{R}^n$ $\probP_1$-almost surely. 

\begin{theo}\label{diffeomorphthm}
 {Let $\{\x_k\}$ be a sequence generated from \eqref{generalds_adv} under any initialization scheme, {i.e., for any $\x_{-1},\x_{0}\in\mathbb{R}^n$}. 
 {Let $f \in \mathcal{C}^1$ be gradient Lipschitz continuous}. 
 Then for $h \in (0,\frac{1}{L})$, where $L$ is the gradient Lipschitz constant for $f(\cdot)$, and $p_k-q_k=1$ for all $k$ we have that:
\begin{enumerate}
\item  The maps $R_k, N_k$ {defined by \eqref{def-Rk}--\eqref{def-Nk}} are invertible 
and satisfy the relation $N_k(\x_k) = G \circ R_k(\x_k) = G(p_k \x_k - q_k N_{k-1}^{-1}(\x_k))$ for all $k \geq 0$.
    \item For $f \in \mathcal{C}^{\omega}_{L}(\mathbb{R}^n)$, the map $N_k$ for any $\z \in \mathbb{R}^n$ is a local diffeomorphism\footnote{ {A diffeomorphism is a map between manifolds which is continuously differentiable and has a continuously differentiable inverse.}} around $\x = N_{k-1} \circ \dots \circ N_{-1}(\z)$ for all $k \geq 0$ $\probP_1$-almost surely, i.e., {the map $D N_k$ is locally invertible in a neighborhood of $\x$ for almost every choice of the step-size $h$}. Furthermore, {in this case we have}  
    $$  D N_k(\x)   =  (\mathbf{I}- h \nabla^2 f(R_k(\x)) )(p_k\mathbf{I} - q_k[DN_{k-1} (N_{k-1}^{-1}(\x))]^{-1})$$ $\probP_1$-almost surely for any $k\geq 0$ where $\x = N_{k-1} \circ \dots \circ N_{-1}(\z)$ and $\z \in \mathbb{R}^n$.
\end{enumerate}
}
\end{theo}
The proof of this theorem is in Appendix \ref{Appendix A}. {
In the statement of Theorem \ref{diffeomorphthm} 
to understand the interplay between local diffeomorphism of the map $N_k$ and the choice of step-size $h$, we first note that the matrix $ DN_k(\x)$ satisfies the relation $ D N_k(\x)  =  (\mathbf{I}- h \nabla^2 f(R_k(\x)) )(p_k \mathbf{I} - q_k D N_{k-1}^{-1}(\x)) $ for any $k$ whenever the map $N_{k-1}^{-1}$ is differentiable. Since $N_k = G \circ R_k$ where $G \equiv \mathrm{id} - h \nabla f$, we get that the map $N_k$ for any $k$ implicitly depends on $h$. Now for any $\x \in \mathbb{R}^n$, the determinant of the matrix $p_k \mathbf{I} - q_k D N_{k-1}^{-1}(\x) $, which is some nonlinear function of $h$, can only vanish for a Lebesgue null set of $h \in \mathbb{R}$ (details in the proof). Then $DN_k$ is a diffeomorphism $\probP_1$-almost surely means that the matrix $DN_k(\x) $ is invertible for almost every choice of $h \in \mathbb{R}  $. As a direct consequence of Theorem \ref{diffeomorphthm}, we get that the sequence $\{\x_k\}$ generated from the update \eqref{generalds_adv} under any initialization scheme satisfies the following relation for all $k \geq 0$ $\probP_1$-almost surely:
 $$  D N_k(\x_k)   =  (\mathbf{I}- h \nabla^2 f(R_k(\x_k)) )(p_k\mathbf{I} - q_k[DN_{k-1} (\x_{k-1})]^{-1}).$$
 } The following corollary describes an important property of the fixed points of the family of maps $\{N_k\}$.

\begin{coro}\label{fixedptcor}
    If in Theorem \ref{diffeomorphthm} with $p_k-q_k=1$ for all $k$ we also have the initialization scheme $\x_0 = \x_{-1}$, then the critical points of the function $f(\cdot)$ belong to the set of fixed points of the map $N_k$ for any $k$.
\end{coro}
The proof of this corollary is in Appendix \ref{Appendix A}. We now make several remarks in relation to Theorem \ref{diffeomorphthm}.
{
\begin{rema}\label{remimp1} Notice that in the second point in Theorem \ref{diffeomorphthm} 
we require analytic function class as opposed to $\mathcal{C}^1$ class. This is because this result requires showing that for a given $\x$, the non-linear equation $$det\bigg(p_k \mathbf{I} - q_k [D N_{k-1}(N_{k-1}^{-1}(\x))]^{-1}\bigg) = 0$$ (viewing $N_{k-1}$ and $DN_{k-1}$ as a function of $h$) holds only when $h$ belongs to a Lebesgue null set. When $f$ is analytic, we can show that the roots of this equation are isolated and countable hence the argument follows. However, a similar property is not easy to show even for $\mathcal{C}^{\infty}$ functions; indeed there are counterexamples where the zeros of a given nonlinear equation, involving a $\mathcal{C}^{\infty}$ nonlinearity, are non-isolated, disconnected and uncountable \cite{conejero2015smooth}. That being said, while dealing with the critical points of $f$, i.e., in the special case when $\x = \x^*$ is a critical point, the terms in the expression of $N_k$ simplify as they admit $\x^*$ as a fixed point and the analytic function class can be relaxed to $ \mathcal{C}^2$ class. But even with this relaxation, one can only achieve local diffeomorphism of the map $N_k$ around the critical points of $f$ and not local diffeomorphism at any $\x$ (see Corollary \ref{corsup1}). 
We also refer the reader to \cite{dixit2020exit, dixit2021boundary} for a detailed description on some well-known practical nonconvex problems where the objectives are in fact analytic \cite{chen2019gradient, ma2020implicit}. 
\end{rema}
}
We now provide a corollary that extends the $\probP_1$-almost sure diffeomorphism of the family of maps $\{N_k\}$ from Theorem \ref{diffeomorphthm} for analytic functions to a local $\probP_1$-almost sure diffeomorphism for $\mathcal{C}^2$ functions around critical points.
\begin{coro}\label{corsup1}
    For $f \in \mathcal{C}^{2,1}_L(\mathbb{R}^n)$, $h \in (0, \frac{1}{L})$ and $p_k-q_k=1$ for all $k \geq 0$, the maps $N_k$ for any $k$ from Theorem \ref{diffeomorphthm} are local diffeomorphisms around the critical points of $f$ $\probP_1$ almost surely, where $N_k \equiv G \circ (p_k \mathrm{id} - q_k N_{k-1}^{-1})$ for all $k \geq 0$ and $ G \equiv \mathrm{id} - h \nabla f$, provided we have the initialization scheme $\x_0 = \x_{-1}$ in Theorem \ref{diffeomorphthm}. In particular, for any critical point $\x^*$ of $f$, $DN_k(\x^*)$ for any $k \geq 0$ satisfies the following recursion:
    $$  D N_k(\x^*)   =  (\mathbf{I}- h \nabla^2 f(\x^*) )(p_k \mathbf{I} - q_k [D N_{k-1}(\x^*)]^{-1})$$ $\probP_1$-almost surely.
\end{coro}
The proof of this corollary is in Appendix \ref{Appendix A}. Note that although Corollary \ref{corsup1} states that for $\mathcal{C}^2$ functions, the maps $\{N_k\}$ from Theorem \ref{diffeomorphthm} are diffeomorphisms in some neighborhoods of the critical points of $f$ $\probP_1$ almost surely, Corollary \ref{corsup1} does not characterize the size of these neighborhoods. It is even possible that the size of these neighborhoods, on which the maps $\{N_k\}$ are diffeomorphisms, may shrink to $0$ as $k \to \infty$. Hence, even with Corollary \ref{corsup1}, the size of these neighborhoods are not guaranteed to be uniformly bounded from below for all $k \geq 0$ in the case of $\mathcal{C}^2$ functions. 

 {It is needless to state that the algorithmic framework from \eqref{generalds_adv} will cover Nesterov accelerated method, constant momentum method, gradient descent, etc. The novelty of this result lies in the fact that using Theorem \ref{diffeomorphthm} we can now analyze the map $ N_k : \x_k \mapsto \x_{k+1}$ which is in $n$-dimensional vector space instead of working with the map $ [\x_k; \x_{k-1} ]\mapsto [\x_{k+1}; \x_k]$ in $2n$-dimensional vector space, and therefore it is a new approach different from those taken in the existing works like \cite{o2019behavior}. This result will prove to be very powerful while evaluating certain asymptotic properties of a sub-family of the algorithm \eqref{generalds_adv} in Section \ref{metricsection}.} 

We now extend the result in Theorem \ref{diffeomorphthm} from the class of analytic functions to $\mathcal{C}^2$ functions for trajectories of the augmented vector $\{[\x_k;\x_{k-1}]\}$ in $2n$-dimensional space where the sequence $\{\x_k\}$ is generated from algorithm \eqref{generalds_adv}. We also need the following result (Theorem \ref{injectivitythm}) from \cite{xinghua1999convergence} in order to prove that maps associated with the update $[\x_k;\x_{k-1}] \mapsto [\x_{k+1};\x_{k}]$ are $\probP_1$-almost sure global diffeomorphisms on $\mathbb{R}^{2n}$. The following result from \cite{xinghua1999convergence} provides a lower bound on the injectivity radius\footnote{ The injectivity radius at a point $\p$ of a Riemannian manifold is the largest radius for which the exponential map at $\p$ is a diffeomorphism.} of a local diffeomorphism. 

\begin{theo}(Theorem 4.1 in \cite{xinghua1999convergence})\label{injectivitythm}
    Let $f : X \rightarrow Y$ for some Banach spaces $X,Y$. Let $J$ be a positive constant. Assume that for some $\x_0 \in X$, $f$ satisfies the condition
    \begin{align}
        \norm{ [Df(\x_0)]^{-1} Df(\x) - \mathbf{I}} \leq J \norm{\x -\x_0}, \hspace{0.1cm} \forall \hspace{0.1cm} \x \in \mathcal{B}_{1/J}(\x_0). \label{ballinjectivityco}
    \end{align}
Then $f^{-1}\vert_{\x_0}$ exists and is differentiable in the open ball
\begin{align}
    \mathcal{B}_{ 1/(2J \norm{ [Df(\x_0)]^{-1}})}(f(\x_0)) \subset f(\mathcal{B}_{1/J}(\x_0)). \label{ballinjectivity}
\end{align}
Moreover, the radius of this ball (the left one in \eqref{ballinjectivity}) is the best possible.
\end{theo}
The next lemma states certain crucial properties of the maps $P_k : [\x_k;\x_{k-1}] \mapsto [\x_{k+1};\x_{k}] $ corresponding to the recursion $\x_{k+1} = N_k(\x_k)$ defined in Theorem \ref{diffeomorphthm} and is proved in part using Theorem \ref{injectivitythm}. 
\begin{lemm}\label{lemma_pk}
    Suppose $P_k : [\x_k;\x_{k-1}] \mapsto [\x_{k+1};\x_{k}] $ where the sequence $\{\x_k\}$ is generated from the recursion \eqref{generalds_adv} for $h \in \mathbb{R}$ on a function $f \in \mathcal{C}^{2,1}_L(\mathbb{R}^n)$. Then the map $P_k$ from $\mathbb{R}^{2n} \cong \mathbb{R}^{n} \times \mathbb{R}^{n}$ to $\mathbb{R}^{2n} \cong \mathbb{R}^{n} \times \mathbb{R}^{n}$ is a $\mathcal{C}^1$-smooth map which satisfies 
    $$ P_k([\x;\y]) =  \begin{bmatrix}
p_k\x -  q_k\y - h \nabla f\bigg(p_k\x -  q_k\y \bigg) \\  \x
\end{bmatrix},$$
$$ D P_k([\x;\y]) =  \begin{bmatrix}
p_k\bigg(\mathbf{I} - h \nabla^2 f\bigg(p_k\x - q_k\y\bigg)\bigg) \hspace{0.1cm} &  -q_k\bigg(\mathbf{I} - h \nabla^2 f\bigg(p_k\x - q_k\y\bigg)\bigg)\\  \mathbf{I} \hspace{0.2cm} & \mathbf{0}
\end{bmatrix}.  $$ 
Moreover, if $p_k \to p$, $q_k \to q$ as $k \to \infty$ then $P_k \darrow P$, $DP_k \darrow DP$ over compact sets of $\mathbb{R}^{2n}$ where $P$ is $\mathcal{C}^1$-smooth map. Additionally, for $h \in (0, \frac{1}{L})$, $q_k \neq 0$ and $q \neq 0$, the sequence of maps $\{P_k\}$ and the map $P$ are homeomorphisms on compact sets and hence proper\footnote{A function between topological spaces is called proper if inverse images of compact subsets are compact.} maps. If, in addition, $f$ is also Hessian Lipschitz continuous, then the sequence of maps $\{P_k\}$ and the map $P$ are diffeomorphisms on $\mathbb{R}^{2n}$ $\probP_1$-almost surely. 
\end{lemm}
The proof of this lemma is in Appendix \ref{Appendix A}. It is easy to check from Lemma \ref{lemma_pk} that if $p_k-q_k =1$ for any $k$, then the set of fixed points of the map $P_k$ is exactly equal to the set $\{[\x^*;\x^*] \hspace{0.1cm}: \hspace{0.1cm} \nabla f(\x^*) = \mathbf{0}  \}$, provided $h < \frac{1}{L}$. The same conclusion holds for the map $P$. Further, the map $P_k$ from Lemma \ref{lemma_pk} and the map $N_k$ from Theorem \ref{diffeomorphthm} are related as $P_k \equiv [N_k; N_{k-1}]$ for any $k$ in the sense that the map $P_k$ can be viewed as the Cartesian product of the maps $N_k, N_{k-1}$.

\begin{coro}\label{corsup2}
    Suppose in Lemma \ref{lemma_pk}, $f \in \mathcal{C}^{2,1}_L(\mathbb{R}^n)$ is Hessian Lipschitz continuous in every compact set of $\mathbb{R}^n$ and $h \in (0, \frac{1}{L})$, $q_k \neq 0$ with $p_k \to p$, $q_k \to q \neq 0$ as $k \to \infty$. Then for any compact set $\mathcal{D}$ of $\mathbb{R}^n$, the maps $P_k$ for any $k$ and the map $P$ are $\probP_1$-almost sure diffeomorphisms on the compact set $\mathcal{D} \times \mathcal{D}$. Further, if $f $ is coercive, i.e., $\lim_{\norm{\x} \to \infty} f(\x) = \infty$, $q_k > 0$, $p_k \leq 1 + \frac{1}{\sqrt{2}}$ for all $k$, $p_k-q_k =1$ for all $k$ and $\x_0=\x_{-1} \in \mathcal{D}$ in the algorithm \eqref{generalds_adv} for any sublevel set $\mathcal{D}$ of $f$, then $P_k : \mathcal{D} \times \mathcal{D} \rightarrow \mathcal{D} \times \mathcal{D}$ for all $k$ and $P : \mathcal{D} \times \mathcal{D} \rightarrow \mathcal{D} \times \mathcal{D}$.
\end{coro}
The proof of this corollary is in Appendix \ref{Appendix A}. Corollary \ref{corsup2} will help in establishing the almost sure non-convergence guarantees on functions which are not globally Hessian Lipschitz continuous.

We now introduce the Stable Center Manifold theorem developed in \cite{smale1967differentiable,hirsch1977invariant,shub2013global} for the theory of invariant manifolds which will be used to establish almost sure non-convergence guarantees to strict saddle points. It should be noted that this theorem has become a standard tool for establishing non-convergence guarantees for first order methods to strict saddle points (see \cite{lee2016gradient,panageas2016gradient,lee2019first,o2019behavior}). 

\begin{theo}\label{measuretheorem1}
(Adaptation of Theorem III.7 of \cite{shub2013global}). Let $\boldsymbol{0}$ be a fixed point for the $\mathcal{C}
^m$
local
diffeomorphism $\phi : \mathcal{U} \to  \mathcal{E}$ for $m \geq 1$ where $\mathcal{U}$ is a neighborhood of $\boldsymbol{0}$ in the Banach space $\mathcal{E}$.
Suppose that $\mathcal{E} = \mathcal{E}_{CS} \bigoplus \mathcal{E}_{US}$, where $\mathcal{E}_{CS}$ is the invariant subspace corresponding to
the eigenvalues of $D\phi(\boldsymbol{0})$ whose magnitude is less than or equal to $1$, and $\mathcal{E}_{US}$ is the
invariant subspace corresponding to eigenvalues of $D\phi(\boldsymbol{0})$ whose magnitude is greater
than $1$. Then there exists a $\mathcal{C}
^m$
embedded disc $\mathcal{W}^{CS}_{loc}$
 that is tangent to $\mathcal{E}_{CS}$ at $\boldsymbol{0}$ called
the local stable center manifold. Additionally, there exists a neighborhood $\mathcal{B}$ of $\boldsymbol{0}$ such
that\footnote{ {Here $\phi^k $ denotes the composition of $\phi$ map $k$-times.}} $\phi(\mathcal{W}^{CS}_{loc}) \cap \mathcal{B} \subset \mathcal{W}^{CS}_{loc}$, and that if $\z$ is a point such that $\phi^k
(\z) \in \mathcal{B}$ for all $k \geq 0$,
then $\z \in  \mathcal{W}^{CS}_{loc}$.
\end{theo}
Observe that we have a slight modification in Theorem \ref{measuretheorem1} where we consider the magnitude of eigenvalues when compared to Theorem III.7 of \cite{shub2013global} which only deals with the eigenvalues. This modification is not new and has been used before in \cite{o2019behavior} to establish non-convergence of Polyak's heavy ball method to strict saddle points. Moreover this modification generalises the result of the Stable Center Manifold theorem to complex-valued dynamical systems. Using Theorem \ref{measuretheorem1}, we now show that for the $\mathcal{C}^1$ smooth maps $P$ and $ P_k$ for any $k$ defined in Lemma \ref{lemma_pk}, the trajectories generated by the recursion \begin{align}
     [\x_{k+1};\x_k] =  \begin{cases} 
      P_k([\x_k; \x_{k-1}]) & 0 \leq k\leq r, \\
      P([\x_k; \x_{k-1}]) & k > r,
   \end{cases} \label{switchds}
\end{align} for any $r \geq 0$, almost surely do not converge to the unstable fixed points of $P$.   {Note that our main objective is to show that the sequence $[\x_k; \x_{k-1}] $ generated by the recursion $  [\x_{k+1};\x_k] = P_k([\x_k; \x_{k-1}])$ does not converge to $[\x^*;\x^*]$ where $\x^*$ is any strict saddle point of $f$. But we cannot prove this result directly using Theorem \ref{measuretheorem1} since the map in Theorem \ref{measuretheorem1} is $k$-independent. Instead, we first prove the non-convergence result for the dynamical system from \eqref{switchds} where the forward map $[\x_k; \x_{k-1}] \mapsto [\x_{k+1};\x_k] $ eventually becomes $k$ independent. Then using the property that $DP_k \darrow DP$ on compact sets from Lemma \ref{lemma_pk} and tools from Banach space theory, we will be able to prove the desired result. Observe that the recursion \eqref{switchds} will look like the recursion $  [\x_{k+1};\x_k] = P_k([\x_k; \x_{k-1}])$ as $r \to \infty$ and so by analyzing the limiting behavior of recursion \eqref{switchds} for any $r \geq 0$, we will be able to understand the limiting behavior of the recursion $  [\x_{k+1};\x_k] = P_k([\x_k; \x_{k-1}])$.} Our next result can be used for accelerated gradient methods where the map $P_k :[\x_k; \x_{k-1}] \mapsto [\x_{k+1};\x_k] $ eventually becomes constant with $k$ such as the Nesterov constant momentum method \ref{generaldsconst}. 
\begin{theo}\label{measuretheorem2}
Let $P : \mathbb{R}^{2n} \to \mathbb{R}^{2n} $ be a proper, invertible, $\mathcal{C}^1$ map such that $P$ is a diffeomorphism on every compact set of $\mathbb{R}^{2n}$ and $\{P_k\}_{k=0}^{\infty} $ be a sequence of proper, invertible, $\mathcal{C}^1$ maps where 
 $P_k : \mathbb{R}^{2n} \to \mathbb{R}^{2n} $ is a diffeomorphism on every compact set of $\mathbb{R}^{2n} $ for all $k$. Suppose $\{[\x_k; \x_{k-1}]\}$ is any sequence generated by the recursion\footnote{In Theorem \ref{measuretheorem2} although the sequence $\{ [\x_{k+1};\x_k]\}$ generated from the $r$-parameterized recursion will depend on $r$, we purposefully omit $r$ as a superscript on $\x_k$ to avoid notation overload. However, next theorem onward we introduce $r$-parameterized expressions for ease of analysis.}
\begin{align*}
    [\x_{k+1};\x_k] =  \begin{cases} 
      P_k([\x_k; \x_{k-1}]) & 0 \leq k\leq r \\
      P([\x_k; \x_{k-1}]) & k > r 
   \end{cases}
\end{align*} for any $r \geq 0$ and for all $k \geq 0$. Let $[\x^*;\x^*] $ be an unstable fixed point of the map $P$ where $D P([\x^*;\x^*])$ has at least one eigenvalue with magnitude greater than $1$. Then for any $r\geq 0$ and any bounded neighborhood $\mathcal{U} $ of $[\x^*;\x^*]$ such that $[\x_0;\x_{-1}] \in \mathcal{U} \backslash [\x^*;\x^*]$, we have that the  {Lebesgue measure of the set $\{[\x_0; \x_{-1}] \in \mathcal{U} \backslash [\x^*;\x^*] \hspace{0.1cm}\vert \hspace{0.1cm} [\x_k;\x_{k-1}] \to [\x^*;\x^*] \} $ is zero or equivalently} $\probP_2(\{[\x_k;\x_{k-1}] \to [\x^*;\x^*]\}) = 0 $ {where the initialization $[\x_0;\x_{-1}]$ is $\probP_2$-measurable and the probability measure $\probP_2$ is absolutely continuous with respect to the Lebesgue measure on $\mathbb{R}^{2n}$.}
\end{theo}
The proof of Theorem \ref{measuretheorem2} is in Appendix \ref{Appendix A}. Note that Theorem~\ref{measuretheorem2} can be used to determine whether the trajectories generated by the class of general accelerated methods \eqref{generalds} avoid strict saddle points $\probP_2$-almost surely, provided their momentum parameter $\beta_k$ eventually becomes constant with $k$, {as would be the case if Theorem~\ref{measurethm7} were written specifically for the Nesterov constant momentum method \eqref{generaldsconst} (to be formally defined in the next section, Section~\ref{subsec-eig-jacobian}). Most importantly, Theorem~\ref{measuretheorem2} is used to establish the next result, Theorem~\ref{measuretheorem3}. Notice that Theorem~\ref{measuretheorem2} does not apply to general accelerated gradient methods where $\beta_k$ continues to vary with $k$, even as $k$ becomes very large. Therefore, we develop a more general result in the next theorem, extending Theorem~\ref{measuretheorem2} to the class of accelerated methods \eqref{generalds} with time-varying momentum parameters $\beta_k$.}

{We first define $\probP = \probP_1 \bigotimes \probP_2$ as the product probability measure on the product space of the step size $h \in \mathbb{R}$ and the initialization $[\mathbf{x}_0; \mathbf{x}_{-1}] \in \mathbb{R}^{2n}$.} Next, for notational brevity, we define the sequence $\{\mathbf{w}_k\}$ where $\mathbf{w}_k = [\mathbf{x}_k; \mathbf{x}_{k-1}]$ for all $k$, and the sequence $\{\mathbf{x}_k\}$ is generated from the update \eqref{generalds_adv}. From this point onward, we take the initialization $[\mathbf{x}_0; \mathbf{x}_{-1}]$ to be $\probP_2$-measurable throughout the paper.

\begin{theo}\label{measuretheorem3}
 Suppose the function $f(\cdot) \in \mathcal{C}^{2,1}_L(\mathbb{R}^n)$ is Hessian Lipschitz continuous in every compact set, $\inf_{\x} f(\x) > -\infty $, the critical points of $f$ are isolated and the sequence of maps given by $\{P_k\}_{k=0}^{\infty}$ and the map $P$ satisfy the assumptions of Lemma \ref{lemma_pk} with $p_k-q_k =1$, $q_k \neq 0$ for all $k$ and $h < \frac{1}{L}$. Let $ \mathcal{I} = \bigg\{ [\x^*;\x^*]  \hspace{0.1cm} : \hspace{0.1cm} \nabla f(\x^*)=\mathbf{0}  \bigg\}$. Suppose the $r$ parameterized sequence $\{\w^r_k\}$ generated by the recursion \begin{align}
    \w^r_{k+1} =  \begin{cases} 
      P_k(\w^r_k) & 0 \leq k\leq r \\
      P(\w^r_k) & k > r, 
   \end{cases} \label{thmrecurcase1}
\end{align} for any $r \geq 0$, and the sequence $\{\w_k\}$ generated by the recursion 
\begin{align}
    \w_{k+1} = P_k(\w_k) \hspace{0.3cm} \forall \hspace{0.1cm} k \geq 0, \label{thmrecurcase2}
\end{align}
 when initialized in any compact set $\mathcal{U}$ always stay in some compact subset $\mathcal{V} \supsetneq \mathcal{U}$ of $\mathbb{R}^{2n}$.  
 \begin{itemize}
\item[a.]  Let $\w^*=[\x^*;\x^*] \in \mathcal{I} \bigcap \mathcal{V}$ and $\norm{D P(\w^*)}_2 > 1$. Then if $\w_0 \in \mathcal{U} $, $ \mathcal{I}\bigcap \mathcal{U} = \emptyset$ and the sequence $\{\w_k\}$ is generated from \eqref{thmrecurcase2}, we have that $$\probP(\{\w_k \to \w^*\}) = 0. $$
 \item[b.] Suppose \colb{$f$ is Morse,} $ \mathcal{I} \bigcap \mathcal{V}  \neq \emptyset$ and the sequence $\{\w^r_k\}$ generated from \eqref{thmrecurcase1} for any $r \geq 0$ satisfies $ \lim_{k \to \infty}\w^r_k \in \mathcal{I} \bigcap \mathcal{V} $ $\probP_1$-almost surely. Let $ \mathcal{I}_+ =\mathcal{I} \bigcap \bigg\{ \w^* :  \norm{D P(\w^*)}_2 <1 \bigg\}$ and suppose $ \mathcal{I}_+ \neq \emptyset$. Then if $\w_0 \in \mathcal{U} $, $ \mathcal{I}\bigcap \mathcal{U} = \emptyset$ and the sequence $\{\w_k\}$ generated from \eqref{thmrecurcase2} converges $\probP_1$-almost surely, we have that $$\probP(\{\lim_{k \to \infty}\w_k \in \mathcal{I}_+ \}) = 1. $$
 \end{itemize}
\end{theo}
The proof of Theorem \ref{measuretheorem3} is in Appendix \ref{Appendix A} and it uses tools from dynamical systems theory \cite{shub2013global} and Banach space theory \cite{dunford1988linear, megginson2012introduction}. Observe that Theorem \ref{measuretheorem3} requires the iterates to stay bounded within the compact set $\mathcal{V} \supsetneq \mathcal{U}$ but the theorem statement does not characterize the size of this set $ \mathcal{V}$ in terms of $ \mathcal{U}$. Also, the case of $\lim_{k \to \infty}\w^r_k \in \mathcal{I} \cap \bigg\{ \w^* :  \norm{D P(\w^*)}_2 =1 \bigg\}$ for any $r$ and the case of $\lim_{k \to \infty}\w_k \in \mathcal{I} \cap \bigg\{ \w^* :  \norm{D P(\w^*)}_2 = 1 \bigg\} $ is not discussed in the theorem statement since the set $ \mathcal{I} \cap \bigg\{ \w^* :  \norm{D P(\w^*)}_2 =1 \bigg\}$ is $\probP_1$-null.
\begin{rema}
\colb{With a slight abuse of notation, we write $\norm{DP(\w^*)}_2$ for the
spectral radius $\rho(DP(\w^*))$ in the stability conditions of
Theorem~\ref{measuretheorem3} and subsequent results that use it.
This convention also applies to the definitions of $\mathcal{I}_+$,
$\mathcal{I}_-$ and $\mathcal{I}_0$ in its proof.
In Lemmas~\ref{measurelemmasup2_a} and~\ref{measurelemmasup2} and their proofs,
$\norm{\cdot}_2$ denotes the matrix norm induced by a fixed equivalent
vector norm chosen at each $\w^*\in\mathcal{I}_+$ so that
\mbox{$\norm{DP(\w^*)}_2<1$}. The vector norms and balls in these lemmas
are defined using the same norm. For a diagonalizable Jacobian,
the eigenbasis norm gives $\norm{DP(\w^*)}_2=\rho(DP(\w^*))$;
in the defective case, an equivalent norm with
$\norm{DP(\w^*)}_2<1$ can still be chosen.
Here $\w^*$ is fixed by $P$ and every $P_k$, since $p_k-q_k=1$,
and $K_0$ is chosen so that the uniform bound on $DP_k-DP$ in
Lemma~\ref{measurelemmasup2_a} holds for every $k\geq K_0$.
Since all norms on $\mathbb{R}^{2n}$ are equivalent, this change of norm
\mbox{preserves convergence and Lebesgue null sets.}}

\colb{Next, to exclude unit-modulus eigenvalues almost surely in part (b),
we use the Morse assumption, which ensures that the Hessian eigenvalues
at each critical point are nonzero. For each such eigenvalue $\nu$ and
fixed limiting parameters $p,q$, the polynomial
$z^2-p(1-h\nu)z+q(1-h\nu)$ has a root on the unit circle for only
finitely many stepsizes $h$. Since the critical set is countable,
these exceptional stepsizes form a $\probP_1$-null set.}
\end{rema}
{Theorem~\ref{measuretheorem3} extends the result from Theorem~\ref{measuretheorem2} to the class of accelerated methods \eqref{generalds} where $\beta_k$ continues to vary with $k$, one such method being the Nesterov accelerated method \eqref{originalnesterov}, which is formally defined in the next section.} It is important to note here that, unlike \cite{o2019behavior}, where the main result is developed for initialization $[\mathbf{x}_0; \mathbf{x}_{-1}]$ very close to $[\mathbf{x}^*; \mathbf{x}^*]$, with $\mathbf{x}^*$ a strict saddle point, Theorem~\ref{measuretheorem3} holds for any bounded initialization $[\mathbf{x}_0; \mathbf{x}_{-1}]$, provided that $f$ is Hessian Lipschitz continuous on compact sets. This extension to arbitrary bounded initialization sets is only possible because of the almost sure diffeomorphism property of the map $P_k : [\x_k;\x_{k-1}] \mapsto [\x_{k+1}; \x_k]$ on compact sets from Corollary~\ref{corsup2} for locally Hessian Lipschitz functions. This property allows us to pull back Lebesgue null sets, under the map $P_k $ for any $k$, from locally around $[\x^*;\x^*]$ for any critical point $\x^*$ of $f$, to any bounded sets in the $\mathbb{R}^{2n}$ space.

  {In order to apply Theorem \ref{measuretheorem3} on some standard accelerated gradient methods so as to establish their almost sure non-convergence property to any strict saddle point $\x^*$, we need to evaluate eigenvalues of the asymptotic Jacobian map $D P([\x^*;\x^*])$ for these accelerated gradient methods. The next section derives such eigenvalues. }

\subsection{Eigenvalues of the asymptotic Jacobian map for some standard accelerated gradient methods}\label{subsec-eig-jacobian}

Recall that in \eqref{generalds}, the momentum sequence $\{\beta_k\}$ satisfies $\beta_k \to \beta$ where we can have $\beta \in [0, \infty) $ (this allows both small momentum $\beta \leq 1$ and large momentum $\beta>1$ regimes). Since \eqref{generalds} can be represented in the form \eqref{generalds_adv} for $p_k = 1+ \beta_k$ and $q_k=\beta_k$, the expressions for the map $P_k: [\x_{k+1};\x_k] \mapsto [\x_{k};\x_{k-1}]$ and the corresponding Jacobian map $D P_k$ for \eqref{generalds} follow directly from Lemma \ref{lemma_pk}. {We want to prove the almost sure non-convergence of \eqref{generalds} to strict saddle points of $f(\cdot) $. To do so we first need to evaluate the eigenvalues of the Jacobian of the map $P$ at $[\x^*;\x^*]$ where $\x^*$ is any critical point of $f$ and $ P$ is the uniform limit of the sequence of maps $\{P_k\}$. The next theorem provides the eigenvalues of the Jacobian matrix $DP([\x^*;\x^*])$ where $DP$ is the uniform limit of the sequence of Jacobian maps $\{DP_k\}$. 
\begin{theo}\label{generalacclimiteigen}
Let $\{\x_k\}$ be the iterate sequence generated by the general accelerated gradient method \eqref{generalds} with $\beta_k \to \beta$ and $h \in (0, \frac{1}{L}]$. Let $\x^*$ be any critical point of the function $f \in \mathcal{C}^{2,1}_{L}(\mathbb{R}^n)$. Then, for $\x_k \to \x^* $, the eigenvalues of the asymptotic Jacobian matrix $\lim_{k \to \infty}D P_k([\x_k; \x_{k-1}]) = D P([\x^*, \x^*])$, where the maps $\{P_k\}$ are defined in Lemma \ref{lemma_pk}, are given by:
\begin{align}
   \hspace{-0.5cm} \lambda_i(D P([\x^*;\x^*])) = \begin{cases}
   \frac{1}{2}\bigg((1+\beta)\lambda_i(\M) \pm \sqrt{(1+\beta)^2\lambda_i(\M)^2 - 4\beta\lambda_i(\M)}\bigg)  &  ; \hspace{0.1cm}\lambda_i(\M) > \frac{4 \beta}{(1+ \beta)^2} \\
      \frac{1}{2}\bigg((1+\beta)\lambda_i(\M) \pm \textbf{\textit{i}}\sqrt{4\beta\lambda_i(\M)-(1+\beta)^2\lambda_i(\M)^2 }\bigg)  &  ; \hspace{0.1cm}\lambda_i(\M) \in (0, \frac{4 \beta}{(1+ \beta)^2}]
    \end{cases}
\end{align} where $\M = \mathbf{I} - h \nabla^2 f(\x^*) $. If $\x^*$ is a local minimum of the function $f \in \mathcal{C}^{2,1}_{L}(\mathbb{R}^n)$ and $\beta \leq 1$ then we have
\begin{align}
  \hspace{-0.5cm}  \lambda_i(D P([\x^*;\x^*])) = 
      \frac{1}{2}\bigg((1+\beta)\lambda_i(\M) \pm \textbf{\textit{i}}\sqrt{4\beta\lambda_i(\M)-(1+\beta)^2\lambda_i(\M)^2 }\bigg)  &  ; \hspace{0.1cm}\lambda_i(\M) \in (0, \frac{4 \beta}{(1+ \beta)^2}],
\end{align}
with $  \max_i\abs{\lambda_i{(D P([\x^*;\x^*]))}} \leq 1$ and the inequality is strict only when $\nabla^2 f(\x^*)$ has no zero eigenvalues.
\end{theo}
The proof of Theorem \ref{generalacclimiteigen} is in Appendix \ref{Appendix B}. Furthermore, the complex eigenvalues from Theorem \ref{generalacclimiteigen} {have magnitude $ \sqrt{\beta\lambda_i(\M)}$ which is less than $1$ when $\beta < \frac{1}{\lambda_i(\M)} $ and greater than $1$ when $\beta > \frac{1}{\lambda_i(\M)}$.} Hence the complex eigenvalues can impart both contractive and expansive dynamics depending upon the value of momentum parameter $\beta$. Also, these eigenvalues are responsible for spiralling of trajectories in the vicinity of $\x^*$ due to the rotation of trajectories induced by the action of complex eigenvalues. {From Theorem \ref{generalacclimiteigen} it is clear that for any given function $f(\cdot)$ or for a fixed matrix $\M$, having a large asymptotic momentum $\beta$ can possibly reduce the number of complex eigenvalues of $ D P([\x^*;\x^*])$ thereby decreasing the occurrence of spiralling behavior in trajectories very close to $\x^*$. The reasoning behind the lower number of complex eigenvalues is quite simple as one can see from Theorem \ref{generalacclimiteigen} that the complex eigenvalues occur when $\lambda_i(\M) \in (0, \frac{4 \beta}{(1+ \beta)^2}] $ and hence by increasing $\beta$ one decreases $ \frac{4 \beta}{(1+ \beta)^2}$ which is smaller than $1$. {For a fixed step size $h$,} this leaves little room for the eigenvalues of $\M$ smaller than $1$ to lie in the interval $(0, \frac{4 \beta}{(1+ \beta)^2}] $, thereby reducing the number of possible complex eigenvalues of $ DP([\x^*;\x^*])$.}

\begin{rema}
   {Note that it may be the case that the Jacobian matrix $D P([\x^*,\x^*]) $ from Theorem \ref{generalacclimiteigen} is a defective matrix and does not have a complete eigenbasis. Then we can always use the generalized eigenvectors in order to extend the incomplete basis of eigenvectors to a complete basis so that the eigenspace of $D P([\x^*,\x^*])$ from Theorem \ref{generalacclimiteigen} spans $\mathbb{R}^{2n}$. The eigenvalues evaluated will still be the same which can be easily checked by using the Jordan normal form of the matrix $D P([\x^*,\x^*]) $.}
\end{rema}

Using Theorem \ref{generalacclimiteigen} we can now obtain the eigenvalues of the asymptotic Jacobian map $DP$ for some well known methods which can be represented as \eqref{generalds}, such as the Nesterov accelerated gradient method and the Nesterov constant momentum method. Formally, the Nesterov accelerated gradient method for $h \in (0, \frac{1}{L}]$ is given by:
\begin{align}
\tag{\textbf{NAG}}
\begin{aligned}\label{originalnesterov}
\y_k &= \x_k + \frac{k}{k+3}(\x_k-\x_{k-1}) \\
\x_{k+1} &= \y_k - h \nabla f(\y_k).
\end{aligned}
\end{align}
Similarly, the Nesterov constant momentum method \cite{nesterov2003introductory} for $h \in (0, \frac{1}{L}]$ is given by:
\begin{align}
\tag{\textbf{NCM}}
    \begin{aligned}
\y_k &= \x_k + \beta(\x_k-\x_{k-1}) \\
\x_{k+1} &= \y_k - h \nabla f(\y_k), \label{generaldsconst}
\end{aligned}
\end{align}
where $\beta \in (0,1)$. Clearly, the method \eqref{originalnesterov} can be represented as \eqref{generalds} with $ \beta_k = \frac{k}{k+3}$, $\beta_k \to 1$ and also \eqref{generaldsconst} can be represented as \eqref{generalds} with $ \beta_k = \beta$. Then Theorem \ref{generalacclimiteigen} gives the eigenvalues of the asymptotic Jacobian map $DP$ at $[\x^*;\x^*]$ for the updates \eqref{originalnesterov} and \eqref{generaldsconst}. The following corollaries describes these eigenvalues for the two methods. 
\begin{coro}\label{nesterovlimiteigen}
Let $\{\x_k\}$ be the iterate sequence generated by the Nesterov accelerated gradient method \eqref{originalnesterov} with $\beta_k = \frac{k}{k+3}$ in \eqref{generalds}, $\beta_k \to \beta = 1$ and $h \in (0, \frac{1}{L}]$. Let $\x^*$ be a strict saddle point of the function $f \in \mathcal{C}^{2,1}_{L}(\mathbb{R}^n)$. Then, for $\x_k \to \x^* $, the asymptotic Jacobian matrix given by $\lim_{k \to \infty}D P_k([\x_k;\x_{k-1}]) = D P([\x^*;\x^*])$ has both real and complex eigenvalues given by:
\begin{align}
    \lambda_i(D P([\x^*;\x^*])) = \begin{cases}
    \lambda_i(\M) \pm \sqrt{\lambda_i(\M)^2 - \lambda_i(\M)}  & \hspace{0.2cm} ; \hspace{0.2cm}\lambda_i(\M) \in (1,2) \\
     \lambda_i(\M) \pm \textbf{i}\sqrt{\lambda_i(\M) - \lambda_i(\M)^2}  & \hspace{0.2cm} ; \hspace{0.2cm}\lambda_i(\M) \in (0,1] 
    \end{cases}
\end{align}
 where $\M = \mathbf{I} - h \nabla^2 f(\x^*) $.
\end{coro}
Observe that the complex eigenvalues in Corollary \ref{nesterovlimiteigen} have magnitude less than or equal to $1$ and hence impart non-expansive dynamics. Also, these eigenvalues are responsible for spiralling of trajectories in the vicinity of $\x^*$ due to the rotation of trajectories induced by the action of complex eigenvalues. For the case of real eigenvalues, the smaller root satisfies the condition $\lambda_i(D P([\x^*;\x^*])) =\lambda_i(\M) - \sqrt{\lambda_i(\M)^2 - \lambda_i(\M)}  < 1$ and the larger root satisfies $\lambda_i(D P([\x^*;\x^*])) =\lambda_i(\M) + \sqrt{\lambda_i(\M)^2 - \lambda_i(\M)}  > 1$ where $\lambda_i(\M) \in (1,2) $.

\begin{coro}\label{contmomentumeigen}
Let $\{\x_k\}$ be the iterate sequence generated by the constant momentum method \ref{generaldsconst} with $ \beta_k =\beta$ in \eqref{generalds} and $\beta \in (0,1)$, $h \in (0, \frac{1}{L}]$. Let $\x^*$ be a strict saddle point of the function $f \in \mathcal{C}^{2,1}_{L}(\mathbb{R}^n)$. Then, for $\x_k \to \x^* $, the asymptotic Jacobian matrix given by $\lim_{k \to \infty}D P([\x_k;\x_{k-1}]) = D P([\x^*;\x^*])$ has both real and complex eigenvalues given by:
\begin{align}
    \hspace{-0.5cm} \lambda_i(D P([\x^*;\x^*])) = \begin{cases}
   \frac{1}{2}\bigg((1+\beta)\lambda_i(\M) \pm \sqrt{(1+\beta)^2\lambda_i(\M)^2 - 4\beta\lambda_i(\M)}\bigg)  & \hspace{0.1cm} ; \hspace{0.1cm}\lambda_i(\M) > \frac{4 \beta}{(1+ \beta)^2} \\
      \frac{1}{2}\bigg((1+\beta)\lambda_i(\M) \pm \textbf{\textit{i}}\sqrt{4\beta\lambda_i(\M)-(1+\beta)^2\lambda_i(\M)^2 }\bigg)  & \hspace{0.1cm} ; \hspace{0.1cm}\lambda_i(\M) \in (0, \frac{4 \beta}{(1+ \beta)^2}]
    \end{cases}
\end{align}
 where $\M = \mathbf{I} - h \nabla^2 f(\x^*) $. 
\end{coro}
\begin{proof}
Setting $\beta_k = \beta$ in \eqref{generalds} for all $k \geq 0$ we get the constant momentum method \eqref{generaldsconst}. Then from Theorem \ref{generalacclimiteigen} we recover the eigenvalues of $D P([\x^*;\x^*])$. The magnitude of complex eigenvalues is given by:
\begin{align}
    \abs{\lambda_i(D P([\x^*;\x^*]))} =  \sqrt{\beta \lambda_i(\M)} < 1
\end{align}
since $ \beta \in (0,1)$ and $\lambda_i(\M) \leq \frac{4 \beta}{(1+ \beta)^2} \leq 1 $.
\end{proof}    
Note that the complex eigenvalues from Corollary \ref{contmomentumeigen} have magnitude strictly less than $1$ and hence impart contractive dynamics. Also, these eigenvalues are responsible for spiralling of trajectories in the vicinity of $\x^*$ due to the rotation of trajectories induced by the action of complex eigenvalues.

\subsubsection{Note on the existence of converging trajectories}\label{convergingtrajsec}
Recall that in Theorem \ref{generalacclimiteigen} we assumed the case of converging trajectories, i.e., $\x_k \to \x^*$ while evaluating the limiting eigenvalues of the Jacobian map $D P([\x^*;\x^*])$. However it may even be the case that no such trajectory really exists. If that is the case then the ensuing asymptotic analysis will be rendered useless. Hence it becomes imperative to discuss when such analysis works. The next lemma provides one such necessary condition on the existence of converging trajectories to any critical point $\x^*$ of $f$.

\begin{lemm}\label{lemmaasymnec}
 Let $f \in \mathcal{C}^{2,1}_L(\mathbb{R}^n)$ be Hessian Lipschitz continuous on compact sets, $\x^*$ is any critical point of $f$ and the sequence of maps $\{P_k\}$ and the map $P$ for \eqref{generalds} are defined from Lemma \ref{lemma_pk} with $p_k -q_k =1$, $q_k > 0$ for all $k$, $q>0$ and $h < \frac{1}{L}$. Then a necessary condition for the existence of a non-empty initialization set $$\bigg\{[\x_0;\x_{-1}]  \in \mathbb{R}^{2n}\backslash [\x^*;\x^*] \hspace{0.1cm} \bigg\vert \hspace{0.1cm} [\x_{k+1};\x_k] = P_k([\x_k;\x_{k-1}]); \hspace{0.1cm} [\x_k;\x_{k-1}] \to [\x^*;\x^*]\bigg\}$$ is given by:
\begin{align}
   \lim_{\delta \downarrow 0} \inf_{\substack{{[\x_k;\x_{k-1}] \in \mathcal{B}_{\delta}([\x^*;\x^*])}}} \norm{[D P_k([\x_k;\x_{k-1}])]^{-1}}^{-1}_2 \leq 1  \hspace{0.5cm} 
 \probP_1-\text{a.s.}, \label{necessaryconvergencestability}
\end{align}
where $\mathcal{B}_{\delta}([\x^*;\x^*]) $ is an open ball of radius $\delta$ around $[\x^*;\x^*]$. If the Jacobian $D P([\x^*;\x^*]) $ satisfies
\begin{align}
     \min_i\abs{\lambda_i(D P([\x^*;\x^*]) )} < 1, \label{sufficientconvergencestability}
\end{align}
 then the necessary condition \eqref{necessaryconvergencestability} for converging trajectories is automatically satisfied.   
\end{lemm}
 The proof of Lemma \ref{lemmaasymnec} is in Appendix \ref{Appendix C_a}.
 {Observe that in the left hand side of \eqref{necessaryconvergencestability} we cannot directly evaluate $\norm{[D P([\x^*;\x^*])]^{-1}}^{-1}_2$ because then we would be implicitly assuming that $\x_k \to \x^*$, thus violating our own hypothesis according to which we do not know whether $\x_k $ converges to $ \x^*$ or not. It may be the case that $\x_k  $ does not converge to $\x^*$ yet the limit evaluated in \eqref{necessaryconvergencestability} still exists.} The limit in \eqref{necessaryconvergencestability} exists by the facts that for any function $F : \mathbb{R}^n \rightarrow \mathbb{R}$ we have that:
\begin{align}
   \lim_{\delta \downarrow 0} \inf_{\w \in \mathcal{B}_{\delta}([\x^*;\x^*])} F(\w) &= \sup_{\delta>0} \inf_{\w \in \mathcal{B}_{\delta}([\x^*;\x^*])} F(\w) \label{semicont1c}
\end{align}
from the definition of $\liminf$\footnote{ {The right hand sides of the equation \eqref{semicont1c} may not necessarily be equal to $F([\x^*;\x^*])$ since we never assumed that $F$ is a lower semi-continuous function.}} and the constraint set on the left hand side of \eqref{necessaryconvergencestability} given by $$ {[\x_k;\x_{k-1}] \in \mathcal{B}_{\delta}([\x^*;\x^*]); \hspace{0.1cm} \x_0 \in \mathbb{R}^n\backslash \x^*} $$ is non-empty for any $\delta>0$ which holds from Lemma \ref{lemmaSdelta} proved later in Section \ref{metricsection}.

 {Using the machinery developed in Theorems \ref{measuretheorem3} and \ref{generalacclimiteigen} we are now ready to establish the almost sure non-convergence of some standard acceleration methods such as the Nesterov accelerated method, the constant momentum method, etc., to strict saddle points of smooth nonconvex functions.}

\subsection{Almost sure non-convergence to strict saddle points and convergence to local minima}

  {Note that for \eqref{generalds}, establishing the almost sure non-convergence guarantee to strict saddle points using part $a.$ of Theorem \ref{measuretheorem3} is straightforward. But in order to establish almost sure convergence guarantees to local minimum from part $b.$ of Theorem~\ref{measuretheorem3} we require that the iterate sequence generated by our algorithm converges in some compact set. This requirement can be met easily for a sub-class of \eqref{generalds} on coercive functions. We first establish that the iterate sequence generated by a sub-class of \eqref{generalds} converges in some compact set.
 We now state the Global Convergence Theorem from \cite{luenberger1984linear} which is instrumental in establishing the almost sure non-convergence result to strict saddle points for \eqref{generalds}. Its proof is detailed in Section 7.7 of \cite{luenberger1984linear} so we do not present its proof here and directly use this theorem.
\begin{theo}[\textbf{Global Convergence Theorem~\cite{luenberger1984linear}}]\label{thmglob}
Let $A$ be an algorithm on a vector space ${X}$, and suppose that, given $\w_0$ the
sequence $\{\w_k\}_{k=0}^{\infty}$ is generated satisfying $\w_{k+1} \in A(\w_k)$.
Let a solution set ${S} \subset {X}$ be given, and suppose:
\begin{enumerate}
    \item all points $\w_k$ are contained in a compact set $ {D} \subset {X}$,
    \item there is a continuous function $Z$ on ${X}$ such that:
    \begin{itemize}
        \item if $\w \not\in {S}$, then $Z(\y) < Z(\w)$ for all $\y \in A(\w)$,
        \item if $\w \in {S}$, then $Z(\y) \leq Z(\w)$ for all $\y \in A(\w)$,
    \end{itemize}
    \item the map $A$ is closed at points outside ${S}$.
\end{enumerate}
Then the limit of any convergent subsequence of $\{\w_k\}$ is a solution. If under the conditions of the Global Convergence Theorem, $S$ consists of a single point $\bar{\w}$, then the sequence $\{\w_k\}$ converges to $\bar{\w}$.
\end{theo}
For our case, by solution set $S$ of the algorithm $A$, we mean the set of fixed points of the algorithm \eqref{generalds}. We also need a supporting lemma in order to prove the main result in this section (Theorem \ref{measurethm7}). Before presenting the next lemma we define the term `uniform equicontinuity'.}
\begin{defi}
    Let $X$ and $Y$ be two metric spaces, and let $\mathcal{F}$ be a family of functions from $X$ to $Y$. We shall denote by $d_X, d_Y$ the respective metrics of these spaces. Then the family $\mathcal{F}$ is uniformly equicontinuous if for every $\epsilon > 0$, there exists a $\delta>0$ such that $d_Y(f(x_1), f(x_2)) < \epsilon$ for all $ f \in \mathcal{F}$ and all $x_1, x_2 \in X$ such that $d_X(x_1, x_2) < \delta$.
\end{defi}

\begin{lemm}\label{suplem_3}
Suppose that the set of accumulation points of any sequence $\{\w_k\}$ in some compact metric space $X$ generated from the relation $\w_{k+1} = A_k(\w_k)$ are the set of fixed points of the map $A_k$ for any $k$. Then this set of accumulation points is connected, provided the family of maps $\{A_k\}$ is uniformly equicontinuous in $X$ and this set of accumulation points is connected $\probP_1$-almost surely if the family of maps $\{A_k\}$ is uniformly equicontinuous in $X$ $\probP_1$-almost surely. 
\end{lemm}
The proof of this lemma is in Appendix \ref{Appendix C_a}. Using Theorem \ref{thmglob}, Lemma \ref{suplem_3} and Lemma \ref{lemmalyapunov} (proved later in Section \ref{globalsection}) we can establish that the sequence $\{\w^r_k\}$ from Theorem \ref{measuretheorem3} generated by the recursion  \begin{align*}
    \w^r_{k+1} =  \begin{cases} 
      P_k(\w^r_k) &; \hspace{0.1cm} 0 \leq k\leq r \\
      P(\w^r_k) &; \hspace{0.1cm} k > r 
   \end{cases},
\end{align*} for any $r\geq 0$, converges to a critical point of $f$. We are now ready to state the almost sure non-convergence result to strict saddle points for \eqref{generalds}.    

\begin{theo}\label{measurethm7}
Let $f \in \mathcal{C}^{2,1}_{L}(\mathbb{R}^n)$ be any coercive, Morse\footnote{A function $f$ is coercive if $ \lim_{\norm{\x} \to \infty} f(\x) = \infty$. A Morse function is a $\mathcal{C}^2$ function whose Hessian is always invertible at its critical points.} function that is Hessian Lipschitz continuous in every compact set. Then the sequence $\{[\x_k;\x_{k-1}]\}$ from \eqref{generalds} with $\beta_k \to \beta$, $\beta_k \leq \frac{1}{\sqrt{2}}$ for all $k$ and $h \in (0, \frac{1}{L})$, when initialized in any compact set $\mathcal{U}_1' \times \mathcal{U}_1' $ converges to $[\x^*;\x^*]$ $\probP$-almost surely where $\x^*$ is any local minimum of $f(\cdot)$, provided $ \mathcal{U}_1'$ does not contain any critical points of $f$. Next, {suppose that $f \in \mathcal{C}^{2}$ be any coercive, Morse function that is Hessian Lipschitz continuous in every compact set. Also, suppose the sequence $\{\x_k\}$ from \eqref{generalds} for any momentum sequence $\{\beta_k\}$ with $\beta_k \leq 1$ for all $k$ and $h \in (0, \frac{1}{\tilde{L}})$ for some $\tilde{L}>0$, when initialized in any sublevel set $\mathcal{U}_1 $ of $f$, stays within some sublevel set $ \mathcal{V}_1 \supsetneq \mathcal{U}_1$ of $f$ where $f$ is locally ${L}$-gradient Lipschitz continuous in $ \mathcal{V}_1$ and $\tilde{L}>L$. }Then the sequence $\{[\x_k;\x_{k-1}]\}$ from \eqref{generalds} with $\beta_k \to \beta$, $\beta_k \leq 1$ for all $k$ and $h< \frac{1}{\tilde{L}}$, under any compact initialization of $[\x_0;\x_{-1}] \in \mathcal{U}_1' \times \mathcal{U}_1'  \subset \mathcal{U}_1 \times \mathcal{U}_1 $ such that $ \mathcal{U}_1'$ does not contain any critical points of $f$, does not converge to $[\x^*;\x^*]$ $\probP$-almost surely where $\x^*$ is any strict saddle point of $f(\cdot)$. 
\end{theo}

The proof of this theorem is in Appendix \ref{Appendix C_a}. As a consequence of the second part of Theorem \ref{measurethm7} and the eigenvalues from Corollary \ref{nesterovlimiteigen}, Nesterov accelerated gradient method \eqref{originalnesterov} for sufficiently small $h$ does not converge to $[\x^*;\x^*]$ $\probP$-almost surely where $\x^*$ is any strict saddle point of the function $f(\cdot)$. The same conclusion holds for \ref{generaldsconst} with $\beta \leq {1}$ by virtue of Theorem \ref{measurethm7} and the eigenvalues from Corollary \ref{contmomentumeigen}. Note that the bounded iterates condition from Theorem \ref{measurethm7} may seem to be restrictive, but in general, this condition can be met for $\mathcal{C}^2$ coercive functions which satisfy the $(\rho,a,b)$ dissipative property given by $$\langle \x, \nabla f(\x)\rangle \geq a \norm{\x}^{2+\rho} - b,$$ for any $\rho>0$ and large enough $a$, $b$ along with some bounded gradient growth assumption (see Section~\ref{sub-disp} in Appendix~\ref{Appendix C_a} {for an explicit characterization of this assumption, along with examples that satisfy both the dissipative property and this assumption}). A milder version of this dissipative property with $\rho =0$ has been used routinely in literature for ensuring boundedness of iterates (see Assumption A.3 in \cite{raginsky2017non}). Moreover, many practical applications such as binary linear classification, robust ridge regression satisfy this $(\rho,a,b)$ dissipative property for $\rho =0$ (see \cite{gao2022global}). {While the $(\rho,a,b)$ dissipative property for $\rho >0$ will cause $f$ to lose global gradient Lipschitz continuity, we recall that functions that satisfy this dissipative property and are not globally smooth will belong to the $\mathcal{C}^{2}$ coercive function class covered in the second part of Theorem \ref{measurethm7}. Hence, the bounded iterates condition from the second part of Theorem \ref{measurethm7} will hold whenever we have $\mathcal{C}^{2}$ coercive functions which satisfy the $(\rho,a,b)$ dissipative property and a bounded gradient growth assumption without the need of gradient Lipschitz continuity. }

\begin{rema}
{  Note that in \cite{o2019behavior}, the almost sure non-convergence result only holds when the initialization is done from a sufficiently small neighborhood of the strict saddle point. This is because \cite{o2019behavior} only uses local diffeomorphism property of the map $P$ around $[\x^*;\x^*]$. Also, their result does not provide any guarantees for methods where momentum parameter and therefore the map $P_k$ varies with $k$. But since Theorem \ref{measurethm7} proves a stronger result of almost sure non-convergence from any bounded initialization, we need the maps $P_k$ to be diffeomorphisms on any compact set so as to pull back $\probP_1$-null sets from locally around $[\x^*;\x^*]$ to arbitrary sets in $\mathbb{R}^{2n}$. Now for the functions in the class $ \mathcal{C}^{2,1}_{L}(\mathbb{R}^n) $ that are locally Hessian Lipschitz continuous, from Corollary \ref{corsup2}, the map $P_k$ is diffeomorphism on any compact set for almost every real $h$ or equivalently $P_k$ is a diffeomorphism $\probP_1$-almost surely. Therefore, in Theorem \ref{measurethm7}, the almost sure non-convergence guarantee is for any bounded initialization of \eqref{generalds}. }
\end{rema}

So far we have explored certain qualitative properties of the family of accelerated methods \eqref{generalds} such as almost sure non-convergence to strict saddle points and the asymptotic eigenvalues of the Jacobian of the map $P_k : [\x_k;\x_{k-1}] \mapsto [\x_{k+1};\x_k]$, where $\{\x_k\}$ is the sequence generated by the update \eqref{generalds_adv}. However the result from Theorem \ref{measurethm7} does not lend us any insight into the escape/convergence behavior of these accelerated methods from/to arbitrary small strict saddle neighborhoods. More importantly it is not clear whether these accelerated methods have a better escape rate asymptotically close to the strict saddle point when compared to standard algorithms like gradient descent or Nesterov accelerated method and also whether larger momentum amplifies the escape rate. In order to investigate such quantitative properties we conduct a rigorous analysis of the asymptotic convergence and divergence rate for these algorithms in the next section. 

\subsection{Metrics of asymptotic comparison}\label{metricsection}
In this section we are interested in determining the asymptotic escape/convergence rate for the trajectories of \eqref{generalds} around critical points of nonconvex functions. In particular, our goal is to derive these asymptotic rates for the iterate sequence $\{\x_k\}$ in the ambient $\mathbb{R}^n$ space as opposed to the rates for the augmented sequence $\{[\x_k;\x_{k-1}]\}$ in the $\mathbb{R}^{2n}$ space. Note that in {subsequent sections (Section~\ref{counterexsecrev} and Section \ref{monotnewds1}),} we show that the behavior of trajectories of \eqref{generalds} in the $\mathbb{R}^{2n}$ space can be very different from their behavior in the $\mathbb{R}^{n}$ space near strict saddle points. Hence, in order to better understand the behavior of \eqref{generalds} in the $\mathbb{R}^{n}$ space we have to work with the forward map $\x_k \mapsto \x_{k+1}$ in $\mathbb{R}^n$ space instead of the forward map $[\x_k;\x_{k-1}] \mapsto [\x_{k+1};\x_k]$ in $\mathbb{R}^{2n}$ space. Then Lemma \ref{lemma_pk} cannot be invoked and we need to use Theorem \ref{diffeomorphthm} in this section.
For a dynamical system corresponding to \eqref{generalds} on a $\mathcal{C}^{2,1}_L(\mathbb{R}^n)$ function $f(\cdot)$, let the forward map $N_k : \x_k \mapsto \x_{k+1}$ be $ N_k := G\circ R_k$ and the maps $G, R_k$ satisfy $ G := \mathrm{id}  - h \nabla f$, $ R_k := (1+\beta_k)\mathrm{id} - \beta_k N_{k-1}^{-1}$ from Theorem \ref{diffeomorphthm} with $h< \frac{1}{L}$. {Suppose 
 $$ J_{\tau}(f) = \bigcup_{\x_0 \in  \mathbb{R}^n \backslash \x^*}\bigg\{ \{\x_k\}_{k=0}^{\infty} \hspace{0.1cm}\bigg\vert \hspace{0.1cm} \x_{k+1} = N_k(\x_k); \hspace{0.1cm} \tau = \{\beta_k\}_{k=0}^{\infty}\bigg\} ,$$
 is the set of all possible trajectories generated by the dynamical system \eqref{generalds} on the function $f(\cdot)$ for all possible initializations in $\mathbb{R}^n \backslash \x^* $ for a fixed step-size $h$ where $\x^*$ is any critical point of $f$ and $ \{\x_k\}_{k=0}^{\infty}$ be any arbitrary trajectory from the set $J_{\tau}(f)$.} 
 Furthermore suppose $f(\cdot) \in \mathcal{F} $ where $ \mathcal{F} \subset  \mathcal{C}^{2,1}_{L}(\mathbb{R}^n)$ is some function class contained within $ \mathcal{C}^{2,1}_{L}(\mathbb{R}^n)$.
Then the metrics of asymptotic divergence and convergence respectively over the function class $\mathcal{F} $ are defined as\footnote{The supremum/infimum is evaluated by first taking some $f$ in the class $\mathcal{F} $ and then generating the trajectories $\{\x_k\}_{k=0}^{\infty} \in J_{\tau}(f)$ from the relation $ \x_{k+1} = N_k(\x_k)$ for all possible initializations $\x_0$ in the set $\mathbb{R}^n \backslash \x^* $.}:
\begin{align}
\mathcal{{M}}^{\star}(f) & = \lim_{\delta \downarrow 0} \sup_{\substack{\x_k \in \mathcal{B}_{\delta}(\x^*) \\  {\{\x_k\}_{k=0}^{\infty} \in J_{\tau}(f)}}} \norm{\frac{\partial \x_{k+1}}{\partial \x_{k-1} }}_2, \label{asymptot1} \\
{\mathcal{M}_{\star}}(f) & = \lim_{\delta \downarrow 0}  \inf_{\substack{\x_k \in \mathcal{B}_{\delta}(\x^*) \\  {\{\x_k\}_{k=0}^{\infty} \in J_{\tau}(f)}}}\norm{\bigg(\frac{\partial \x_{k+1}}{\partial \x_{k-1} }\bigg)^{-1}}^{-1}_2. \label{asymptot2} 
\end{align}

These metrics provide two step best possible asymptotic rates of divergence from and convergence\footnote{ {Note that the metrics defined in \eqref{asymptot1}, \eqref{asymptot2} can be evaluated at any critical point $\x^*$ of $f$ and not just its strict saddle points.}} to the critical point $\x^*$ of $f$ respectively under every possible initialization of $\x_0$ in any neighborhood of $\x^*$ for the \eqref{generalds}. {To unpack the notation in \eqref{asymptot1}, \eqref{asymptot2} first observe that $ \norm{\frac{\partial \x_{k+1}}{\partial \x_{k-1} }}_2$ is the largest singular value of the two step Jacobian matrix $ \frac{\partial \x_{k+1}}{\partial \x_{k-1} } = D N_k(\x_k) D N_{k-1}(\x_{k-1}) $ and similarly $  \norm{\bigg(\frac{\partial \x_{k+1}}{\partial \x_{k-1} }\bigg)^{-1}}^{-1}_2$ is the smallest singular value of the same matrix $ \frac{\partial \x_{k+1}}{\partial \x_{k-1} }  $. Note that in any sufficiently small neighborhood of some critical point $\x^*$ of $f(\cdot)$ and under some mild assumptions on the sequence of maps $\{N_k\}$, $ \norm{\frac{\partial \x_{k+1}}{\partial \x_{k-1} }}_2$ will determine the rate of divergence of $\x_k$ from $\x^*$ while $  \norm{\bigg(\frac{\partial \x_{k+1}}{\partial \x_{k-1} }\bigg)^{-1}}^{-1}_2$ will determine the rate of convergence to $\x^*$ (see Section \ref{relcontractexpand}, Appendix \ref{Appendix C}). Now the supremum/infimum of these singular values is evaluated over those trajectories of $\x_k$ which come $\delta$ close to the critical point $\x^*$. Recall that $J_{\tau}(f)$ is the set of all possible trajectories generated by the dynamical system \eqref{generalds} where $\x_0 \neq \x^*$ and for evaluating the supremum/ infimum in \eqref{asymptot1}, \eqref{asymptot2} we require only those trajectories from the set $J_{\tau}(f)$ which come $\delta$ close to the critical point $\x^*$. Finally taking $\delta$ to $0$ we get the best possible asymptotic rates of divergence from and convergence to $\x^*$ respectively from \eqref{asymptot1}, \eqref{asymptot2}.} 

 Note that the support set $\bigg\{\x_k \hspace{0.1cm}\bigg\vert\hspace{0.1cm} \x_k \in \mathcal{B}_{\delta}(\x^*); \hspace{0.1cm}\{\x_k\}_{k=0}^{\infty} \in J_{\tau}(f)\bigg\} $ over which the supremum/infimum is evaluated in \eqref{asymptot1}, \eqref{asymptot2} is non-empty for any $\delta >0$ (established later from Lemma \ref{lemmaSdelta} in this section). Then the limits defined in \eqref{asymptot1}, \eqref{asymptot2} exist by the fact that for any function $F : \mathbb{R}^n \rightarrow \mathbb{R}$ we have
\begin{align}
   \lim_{\delta \downarrow 0} \sup_{\x \in \mathcal{B}_{\delta}(\x^*)} F(\x) &= \inf_{\delta>0} \sup_{\x \in \mathcal{B}_{\delta}(\x^*)} F(\x), \label{semicont1a}\\
   \lim_{\delta \downarrow 0} \inf_{\x \in \mathcal{B}_{\delta}(\x^*)} F(\x) &= \sup_{\delta>0} \inf_{\x \in \mathcal{B}_{\delta}(\x^*)} F(\x), \label{semicont1b}
\end{align}
from the definition of $\limsup$ and $\liminf$ respectively\footnote{ {The right hand sides of the two equations \eqref{semicont1a}, \eqref{semicont1b} may not necessarily be equal to $F(\x^*)$ since we never assumed that $F$ is an upper or a lower semi-continuous function.}}. {Since the norm of the two step Jacobian matrices evaluated within \eqref{asymptot1}, \eqref{asymptot2} may not necessarily be upper or lower semi-continuous functions on the support set $\bigg\{\x_k \hspace{0.1cm}\bigg\vert\hspace{0.1cm} \x_k \in \mathcal{B}_{\delta}(\x^*); \hspace{0.1cm}\{\x_k\}_{k=0}^{\infty} \in J_{\tau}(f)\bigg\} $, we cannot conclude that these norm functions will pick up their respective values at $\x^*$ as $\delta \downarrow 0$. Thus, we cannot directly evaluate the metrics in \eqref{asymptot1}, \eqref{asymptot2} at $\x^*$.} 

Next, observe that the metric defined in \eqref{asymptot2} relies on the invertibility of the two step Jacobian matrix $ {\frac{\partial \x_{k+1}}{\partial \x_{k-1} }}$. For the updates of the form \eqref{generalds} and the function class $\mathcal{F} = \mathcal{C}^{\omega}_{L}(\mathbb{R}^n)$, this two step Jacobian matrix is equal to $DN_k(\x_k) DN_{k-1}(\x_{k-1})$ and therefore invertible $\probP_1$-almost surely from the fact that the map $ N_k$ is a local diffeomorphism for all $k$ $\probP_1$-almost surely from Theorem \ref{diffeomorphthm} for $h < \frac{1}{L}$. From our next result (Proposition \ref{prop_p1}) it will become apparent that solving the two step dynamics is much easier as compared to single step dynamics due to the inherent coupling of consecutive Jacobian maps $D N_k, D N_{k-1}$ for the dynamical system \eqref{generalds}. 

 {We now derive the expression for the two step Jacobian that will be used repeatedly in the derivation of asymptotic rates. In particular, the next proposition shows that the two step Jacobian is independent of the inverse term $[D N_{k-1}(\x_{k-1})]^{-1} $ whereas the single step Jacobian isn't. Hence analysis of the two step Jacobian map is relatively easier compared to the single step Jacobian map.}
 {
\begin{prop}\label{prop_p1}
For the dynamical system \eqref{generalds} with $\x_{k+1}= N_k(\x_k)$, where $ N_k := G\circ R_k$ and the maps $G, R_k$ from \eqref{def-Rk}--\eqref{def-Nk} are given by $ G := \mathrm{id} - h \nabla f$ for $f \in \mathcal{C}^{\omega}_{L}(\mathbb{R}^n)$ and $ R_k := (1+\beta_k)\mathrm{id} - \beta_k N_{k-1}^{-1}$, the single step Jacobian map $D N_k$ for any $k \geq 0$ satisfies the following recursion $\probP_1$-almost surely:
\begin{align}
   \frac{\partial \x_{k+1}}{\partial \x_{k}} =  D N_k(\x_{k}) & = \bigg(\mathbf{I} -   h \nabla^2 f( R_k(\x_{k}))\bigg)\bigg((1+ \beta_k)\mathbf{I} -  \beta_k[D N_{k-1}(\x_{k-1})]^{-1}\bigg). \label{prop_p1b}
\end{align}
Since the map $N_k$ for any $k \geq 0$ is a local diffeomorphism $\probP_1$-almost surely from Theorem \ref{diffeomorphthm} we can right multiply\footnote{For $k=0$ and initialization scheme of $\x_0 = \x_{-1}$, the Jacobian map $DN_{-1}$ is taken to be identity.} $ D N_{k-1}(\x_{k-1})$ on both sides of the above recursion to obtain the two step Jacobian map $\probP_1$-almost surely:
\begin{align}
   \frac{\partial \x_{k+1}}{\partial \x_{k-1}} = D N_k(\x_{k})D N_{k-1}(\x_{k-1}) & = \bigg(\mathbf{I} -   h \nabla^2 f( R_k(\x_{k}))\bigg)\bigg((1+ \beta_k)D N_{k-1}(\x_{k-1}) -  \beta_k\mathbf{I}\bigg). \label{prop_p1a}
\end{align}
\end{prop}
\begin{proof}
 Using the map $N_k$ with $\x_{k+1} = N_k(\x_{k})$ we get:
\begin{align}
N_k(\x_k) &= R_k(\x_k) -h \nabla f(R_k(\x_k)) \\
\implies N_k(\x_k) &= (1+ \beta_k) \x_k - \beta_k\x_{k-1} - h\nabla f(R_k(\x_k)) \\
\implies N_k(\x_k) &= (1+ \beta_k) \x_k - \beta_k N_{k-1}^{-1}(\x_{k}) - h \nabla f(R_k(\x_k)) \label{dynamics11a}
\end{align} 
where in the last step, we used the fact that $N_k$ is invertible for all $k$ hence $\x_{k-1} = N_{k-1}^{-1}(\x_{k})$. Next, differentiating \eqref{dynamics11a} with respect to $\x_{k}$ yields the following recursion:
\begin{align}
D N_k(\x_{k}) & =\underbrace{(1+ \beta_k)\mathbf{I} -  \beta_k[D N_{k-1}(\x_{k-1})]^{-1}}_{=DR_k(\x_k)} - h \nabla^2 f(R_k(\x_{k}))D R_k(\x_k) , \label{prop_p1a00}
\end{align}
where we used the relation $\x_{k} = N_{k-1}(\x_{k-1})$ and the local diffeomorphism of map $N_{k-1}$ $\probP_1$-almost surely to get $D N_{k-1}^{-1}(\x_{k})=[D N_{k-1}(\x_{k-1})]^{-1}$ by the inverse function theorem. Finally, \eqref{prop_p1a} follows directly from \eqref{prop_p1a00}.
\end{proof}   
}
We note that \eqref{prop_p1a} is easier to handle as compared to \eqref{prop_p1b} due to the fact that evaluating norm of the two step Jacobian $  \frac{\partial \x_{k+1}}{\partial \x_{k-1}}$ from \eqref{prop_p1a} in the asymptotic metric from \eqref{asymptot1} will not require inverse computations whereas in order to evaluate norm of $  \frac{\partial \x_{k+1}}{\partial \x_{k}}$ from \eqref{prop_p1b} inverse computations are required.  

Having discussed the coupling aspect of the two step Jacobian maps, we now define some function classes over which the asymptotic metrics will be evaluated. Let $ \mathcal{C}^{2,1}_{\mu, L}(\mathbb{R}^n) $ with $L \geq \mu>0 $ represent the class of strict saddle $L$-gradient Lipschitz functions $f : \mathbb{R}^n \to \mathbb{R}$ which satisfy the following property:

If $\lambda_n \leq \lambda_{n-1} \leq \dots \leq \lambda_1$ are eigenvalues of $ \nabla^2 f(\x^*)$ for any strict saddle point $\x^*$ of $f$ then we have:
\begin{align}
    -\mu\leq\lambda_n \leq \lambda_{n-1} \leq \dots \leq \lambda_1\leq L \hspace{0.5cm}; \hspace{0.5cm} {\min_{i} \lambda_i < 0.} \label{spectrumrange1}
\end{align}

Further, let $\M  = \mathbf{I} - h \nabla^2 f(\x^*)$ then we have the following:
\begin{align}
 (1- Lh)\leq     \norm{\M^{-1}}^{-1}_2 < \norm{\M}_2 \leq  1+\mu h, \label{spectrumrange2}
 \end{align}
 where strict equalities $\norm{\M^{-1}}^{-1}_2 =(1- Lh) $ and $\norm{\M}_2 =  1+\mu h $ hold for at least one function in the class $\mathcal{C}^{2,1}_{\mu, L}(\mathbb{R}^n) $. Moreover, the lower bound of $(1- Lh)\leq     \norm{\M^{-1}}^{-1}_2 $ from \eqref{spectrumrange2} will hold even when $\x^*$ is a local minimum of $f\in \mathcal{C}^{2,1}_{\mu, L}(\mathbb{R}^n) $. Clearly one can find quadratic functions in this class which satisfy this property. Also, we require that the function $f$ in this section satisfies the property that the critical points of $f$ are isolated. {This is a commonly made assumption that can often be satisfied in practice; for instance, Morse functions which are dense in the class of $\mathcal{C}^2$ functions \cite{matsumoto2002introduction} satisfy this property (see \cite{mei2018landscape,mokhtari2019efficient,yang2021fast,kurochkin2021neural}).} Using the definition of class $ \mathcal{C}^{2,1}_{\mu, L}(\mathbb{R}^n) $, the corresponding analytic function class is given by $  \mathcal{C}^{\omega}_{\mu, L}(\mathbb{R}^n) = \mathcal{C}^{2,1}_{\mu, L}(\mathbb{R}^n) \cap \mathcal{C}^{\omega}$. Since analytic Morse functions have isolated critical points, the class $ \mathcal{C}^{\omega}_{\mu, L}(\mathbb{R}^n)$ will contain analytic Morse functions that satisfy \eqref{spectrumrange1}.

 Finally, we are ready to derive the bounds on the asymptotic metrics ${\mathcal{M}^{\star}}{}(f) $ from \eqref{asymptot1} and ${\mathcal{M}_{\star}}{}(f) $ from \eqref{asymptot2} over the given class of functions. We first upper bound the convergence metric ${\mathcal{M}_{\star}}{}(f) $ from \eqref{asymptot2} for the function class $ \mathcal{C}^{\omega}_{\mu, L}(\mathbb{R}^n)$. Since we cannot assume that $\x_k$ converges to a critical point $\x^*$ of $f$, we cannot simply take the limit $\lim_{k\to \infty} DN_k(\x_k) $ $\probP_1$-almost surely and evaluate the metric ${\mathcal{M}_{\star}}{}(f) $. Moreover, from Theorem \ref{diffeomorphthm} it is not clear whether the sequence of the Jacobian matrices $\{DN_k(\x_k)\}$ will converge even if $\{\x_k\}$ converges. In particular, we can only evaluate ${\mathcal{M}_{\star}}{}(f) $ at those iterations $k$ for which any sequence $\{\x_k\}$ comes $\delta$ close to $\x^*$ for any given $\delta>0$. The next lemma derives this iteration index set and shows that $\infty$ is always contained in the closure of this iteration index set.    
\begin{lemm}\label{lemmaSdelta}
For any fixed $h<\frac{1}{L}$, let $ S_{\delta} =  \bigg\{k \hspace{0.1cm}\bigg\vert\hspace{0.1cm} \x_k \in \mathcal{B}_{\delta}(\x^*); \hspace{0.1cm}\{\x_k\}_{k=0}^{\infty} \in J_{\tau}(f)\bigg\}$ for any $\delta>0$ be the set of indices for which any trajectory $\{\x_k\}_{k=0}^{\infty}$ from the set $J_{\tau}(f)$ under any possible initialization $\x_0 \in \mathbb{R}^n \backslash \x^*$ stays within the $\delta$ neighborhood of a critical point $\x^*$ of the function $f(\cdot)$ in the class $ \mathcal{C}^{2,1}_{ L}(\mathbb{R}^n)$. Then $\{k \hspace{0.1cm} \vert \hspace{0.1cm} k \in S_{\delta}\} = \mathbb{Z}^{*}$, i.e., $ S_{\delta} = \{k \hspace{0.1cm} \vert \hspace{0.1cm} k \geq 0 \}$.
\end{lemm}
The proof of this lemma is in Appendix \ref{Appendix C}. As a consequence of this lemma the set $ S_{\delta}$ is non-empty for every $\delta>0$. Lemma \ref{lemmaSdelta} will assist us in deriving an upper bound on metric ${\mathcal{M}_{\star}}{}(f) $ for all possible trajectories (not just converging) by allowing us to evaluate $\norm{\bigg(\frac{\partial \x_{k+1}}{\partial \x_{k-1} }\bigg)^{-1}}^{-1}_2 $ for arbitrarily large values of iteration index $k$.

\begin{theo}\label{metricconvergethm2}
For the general accelerated method \eqref{generalds} on $f \in  \mathcal{C}^{\omega}_{\mu, L}(\mathbb{R}^n)$ with $\beta_k \to \beta$ and initialization $\x_0=\x_{-1}$, suppose the sequence of forward maps $\{N_k\}$ defined in Theorem \ref{diffeomorphthm} are $\probP_1$-almost sure local diffeomorphisms in the open ball $\mathcal{B}_{\Delta}(\x^*)$ for some $\Delta>0$ where $\x^*$ can be local minimum or strict saddle point of $f$, the sequence of maps $ \{DN_k^{-1}\}$ are equicontinuous on $\mathcal{B}_{\Delta}(\x^*)$ $\probP_1$-almost surely, and for any $ \y \in \mathcal{B}_{\Delta}(\x^*)$, $ \norm{[DN_k(\x^*)]^{-1}}_2$ satisfies the following growth condition\footnote{{Note that the equicontinuity condition of the inverse Jacobian maps as well as the growth condition are not vacuous and are satisfied for quadratic functions (see Lemma \ref{lemsupab} in Appendix \ref{Appendix C}). However, proving this condition for general non-quadratic functions in the analytic class is not that straightforward and so we refrain from such analysis here.}} in $k$:
$$ \lim_{\norm{\y -\x^*} \to 0 }\sup_{k \geq 0} \norm{[DN_k(\x^*)]^{-1} - [DN_k(\y)]^{-1}}_2 \norm{[DN_k(\x^*)]^{-1}}_2 = 0 \quad \probP_1 \text{ almost surely}.$$ Then the two step asymptotic convergence metric ${\mathcal{M}_{\star}}{}(f)$ defined in \eqref{asymptot2} at $\x^*$ is upper bounded $\probP_1$-almost surely as follows:
\begin{align}
   {\mathcal{M}_{\star}}{}(f)  & \leq   \beta \norm{\M^{-1}}^{-1}_2\bigg( \frac{{\sqrt{\frac{4\beta}{(1+ \beta)^2}\norm{\M^{-1}}_2+1}} -1}{{\sqrt{\frac{4\beta}{(1+ \beta)^2}\norm{\M^{-1}}_2+1}}+1} \bigg)^{-1} ,
\end{align}
where $h \in (0, \frac{1}{L})$ and $\M  = \mathbf{I} - h \nabla^2 f(\x^*)$.
\end{theo}
The proof of this theorem is in Appendix \ref{Appendix C}.  
  {From Theorem \ref{metricconvergethm2} it appears that the upper bound on ${\mathcal{M}_{\star}}{}(f)$ increases by increasing the asymptotic momentum parameter $\beta$. Hence by increasing $\beta$ the asymptotic convergence speed to the strict saddle point $\x^*$ can possibly be worsened for any algorithm of the form \eqref{generalds} in the sense of upper bound on $ {\mathcal{M}_{\star}}{}(f) $. Note that with larger $\beta$, the upper bound on $ {\mathcal{M}_{\star}}{}(f) $ increases thereby making $\x^*$ a weak attractor. Then the trajectories of the dynamical system are pulled much slower towards the strict saddle point $\x^*$.  Also, since the trajectories do not converge to $\x^*$ almost surely, they must eventually have expansive dynamics. Then with larger $\beta$ or equivalently smaller convergence speed, these trajectories have more time to possibly pick expansive dynamics and avoid entering smaller neighborhoods of $\x^*$ from where escape can take much longer. Thus, large $\beta$ can possibly improve saddle escape behavior by converging slowly to a neighborhood of any strict saddle point. This is inline with other results in literature which showed that momentum methods can escape strict saddle points faster than the gradient descent which has a zero momentum term \cite{jin2017accelerated}. However a fast saddle escape can trade-off the convergence rate to a local minimum, i.e., larger $\beta$ will increase the upper bound of ${\mathcal{M}_{\star}}{}(f) $ thereby allowing ${\mathcal{M}_{\star}}{}(f) $ to take values closer to or even greater than $1$. This can result in slower convergence to a local minimum or even worse where no trajectory ends up converging to local minimum if ${\mathcal{M}_{\star}}{}(f) >1$ for arbitrarily large $\beta$. }

We now provide a lower bound on the asymptotic divergence metric ${\mathcal{M}^{\star}}{}(f)$. 
\begin{theo}\label{metricedivergethm}
For the general accelerated method \eqref{generalds} on $f \in \mathcal{C}^{\omega}_{\mu, L}(\mathbb{R}^n)$ with $\beta_k \to \beta$ and initialization $\x_0=\x_{-1}$, suppose that $\beta_k$ is non-decreasing, the sequence of forward maps $\{N_k\}$ defined in Theorem~\ref{diffeomorphthm} are $\probP_1$-almost sure local diffeomorphisms in the open ball $\mathcal{B}_{\Delta}(\x^*)$ for some $\Delta>0$ where $\x^*$ is a strict saddle point of $f$, the sequence of maps $ \{DN_k\}$ are equicontinuous on $\mathcal{B}_{\Delta}(\x^*)$ $\probP_1$-almost surely and for any $ \y \in \mathcal{B}_{\Delta}(\x^*)$, $ \norm{DN_k(\x^*)}_2$ satisfies the following growth condition in $k$:
$$ \lim_{\norm{\y -\x^*} \to 0 }\sup_{k \geq 0} \norm{DN_k(\x^*) - DN_k(\y)}_2 \norm{DN_k(\x^*)}_2 = 0 \quad \probP_1 \text{ almost surely}.$$ Then the two step asymptotic divergence metric ${\mathcal{M}^{\star}}{}(f)$ defined in \eqref{asymptot2} is lower bounded $\probP_1$-almost surely as follows:
\begin{align}
   {\mathcal{M}^{\star}}{}(f)  & \geq  \norm{\M}_2\bigg((1+ \beta)\norm{\M}_2-  \beta \bigg) ,
\end{align}
where $h \in (0, \frac{1}{L})$ and $\M  = \mathbf{I} - h \nabla^2 f(\x^*)$.
\end{theo}
The proof of this theorem is in Appendix \ref{Appendix C}.

\begin{rema}
It should be noted that while evaluating $\mathcal{{M}}^{\star}{}(f)$ we purposefully did not analyze the case of converging trajectories $\x_k \to \x^*$ since by definition, the metric $\mathcal{{M}}^{\star}{}(f)$ corresponds to asymptotic rate of divergence. Hence obtaining the asymptotic divergence rate for converging trajectories is meaningless. Also, as noted before, the uniform local equicontinuity and the growth condition of the maps $\{DN_k^{-1}\}$, $\{DN_k\}$ in Theorems \ref{metricconvergethm2}, \ref{metricedivergethm} respectively are not vacuous conditions and they hold for quadratic functions\footnote{The maps $\{DN_k^{-1}\}$, $\{DN_k\}$ for a quadratic function are constant everywhere $\probP_1$ almost surely (Lemma \ref{lemsupab}).}. For a general analytic function, we require $ \sup_{k \geq 0}\norm{DN_k(\x^*)}_2$, $ \sup_{k \geq 0}\norm{[DN_k(\x^*)]^{-1}}_2$ to be bounded from growth conditions in Theorems \ref{metricconvergethm2}, \ref{metricedivergethm}. Also, the uniform local equicontinuity and the growth conditions are not sufficient to establish subsequential convergence of these maps. For instance, from the growth condition in Theorem \ref{metricedivergethm} we cannot conclude if the sequence of maps $\{DN_k\}$ are uniformly bounded and in fact from the growth condition, for quadratic functions $ \sup_{k \geq 0}\norm{DN_k(\x^*)}_2$ can also be unbounded. Hence even with the uniform local equicontinuity of these maps, the existence of any uniformly convergent subsequence for these maps cannot be proved (see Arzela-Ascoli theorem \cite{dunford1988linear}).
\end{rema}

If for some $f$ we have $\norm{\M}_2 = 1+\mu h $, $ \norm{\M^{-1}}^{-1}_2 = 1-Lh$, then from Theorems \ref{metricconvergethm2} and \ref{metricedivergethm} we get:
\begin{align}
  \mathcal{{M}}_{\star}{}(f)  &\leq \beta (1-Lh)\bigg( \frac{{\sqrt{\frac{4\beta}{(1+ \beta)^2(1-Lh)}+1}} -1}{{\sqrt{\frac{4\beta}{(1+ \beta)^2(1-Lh)}+1}}+1} \bigg)^{-1},\\
 \mathcal{{M}}^{\star}{}(f) & \geq ( 1+\mu h)\bigg((1+ \beta)( 1+\mu h)-  \beta \bigg) ,\label{finalmetric}
\end{align}
which hold $\probP_1$-almost surely by Theorem \ref{diffeomorphthm} since we crucially used the fact that every $N_k$ is a local diffeomorphism for any $k$ while deriving Theorems \ref{metricconvergethm2} and \ref{metricedivergethm}.

 {
\begin{rema}
Note that we only have an upper bound for the metric $  \mathcal{{M}}_{\star}(f)$ and similarly only a lower bound for the metric $  \mathcal{{M}}^{\star}(f)$. Now recall that from \eqref{prop_p1a} of Proposition \ref{prop_p1} we have $ \frac{\partial \x_{k+1}}{\partial \x_{k-1}}  = \bigg(\mathbf{I} -   h \nabla^2 f( R_k(\x_{k}))\bigg)\bigg((1+ \beta_k)D N_{k-1}(\x_{k-1}) -  \beta_k\mathbf{I}\bigg) $ where $\mathbf{I} -   h \nabla^2 f( R_k(\x_{k})) $ is positive definite for any $f \in  \mathcal{C}^{\omega}_{\mu, L}(\mathbb{R}^n) $ and $N_k$ is a local diffeomorphism $\probP_1$-almost surely for all $k \geq 0$. However it is possible that for some $(\x_k,k)$ pair, any one of the eigenvalues of $ D N_{k-1}(\x_{k-1})$ could be arbitrary close\footnote{ {By arbitrary close we mean $0< \lvert\lambda_i( D N_{k-1}(\x_{k-1})) -  \frac{\beta_k}{(1+ \beta_k)} \rvert \ll 1 $. The eigenvalue cannot be equal to $ \frac{\beta_k}{(1+ \beta_k)}$ $\probP_1$-almost surely since then $ D N_k$ fails to be a diffeomorphism at $\x_k$ from \eqref{prop_p1b} in Proposition \ref{prop_p1}.}} to $ \frac{\beta_k}{(1+ \beta_k)}$. Then the smallest singular value of $\frac{\partial \x_{k+1}}{\partial \x_{k-1}} $ becomes arbitrary small which makes it impossible to lower bound $ \mathcal{{M}}_{\star} (f)$ by a fixed positive constant. Using similar argument in the other direction by taking inverse on both sides of \eqref{prop_p1a} it can be shown that one cannot upper bound $ \mathcal{{M}}^{\star}(f) $ by a finite positive constant.    
\end{rema}
}

The next lemma evaluates the bounds on the asymptotic metrics $  \mathcal{{M}}_{\star}(f),  \mathcal{{M}}^{\star}(f)$ for the gradient descent method, \eqref{generaldsconst} and \eqref{originalnesterov} at strict saddle points.
\begin{lemm}\label{lemmaasympevalm}
    Let $\x^*$ be any strict saddle point of a function $f \in \mathcal{C}^{\omega}_{\mu,L}(\mathbb{R}^n)$ where $\norm{\M}_2 = 1+\mu h $, $ \norm{\M^{-1}}^{-1}_2 = 1-Lh$ and $\M  = \mathbf{I} - h \nabla^2 f(\x^*)$. Also, suppose that the maps $\{DN_k^{-1}\}$, $\{DN_k\}$ for \eqref{generalds} satisfy the assumption of uniform local equicontinuity from Theorems \ref{metricconvergethm2}, \ref{metricedivergethm} and the respective growth conditions. Then the following hold $\probP_1$-almost surely:
    \begin{itemize}
        \item[(i)] For the gradient descent method with constant step-size of $h < \frac{1}{L}$ we have that
        \begin{align*}
             \mathcal{{M}}_{\star}{}(f) =  (1- L h)^2  < 1 < (1 + \mu h)^2 \leq \mathcal{{M}}^{\star}{}(f). 
        \end{align*}
        \item[(ii)] For the Nesterov constant momentum \eqref{generaldsconst} with momentum parameter $\beta = \frac{1-\sqrt{Lh}}{1+ \sqrt{Lh}}$ and constant step-size of $h < \frac{1}{L}$ we have that
        \begin{align*}
             \mathcal{{M}}_{\star}{}(f) & \leq  (1-\sqrt{Lh})^2\frac{\sqrt{2} +1}{\sqrt{2}-1} , \\
             \mathcal{{M}}^{\star}{}(f) & \geq (1 + \mu h)^2 + \frac{1 - \sqrt{Lh}}{1 + \sqrt{Lh}} \bigg((1 + \mu h)^2 -(1 + \mu h)\bigg).
        \end{align*}
         \item[(iii)] For the Nesterov accelerated gradient method \eqref{originalnesterov} with momentum parameter sequence $\beta_k = \frac{k}{k+3}$ and constant step-size of $h < \frac{1}{L}$ we have that
       \begin{align*}
  \mathcal{{M}}_{\star}{}(f) &\leq  (1-Lh) \bigg( \frac{{\sqrt{(1-Lh)^{-1}+1}} -1}{{\sqrt{(1-Lh)^{-1}+1}}+1} \bigg)^{-1},\\  
  \mathcal{{M}}^{\star}{}(f) &\geq 2(1 + \mu h)^2 - (1+ \mu h).
\end{align*}
    \end{itemize}
\end{lemm}
The proof of this lemma is in Appendix \ref{Appendix C}. Note that the results obtained in this section are relevant to our discussion due to the fact that they provide insights into the asymptotic behavior of the trajectories of the iterate sequence $\{\x_k\}$ from \eqref{generalds} in the ambient $\mathbb{R}^{n}$ space as opposed to the behavior of augmented sequence $\{[\x_k; \x_{k-1}]\}$ in the $\mathbb{R}^{2n}$ space. To the best of our knowledge, this is the first work that is able to analyze the asymptotic escape/convergence behavior of accelerated methods near strict saddle points without transforming the dynamics from $\mathbb{R}^{n}$ space to the $\mathbb{R}^{2n}$ space. However, this comes at the cost of regularity of the function class where we need analytic functions as opposed to $\mathcal{C}^2$ functions in this section. {Before concluding the asymptotic analysis of \eqref{generalds}, we remind the reader that analyzing the dynamics of the trajectory $\{\x_k\}_{k=0}^{\infty}$ in the ambient space $\mathbb{R}^{n}$ is not a straightforward extension of analyzing the dynamics of the trajectory $\{[\x_{k};\x_{k-1}]\}_{k=0}^{\infty}$ in $\mathbb{R}^{2n}$. Moreover, one can construct simple counterexamples in which even the transient and asymptotic behavior of the dynamics in $\mathbb{R}^{n}$ and $\mathbb{R}^{2n}$ differs significantly. The next section presents one such example.}

\begin{rema}
{Note that our asymptotic analysis pertains to the dynamics in $\mathbb{R}^n$, as opposed to \cite{o2019behavior}, where analysis of trajectories is performed in $\mathbb{R}^{2n}$ for a quadratic function. Even if we restrict attention to quadratic functions, trajectory analysis in $\mathbb{R}^n$ is not a trivial extension of analysis in $\mathbb{R}^{2n}$. In $\mathbb{R}^n$, the forward maps $N_k : \mathbb{R}^n \to \mathbb{R}^n$ may not be diffeomorphisms, and when these maps are local diffeomorphisms (Theorem~\ref{diffeomorphthm}), their corresponding Jacobian maps $DN_k$ satisfy the following recursion from Theorem~\ref{diffeomorphthm}:
\[
DN_k(\mathbf{x}) = (\mathbf{I} - h \nabla^2 f(R_k(\mathbf{x}))) \left(p_k \mathbf{I} - q_k \left[DN_{k-1}(N_{k-1}^{-1}(\mathbf{x}))\right]^{-1}\right).
\]
These Jacobian maps are time-varying in general, and it does not seem possible to compute a closed-form expression of the Jacobian at any given $k$ from the recursion, even for the simplest quadratic case with constant momentum (one such example is discussed in the next section, Section~\ref{counterexsecrev}). Therefore, extending the techniques in \cite{o2019behavior} to analyze the asymptotic dynamics in $\mathbb{R}^n$ is not feasible.}
\end{rema}

\begin{rema}
{It is worth pointing out that throughout Sections~\ref{ssec:prelim.asymptotics} to~\ref{metricsection}, most results (see, e.g., Theorems~\ref{diffeomorphthm}, \ref{measuretheorem3}, \ref{measurethm7}, \ref{metricconvergethm2}, and others) require the step size $h$ to lie strictly below $\frac{1}{L}$, whereas in a few cases (see, e.g., Theorem~\ref{generalacclimiteigen}), the results allow for $h$ to include the endpoint, so that $h \in (0, \frac{1}{L}]$. The restriction $h < \frac{1}{L}$ arises from the need for the Jacobian maps $N_k$, $P_k$, and/or $P$ to be diffeomorphisms, a property that holds only when $h$ is strictly less than $\frac{1}{L}$. Consider the Jacobian map $P$, for instance. From Theorem~\ref{generalacclimiteigen}, we observe that the $i$-th eigenvalue pair (in either the real or complex case) of $DP[\mathbf{x}^*; \mathbf{x}^*]$ is zero whenever the $i$-th eigenvalue of $\mathbf{M} = \mathbf{I} - h \nabla^2 f(\mathbf{x}^*)$ is zero. Conversely, it follows from simple algebra that the $i$-th eigenvalue of $\mathbf{M}$ is zero whenever the corresponding eigenvalue pair of $DP[\mathbf{x}^*; \mathbf{x}^*]$ is zero. Hence, $DP[\mathbf{x}^*; \mathbf{x}^*]$ is invertible if and only if $\mathbf{I} - h \nabla^2 f(\mathbf{x}^*)$ is invertible, which is equivalent to $h < \frac{1}{L}$. Similar conclusions apply to the Jacobian maps $N_k$ and $P_k$. In contrast, when the results are concerned only with computing the eigenvalues of these Jacobian maps, such as in Theorem~\ref{generalacclimiteigen}, rather than establishing their invertibility, the endpoint $h = \frac{1}{L}$ can be included in the analysis.}
\end{rema}

\subsection{{On the non-equivalence of  
the dynamics in $\mathbb{R}^{n}$ and $\mathbb{R}^{2n}$
}}\label{counterexsecrev}

{Let $f(\x)$ be a quadratic function of the form $f(\x) = \frac{1}{2}\langle \x - \x^*, \A (\x - \x^*) \rangle$ for a diagonal matrix $\A \in \mathbb{R}^{n \times n}$ with both positive and negative diagonal entries, and thus having $\x^*$ as a strict saddle point. Note that the Hessian of $f$ is given by $\nabla^2 f(\x) = \A = \nabla^2 f(\x^*)$ for all $\x \in \mathbb{R}^n$. Let $\beta_k = \beta$ for all $k \geq 0$, where $\beta \leq 1$ in \eqref{generalds}, and let the initialization scheme be $\x_0 = \x_{-1}$. Then the Jacobian map $DP_k$ defined in Lemma \ref{lemma_pk} satisfies
\begin{align}
D P_k([\x_{k};\x_{k-1}]) &=  \begin{bmatrix}
(1+ \beta)\bigg(\mathbf{I} - h \nabla^2 f\bigg((1+ \beta)\x_{k} - \beta \x_{k-1}\bigg)\bigg) \hspace{0.1cm} &  - \beta \bigg(\mathbf{I} - h \nabla^2 f\bigg((1+ \beta)\x_{k} - \beta\x_{k-1}\bigg)\bigg)\\  \mathbf{I} \hspace{0.2cm} & \mathbf{0}  
\end{bmatrix}  \\
& =  \begin{bmatrix}
(1+ \beta)\bigg(\mathbf{I} - h \A \bigg) \hspace{0.1cm} &  - \beta \bigg(\mathbf{I} - h \A\bigg)\\  \mathbf{I} \hspace{0.2cm} & \mathbf{0}
\end{bmatrix},   \label{counterexjac1}
\end{align}
where we have used $\nabla^2 f(\x) = \A$ for all $\x$. This implies that the Jacobian in $\mathbb{R}^{2n}$ is constant for all $k$. Therefore, in this case, it trivially follows that the constant Jacobian in \eqref{counterexjac1} is equal to the limiting Jacobian $DP([\x^*;\x^*])$. Next, we evaluate the Jacobian along the trajectory $\{\x_k\}_{k=0}^{\infty}$ in the $\mathbb{R}^{n}$ space. Using the recursion for $DN_k(\x_k)$ from Theorem \ref{diffeomorphthm} for a quadratic function $f$, we obtain for any $k \geq 0$:
\begin{align}
    D N_k(\x_k)   &=  \bigg(\mathbf{I}- h \nabla^2 f(\x^* ) \bigg)\bigg(  (1+ \beta)\mathbf{I} - \beta [DN_{k-1} (\x_{k-1})]^{-1} \bigg)  \hspace{0.2cm} \probP_1 \text{ a.s.} \label{counterexjac2}
\end{align}
with $DN_{-1}(\x_{-1}) = \mathbf{I}$ since $\x_0 = \x_{-1}$, and $DN_k(\x_k)$ is invertible for all $k \geq 0$ $\probP_1$ a.s. From the recurrence in \eqref{counterexjac2}, it is clear that $DN_k(\x_k)$ varies with $k$ and is therefore not constant, unlike the Jacobian $DP_k([\x_k; \x_{k-1}])$ which is constant in $\mathbb{R}^{2n}$ as shown in \eqref{counterexjac1}. Therefore, the non-transient behaviors of the Jacobians in $\mathbb{R}^{2n}$ and $\mathbb{R}^{n}$ are not equivalent.}

{We now additionally show that, unlike in $\mathbb{R}^{2n}$ where the Jacobian converges to $DP([\x^*;\x^*])$, the sequence $\{DN_k(\x_k)\}_{k=0}^{\infty}$ fails to converge. To do so, we proceed by contradiction. Suppose that $\lim_{k \to \infty} DN_k(\x_k) = \X$ $ \hspace{0.2cm} \probP_1$ a.s. Rearranging \eqref{counterexjac2} and taking the limit yields:
\begin{align}
    D N_k(\x_k) DN_{k-1} (\x_{k-1}) & = \bigg(\mathbf{I}- h \nabla^2 f(\x^* ) \bigg)\bigg(  (1+ \beta) DN_{k-1} (\x_{k-1}) - \beta \mathbf{I}\bigg)   \hspace{0.2cm} \probP_1 \text{ a.s.} \\
    \implies \lim_{k \to \infty}  D N_k(\x_k) DN_{k-1} (\x_{k-1}) & = \bigg(\mathbf{I}- h \nabla^2 f(\x^* ) \bigg)\bigg(  (1+ \beta)\lim_{k \to \infty} DN_{k-1} (\x_{k-1}) - \beta \mathbf{I}\bigg)   \hspace{0.2cm} \probP_1 \text{ a.s.} \\
    \implies \X^2 &= \bigg(\mathbf{I}- h \nabla^2 f(\x^* ) \bigg)\bigg(  (1+ \beta)\X - \beta \mathbf{I}\bigg)   \hspace{0.2cm} \probP_1 \text{ a.s.}  \label{counterexjac3}
\end{align}
Since $\bigg(\mathbf{I}- h \nabla^2 f(\x^* ) \bigg)$ is diagonal and $DN_{-1}(\x_{-1}) = \mathbf{I}$, we get from \eqref{counterexjac2} that $DN_{0}(\x_{0}) = \bigg(\mathbf{I}- h \nabla^2 f(\x^* ) \bigg)$ is diagonal. Then, if $DN_{k-1}(\x_{k-1})$ is diagonal for some $k > 1$, it follows from \eqref{counterexjac2} that $DN_k(\x_k)$ is also diagonal. By induction, we conclude that $DN_k(\x_k)$ is diagonal for all $k \geq 0$. Since the assumption is that $\lim_{k \to \infty} DN_k(\x_k) = \X$, the limiting matrix $\X$ therefore must also be diagonal, as the limit of a convergent sequence of diagonal matrices is necessarily diagonal. It then follows from \eqref{counterexjac3} that the diagonal entries $x_j$ of $\X$ for $1 \leq j \leq n$ satisfy the quadratic equation:
\begin{align}
    x^2_j &= (  (1+ \beta)x_j - \beta ) \lambda_j   \hspace{0.2cm} \probP_1 \text{ a.s.},
\end{align}
where $\lambda_j$ is the $j$-th eigenvalue of $\mathbf{I} - h \nabla^2 f(\x^*)$. The explicit solutions to this quadratic equation are:
\begin{align}
     x_j = \begin{cases}
   \frac{1}{2}\Big((1+\beta)\lambda_j \pm \sqrt{(1+\beta)^2\lambda_j^2 - 4\beta\lambda_j}\Big)  & \hspace{0.1cm} ; \hspace{0.1cm}\lambda_j > \frac{4 \beta}{(1+ \beta)^2} \\
      \frac{1}{2}\Big((1+\beta)\lambda_j \pm \textbf{\textit{i}}\sqrt{4\beta\lambda_j-(1+\beta)^2\lambda_j^2 }\Big)  & \hspace{0.1cm} ; \hspace{0.1cm}\lambda_j \in \left(0, \frac{4 \beta}{(1+ \beta)^2}\right].   \label{counterexjac4}
    \end{cases}
\end{align}
Thus, $x_j$ is complex whenever $ \lambda_j \in \left(0, \frac{4 \beta}{(1+ \beta)^2}\right]$, implying the matrix $\X$ has complex entries. However, by construction, each $DN_k(\x_k)$ is a real matrix from \eqref{counterexjac2}, and the limit of a sequence of real matrices must also be real. This contradiction shows that the sequence $\{DN_k(\x_k)\}_{k=0}^{\infty}$ cannot converge. Therefore, even when the Jacobian sequence in $\mathbb{R}^{2n}$ converges, the Jacobian sequence in $\mathbb{R}^{n}$ may fail to converge, highlighting a fundamental non-equivalence in the asymptotic dynamics between $\mathbb{R}^{n}$ and $\mathbb{R}^{2n}$.
}

\begin{rema}
{Note that in the above example, the failure of convergence of the sequence $\{DN_k(\x_k)\}_{k=0}^{\infty}$ is supported by the Arzela--Ascoli theorem. The Jacobian maps $\{DN_k\}_{k=0}^{\infty}$ are constant maps $\probP_1$ a.s. for any quadratic function (Lemma \ref{lemsupab}) and are therefore uniformly equicontinuous. However, to claim convergence of the sequence $\{DN_k(\x_k)\}_{k=0}^{\infty}$ using the Arzela--Ascoli theorem, uniform local equicontinuity must be accompanied by local uniform boundedness of the maps $\{DN_k\}_{k=0}^{\infty}$. In the above example, this second condition fails to hold. Remarkably, the asymptotic eigenvalues from \eqref{counterexjac4}, derived under the incorrect supposition of convergence, have the same expression as those in Corollary \ref{contmomentumeigen} for constant momentum $\beta$.}
\end{rema}

\subsubsection{{Explicit dynamics of the trajectories around the strict saddle point}}
{Continuing with the example quadratic function, we next develop explicit characterizations of the trajectories in both $\mathbb{R}^{2n}$ and $\mathbb{R}^n$ to further highlight differences in dynamics within the ambient space and the $2n$-dimensional lifted space. Within the $\mathbb{R}^{2n}$ space, let us denote $\mathbf{w}_k := [\mathbf{x}_k;\mathbf{x}_{k-1}]$ and $\mathbf{w}^* := [\mathbf{x}^*; \mathbf{x}^*]$. Then by Taylor's expansion with integral remainder about the fixed point $\mathbf{w}^*$ of $P_k$, for any $k$ we get:
\begin{align}
    \mathbf{w}_{k+1} - \mathbf{w}^* &= \left( \int_{t=0}^1 DP_k(\mathbf{w}^* + t (\mathbf{w}_k - \mathbf{w}^*)) \, dt \right)(\mathbf{w}_k - \mathbf{w}^*) = DP(\mathbf{w}^*) (\mathbf{w}_k - \mathbf{w}^*) \\
    \implies \mathbf{w}_K - \mathbf{w}^* &= \left(DP(\mathbf{w}^*)\right)^K (\mathbf{w}_0 - \mathbf{w}^*). \label{counterexjacaz1}
\end{align}}
{In contrast, for the trajectory in $\mathbb{R}^n$, we have:
\begin{align}
    \mathbf{x}_{k+1} - \mathbf{x}^* &= \left( \int_{t=0}^1 DN_k(\mathbf{x}^* + t (\mathbf{x}_k - \mathbf{x}^*)) \, dt \right)(\mathbf{x}_k - \mathbf{x}^*) = DN_k(\mathbf{x}^*) (\mathbf{x}_k - \mathbf{x}^*) \hspace{0.2cm} \probP_1 \text{ a.s.} \\
    \implies \mathbf{x}_K - \mathbf{x}^* &= \left( \prod_{k=0}^{K-1} DN_k(\mathbf{x}^*) \right)(\mathbf{x}_0 - \mathbf{x}^*) \hspace{0.2cm} \probP_1 \text{ a.s.} \label{counterexjacaz2}
\end{align}
where in the first step we used the fact that the Jacobian maps $\{DN_k\}_{k=0}^{\infty}$ are constant maps $\probP_1$ a.s. for any quadratic function (Lemma~\ref{lemsupab}). Since the Jacobian maps are constant almost surely, $DN_k(\mathbf{x}^*)$ can be obtained from the recursion \eqref{counterexjac2} by substituting $\mathbf{x}_k = \mathbf{x}^*$ for all $k$:
\begin{align}
    DN_k(\mathbf{x}^*) &= \left( \mathbf{I} - h \nabla^2 f(\mathbf{x}^*) \right)\left( (1 + \beta)\mathbf{I} - \beta \left[DN_{k-1}(\mathbf{x}^*)\right]^{-1} \right) \hspace{0.2cm} \probP_1 \text{ a.s.} \label{counterexjacaz3}
\end{align}
Observe that from \eqref{counterexjacaz3}, for any $k$, $DN_k(\mathbf{x}^*) \neq DN_{k-1}(\mathbf{x}^*)$. If not, suppose for some $k$ we have $DN_k(\mathbf{x}^*) = DN_{k-1}(\mathbf{x}^*) = \mathbf{X}$. Substituting $\mathbf{X}$ into \eqref{counterexjacaz3} and rearranging yields the matrix quadratic equation \eqref{counterexjac3}. This equation admits complex-valued solutions, which leads to a contradiction---as before---since the matrix $DN_k(\mathbf{x}^*)$ generated from \eqref{counterexjacaz3} is a real diagonal matrix for any $k$.
}

{Thus, while the dynamics in $\mathbb{R}^{2n}$ evolve via the constant matrix $DP(\mathbf{w}^*)$ in \eqref{counterexjacaz1}, the same is not true in $\mathbb{R}^n$, as \eqref{counterexjacaz2} and \eqref{counterexjacaz3} show that $DN_k(\mathbf{x}^*)$ varies with $k$. This discussion explicitly highlights the difficulty in developing measure-theoretic results of almost sure saddle avoidance in $\mathbb{R}^n$ for \eqref{generalds}. To establish an almost sure non-convergence result in $\mathbb{R}^n$, one would need to invoke a version of the stable manifold theorem under the condition that $DN_k(\mathbf{x}^*)$ has at least one eigenvalue outside the unit disk in the complex plane for all $k$. Then, if $N_k$ and $DN_k$ were to converge uniformly to some $N$ and $DN$, respectively, one could define switched dynamics:
\begin{align}
    \mathbf{x}_{k+1} =  
    \begin{cases} 
        N_k(\mathbf{x}_k) & 0 \leq k \leq r, \\
        N(\mathbf{x}_k) & k > r,
    \end{cases}
    \label{switchdsrev1}
\end{align}
for any $r \geq 0$, and show almost sure saddle avoidance for such dynamics, as done in Theorem~\ref{measuretheorem2}. Then, using tools from Banach space theory similar to the proof of Theorem~\ref{measuretheorem3}, one could argue almost sure saddle avoidance for the original dynamics $\mathbf{x}_{k+1} = N_k(\mathbf{x}_k)$. However, this is not possible since the example in Section~\ref{counterexsecrev} clearly shows that even for the simplest quadratic function with constant momentum, the Jacobian map $DN_k$ fails to converge. Thus, the intrinsic complexity of the dynamics in $\mathbb{R}^n$ suggests that this setting may be substantially more interesting---and more challenging---to analyze than the lifted $\mathbb{R}^{2n}$ formulation, especially since the iterates themselves evolve in $\mathbb{R}^n$.
}

{We next turn to studying certain non-asymptotic properties of a class of accelerated methods near strict saddle points and derive key estimates that quantify these properties. For technical convenience, and notwithstanding the preceding discussion, the following analysis is conducted in $\mathbb{R}^{2n}$, and the rationale for this choice is explained in the next section.
}

\section{Exit time analysis of trajectories around weakly hyperbolic fixed points}\label{exittimesection}
{Roughly speaking, weakly hyperbolic fixed points (which we will formally define in Definition \ref{defweakhyp}) can be considered to be strict saddle points of a dynamical system in a complex vector space.} In this section we are interested in a non-asymptotic escape analysis for a class of accelerated methods \eqref{generalds} from strict saddle neighborhoods {around such points}. Recall that unlike the previous section where the asymptotic metrics could be evaluated by working in the $n$-dimensional vector space using the map $ N_k : \x_k \mapsto \x_{k+1}$, it is not clear how such an analysis can be used to obtain non-asymptotic properties such as escape rate. The reason behind this is the fact that in the previous section, the closed form expression of the eigenvalues of the non-asymptotic Jacobian map $D N_k $ could not be computed. In particular, for obtaining exit time expressions with our techniques we need the closed form expression of the eigenvalues of the non-asymptotic Jacobian map $D N_k $ which {does not seem to be simple to estimate} due to the coupling between $D N_k $ and $D N_{k-1} $. 

However, by leveraging the properties of the \eqref{generalds} dynamical system in the $2n$-dimensional vector space, we can easily {resolve} the coupling issue and obtain the eigenvalues of the non-asymptotic Jacobian map. Moreover, if the trajectory of the iterate pair $[\x_k; \x_{k-1}]$ exits some $\epsilon$-neighborhood of $[\x^*;\x^*]$ in $2n$-dimensional vector space for any strict saddle point $\x^*$ of $f$, then it must be that the trajectory of the iterate $\x_k$ also exits some $\mathcal{O}(\epsilon)$ neighborhood of $\x^*$ in the $n$-dimensional vector space. Thus, the exit time estimates in the $n$-dimensional system are equivalent to the estimates in the $2n$-dimensional system upto some constant. Also, by working in the $2n$-dimensional vector space we can use the standard analytic machinery \cite{o2019behavior} while dealing with non-asymptotic properties of accelerated methods. Note that from here onwards, in order to study the non-asymptotic properties and rates of \eqref{generalds}, we will require the $\mathcal{C}^2$ function class as opposed to the analytic class unless otherwise stated. We first develop a theoretical framework for a general complex dynamical system where we compute the exit times following which we then express our accelerated method in the complex dynamics and compute its exit time from any strict saddle neighborhood.

\subsection{ {Dynamics of \eqref{generalds} in $2n$-dimensional vector space}}\label{intuitioncomplex}
 {Consider a general acceleration scheme given by:
\begin{align}
\tag{\textbf{G-AGM1}}
    \begin{aligned}
    \y_{k} & = \x_{k} + \beta_k (\x_{k} - \x_{k-1}),  \\
    \x_{k+1} & = \y_k - h \nabla f(\y_k), \label{ds1} \\
    \abs{\beta_k -\beta} &\leq \mathcal{O}(1/k) \hspace{0.2cm} \forall \hspace{0.2cm} k \hspace{0.2cm}, \\ \beta_k &\text{ is non-decreasing with } k.
\end{aligned}
\end{align}
Observe that in \eqref{ds1} above which is a sub-class of \eqref{generalds}, we have assumed that $\beta_k$ is a non-decreasing sequence and $\beta_k \to \beta$ with a rate of $\mathcal{O}(1/k)$ (this assumption on the sequence $\{\beta_k\}$ covers Nesterov accelerated gradient method \eqref{originalnesterov} and constant momentum method \eqref{generaldsconst}). Also we have that $f(\cdot)$ is some $L$-gradient Lipschitz continuous function in the class $\mathcal{C}^{2,1}_{\mu, L}(\mathbb{R}^n) \subset \mathcal{C}^{2,1}_L(\mathbb{R}^n) $ \footnote{Unlike the previous section where $f$ was required to be analytic while doing asymptotic rate analysis, for the non-asymptotic analysis we do not need the analyticity of $f$.}and $h$ is some step size in $(0, \frac{1}{L})$. Recall that the function class $\mathcal{C}^{2,1}_{\mu, L}(\mathbb{R}^n) $ has already been defined in Section \ref{metricsection}.
Further we have the following assumption on $f(\cdot)$:
		\begin{itemize}
	\item[] \textbf{A1.} \textit{The Hessian of function $f(\cdot)$ is locally $M$-Lipschitz continuous around any strict saddle point of $f(\cdot)$ where we have that:\\} $$\norm{\nabla^{2} f(\x) - \nabla^{2} f(\y)}_2 \leq M \norm{\x - \y}$$
	for any $\x, \y$ in some compact neighborhood of a strict saddle point of $f(\cdot)$. 
	\end{itemize}
Using Taylor's formula with an integral remainder we can write $$  \nabla f(\y_k)  = \bigg(\int_{p=0}^{1} \nabla^2 f(\x^*+p(\y_k-\x^*))dp\bigg) (\y_k -\x^*)$$ where $\x^*$ is a strict saddle point of $f(\cdot)$ and $\x_k, \x_{k+1} \in \mathcal{B}_{\epsilon}(\x^*) $ for some $k$ and sufficiently small $\epsilon$ so that $f$ is Hessian Lipschitz continuous on $\mathcal{B}_{\epsilon}(\x^*)  $. For simplicity let $$\D(\y_k) = \int_{p=0}^{1} \nabla^2 f(\x^*+p(\y_k-\x^*))dp, $$ then using \textbf{A1} we have $\norm{\D(\y_k)-\nabla^2 f(\x^*)}_2 = \mathcal{O}(\epsilon)$. 
Then the update in \eqref{ds1} can be compactly written as:
\begin{align}
\begin{bmatrix}
 \x_{k+1} \\  \x_{k}
\end{bmatrix} &= \begin{bmatrix}
\x_k + \beta_k(\x_k-\x_{k-1}) - h \nabla f(\y_k) \\  \x_{k}
\end{bmatrix}\\
\begin{bmatrix}
 \x_{k+1} \\  \x_{k}
\end{bmatrix} & = \begin{bmatrix}
\x_k + \beta_k(\x_k-\x_{k-1}) - h \D(\y_k) (\y_k -\x^*)\\  \x_{k}
\end{bmatrix}\\
\begin{bmatrix}
 \x_{k+1}-\x^* \\  \x_{k}-\x^*
\end{bmatrix} & = \begin{bmatrix}
(1+\beta_k)(\mathbf{I}-h\D(\y_k))(\x_k-\x^*) - \beta_k(\mathbf{I}-h\D(\y_k))(\x_{k-1}-\x^*) \\  \x_{k}-\x^*
\end{bmatrix} \\
\begin{bmatrix}
 \x_{k+1}-\x^* \\  \x_{k}-\x^*
\end{bmatrix} & = \begin{bmatrix}
(1+\beta_k)(\mathbf{I}-h\D(\y_k))\hspace{0.5cm} - \beta_k(\mathbf{I}-h\D(\y_k)) \\  \mathbf{I} \hspace{3.5cm} \boldsymbol{0}
\end{bmatrix}\begin{bmatrix}
 \x_{k}-\x^* \\  \x_{k-1}-\x^*
\end{bmatrix}. 
\end{align}
Simplifying the last step even further yields
\begin{align}
\begin{bmatrix}
 \x_{k+1}-\x^* \\  \x_{k}-\x^*
\end{bmatrix} & = \underbrace{\begin{bmatrix}
(1+\beta)(\mathbf{I}-h\nabla^2f(\x^*))\hspace{0.5cm} - \beta(\mathbf{I}-h\nabla^2f(\x^*)) \\  \mathbf{I} \hspace{3.5cm} \boldsymbol{0}
\end{bmatrix}}_{=\V\Lambda \V^{-1}}\begin{bmatrix}
 \x_{k}-\x^* \\  \x_{k-1}-\x^*
\end{bmatrix} + \nonumber\\
& \hspace{-0.5cm}\underbrace{\begin{bmatrix}
(\beta_k-\beta)(\mathbf{I}-h\nabla^2f(\x^*))\hspace{0.5cm} - (\beta_k-\beta)(\mathbf{I}-h\nabla^2f(\x^*)) \\  \boldsymbol{0} \hspace{3.5cm} \boldsymbol{0}
\end{bmatrix}}_{=\C_k}\begin{bmatrix}
 \x_{k}-\x^* \\  \x_{k-1}-\x^*
\end{bmatrix}+\nonumber\\
& \hspace{-0.5cm} \underbrace{\begin{bmatrix}
(1+\beta_k)h(\nabla^2f(\x^*)-\D(\y_k))\hspace{0.5cm} - \beta_k h(\nabla^2f(\x^*)-\D(\y_k)) \\  \boldsymbol{0} \hspace{3.5cm} \boldsymbol{0}
\end{bmatrix}}_{\mathbf{M}_k}\begin{bmatrix}
 \x_{k}-\x^* \\  \x_{k-1}-\x^*
\end{bmatrix} \label{eigdecomposemetric1} \\
\underbrace{\V^{-1} \begin{bmatrix}
 \x_{k+1}-\x^* \\  \x_{k}-\x^*
\end{bmatrix}}_{\u_{k+1}} & = \Lambda \underbrace{\V^{-1} \begin{bmatrix}
 \x_{k}-\x^* \\  \x_{k-1}-\x^*
\end{bmatrix}}_{\u_k} + \underbrace{\V^{-1}\C_k \V}_{\B_k} \underbrace{\V^{-1}\begin{bmatrix}
 \x_{k}-\x^* \\  \x_{k-1}-\x^*
\end{bmatrix}}_{\u_k} + \nonumber \\ & \hspace{1cm} \underbrace{\V^{-1} \mathbf{M}_k\V}_{=2\epsilon \R(\u_k)} \underbrace{\V^{-1}\begin{bmatrix}
 \x_{k}-\x^* \\  \x_{k-1}-\x^*
\end{bmatrix}}_{\u_k} \\
\u_{k+1} & = \Lambda \u_k + \B_k\u_k + 2\epsilon \R(\u_k)\u_k \label{dynamicsystem1abc}
\end{align}
where $ \R(\u_k)$ is some perturbation matrix that depends on $\u_k$ because the matrix $\M_k$ depends on $\y_k$. {In particular, since $\norm{\D(\y_k)-\nabla^2 f(\x^*)}_2 = \mathcal{O}(\epsilon)$ we get from \eqref{eigdecomposemetric1} that $ \mathbf{M}_k = \mathcal{O}(\epsilon) $ and so $\R(\u_k) $ can be treated as a constant order term in $\epsilon$.} Note that \eqref{dynamicsystem1abc} is a linear complex dynamical system in $\u_k$ for any $k =  \Omega(\epsilon^{-1})$ \footnote{ {Since $k =  \Omega(\epsilon^{-1})$ we will have $\B_k = \mathcal{O}(\epsilon) $ by the definition of $\C_k$ and the fact $\beta_k \to \beta$ with a rate of $\mathcal{O}(1/k)$. Then we can write \eqref{dynamicsystem1abc} as $\u_{k+1}  = (\Lambda  + \mathcal{O}(\epsilon))\u_k $ which is linear for sufficiently small $\epsilon$.}} where the vector $ \u_k \in \mathbb{C}^{2n}$. Observe that we assumed diagonalizability in the third last step \eqref{eigdecomposemetric1} here which will be satisfied provided $\lambda_i \neq \frac{4\beta}{(1+\beta)^2} $ for any $i$ where $\lambda_i $ is the $i^{th}$ eigenvalue of $\mathbf{I}-h\nabla^2f(\x^*)$. But since $\lambda_i \neq \frac{4\beta}{(1+\beta)^2} $ $\probP_1$-almost surely (see Section \ref{seceigendiagonal}), the diagonalization is justified. Therefore the nonlinear dynamics of \eqref{ds1} can be transformed into the linear dynamics of $\u_k$ in complex vector space and by doing so the analysis can be done with significant ease. It should however be noted that the linearization \eqref{dynamicsystem1abc} can only be achieved for $k =  \Omega(\epsilon^{-1})$ (more details in Section \ref{kboundsection}) and therefore we require that $\beta_k \to \beta$ with a rate of $\mathcal{O}(1/k)$ in \eqref{ds1}.}

\subsection{Exit time estimates for a general complex-valued linear dynamical system}

Consider the following $d$-dimensional complex-valued linear dynamical system which evolves around a small local neighborhood of a point $\u^* \in \mathbb{C}^d$ from the relation:
\begin{align}
    \u_{k+1}- \u^* = \Lambda (\u_k- \u^*) + \B_k(\u_k- \u^*) + \norm{\u_k-\u^*} \mathbf{P}(\u_k) (\u_k-\u^*),
\end{align}
where $\u_k-\u^* \in \mathbb{C}^d$ is a complex radial vector originating from some point $\u^*$, $\Lambda$ is a diagonal matrix with diagonal elements given by complex numbers $\{z_1,\dots, z_d\}$ and $\B_k, \mathbf{P}(\u_k) $ can be treated as some perturbation matrices.
 Without loss of generality we can set {$\u^*=\boldsymbol{0}$} and obtain:
\begin{align}
    \u_{k+1} &= \Lambda \u_k + \B_k \u_k + \norm{\u_k} \mathbf{P}(\u_k) \u_k = \Lambda \u_k + \B_k \u_k + \epsilon\R(\u_k) \u_k , \label{exittime1}
\end{align}
where $ \R(\u_k) = \frac{\norm{\u_k}}{\epsilon}\mathbf{P}(\u_k) $, $\epsilon \ll 1$ and suppose the trajectory of $\{\u_k\}$ starts from the surface of the ball $\mathcal{B}_{\epsilon}(\mathbf{0})$, i.e., $\norm{\u_0} = \epsilon$. Next we have the following conditions on $\B_k$ and $\R(\u_k)$:
\begin{enumerate}[label=\textbf{S.\arabic*}]
    \item \label{conditionsdynamicsa} \textit{The matrix sequence $\{\B_k\} $ satisfies $\norm{\B_k}_F \to {0}$ as $k \to \infty$ with\footnote{The little-o notation in Condition \ref{conditionsdynamicsa} is with respect to $\epsilon \to 0$.} $\sup_{k \geq 0}\norm{\B_k}_F = o(\epsilon) $.} \\
    \item \label{conditionsdynamicsb} \textit{The matrix  $\mathbf{R}(\u_k)$ is of bounded norm for all $k$ provided $\norm{\u_k} \leq \epsilon$.}\\
\end{enumerate}
 We are interested in finding the linear exit time upper bound, i.e. $K_{exit} = \mathcal{O}(\log(\epsilon^{-1}))$ where $K_{exit} = \inf_{k>0}\{k \vert \norm{\u_k} > \epsilon\} $, for the trajectory generated by $\{\u_k\}$ from \eqref{exittime1} in the ball $\mathcal{B}_{\epsilon}(\mathbf{0})$.
Recursively writing down \eqref{exittime1} and expanding the product up to first order in $\epsilon$ gives:
\begin{align}
    \u_K & = \Pi_{k=0}^{K-1}(\Lambda  + \B_k + \epsilon \mathbf{R}(\u_k))\u_0 \\
     & = {\bigg( \Lambda^K  +  \epsilon\sum\limits_{{r=1}}^{K}\Lambda^{r-1} (o(1) +\mathbf{R}(\u_{K-r})) \Lambda^{K-r} + \underbrace{\mathcal{O}(\norm{\Lambda}_2^{K-2} (K\epsilon)^2)}_{\textit{tail error}}\bigg)\u_0} \label{exittimetailerror} \\
    & = \underbrace{\bigg( \Lambda^K  +  \epsilon\sum\limits_{{r=1}}^{K}\Lambda^{r-1} (o(1) +\mathbf{R}(\u_{K-r})) \Lambda^{K-r}\bigg)\u_0}_{\tilde{\u}_K} + \mathcal{O}(\norm{\Lambda}_2^{K-2} (K\epsilon)^2\epsilon). \label{exittime2}
\end{align}
Notice that in \eqref{exittime2} we introduced the approximation $ \tilde{\u}_K $ for the exact trajectory ${\u}_K$ so as to work with approximations in order to facilitate the analysis. It should be noted that for the approximation $ {\u}_K \approx \tilde{\u}_K$ to hold, we must necessarily have $K\epsilon \ll 1$ so that the tail error from \eqref{exittime2} can possibly remain bounded. {The approximation error of $\mathcal{O}(\norm{\Lambda}_2^{K-2} (K\epsilon)^2\epsilon)$ in \eqref{exittime2} follows from a straightforward calculation of the tail error bound, as shown below:
\begin{align}
    \u_K - \tilde{\u}_K & = \Pi_{k=0}^{K-1}(\Lambda  + \B_k + \epsilon \mathbf{R}(\u_k))\u_0 - \bigg( \Lambda^K  +  \epsilon\sum\limits_{r=1}^{K}\Lambda^{r-1} (o(1) +\mathbf{R}(\u_{K-r})) \Lambda^{K-r} \bigg)\u_0 \\
    \implies \norm{\u_K - \tilde{\u}_K} & \leq  \sum\limits_{r=2}^K  {K \choose r} \bigg( \sup_{0 \leq k \leq K}\norm{\B_k} + \epsilon \sup_{0 \leq k \leq K} \norm{\mathbf{R}(\u_k)} \bigg)^r \norm{\Lambda}_2^{K-r} \norm{\u_0} \\
    &\!\!\!\!\!\!\!\!\!\!\!\!\!\!\!\!\!\! \underbrace{\leq}_{\textbf{Conditions } \ref{conditionsdynamicsa}, \ref{conditionsdynamicsb} }  \sum\limits_{r=2}^K  K^r \epsilon^r\bigg( o(1) +  \underbrace{\sup_{0 \leq k \leq K} \norm{\mathbf{R}(\u_k)}}_{\leq C_0} \bigg)^r \norm{\Lambda}_2^{K-r} \norm{\u_0} \\
    & \leq \norm{\Lambda}_2^{K-2} \sum\limits_{r=2}^K    \bigg( \underbrace{K \epsilon (o(1) + C_0)}_{\ll 1 \text{ for } K\epsilon \ll 1}\bigg)^r \norm{\u_0} \\ 
    & \leq  \norm{\Lambda}_2^{K-2}  \frac{\bigg( {K \epsilon (o(1) + C_0)}\bigg)^2}{ 1- \bigg( {K \epsilon (o(1) + C_0)}\bigg) } \norm{\u_0} = \mathcal{O}(\norm{\Lambda}_2^{K-2} (K\epsilon)^2\epsilon),
\end{align}
where, in the second-to-last step, we used the fact that \( \| \Lambda \|_2 > 1 \) by assuming \( \u^* = \mathbf{0} \) to be a weakly hyperbolic fixed point (formally defined later in Definition~\ref{defweakhyp}) of the dynamics \eqref{exittime1}.} Since we are looking for linear exit time solutions, i.e., $K_{exit} = \mathcal{O}(\log(\epsilon^{-1}))$, the condition $ K\epsilon \ll 1$ will always be satisfied for any $K \leq K_{exit}$. In particular a sufficient condition for the approximation $ {\u}_K \approx \tilde{\u}_K$ to hold for all $K \leq K_{exit}$ is given by:
\begin{align}
  \mathcal{R} & =  \sup_{0 \leq K \leq K_{exit}} \frac{\norm{{\u}_K - \tilde{\u}_K}}{\norm{ {\u}_K}} \xrightarrow{\epsilon \to 0} 0, \label{relerror}
\end{align}
where $\mathcal{R} $ is the maximum relative error in the trajectory approximation for the duration the actual trajectory stays within an $\epsilon$-neighborhood of origin.
Later on we will show that the relative error condition \eqref{relerror} is satisfied for trajectories with $K_{exit} = \mathcal{O}(\log(\epsilon^{-1}))$. To see that \eqref{relerror} is a sufficient condition for the approximation to hold, we can write $\tilde{\u}_K = {\u}_K + \v$ for some vector $\v$ that is a function of $K$. Then $\norm{\v} \leq  \mathcal{R} \norm{\u_K}$ from \eqref{relerror} and therefore $\v = o(\norm{\u_K})$ because as $\epsilon \to 0$ we have $ \norm{\u_K} \to 0$ but then $\frac{\norm{\v}}{\norm{\u_K}} \leq  \mathcal{R}\xrightarrow{\epsilon \to 0} 0  $. 

We now formally define the term `weakly hyperbolic fixed point' of the complex valued dynamical system \eqref{exittime1}. 
\begin{defi}\label{defweakhyp}
    Suppose $ \p \in \mathbb{C}^d$ is a fixed point of the complex valued dynamical system \eqref{exittime1} and the conditions \eqref{conditionsdynamicsa}, \eqref{conditionsdynamicsb} are satisfied. Let the absolute value of the eigenvalues $z_i$ of $\Lambda$ be divided into 3 disjoint sets using the stable center manifold theorem for a complex valued dynamical system. In other words, for the canonical basis $\{\e_i\}$ of $\mathbb{R}^d$, the complex valued vector $\u_k$ from \eqref{exittime1} belongs to a vector space $\mathcal{E} = \mathcal{E}_{S} \bigoplus \mathcal{E}_{US} \bigoplus \mathcal{E}_{C}$ where we have that: 
    \begin{align*}
  	\mathcal{E}_{S} &= span\{\e_i | \hspace{0.05cm}\abs{z_i} <1\}, \hspace{0.05cm} \mathcal{N}_{S} = \{i | \hspace{0.05cm}\abs{z_i} < 1\}, \\
  		\mathcal{E}_{US} &= span\{\e_i  | \hspace{0.05cm}\abs{z_i} > 1\}, \hspace{0.05cm} \mathcal{N}_{US} = \{i |\hspace{0.05cm}\abs{z_i} >1 \}, \\
  		\mathcal{E}_{C} &= span\{\e_i  | \hspace{0.05cm}\abs{z_i} = 1\}, \hspace{0.05cm} \mathcal{N}_{C} = \{i |\hspace{0.05cm} \abs{z_i} = 1\}
\end{align*}
	and $\mathcal{E}_{S}, \mathcal{E}_{US}, \mathcal{E}_{C} $ represent the stable, unstable and center subspaces respectively. Then if $\Lambda$ satisfies the condition of $ dim(\mathcal{E}_{US}) \geq 1$, the point $ \p$ is referred to as a weakly hyperbolic fixed point of the dynamical system \eqref{exittime1}. The subspace $	\mathcal{E}_{US} \bigoplus 	\mathcal{E}_{C} $ is referred to as the center unstable manifold.
\end{defi}

\begin{rema}
    Note that the above definition is slightly different from the standard definition of hyperbolic fixed points \cite{tabor1989chaos, ott2002chaos} for the $\mathcal{C}^1$ map on real field given by $T : \mathbb{R}^n \rightarrow \mathbb{R}^n$ with $\x_{k+1} = T (\x_k)$. In particular, by the standard definition a fixed point $\p$ of the map $T$ is hyperbolic if no eigenvalue of the Jacobian $D T (\p)$ lies on the unit circle and at least one eigenvalue has magnitude greater than $1$. In other words it says that the linearization of the dynamics $\x_{k+1} = T (\x_k)$ about $\p$ given by $ \x_{k+1} - \p = D T (\p)(\x_k - \p) + o(\norm{\x_k - \p})  $ induces hyperbolic trajectories around $\p$. However on observing the complex valued dynamical system \eqref{exittime1}, we find that $\Lambda$ is similar to $ D T (\p)$ and the residual terms in \eqref{exittime1} given by $\B_k \u_k$, $ \epsilon \R(\u_k) \u_k = \norm{\u_k} \P(\u_k)\u_k$ are $o(\norm{\u_k})$ as $k \to \infty$ by conditions \eqref{conditionsdynamicsa}, \eqref{conditionsdynamicsb}. Hence the definition of `weakly hyperbolic fixed point' in Definition \ref{defweakhyp} is consistent with the standard terminology.
\end{rema}

 {Our next theorem provides an upper bound on the exit time of the approximate trajectory $\{\tilde{\u}_k \} $ from a sufficiently small $\epsilon$ neighborhood of any weakly hyperbolic fixed point and the conditions on the initial radial vector $\u_0$ that guarantee a linear exit time. In particular the theorem also establishes that if the approximate trajectory $\{\tilde{\u}_K \} $ satisfies the relative error condition from \eqref{relerror} then the exact trajectory $\{ {\u}_K\} $ also has a linear exit time.}
\begin{theo}\label{exittimethm1}
Suppose the exact trajectory $\{ {\u}_K\} $ from \eqref{exittime1} starts from the surface of the ball $\mathcal{B}_{\epsilon}(\boldsymbol{0})$ at $k=0$, i.e., $\u_0 = \epsilon\sum\limits_{i=1}^{d} \theta_i \e_i $ with $\e_i$ form the canonical basis of $\mathbb{R}^d$ Euclidean space, $\epsilon\theta_i = \langle \u_0, \e_i\rangle$ for all $i$ with $ \theta_i \in \mathbb{C}$ and $\epsilon \ll 1$. Also, suppose $ \boldsymbol{0} \in \mathbb{C}^d$ is a weakly hyperbolic fixed point of the dynamical system \eqref{exittime1}. Then with Conditions \ref{conditionsdynamicsa} and \ref{conditionsdynamicsb}, the approximate trajectory $\{ \tilde{\u}_K\} $ for $K\epsilon \ll 1$ given by $  \tilde{\u}_K =\bigg( \Lambda^K  +  \epsilon\sum\limits_{ {r=1} }^{K}\Lambda^{r-1} (o(1) +\mathbf{R}(\u_{K-r})) \Lambda^{K-r}\bigg)\u_0$ from \eqref{exittime2} will exit the ball $\mathcal{B}_{\epsilon}(\boldsymbol{0})$ in linear time provided $\u_0 \in \mathcal{K}_{\sigma} \cap \mathcal{B}_{\epsilon}(\boldsymbol{0})$, $\norm{\u_0} =\epsilon$ where $\mathcal{K}_{\sigma} = \{ \z \in \mathbb{C}^d \hspace{0.1cm}\vert \hspace{0.1cm} \frac{\pi_{\mathcal{E}_{US}}(\z)}{\norm{\z}} \geq \sigma^{\frac{1}{2}}; \hspace{0.1cm} \z \neq \boldsymbol{0} \} $ is the double cone\footnote{ {The map $\pi_{\mathcal{E}_{US}}(.) $ gives the $\ell_2$ norm of the orthogonal projection of the argument vector on the subspace $\mathcal{E}_{US}$. By double cone $\mathcal{K}_{\sigma}$ we mean that if $ \z \in \mathcal{K}_{\sigma}$ then $ -\z \in \mathcal{K}_{\sigma}$. }} containing the unstable subspace $\mathcal{E}_{US} $ and $\sigma$ is of constant order, i.e., $0 \ll \sigma < 1$. Formally, the exit time defined as $K_{exit} = \inf_{k>0}\{k \vert \norm{\tilde{\u}_k} > \epsilon\} $ is upper bounded by:
\begin{align}
     K_{exit} \lessapprox \begin{cases}
    \frac{\log\bigg( \frac{2 \log\bigg(\frac{\norm{\Lambda}_2}{\inf_{i  \in \mathcal{N}_{US}}\abs{z_i}}\bigg)}{\epsilon \norm{\Lambda}^{-1}_2 \Gamma}\bigg)}{2\log\bigg(\frac{\norm{\Lambda}_2}{\inf_{i  \in \mathcal{N}_{US}}\abs{z_i}}\bigg)}  \hspace{0.5cm} ; \hspace{0.5cm} & \inf_{i  \in \mathcal{N}_{US}}\abs{z_i} <  \norm{\Lambda}_2 ,
    \\
    \frac{\log\bigg(\frac{2\log \bigg(\bigg(1-\frac{1}{\norm{\Lambda}_2}\bigg)^{-1}\bigg)}{\epsilon\norm{\Lambda}^{-1}_2 \Gamma}\bigg)}{2\log\bigg(\bigg(1-\frac{1}{\norm{\Lambda}_2}\bigg)^{-1}\bigg)}   \hspace{0.5cm} ; \hspace{0.5cm} & \inf_{i  \in \mathcal{N}_{US}}\abs{z_i} =  \norm{\Lambda}_2 ,
    \end{cases}
\end{align}
where $ \Gamma$ is some perturbation parameter satisfying $\sup_{\norm{\u_{k}}\leq \epsilon}\norm{\mathbf{R}(\u_{k})}\leq \Gamma$. Further, if this bound on $K_{exit}$ satisfies the relative error condition \eqref{relerror} {then we have that $$ K_{exit} =  \inf_{k>0}\{k \vert \norm{\tilde{\u}_k} > \epsilon\}\geq  \inf_{k>0}\bigg\{k \bigg\vert \norm{{\u}_k} > \frac{\epsilon}{1 + {\tilde{\gamma}(\epsilon)}}\bigg\}$$ {for some scalar ${\tilde{\gamma}(\epsilon)}\geq 0$ with ${\tilde{\gamma}(\epsilon)}=o(1)$, 
i.e.,} 
${\tilde{\gamma}(\epsilon)} \to 0$ as $\epsilon \to 0$.}
\end{theo}
The proof of this theorem is in Appendix \ref{local analysis appendix}. Theorem \ref{exittimethm1} will play a crucial role in Section \ref{sectiondynamicformula} where it will be used to estimate the exit time of trajectories of \eqref{ds1} from sufficiently small strict saddle neighborhoods. 

\begin{rema}\label{remark:epsilon.bound.exit.time}
{It is important to note that Theorem~\ref{exittimethm1} establishes that both the approximate and exact dynamics can have exit time \( \mathcal{O}(\log(1/\epsilon)) \) under appropriate conditions for any sufficiently small \( \epsilon > 0 \). By ``sufficiently small'' \( \epsilon \), we mean that there exists an explicit upper bound on \( \epsilon \) below which the term \( \tilde{\gamma}(\epsilon) \) remains negligible; however, we do not compute this bound explicitly for brevity. To see that the exit time bound \( \mathcal{O}(\log(1/\epsilon)) \) indeed holds for the exact trajectories, note from the statement and proof of Theorem~\ref{exittimethm1} that, under the relative error condition~\eqref{relerror}, for any sufficiently small \( \epsilon > 0 \), we have \( K_{\text{exit}} = \inf_{k > 0} \left\{ k \mid \| \tilde{\u}_k \| > \epsilon \right\} \geq \inf_{k > 0} \left\{ k \mid \| \u_k \| > \frac{\epsilon}{1 + {\tilde{\gamma}(\epsilon)}} \right\} \), where \( {\tilde{\gamma}(\epsilon)} \to 0 \) as \( \epsilon \to 0 \). Hence, for any sufficiently small \( \epsilon > 0 \), the exit time \( K_{\text{exit}} \) of the approximate trajectory \( \{ \tilde{\u}_k \} \) from the ball \( \mathcal{B}_\epsilon(\boldsymbol{0}) \) provides an upper bound on the exit time of the exact trajectory \( \{ \u_k \} \) from the ball \( \mathcal{B}_{\epsilon / (1 + o(1))}(\boldsymbol{0}) \), when \( K_{\text{exit}} = \mathcal{O}(\log(1/\epsilon)) \). Finally, a simple rescaling of \( \epsilon \) by a factor of \( 1 + o(1) \) implies that the exit time of the exact trajectory \( \{ \u_k \} \) from the ball \( \mathcal{B}_\epsilon(\boldsymbol{0}) \), given by \( \inf_{k > 0} \left\{ k \mid \| \u_k \| > \epsilon \right\} \), is of order \( \mathcal{O}(\log((1 + o(1))/\epsilon)) \), or equivalently \( \mathcal{O}(\log(1/\epsilon)) \), for any sufficiently small \( \epsilon > 0 \). A similar conclusion holds for the result stated in the subsequent Lemma~\ref{lemmaprojectioncondition}.}\looseness=-1
\end{rema}

\begin{rema}
It may be the case that the complex valued dynamical system instead of starting from the surface of the ball $\mathcal{B}_{\epsilon}(\boldsymbol{0})$ actually starts elsewhere and reaches the ball $\mathcal{B}_{\epsilon}(\boldsymbol{0})$ after some $K$ iterations. Then the conditions \ref{conditionsdynamicsa} and \ref{conditionsdynamicsb} need to hold from $k=K$ onward instead of $k=0$ and the exit time bound from Theorem \ref{exittimethm1} remains unaltered where now the exit time will be defined as $K_{exit} = \inf_{k>K}\{k \vert \norm{\tilde{\u}_k} > \epsilon\} $.
\end{rema}

\subsubsection{Relative error bound on the approximate trajectory\\}\label{relerrorsection1}
We now derive the bounds on relative error for the approximate trajectory used in Theorem \ref{exittimethm1} and verify if the relative error condition \eqref{relerror} is satisfied by the approximate trajectory. For the two different cases arising from Theorem \ref{exittimethm1} where the first case corresponds to $\inf_{i  \in \mathcal{N}_{US}}\abs{z_i} <  \norm{\Lambda}_2 $ while the second case corresponds to $\inf_{i  \in \mathcal{N}_{US}}\abs{z_i} =  \norm{\Lambda}_2 $, the next lemma provides a formal expression for the relative error in these two cases.\\

\begin{lemm}\label{lemmarelerrorexact}
Suppose for $K\epsilon \ll 1$, the approximate trajectory $\{\tilde{\u}_K\}$ of the exact trajectory $\{\u_K\} $ exits the ball $\mathcal{B}_{\epsilon}(\boldsymbol{0})$ in a linear time bounded from Theorem \ref{exittimethm1} where we have $\u_0 \in \mathcal{K}_{\sigma} \cap \mathcal{B}_{\epsilon}(\boldsymbol{0})$ with the double cone $\mathcal{K}_{\sigma} $ defined in Theorem \ref{exittimethm1}, $\u_0 = \epsilon\sum\limits_{i=1}^{d} \theta_i \e_i $ with $\e_i$ forming the canonical basis of $\mathbb{R}^d$ Euclidean space and $\epsilon\theta_i = \langle \u_0, \e_i\rangle$ for all $i$. Then the relative error $\mathcal{R}$ in the approximate trajectory $\{\tilde{\u}_K\}$ given by $  \mathcal{R} =  \sup_{0 \leq K \leq K_{exit}} \frac{\norm{{\u}_K - \tilde{\u}_K}}{\norm{ {\u}_K}}$ from \eqref{relerror} is bounded as follows:
\begin{align}
   \mathcal{R}   \leq \begin{cases}
   \frac{{\mathcal{O}\bigg(\frac{1}{\sqrt{\epsilon}} (\log(\epsilon^{-1})\epsilon)^2\bigg)}}{ \sqrt{\sum\limits_{i \in \mathcal{N}_{US}} \abs{\theta_i}^2} -   {\mathcal{O}\bigg(\frac{1}{\sqrt{\epsilon}} (\log(\epsilon^{-1})\epsilon)}\bigg)} \xrightarrow{\epsilon \to 0} 0 \hspace{0.5cm} ; \hspace{0.5cm} & \inf_{i  \in \mathcal{N}_{US}}\abs{z_i} <  \norm{\Lambda}_2
    \\
    \frac{{\mathcal{O}\bigg( (\log(\epsilon^{-1})\epsilon)^2\bigg)}}{ \sqrt{\sum\limits_{i \in \mathcal{N}_{US}} \abs{\theta_i}^2} -   {\mathcal{O}\bigg( (\log(\epsilon^{-1})\epsilon)}\bigg)} \xrightarrow{\epsilon \to 0} 0   \hspace{0.5cm} ; \hspace{0.5cm} & \inf_{i  \in \mathcal{N}_{US}}\abs{z_i} =  \norm{\Lambda}_2
    \end{cases}
\end{align}
and we have that the exit time of exact trajectory $\{\u_K\} $ from the ball $\mathcal{B}_{\epsilon}(\boldsymbol{0})$ for $\epsilon \ll 1$ is approximately equal to the exit time of approximate trajectory $\{\tilde{\u}_K\}$ from this ball.
\end{lemm}
The proof of this Lemma is in Appendix \ref{local analysis appendix}. {We note that the relative error bound from Lemma~\ref{lemmarelerrorexact} holds for any sufficiently small $\epsilon > 0$ (cf.~Remark~\ref{remark:epsilon.bound.exit.time}), while the relative error itself vanishes as $\epsilon \to 0$.}

\subsection{Some conditions for the existence of linear exit time solutions}\label{existencelineartime}
Recall that Theorem \ref{exittimethm1} only provided a generic condition of $\u_0 \in \mathcal{K}_{\sigma} \cap \mathcal{B}_{\epsilon}(\boldsymbol{0})$ for the existence of linear exit time time trajectories where $\mathcal{K}_{\sigma} = \{ \z \in \mathbb{C}^d \hspace{0.1cm}\vert \hspace{0.1cm} \frac{\pi_{\mathcal{E}_{US}}(\z)}{\norm{\z}} \geq \sigma^{\frac{1}{2}}; \hspace{0.1cm} \z \neq \boldsymbol{0} \} $ is the double cone containing the unstable subspace $\mathcal{E}_{US} $ and $\sigma$ is of constant order, i.e., $0 \ll \sigma < 1$. But this condition on $\u_0$ only depends on $\sigma$ where $\sigma$ is some constant. However there must be some dependency of the initial radial vector $\u_0$, or equivalently the norm of projection of unit radial vector $\u_0/\norm{\u_0}$ on the subspace $ \mathcal{E}_{US}$ given by $\frac{\pi_{\mathcal{E}_{US}}(\u_0)}{\norm{\u_0}}  $ where $ \frac{\pi_{\mathcal{E}_{US}}(\u_0)}{\norm{\u_0}} = \sqrt{\sum\limits_{i \in \mathcal{N}_{US}} \abs{\theta_i}^2} $, on key parameters like the eigenvalues of stationary matrix $\Lambda$, radius $\epsilon$ and the perturbation parameter $\Gamma$. The next lemma brings out the dependence of the initial unstable subspace projection value given by $  \sum\limits_{i \in \mathcal{N}_{US}} \abs{\theta_i}^2 $ on these key parameters.
\begin{lemm}\label{lemmaprojectioncondition}
From Theorem \ref{exittimethm1}, for $K\epsilon \ll 1$ and the case $\inf_{i  \in \mathcal{N}_{US}}\abs{z_i} <  \norm{\Lambda}_2 $, the necessary condition for the existence of the linear exit time bound is given by:
\begin{align}
     \sum\limits_{i \in \mathcal{N}_{US}} \abs{\theta_i}^2 &> \frac{\epsilon \Gamma}{\inf_{i  \in \mathcal{N}_{US}}\abs{z_i}} = \Theta(\epsilon \inf_{i  \in \mathcal{N}_{US}}\abs{z_i}^{-1}). 
\end{align}
 Next, the linear exit time bound for the case $\inf_{i  \in \mathcal{N}_{US}}\abs{z_i} =  \norm{\Lambda}_2 $ in Theorem \ref{exittimethm1} will definitely hold provided ${\sum\limits_{i \in \mathcal{N}_{US}} \abs{\theta_i}^2}$ satisfies the minimal\footnote{By `minimal' we mean that this condition is tight and one cannot relax it any further to obtain a different sufficient condition as per our proof technique.} sufficient condition given by:
\begin{align}
    \sum\limits_{i \in \mathcal{N}_{US}} \abs{\theta_i}^2 &\geq \frac{\epsilon\norm{\Lambda}^{-1}_2 \Gamma}{2\log (\norm{\Lambda}_2) } \bigg(1 +  \log \bigg(\frac{2\log (\norm{\Lambda}_2)}{\epsilon\norm{\Lambda}^{-1}_2 \Gamma}\bigg) \bigg)\nonumber\\ &= \Theta \bigg(  \frac{\epsilon}{\norm{\Lambda}_2\log (\norm{\Lambda}_2)}\log \bigg( \frac{\norm{\Lambda}_2\log (\norm{\Lambda}_2)}{\epsilon}\bigg)\bigg). \label{lemmaprojc2}
\end{align}
\end{lemm}
The proof of this lemma is given in Appendix~\ref{local analysis appendix}. Note from the sufficient condition in Lemma~\ref{lemmaprojectioncondition} that we have the following asymptotics:
\[
\Theta \bigg(  \frac{\epsilon}{\norm{\Lambda}_2 \log(\norm{\Lambda}_2)} \log \bigg( \frac{\norm{\Lambda}_2 \log(\norm{\Lambda}_2)}{\epsilon} \bigg) \bigg) \xrightarrow{\epsilon \to 0} 0.
\]
{It is important to point out, however, that similar to Theorem~\ref{exittimethm1}, the sufficient condition for linear exit time from Lemma~\ref{lemmaprojectioncondition} holds for any sufficiently small $\epsilon > 0$ (cf.~Remark~\ref{remark:epsilon.bound.exit.time}), and not just in the limit as $\epsilon \to 0$.}\looseness=-1 

Since the right hand side in \eqref{lemmaprojc2} has order $ \Theta (\epsilon \log (\epsilon^{-1}))$, the relative error condition of $ \sqrt{\sum\limits_{i \in \mathcal{N}_{US}} \abs{\theta_i}^2} >   {\mathcal{O}( \epsilon\log(\epsilon^{-1})}) $ from Section \ref{relerrorsection1} gets automatically satisfied for the case of $\inf_{i  \in \mathcal{N}_{US}}\abs{z_i} =  \norm{\Lambda}_2 $ for any sufficiently small $\epsilon$. Therefore the exit time bound from Theorem \ref{exittimethm1} will hold for the exact trajectory $\{\u_k\}$ in the case of $\inf_{i  \in \mathcal{N}_{US}}\abs{z_i} =  \norm{\Lambda}_2 $ provided the minimal sufficient condition \eqref{lemmaprojc2} holds. 

\subsection{Monotonicity of radial distances}\label{monotonicsection}
In this section we are interested in finding whether the trajectory of $\{\u_k\}$ in its expansion phase expands monotonically. Recall that Theorem \ref{exittimethm1} gives us the linear exit time bound $\mathcal{O}(\log(\epsilon^{-1}))$ from some $\epsilon$-neighborhood of any weakly hyperbolic fixed point for the approximate trajectory and the conditions on the initialization point of the trajectory under which such bound holds. But even after exiting this neighborhood, the trajectory could possibly return back to the same $\epsilon$-neighborhood, or even worse where the trajectory keeps on visiting this neighborhood multiple times. In that case a fast escape practically achieves nothing. Hence it is important for a trajectory to keep on increasing its radial distance even after exiting the $\epsilon$-neighborhood of any weakly hyperbolic fixed point at least for some significant radial distance $\xi \gg \epsilon$ so as to minimize the chances of return. Our goal is to obtain such radius $\xi$. The next lemma shows the existence of such radius and also its dependence on $\Lambda$.
\begin{lemm}\label{xilemma}
Suppose the trajectory $\{\u_k\}$ generated from \eqref{exittime1} as follows $$ \u_{k+1} = \Lambda \u_k + \B_k \u_k + \norm{\u_k} \mathbf{P}(\u_k) \u_k,$$ satisfies Condition \ref{conditionsdynamicsa} with $\sup_{k\geq K}\norm{\B_k}_F = o(\norm{\u_{K}}) $ for some sufficiently large $K$, $\sup_{\norm{\u_{k}}\leq R}\norm{\mathbf{P}(\u_k)} $ is bounded for every constant $R$ and that $\mathbf{0}$ is a weakly hyperbolic fixed point of the dynamical system \eqref{exittime1}. Also, suppose $\{\u_k\}$ has non-contractive dynamics at $k=K$, i.e., $ \norm{\u_{K+1}} \geq \norm{\u_K}$, $\norm{\u_K} >0$ and no eigenvalue of $\Lambda^H\Lambda $ is equal to $1$. Then $ \norm{\u_{k+1}} > \norm{\u_k} $ holds for all $k>K$, i.e., $\{\u_k\}_{k>K}$ has expansive dynamics as long as $\norm{\u_k}\leq \xi$ for any $\xi$ that satisfies:
\begin{align}
    \xi & {<}  { C\norm{\bigg(\Lambda^H\Lambda - \mathbf{I}\bigg)^{-2}}_2^{-\frac{1}{2}} } ,
\end{align}
with some constant $C>0$. In case if any eigenvalue of $\Lambda^H\Lambda $ is equal to $1$ then $\xi =0$.
\end{lemm}
The proof of this lemma is given in Appendix \ref{local analysis appendix}. The exact value of the constant $C$ has been omitted from Lemma \ref{xilemma} for sake of brevity (see Appendix \ref{local analysis appendix} for details). Note that it may be the case that for certain $\Lambda$ we have one of the eigenvalues of $\Lambda^H\Lambda $ equal to $1$ which will set the radius $\xi = 0$ thereby rendering the linear exit time from Theorem~\ref{exittimethm1} useless as monotonicity of the sequence $\{\u_k\}$ cannot be guaranteed even after fast escape from $\epsilon$ ball for $\epsilon \ll 1$. 

Notice that as a consequence of Lemma \ref{xilemma} there cannot be any other fixed point of the dynamical system \eqref{exittime1}, except $\mathbf{0}$, in the ball $ \mathcal{B}_{\xi}(\mathbf{0})$. If there was another fixed point, say $\p \neq \mathbf{0}$ with $\norm{\p}<\xi$, then for $\u_k = \p$ for any $k=K$, it must be that $ \u_{K+1} = \u_{K+2} = \p$ by the definition of fixed point iteration. But that contradicts the monotonicity property of the sequence $\{\u_k\}_{k>K}$ from Lemma \ref{xilemma}.  

 {Having developed an analytical machinery that can estimate the exit time of trajectories of a complex dynamical system from a neighborhood of any weakly hyperbolic fixed point, we can analyze the dynamics of many first order optimization algorithms (deterministic) locally around strict saddle points of nonconvex functions by transforming their dynamics from some real Euclidean space to a complex vector space. In particular, since this work deals with a class of accelerated gradient methods, the next section analyzes this class of algorithms from the lens of a complex dynamical system and finds its exit time from strict saddle neighborhoods.}

\section{Exit time analysis of \eqref{ds1} trajectories from strict saddle neighborhoods}\label{sectiondynamicformula}
Recall that from Section \ref{intuitioncomplex} the update in \eqref{ds1} expressed as:
\begin{align}
    \begin{bmatrix}
 \x_{k+1}-\x^* \\  \x_{k}-\x^*
\end{bmatrix} & = \begin{bmatrix}
(1+\beta_k)(\mathbf{I}-h\D(\y_k))\hspace{0.5cm} - \beta_k(\mathbf{I}-h\D(\y_k)) \\  \mathbf{I} \hspace{3.5cm} \boldsymbol{0}
\end{bmatrix}\begin{bmatrix}
 \x_{k}-\x^* \\  \x_{k-1}-\x^*
\end{bmatrix},
\end{align}
where $\D(\y_k) = \int_{p=0}^{1} \nabla^2 f(\x^*+p(\y_k-\x^*))dp $, was complexified to obtain the following dynamical system:
\begin{align}
\u_{k+1} & = \Lambda \u_k + \B_k\u_k + 2\epsilon \R(\u_k)\u_k \label{dynamicsystem1},
\end{align}
which is in the standard form from \eqref{exittime1}. Notice that we have $2 \epsilon $ instead of $\epsilon$ in \eqref{dynamicsystem1}, hence our exit time expression will be with respect to the $2 \epsilon$ ball. Next, recall that while deriving the complexified dynamical system \eqref{dynamicsystem1} we assumed diagonalizability in the step \eqref{eigdecomposemetric1} which will be satisfied provided $\lambda_i \neq \frac{4\beta}{(1+\beta)^2} $ for any $i$ where $\lambda_i $ is the $i^{th}$ eigenvalue of $\mathbf{I}-h\nabla^2f(\x^*)$. However we still need to make sure that the conditions \ref{conditionsdynamicsa} and \ref{conditionsdynamicsb} hold before we can use Theorem \ref{exittimethm1} to compute exit time for the dynamical system \eqref{dynamicsystem1}. First, from Condition \ref{conditionsdynamicsa} we require that $ \norm{\B_k}_F \to 0$ which is satisfied since $\B_k = \V^{-1}\C_k \V $ and $\C_k \to \boldsymbol{0}$ as $\beta_k \to \beta$ from the definition of $\C_k$ in \eqref{eigdecomposemetric1}. The condition $\B_k = o(\epsilon) $ for all $k\geq K_{0}$\footnote{Note that we have $k\geq K_{0}$ condition instead of $k\geq {0}$ since we now assume that the trajectory no longer starts around $\x^*$, i.e., the trajectory takes some $K_0$ iterations to reach the local neighborhood of $\x^*$.}, where $K_0$ is the time taken by trajectory $\{\x_k\} $ to reach local neighborhood of $\x^*$, will be satisfied by the fact that $\beta_k \to \beta$ with a rate of $\mathcal{O}(1/k)$ and therefore one can always find some $K_0$ such that $\abs{\beta_k -\beta} = o(\epsilon)$ for all $k\geq K_{0}$. Hence Condition \ref{conditionsdynamicsa} is satisfied. Next we have $2 \epsilon \R(\u_k) = \V^{-1}\M_k \V $ from \eqref{dynamicsystem1} which implies:
\begin{align}
   \R(\u_k)  &= \frac{\V^{-1}\M_k \V}{2\epsilon} \\
   \implies \norm{\R(\u_k)} & \leq \frac{\norm{\V^{-1}}\norm{\M_k} \norm{\V}}{2\epsilon} = \mathcal{O}(\norm{\V^{-1}} \norm{\V})
\end{align}
where we used the definition of $\M_k$ from \eqref{eigdecomposemetric1} and the fact that $\norm{\D(\y_k)-\nabla^2 f(\x^*)}_2 = \mathcal{O}(\epsilon)$. Hence Condition \ref{conditionsdynamicsb} gets satisfied.
Therefore we can now use Theorem \ref{exittimethm1} to estimate the exit time of \eqref{ds1} trajectories from strict saddle neighborhoods. But in order to use Theorem \ref{exittimethm1} we first need to evaluate certain constants appearing in its exit time expression. The next section derives bounds on such constants.

\subsection{Eigenvalues of $\Lambda$ and perturbation parameter $\Gamma$ for complex dynamics of \eqref{ds1}}\label{seceigendiagonal}
Let $\lambda$ represent the eigenvalues of $(\mathbf{I}-h\nabla^2f(\x^*)) $ where $\lambda \in (0,2)$ for $h < \frac{1}{L}$ and $f(\cdot) \in \mathcal{C}^{2,1}_{\mu, L}(\mathbb{R}^n) $. Then the eigenvalues $z_i$ of $\Lambda$ are given by the roots of quadratic:
\begin{align}
    z_i&(z_i - (1+ \beta)\lambda) + \beta\lambda = 0 \\
   \implies z_i &= \frac{(1+\beta)\lambda \pm \sqrt{(1+\beta)^2\lambda^2-4\beta \lambda }}{2}, \label{realeigen}
\end{align}
whenever $\lambda > \frac{4\beta}{(1+\beta)^2} $ and 
\begin{align}
   z_i & = \frac{(1+\beta)\lambda \pm \textbf{\textit{i}}\sqrt{4\beta \lambda -(1+\beta)^2\lambda^2}}{2}, \label{complexeigen}
\end{align}
whenever $\lambda < \frac{4\beta}{(1+\beta)^2} $. The case of $\lambda = \frac{4\beta}{(1+\beta)^2} $ is not included since it results in an incomplete eigenvector basis for $\V$ in \eqref{eigdecomposemetric1}. In particular the case $\lambda = \frac{4\beta}{(1+\beta)^2} $ will not occur $\probP_1$ almost surely by the fact that any eigenvalue of the matrix $(\mathbf{I}-h\nabla^2f(\x^*)) $ for any function $f \in  \mathcal{C}^{2,1}_{\mu,L}(\mathbb{R}^n)$ can be equal to $\frac{4\beta}{(1+\beta)^2} $ for at most finite many values of $h \in (0, \frac{1}{L})$ and since a set of finitely many elements in $\mathbb{R}$ has a zero Lebesgue measure, the conclusion follows.

Now for the complex roots from \eqref{complexeigen}, we have $\abs{z_i} = \sqrt{\beta \lambda}< \frac{2\beta}{(1+\beta)}$ which is less than or equal to $1$ for $\beta \leq 1$. Whenever $\lambda > 1$ (eigenvalues corresponding to unstable subspace of $\nabla^2 f(\x^*)$) we will have $\lambda > 1 \geq \frac{4\beta}{(1+\beta)^2}  $ since $1 \geq  \frac{4\beta}{(1+\beta)^2}   $ for any real $\beta$. Therefore, the case of $\lambda>1$ corresponds to real eigenvalues of $\Lambda$ given by \eqref{realeigen}. For $\lambda <1$ (eigenvalues corresponding to stable subspace of $\nabla^2 f(\x^*)$), both real and complex eigenvalues of $\Lambda$ can occur. Since the expression of exit times from Theorem \ref{exittimethm1} only depends on $\norm{\Lambda}_2$ and $\inf_{i \in \mathcal{N}_{US}}\abs{z_i}$, we do not need to compute every eigenvalue. In particular, for the case when $dim( \mathcal{E}_{US})=1 $ from Theorem \ref{exittimethm1} we only require $\norm{\Lambda}_2$. The next lemma provides an upper bound on $\norm{\Lambda}_2$.  

 {
\begin{lemm}\label{lambdaeig_lemma}
For any function $f \in \mathcal{C}^{2,1}_{\mu,L}(\mathbb{R}^n)$, the largest absolute eigenvalue $\norm{\Lambda}_2$ of the matrix $\Lambda$ from the dynamical system \eqref{dynamicsystem1} is upper bounded as follows:
\begin{align}
    \norm{\Lambda}_2   \leq& \frac{(1+\beta)(1+\mu h)}{2} \bigg(1 + \sqrt{1 - \frac{4 \beta}{(1+\beta)^2(1+\mu h)}} \bigg).
\end{align}
\end{lemm}
The proof of this lemma is in Appendix \ref{local analysis appendix}.
}

 {Note that in order to obtain the exit time estimate from Theorem \ref{exittimethm1} for the dynamics \eqref{dynamicsystem1} we still need the parameter $\Gamma$. Recall from Theorem \ref{exittimethm1} that $\Gamma$ is any positive constant satisfying the condition $ \sup_{\norm{\u_{k}}\leq \epsilon} \norm{\R(\u_k)}_2 \leq \Gamma$. Using this fact we now provide a value for the perturbation parameter $\Gamma$.}
 {
\begin{lemm}\label{lambdagamma_lemma}
For any function $f \in \mathcal{C}^{2,1}_{\mu, L}(\mathbb{R}^n)$ that satisfies Assumption \textbf{A1} of local Hessian Lipschitz continuity, the term $\R(\u_k) $ from the dynamical system \eqref{dynamicsystem1} satisfies the condition $ \sup_{\norm{\u_{k}}\leq \epsilon} \norm{\R(\u_k)}_2 \leq \Gamma$ where $ \Gamma = \frac{ M (1+2 \beta)^2h}{4}$.
\end{lemm}
The proof of this lemma is in Appendix \ref{local analysis appendix}.
}

\subsection{Monotonicity of trajectories generated by \eqref{ds1} and exit time estimates}\label{monotnewds1}
Recall from Section \ref{monotonicsection} that any trajectory of $\{\u_k\}$ in its expansion phase is monotonic inside the ball $\mathcal{B}_{\xi}(\boldsymbol{0})$ after some large $k = K$ provided $\abs{z_i} \neq 1$ where $z_i$ is the $i^{th}$ eigenvalue of the matrix $\Lambda$. 
From the eigenvalues \eqref{realeigen} and \eqref{complexeigen}, it is evident that the case $\abs{z_i} = 1$ can occur when $\lambda = 1$ for real $z_i$ or when $\beta = \frac{1}{\lambda}$ for complex $z_i$ where $\lambda$ is the eigenvalue of the matrix $\mathbf{I} - h \nabla^2 f(\x^*) $. Both these cases occur with zero probability where the probability is with respect to measure $\probP_1$. In particular, the first case has zero probability due to the fact that if some eigenvalue of the matrix $\mathbf{I} - h \nabla^2 f(\x^*)  $ is exactly equal to $1$, then that event can occur only for finitely many $h \in (0, \frac{1}{L})$ and therefore $\lambda \neq 1 $ $\probP_1$-almost surely. Similar argument holds for the second case where the eigenvalue set $\{\lambda \hspace{0.1cm} \vert \hspace{0.1cm} \lambda  =\frac{1}{\beta} \hspace{0.1cm} \} $ for any fixed $\beta$ allows only finitely many choices of the step size $h \in (0, \frac{1}{L})$ and so has $\probP_1$ measure zero. Hence it is safe to say that the trajectory $\{\u_k\}$ is $\probP_1$-almost surely monotonic after leaving $2 \epsilon$ neighbourhood of the weakly hyperbolic fixed point $\mathbf{0}$.

However the same conclusion may not hold for the trajectory of the iterate $\{\x_k\}$. To see this observe that from the relation $\u_k = \V^{-1}\begin{bmatrix}
 \x_{k}-\x^* \\  \x_{k-1}-\x^*
\end{bmatrix} $, we have $\norm{\u_k} \leq \norm{ \x_{k}-\x^* } + \norm{\x_{k-1}-\x^*}$ for $\norm{\V^{-1}}_2 =1$ but this does not conclude that the sequence $\{\norm{ \x_{k}-\x^* } + \norm{\x_{k-1}-\x^*}\} $ or the sequence $\{\norm{ \x_{k}-\x^* }\} $ is necessarily monotonic after escape. Since the monotonicity of the trajectory of $ \{\x_k\}$ cannot be ascertained, the notion of exit time for these trajectories from some local neighborhood of $\x^*$ is a moot concept. In particular, one could somehow argue that working with the \textit{`first exit time'} could be justified but that also is rendered useless due to the fact that the trajectory of $\{\x_k\}$ can possibly return back to the local neighborhood of $\x^*$ soon after exiting due to the absence of monotonicity property. Therefore it becomes imperative to redefine the exit time of the trajectory of $\{\x_k\}$ in a way such that there is some universal agreement on the notion of first exit and subsequent no return, at least for some iterations. 

We begin by formally defining the exit time for the trajectories of $\{\x_k\}$ with respect to a metric $g$ where we have that $g(\cdot, \cdot) :  \mathbb{R}^{d} \times  \mathbb{R}^{d} \rightarrow \mathbb{R}$ and 
\begin{align}
    g(\x,\y) = \langle \x, \y \rangle_{\V^{-1}} = (\V^{-1}\y)^H \V^{-1}\x \label{exitimemetrica}
\end{align}
 for any $\x, \y \in \mathbb{R}^{d}$ where $d=2n$ and the matrix $\V$ comes from \eqref{eigdecomposemetric1}. The metric $g$ induces a norm on $\mathbb{R}^{d} $ given by $ g(\x, \x) = \norm{\x}^2_g = \langle \x, \x \rangle_{\V^{-1}}$. Now recall that from Theorem \ref{exittimethm1} if $\u_0 \in \mathcal{K}_{\sigma} \cap \mathcal{B}_{\epsilon}(\boldsymbol{0})$ where $\mathcal{K}_{\sigma} = \{ \z \in \mathbb{C}^d \hspace{0.1cm}\vert \hspace{0.1cm} \frac{\pi_{\mathcal{E}_{US}}(\z)}{\norm{\z}} \geq \sigma^{\frac{1}{2}}; \hspace{0.1cm} \z \neq \boldsymbol{0} \} $ is the double cone with $0 \ll \sigma < 1$ and $\mathcal{E}_{US} $ is the unstable subspace of $\Lambda$, then for $\epsilon \ll 1$ the approximate trajectory $\{\tilde{\u}_k\}$ exits $ \mathcal{B}_{\epsilon}(\boldsymbol{0})$ in linear time. Further since the relative error condition \eqref{relerror} is satisfied from Section \ref{relerrorsection1}, the exact trajectory $\{{\u}_k\}$ also exits $ \mathcal{B}_{\epsilon}(\boldsymbol{0})$ in approximately linear time.  For the case of our particular dynamical system from \eqref{dynamicsystem1} let the trajectory $\{{\u}_k\}$ enter the $2 \epsilon$ radius ball $ \mathcal{B}_{2\epsilon}(\boldsymbol{0})$ at $k = K_0$ where we have the assumption that $K_0 = \Omega(\epsilon^{-1-a})$ for any $a>0$. Since $\beta_k \to \beta$ with a rate of $\mathcal{O}(1/k)$ we can write $\beta_k = \beta \pm \mathcal{O}(1/k)$ and then it can be easily verified that $ \B_k = o(\epsilon)$ in \eqref{dynamicsystem1} for all $k \geq K_0$. Now at $k = K_0+K_{exit}$ we will get $\norm{\u_{K_0+K_{exit}}}\geq  2 \epsilon $ or equivalently $\norm{\V^{-1}}_2(\norm{ \x_{K_0+K_{exit}}-\x^* }  + \norm{ \x_{K_0+K_{exit}-1}-\x^* }) \geq \norm{\u_{K_0+K_{exit}}} \geq 2 \epsilon$ which implies the following: $$\max \bigg\{\norm{ \x_{K_0+K_{exit}}-\x^* }, \norm{ \x_{K_0+K_{exit}-1}-\x^* }  \bigg\}\geq \frac{\epsilon}{\norm{\V^{-1}}_2},  $$
 for a trajectory $\{\x_k\}$ exiting the ball $\mathcal{B}_{{\frac{\epsilon}{\norm{\V^{-1}}_2}}}(\x^*)$ or equivalently the ball $\mathcal{B}_{\epsilon}(\x^*)$\footnote{Since the matrix $\V$ can be scaled in the eigendecomposition \eqref{eigdecomposemetric1}, we can set $\norm{\V^{-1}}_2 =1$.} at $ k=K_0+K_{exit} $. Therefore using the trajectories of $\{\u_k\}$ entering $\mathcal{B}_{{2\epsilon}}(\boldsymbol{0}) $ from the double cone $\mathcal{K}_{\sigma}$, we define exit time of the trajectories of $\{\x_k\}$ in the metric $g$ for $\epsilon \ll 1$ as follows:\footnote{Observe that in \eqref{exitimeformal} we use `$K_{exit}(\sigma)$' to denote exit time. This is done so as to distinguish the exit time variable for \eqref{ds1} trajectories from the exit time variable `$ K_{exit}$' used in Theorem~\ref{exittimethm1}.}
\begin{align}
  \hspace{-0.5cm}  K_{exit}(\sigma) = \inf_{K>0}\left\{ K  \hspace{0.1cm} \ \middle\vert  \hspace{0.1cm} \begin{array}{l} \norm{[\x_{K+K_0}-\x^*; \x_{K+K_0-1} - \x^*]}_g \geq 2 \epsilon;  \\   \u_{K_0}   \in \mathcal{K}_{\sigma} \cap \bar{\mathcal{B}}_{{2\epsilon}}(\boldsymbol{0}) \backslash {\mathcal{B}}_{{2\epsilon}}(\boldsymbol{0}) ; \hspace{0.1cm}  {K_0 = \Omega(\epsilon^{-1-a})}\end{array} \right\}. \label{exitimeformal} 
\end{align}
Here $\u_{K_0} = \V^{-1} \begin{bmatrix}
 \x_{K_0}-\x^* \\  \x_{K_0 -1}-\x^*
\end{bmatrix} $ and $ \mathcal{E}_{US}$ is the unstable subspace of $\Lambda$ from \eqref{dynamicsystem1}. The sequence $\bigg\{ \begin{bmatrix}
 \x_{K}-\x^* \\  \x_{K -1}-\x^*
\end{bmatrix} \bigg\} $ is monotonic after escape with respect to metric $g$ since we have $ \norm{\begin{bmatrix}
 \x_{K}-\x^* \\  \x_{K -1}-\x^*
\end{bmatrix}}_g = \norm{\u_K}$ and $\norm{\u_K} $ is monotonic after escape from Section \ref{monotonicsection}. 

In the worst case,\footnote{The case of $ dim(\mathcal{E}_{US}) = 1$ is called the `worst case' because here the saddle escape can occur only from one direction which is the one-dimensional subspace $ \mathcal{E}_{US}$. The remaining complementary subspace $ \mathcal{E}_{S} \bigoplus \mathcal{E}_{C} $ does not allow any expansive dynamics.} we assume that $ dim(\mathcal{E}_{US}) = 1$, i.e., $ \Lambda$ has only one eigenvalue with magnitude strictly greater than $1$. Then for any function $f \in \mathcal{C}^{2,1}_{\mu, L}(\mathbb{R}^n)$ that satisfies Assumption \textbf{A1} of local Hessian Lipschitz continuity, substituting the bound on $\norm{\Lambda}_2 $ from Lemma \ref{lambdaeig_lemma} and value of $\Gamma$ from \eqref{lambdagamma_lemma} in the expression of exit time from Theorem \ref{exittimethm1} for $2\epsilon$ when ${\inf_{i  \in \mathcal{N}_{US}}\abs{z_i}} = {\norm{\Lambda}_2} $ we get
\begin{align}
K_{exit}(\sigma) \lessapprox &  \underbrace{\frac{\log\bigg( { \log\bigg(1-\frac{1}{\norm{\Lambda}_2}\bigg)^{-1}}\bigg)}{2\log\bigg(\bigg(1-\frac{1}{\norm{\Lambda}_2}\bigg)^{-1}\bigg)}}_{<1} + \frac{\log\bigg( \frac{1}{\epsilon \norm{\Lambda}^{-1}_2 \Gamma}\bigg)}{2\log\bigg(\bigg(1-\frac{1}{\norm{\Lambda}_2}\bigg)^{-1}\bigg)} \\    
 \implies K_{exit}(\sigma)  \lessapprox & 1 +\frac{\log\bigg( \frac{ \frac{(1+\beta)(1+\mu h)}{2} \bigg(1 + \sqrt{1 - \frac{4 \beta}{(1+\beta)^2(1+\mu h)}} \bigg) }{\epsilon  \frac{ M (1+2 \beta)^2h}{4} }\bigg)}{2\log\bigg(\bigg(1-\frac{1}{ \frac{(1+\beta)(1+\mu h)}{2} \bigg(1 + \sqrt{1 - \frac{4 \beta}{(1+\beta)^2(1+\mu h)}} \bigg) }\bigg)^{-1}\bigg)} 
 \end{align}
 \begin{align}
 \implies K_{exit}(\sigma)  \lessapprox & 1 +\frac{\log\bigg( \frac{{4(1+\beta)(1+Lh)} }{\epsilon  { M (1+2 \beta)^2 h} }\bigg)}{2\log\bigg(\bigg(1-\frac{1}{(1+\beta)(1+\mu h)  }\bigg)^{-1}\bigg)} \label{exittimetradeoffbound},
\end{align}
 where in the last step we used the bound $ \sqrt{1 - \frac{4 \beta}{(1+\beta)^2(1+\mu h)}} < 1$ in both numerator and denominator.
 This exit time bound holds for the exact trajectory of $\bigg\{ \begin{bmatrix}
 \x_{K}-\x^* \\  \x_{K -1}-\x^*
\end{bmatrix} \bigg\} $ in metric $g$ by Theorem \ref{exittimethm1} since the relative error condition from \eqref{relerror} gets satisfied by the definition of $K_{exit}(\sigma)$. 

\subsubsection{{Discussion of the exit time bound \eqref{exittimetradeoffbound}}}\label{sssec:discussion.exit.time}
{Several remarks are in order concerning the exit time bound \eqref{exittimetradeoffbound}. First, it is important to remind the reader that this bound corresponds to fixed values of the function parameters for any nonconvex function \( f \in \mathcal{C}^{2,1}_{\mu, L}(\mathbb{R}^n) \) that satisfies Assumption~\textbf{A1}, as well as fixed values of the step size \( h \) and the limiting momentum parameter \( \beta \). Second, to the best of our knowledge, this result is the first in the literature to establish a strict saddle escape rate of \(\mathcal{O}(\log(\epsilon^{-1}))\) for the trajectories of a class of accelerated methods \eqref{ds1}, for nonconvex functions that are not necessarily quadratic (see also the discussion in Section~\ref{sec:intro} in relation to \cite{o2019behavior}). Prior to this, for the class of nonconvex, smooth Morse functions---beyond quadratics---a \(\mathcal{O}(\log(\epsilon^{-1}))\) exit time bound for gradient descent from a strict saddle point neighborhood was first derived using a trajectory-based analysis in our earlier works \cite{dixit2020exit,dixit2021boundary}. These results substantially advanced the understanding of the behavior of first-order methods near strict saddle points by extending insights from works such as \cite{jin2017escape}, which established convergence rates for a perturbed version of gradient descent to \(\epsilon\)-second-order stationary points with high probability, but did not explicitly analyze the trajectories \( \mathbf{x}_k \) near strict saddles.}\looseness=-1

{Third, it is worthwhile to understand the relationship between the lower bound on the asymptotic divergence metric \( \mathcal{M}^*(f) \), derived in Theorem~\ref{metricedivergethm} in Section~\ref{metricsection}, and the upper bound on the non-asymptotic exit time metric \( K_{\text{exit}}(\sigma) \) provided in \eqref{exittimetradeoffbound}. To this end, we compare the two bounds numerically in Figure~\ref{fig10}(a) and Figure~\ref{fig10}(b), respectively, presented as heatmaps over varying values of the asymptotic momentum parameter \( \beta \) and step size \( h \), with fixed values of \( \epsilon \), \( \mu \), and \( L \). While the bounds are not on the same scale due to different underlying constants, their order-wise behavior reveals an inverse relationship: as \( \beta \) and \( h \) increase, the lower bound on the divergence metric increases, while the upper bound on the exit time decreases. This suggests that both metrics---one asymptotic in \( \mathbb{R}^n \), the other non-asymptotic in \( \mathbb{R}^{2n} \)---capture complementary aspects of the trajectory behavior of momentum-based accelerated methods around strict saddle points.}\looseness=-1

\begin{figure}[t]
\centering
\begin{tabular}{c}
   \includegraphics[height=2.45in]{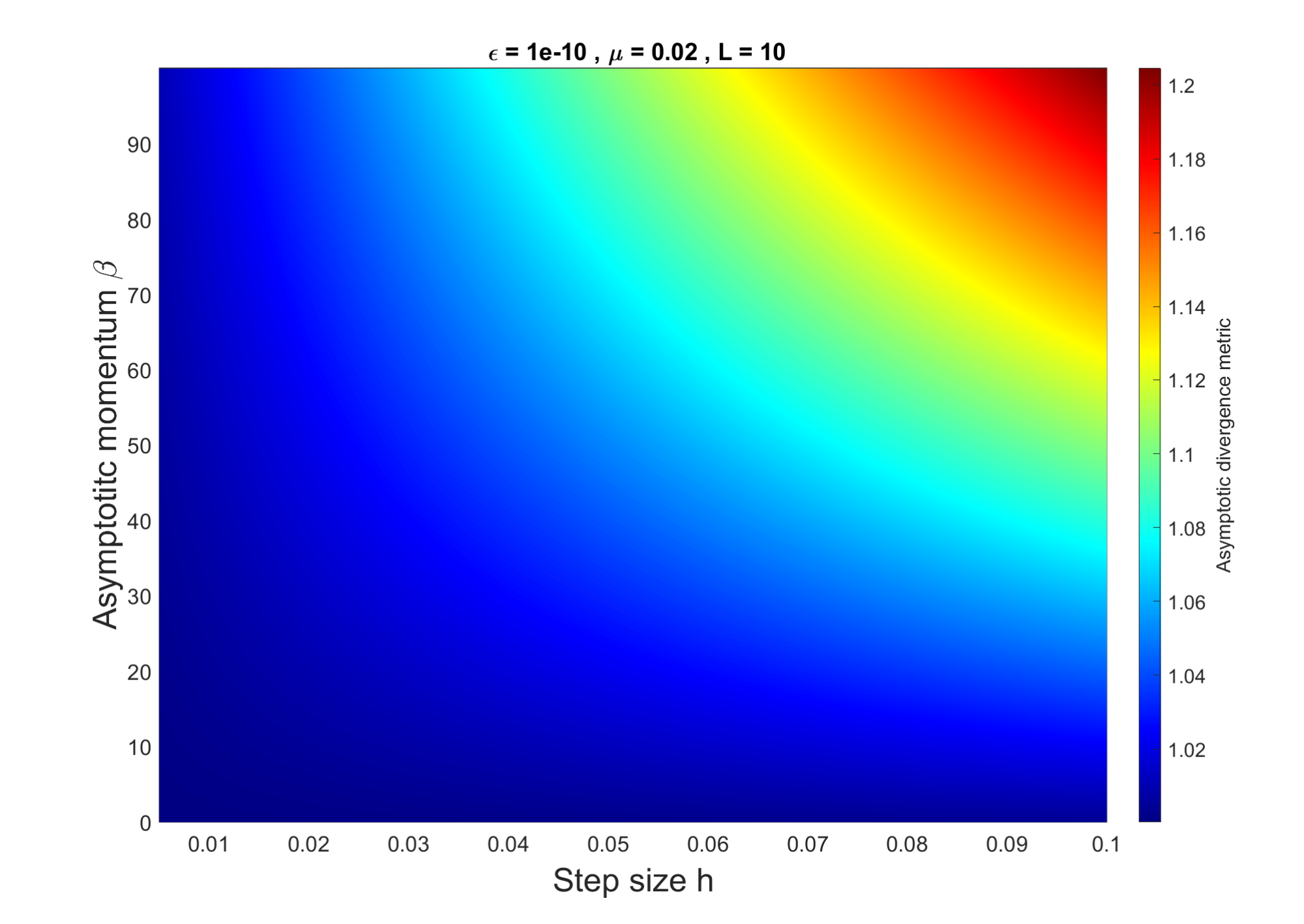} \\
    (a) \\
    \includegraphics[height=2.45in]{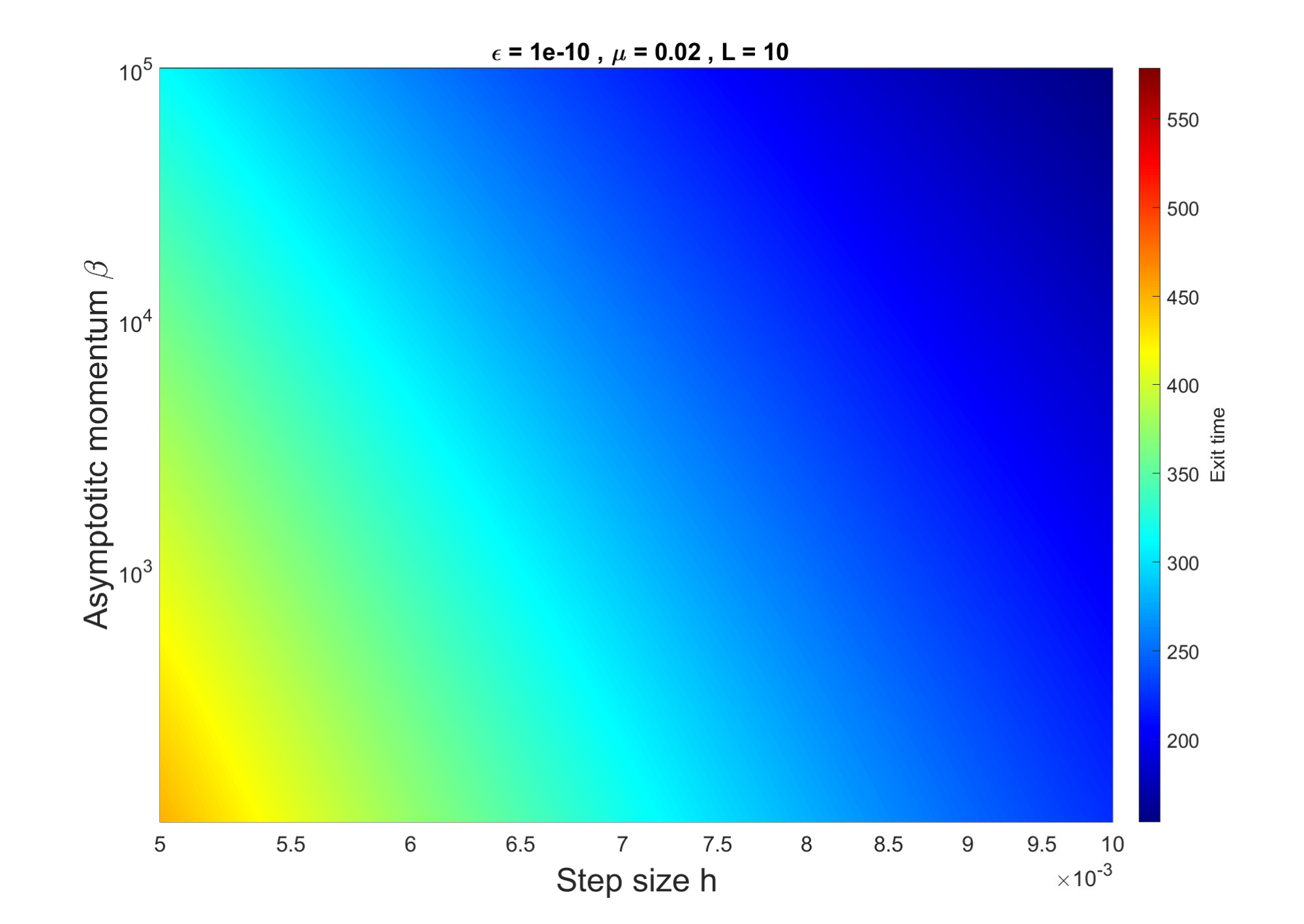} \\
    (b) 
\end{tabular}
\caption{{Heatmaps of (a) the lower bound on the asymptotic divergence metric \( \mathcal{M}^{\star}(f) \) from Theorem~\ref{metricedivergethm} and (b) the upper bound on the exit time metric \( K_{\text{exit}}(\sigma) \) from \eqref{exittimetradeoffbound}, shown as functions of the asymptotic momentum parameter \( \beta \) and step size \( h \). While the scales differ across the two plots, the bounds exhibit an order-wise inverse relationship: as \( \beta \) and \( h \) increase, the lower bound on \( \mathcal{M}^{\star}(f) \) increases, whereas the upper bound on \( K_{\text{exit}}(\sigma) \) decreases.}
}
\label{fig10}
\end{figure}

{This inverse relationship, observed in Figure~\ref{fig10}, can also be shown rigorously in the quadratic setting for gradient descent, where the iteration matrix is symmetric and constant. In that case, the exit time is governed by the spectral radius of the iteration matrix, which captures the fastest rate of expansion of the iterates, while \( \mathcal{M}^*(f) \)---as defined in \eqref{asymptot1} in terms of the norm of the iterates---is determined by the operator norm of the same matrix. Due to the symmetry of the iteration matrix, the spectral radius coincides with the operator norm, allowing a direct comparison of the two metrics. However, the analysis becomes more difficult for momentum-based accelerated methods, even in the quadratic setting. Although the dynamics with constant momentum \( \beta \) on quadratics can still be described by a constant iteration matrix in the augmented space \( \mathbb{R}^{2n} \), as discussed in Section~\ref{counterexsecrev}, these matrices are generally non-symmetric. As a result, the spectral radius and the operator norm---which govern \( K_{\text{exit}}(\sigma) \) and \( \mathcal{M}^*(f) \), respectively---no longer coincide and lack a simple analytical relationship. The situation becomes even more complex for non-quadratic functions, where the iteration matrix becomes state-dependent and nonlinear, making spectral analysis particularly challenging. For these reasons, a rigorous comparison between \( \mathcal{M}^*(f) \) and \( K_{\text{exit}}(\sigma) \) remains difficult beyond the quadratic gradient descent setting. Nevertheless, the numerical comparisons presented in Figure~\ref{fig10} suggest that these two metrics continue to exhibit an inverse relationship even for momentum-based accelerated methods.}

{Finally, it is important to understand how the limiting momentum parameter \( \beta \) affects the exit time bound \eqref{exittimetradeoffbound}; in particular, this sheds light on the potential advantages of momentum-based accelerated gradient methods over vanilla gradient descent in escaping strict saddle points. These effects can be investigated along two different axes: the first concerns the allowable projection of the initial iterate onto the unstable subspace as a function of \( \beta \), and the second concerns the exit time itself as a function of \( \beta \). It turns out that in both aspects, under the assumption that \( \epsilon \) is sufficiently small and \( \beta = \mathcal{O}(1) \) as \( \epsilon \to 0 \), having a larger \( \beta \) is beneficial for escaping strict saddles, as discussed below and as observed in the numerical results reported in Section~\ref{numericalsection}.}

{Regarding the initial unstable projection, note that the exit time bound in \eqref{exittimetradeoffbound} was derived under the condition that \( \inf_{i \in \mathcal{N}_{US}} |z_i| = \norm{\Lambda}_2 \) in the complex dynamics. 
Consequently, this exit time bound remains valid under a significantly weaker condition on \( \u_{K_0} \), as given in Lemma~\ref{lemmaprojectioncondition}. Specifically, it suffices to require:
\begin{align}
 \left(\frac{\pi_{\mathcal{E}_{US}} (\u_{K_0})}{\norm{\u_{K_0}}}\right)^2 = \sum\limits_{i \in \mathcal{N}_{US}} |\theta_i|^2 \geq \Theta \left( \frac{\epsilon}{\norm{\Lambda}_2 \log (\norm{\Lambda}_2)} \log \left( \frac{\norm{\Lambda}_2 \log (\norm{\Lambda}_2)}{\epsilon} \right) \right),
\end{align}
in contrast to the stronger requirement \( \left(\frac{\pi_{\mathcal{E}_{US}} (\u_{K_0})}{\norm{\u_{K_0}}}\right)^2 \geq \sigma \) needed for Theorem~\ref{exittimethm1} and the validity of~\eqref{exittimetradeoffbound}, where \( 0 \ll \sigma < 1 \).
Now, since \( \norm{\Lambda}_2 = \sup_i |z_i| = \Theta(\beta) \) from \eqref{realeigen}, and since
\[
\Theta \left( \frac{\epsilon}{\norm{\Lambda}_2 \log (\norm{\Lambda}_2)} \log \left( \frac{\norm{\Lambda}_2 \log (\norm{\Lambda}_2)}{\epsilon} \right) \right) \downarrow 0 \quad \text{as } \epsilon \to 0,
\]
by Lemma~\ref{lemmaprojectioncondition}, and noting that small \( \epsilon \) allows for large values of \( \norm{\Lambda}_2 \), or equivalently \( \beta \gg 1 \), we can conclude that a larger momentum \( \beta \) permits a much smaller required initial projection \( \left(\frac{\pi_{\mathcal{E}_{US}} (\u_{K_0})}{\norm{\u_{K_0}}}\right)^2 \). Furthermore, it achieves a sharper exit time bound from \eqref{exittimetradeoffbound}, since the upper bound on exit time decreases with increasing \( \beta \), as discussed next. This advantage, however, comes at the cost of requiring smaller values of \( \epsilon \). We refer the reader to Figure~\ref{fig11abc} in Section~\ref{numericalsection} for a numerical validation of this conclusion concerning how larger momentum allows for smaller initial projections onto the unstable subspace.}}

{Last but not least, we examine the exit time itself as a function of the limiting momentum parameter \( \beta \), while keeping all other quantities fixed. It can be argued from the general form of the upper bound in \eqref{exittimetradeoffbound} that increasing \( \beta \) reduces the number of iterations required to exit the strict saddle neighborhood (see also Figure~\ref{fig10}(b)). This observation is also supported by the numerical results reported in Section~\ref{numericalsection} for the phase retrieval and low-rank matrix factorization problems, where larger values of \( \beta \) tend to lead to faster escape.}

{We can also justify this trend analytically by fixing a sufficiently small \( \epsilon \ll 1 \) and taking \( \beta = \mathcal{O}(1) \), and expressing the upper bound in \eqref{exittimetradeoffbound} as a function of \( \beta \):
\[
\tilde{K}(\beta) = 1 + \frac{\log\left( \frac{4(1+\beta)(1+Lh)}{\epsilon M (1+2\beta)^2 h} \right)}{2 \log\left( \left(1 - \frac{1}{(1+\beta)(1+\mu h)}\right)^{-1} \right)}.
\]
We now compute and simplify the derivative of \( \tilde{K}(\beta) \) with respect to \( \beta \) as follows:
\begin{align}
    \frac{d}{d \beta} \tilde{K}(\beta) &= \frac{ \frac{(1+2 \beta)^2}{(1+\beta)} \cdot \frac{(1+2 \beta)^2 - 4(1+2 \beta)(1+\beta)}{(1+2 \beta)^4} \cdot \log\left(\left(1 - \frac{1}{(1+\beta)(1+\mu h)}\right)^{-1}\right)}{2 \left(\log\left(\left(1 - \frac{1}{(1+\beta)(1+\mu h)}\right)^{-1}\right)\right)^2} \nonumber \\
        &\qquad\qquad\qquad - \frac{\left(1 - \frac{1}{(1+\beta)(1+\mu h)}\right) \cdot \frac{1}{(1+\beta)^2(1+\mu h)} \cdot \log\left( \frac{4(1+\beta)(1+Lh)}{\epsilon M (1+2\beta)^2 h} \right)}{2 \left(\log\left(\left(1 - \frac{1}{(1+\beta)(1+\mu h)}\right)^{-1}\right)\right)^2} \\
        &\underbrace{=}_{\text{for } \beta = \mathcal{O}(1)} \frac{ \mathcal{O}(1) - \tilde{c}_1 \log(\epsilon^{-1}) - \tilde{c}_2 }{2 \left(\log\left(\left(1 - \frac{1}{(1+\beta)(1+\mu h)}\right)^{-1}\right)\right)^2}, \qquad \tilde{c}_1 > 0. \label{quantspeedup1*}
\end{align}
In the final step above, we used the following substitutions:
\[
\frac{(1+2\beta)^2}{(1+\beta)} \cdot \frac{(1+2\beta)^2 - 4(1+2\beta)(1+\beta)}{(1+2\beta)^4} \cdot \log\left(\left(1 - \frac{1}{(1+\beta)(1+\mu h)}\right)^{-1}\right) = \mathcal{O}(1),
\]
\[
\left(1 - \frac{1}{(1+\beta)(1+\mu h)}\right) \cdot \frac{1}{(1+\beta)^2(1+\mu h)} \cdot \log\left( \frac{4(1+\beta)(1+Lh)}{\epsilon M (1+2\beta)^2 h} \right) = \tilde{c}_1 \log(\epsilon^{-1}) + \tilde{c}_2,
\]
for \( \beta = \mathcal{O}(1) \), where the constants \( \tilde{c}_1 \) and \( \tilde{c}_2 \) are absolute with respect to \( \epsilon \) and depend only on \( \beta \), \( \mu \), \( h \), and \( L \). The fact that \( \tilde{c}_1 > 0 \) can be verified by a straightforward coefficient comparison. Therefore, for any \( \epsilon \ll 1 \), the term \( \tilde{c}_1 \log(\epsilon^{-1}) \) dominates \( \mathcal{O}(1) - \tilde{c}_2 \), which implies \( \frac{d}{d \beta} \tilde{K}(\beta) < 0 \) from \eqref{quantspeedup1*}. Hence, for \( \beta = \mathcal{O}(1) \) and sufficiently small \( \epsilon \), the upper bound on exit time given by \( \tilde{K}(\beta) \) is locally decreasing in \( \beta \). This provides a quantitative justification for the exit time speedup that results from increasing the momentum parameter~\( \beta \).}

\subsubsection{Comments on the case when $K_0 \ll \Theta(\epsilon^{-1-a})$ in the exit time expression  \eqref{exitimeformal}\\}\label{kboundsection}
Observe that it may not always be the case that the trajectory starts somewhere far from the strict saddle neighborhood $\mathcal{B}_{\epsilon}(\x^*)$ and only approaches this neighborhood after some $K_0 = \Omega(\epsilon^{-1-a})$ iterations. For instance with some non-zero probability one could always initialize the algorithm \eqref{ds1} where $\beta_k \to \beta$ with $\mathcal{O}(1/k)$ rate, on the surface of ball $\mathcal{B}_{\epsilon}(\x^*)$ and then the exit time bound from \eqref{exittimetradeoffbound} will not hold. This particular drawback arises due to the fact that in \eqref{eigdecomposemetric1} we require the matrix $\B_k$ or equivalently the term $\beta_k -\beta$ to be $o(\epsilon)$ which is only possible after some $K_0$ iterations when the momentum sequence $\{\beta_k\}$ has converged to an $\epsilon$ neighborhood of its limit $\beta$. Recall the linearized recursion from \eqref{exittime1} where we had that
\begin{align}
    \u_{k+1} &= \Lambda \u_k + \B_k \u_k + \norm{\u_k} \mathbf{P}(\u_k) \u_k = \Lambda \u_k + \B_k \u_k + \epsilon\R(\u_k) \u_k.
\end{align}
Now if the matrix $\B_k$ is a dominant term then we cannot linearize the complex dynamics\footnote{In order to linearize the dynamics, the state matrix must be independent of $k$, however when the matrix $\B_k$ becomes dominant, the state matrix in \eqref{exittime1} becomes $\Lambda+\B_k$ which depends on $k$.} about $\Lambda$ in \eqref{exittime1} and all the subsequent analysis fails. Hence without the assumption of $K_0 = \Omega(\epsilon^{-1-a}) $ for any $a>0$, we cannot bound the exit time from the saddle neighborhood. However we could still comment on the relation between the escape behavior and the non-asymptotic momentum $\beta_k$ when $K_0 \ll \Theta(\epsilon^{-1-a})$. Modifying \eqref{eigdecomposemetric1} by cancelling the $\beta$ dependent terms after grouping them together and writing \eqref{eigdecomposemetric1} only as a function of $\beta_k$ we get that:
\begin{align}
    \begin{bmatrix}
        \x_{k+1}-\x^* \\  \x_{k}-\x^*
    \end{bmatrix} & = \underbrace{\begin{bmatrix}
(1+\beta_k)(\mathbf{I}-h\nabla^2f(\x^*))\hspace{0.5cm} - \beta_k(\mathbf{I}-h\nabla^2f(\x^*)) \\  \mathbf{I} \hspace{3.5cm} \boldsymbol{0}
\end{bmatrix}}_{=\V_k\Lambda_k \V_k^{-1}}\begin{bmatrix}
 \x_{k}-\x^* \\  \x_{k-1}-\x^*
\end{bmatrix} + \nonumber\\
& \hspace{-0.5cm} \underbrace{\begin{bmatrix}
(1+\beta_k)h(\nabla^2f(\x^*)-\D(\y_k))\hspace{0.5cm} - \beta_k h(\nabla^2f(\x^*)-\D(\y_k)) \\  \boldsymbol{0} \hspace{3.5cm} \boldsymbol{0}
\end{bmatrix}}_{\mathbf{M}_k}\begin{bmatrix}
 \x_{k}-\x^* \\  \x_{k-1}-\x^*
\end{bmatrix} \\
\underbrace{\V_k^{-1} \begin{bmatrix}
 \x_{k+1}-\x^* \\  \x_{k}-\x^*
\end{bmatrix}}_{\u_{k+1}} & = \Lambda_k \underbrace{\V_k^{-1} \begin{bmatrix}
 \x_{k}-\x^* \\  \x_{k-1}-\x^*
\end{bmatrix}}_{\u_k}  +  \underbrace{\V_k^{-1} \mathbf{M}_k\V_k}_{=2\epsilon \R_k(\u_k)} \underbrace{\V_k^{-1}\begin{bmatrix}
 \x_{k}-\x^* \\  \x_{k-1}-\x^*
\end{bmatrix}}_{\u_k} \\
\u_{k+1} & = \Lambda_k \u_k + 2\epsilon \R_k(\u_k)\u_k. \label{commentsec1}
 \end{align}
 Next taking norm squared on both sides in \eqref{commentsec1} and dividing by $\norm{\u_k}^2$ yields:
 \begin{align}
     \frac{\norm{\u_{k+1}}^2}{\norm{\u_k}^2} & =  \frac{ \u_k^H \Lambda_k^H\Lambda_k \u_k }{\norm{\u_k}^2} + \mathcal{O}(\epsilon) = \sum\limits_{\abs{z_i(k)}>1} \abs{z_i(k)}^2 \theta^2_i(k) + \underbrace{\sum\limits_{\abs{z_j(k)}\leq 1} \abs{z_j(k)}^2 \theta^2_j(k)}_{<1} + \mathcal{O}(\epsilon) \label{commentsec2}
 \end{align}
 where $(z_i(k), \v_i(k))$ is the $i^{th}$ eigenvalue-eigenvector pair of the diagonal matrix $\Lambda_k$ and $\u_k = \sum\limits_{\abs{z_i(k)}>1}  \theta_i(k)\v_i(k) + \sum\limits_{\abs{z_j(k)}\leq 1} \theta_j(k) \v_j(k)$ with $ \sum\limits_{\abs{z_i(k)}>1}  \theta^2_i(k) + \sum\limits_{\abs{z_j(k)}\leq 1}  \theta^2_j(k) =1$. Since the matrix $\Lambda_k$ can be obtained from $\Lambda$ just by replacing $\beta$ in the matrix $\Lambda$ with $\beta_k$, the eigenvalues of $\Lambda_k$ can be written directly using \eqref{realeigen}, \eqref{complexeigen} as follows: 
 \begin{align}
 z_i(k) &= \frac{(1+\beta_k)\lambda \pm \sqrt{(1+\beta_k)^2\lambda^2-4\beta_k \lambda }}{2} 
\end{align}
whenever $\lambda > \frac{4\beta_k}{(1+\beta_k)^2} $ and 
\begin{align}
   z_i(k) & = \frac{(1+\beta_k)\lambda \pm \textbf{\textit{i}}\sqrt{4\beta_k \lambda -(1+\beta_k)^2\lambda^2}}{2} 
\end{align}
whenever $\lambda < \frac{4\beta_k}{(1+\beta_k)^2} $ where $\lambda$ represents the eigenvalues of $(\mathbf{I}-h\nabla^2f(\x^*)) $. Clearly $\abs{z_i(k)}$ increases as $\beta_k$ increases. Now if $ \norm{\u_{k+1}} > \norm{\u_k}$, i.e., the trajectory is escaping then the term $\sum\limits_{\abs{z_i(k)}>1} \abs{z_i(k)}^2 \theta^2_i(k) $ on the right hand side of \eqref{commentsec2} has to be strictly greater than $1$. The eigenvalues for the case $\abs{z_i(k)}>1 $ are given by $ z_i(k) = \frac{(1+\beta_k)\lambda + \sqrt{(1+\beta_k)^2\lambda^2-4\beta_k \lambda }}{2}  $ and hence a larger $\beta_k$ will result in a larger ratio $\frac{\norm{\u_{k+1}}^2}{\norm{\u_k}^2} $ from \eqref{commentsec2} thereby improving the escape behavior in terms of iterations. To see this suppose $\frac{\norm{\u_{k+1}}}{\norm{\u_k}} = a_k$ where $a_k>1$, then for any $K>0$ we have $\norm{\u_{K}} = \Pi_{k=0}^{K-1} a_k \norm{\u_{0}}$. Then for any $R>\norm{\u_{0}}$, we will have that $\inf \{K:\norm{\u_{K}} \geq R\} $ is a decreasing function of $\Pi_{k=0}^{K-1} a_k  $. Thus, with larger values in the sequence $\{a_k\}$, it will take fewer iterations for $\norm{\u_{K}} $ to be greater than $R$. The next corollary brings out an important observation from this analysis.

\begin{coro}\label{commentseccorr}
Suppose in the setting of \eqref{ds1} for any function $f \in \mathcal{C}^{2,1}_{\mu, L}(\mathbb{R}^n)$ that satisfies Assumption \textbf{A1} of local Hessian Lipschitz continuity, we have two algorithms namely Algorithm $\mathcal{A}_1$, Algorithm $\mathcal{A}_2$ with different momentum sequences $\{\beta^1_{k}\},\{\beta^2_{k}\}$ respectively and the same step-size $h$ such that $\beta^1_{k} \to \beta $, $\beta^2_{k} \to \beta $ and $\beta^2_k$ dominates $\beta^1_k$, i.e., $\beta^2_{k} > \beta^1_{k} $ for all $k \geq 0$. If both the algorithms are initialized from the same point under the initialization scheme $\x_0 = \x_{-1}$ in an $\epsilon$ neighborhood of the strict saddle point $\x^*$ then in terms of the metric $g$ defined in \eqref{exitimemetrica}, the escape behavior of Algorithm $\mathcal{A}_2$ with respect to number of iterations will be better compared to that of Algorithm $\mathcal{A}_1$ from the strict saddle point $\x^*$.
\end{coro}

\section{Local convergence rates in convex neighborhoods of nonconvex functions for a sub-family of \eqref{generalds}}\label{convexsectionintro}
So far we have explored the exit time and the escape behavior for a class of accelerated methods \eqref{generalds} from strict saddle points and established that an algorithm with a larger asymptotic momentum parameter $\beta$ can decrease the upper bound on the exit time (Theorem \ref{exittimethm1} and the bound \eqref{exittimetradeoffbound}) from the strict saddle neighborhood provided the algorithm reaches this neighborhood only after sufficiently large time. We have also established that within a family of algorithms with {the} same asymptotic momentum parameter $\beta$, the algorithm with the dominant momentum sequence exhibits better escape behavior from the strict saddle neighborhood compared to others whose momentum term $\beta_k$ remains small (Corollary \ref{commentseccorr}). We are now interested in finding whether there exists a class of accelerated methods either with a large asymptotic momentum parameter $\beta$ or with a momentum sequence $\{\beta_k\}$ that dominates the momentum sequence $\{\frac{k}{k+3}\}$ of the Nesterov accelerated gradient method \eqref{originalnesterov} in the sense of Corollary \ref{commentseccorr} and yet still this class is able to converge to a local minimum of convex  {neighborhoods of nonconvex functions. The purpose is to increase the momentum $\beta_k$ as much as possible to escape saddle faster, but this could then deteriorate the behavior around local minima. Therefore we will look for a momentum sequence $\{\beta_k\}$ that will achieve convergence rate close to the \eqref{originalnesterov}'s optimal rate of $\mathcal{O}(1/k^2)$ in convex neighborhoods while having larger $\beta_k$ values than \eqref{originalnesterov}. To do so we first present a recent result \cite{apidopoulos2020convergence} that will give us convergence rates for a sub family of accelerated gradient methods \eqref{generalds} on convex functions. In particular, this sub-family corresponds to the case of sub-critical Nesterov update (see details in \cite{apidopoulos2020convergence}). Then we will show that this sub-family indeed preserves the same convergence rate in convex neighborhoods of nonconvex functions.}

 Before stating the next theorem we define the class $\mathcal{S}^{1,1}_{L}(\mathbb{R}^n) $ which represents the class of functions which are convex and $L$-gradient Lipschitz continuous\footnote{Note that $\mathcal{S}^{1,1}_{L}(\mathbb{R}^n) \not\subset \mathcal{C}^{2,1}_{L}(\mathbb{R}^n) $ since the class $ \mathcal{S}^{1,1}_{L}(\mathbb{R}^n)$ contains functions that are convex but not twice continuously differentiable. Also,  $\mathcal{C}^{2,1}_{L}(\mathbb{R}^n) \not\subset \mathcal{S}^{1,1}_{L}(\mathbb{R}^n)  $ since the class $ \mathcal{C}^{2,1}_{L}(\mathbb{R}^n)$ contains gradient Lipschitz functions which are nonconvex.} where $L \geq 0$. Formally, the function class $\mathcal{S}^{1,1}_{L}(\mathbb{R}^n) $ is defined as follows:
  \begin{align*}
   \mathcal{S}^{1,1}_{L}(\mathbb{R}^n)  = \left\{ f : \mathbb{R}^n \rightarrow \mathbb{R}; \hspace{0.1cm} f \in \mathcal{C}^1 \hspace{0.1cm} \ \middle\vert \begin{array}{l}
   \inf_{\substack{\x, \y \in \mathbb{R}^n \\ {\x \neq \y} }}\frac{{\langle \nabla f(\x) - \nabla f(\y), \x -\y \rangle }}{\norm{\x -\y}^2} \geq 0, \\
    L = \sup_{\substack{\x, \y \in \mathbb{R}^n \\ {\x \neq \y} }}\frac{{\norm{\nabla f(\x) - \nabla f(\y)} }}{\norm{\x -\y}}  
  \end{array}\right\}.
  \end{align*}
 
\begin{theo}\label{thmconvexrate}[Adaptation of Theorem 2, Corollary 2 in \cite{apidopoulos2020convergence}]
Let $\{\x_k\}$ be the iterate sequence of the general accelerated method \eqref{generalds} on any function $f(\cdot)$ in the class $ \mathcal{S}^{1,1}_{L}(\mathbb{R}^n)$. Suppose $$\beta_{k} = \frac{k}{k+3-r} ,$$ for $ r \in [0,3)$ and $h \in (0, \frac{1}{L}]$. Then for any minimum $\x^*$ of $f(\cdot)$, the function sequence $\{f(\x_k)\}$ satisfies:
\begin{align}
     f(\x_k) - f(\x^*) &\leq  \frac{C}{\bigg(k + 2-r \bigg)^{(2-\frac{2r}{3})}},
\end{align}
for some constant $C$.
In terms of order notation, this rate of convergence has the following order:
\begin{align}
     f(\x_k) - f(\x^*) &= \mathcal{O} \bigg({k^{-(2-\frac{2r}{3} )}}\bigg).
\end{align}
\end{theo}
 Notice that $r=0$ corresponds to the standard Nesterov acceleration \eqref{originalnesterov} for which we recover the convergence rate of order $\mathcal{O} \bigg( \frac{1}{k^{2}}\bigg) $ for any convex function $f(\cdot)$. We also note that Theorem \ref{thmconvexrate} [Theorem 2, Corollary 2 in \cite{apidopoulos2020convergence}] is proved by introducing an over-relaxation term to the Forward–Backward algorithm \cite{apidopoulos2020convergence} and we omit those details for sake of brevity.\footnote{The original algorithmic update from \cite{apidopoulos2020convergence} consists of a proximal step arising due to the presence of an additive convex non-smooth function. In the absence of such non-smoothness, the proximal step in the algorithm from  \cite{apidopoulos2020convergence} gets omitted and then their algorithmic update matches the \eqref{generalds} with $\beta_{k} = \frac{k}{k+3-r}$.  }

 Therefore for $r \in[0,3)$ and $h \in (0, \frac{1}{L}]$ we have a sub-family of the general acceleration methods \eqref{generalds} that guarantees the convergence rate from Theorem \ref{thmconvexrate} and is given by:
\begin{align}
\tag{\textbf{G-AGM2}}
    \begin{aligned}\label{familyof momentum}
     \y_{k}  &= \x_{k} + \frac{k}{(k+3-r)}(\x_{k} - \x_{k-1}) \hspace{0.3cm}\text{for any fixed }r \in[0,3),\\
     \x_{k+1} & = \y_{k} - h \nabla f(\y_{k}). \\
\end{aligned}
\end{align}
\begin{rema}
Observe that this class of accelerated methods \eqref{familyof momentum} with $\beta_k = \frac{k}{(k+3-r)}$ where $r \in (0,3)$ will dominate the Nesterov accelerated gradient method \eqref{originalnesterov} with $ \beta_k = \frac{k }{(k+3)}$ in the sense of Corollary \ref{commentseccorr}. Hence the family of accelerated gradient methods \eqref{familyof momentum} for $r \in (0,3)$ provably has better escape dynamics compared to the Nesterov accelerated gradient method \eqref{originalnesterov}. 
\end{rema}
Recall that Theorem \ref{thmconvexrate} applies to convex functions.
  {We now show that the class of accelerated methods \eqref{familyof momentum} achieves the convergence rate from Theorem \ref{thmconvexrate} in convex neighborhoods of nonconvex functions. We first define the class $\mathcal{S}^{loc}_{L}(\mathbb{R}^n) $ which represents the class of functions which are locally convex in some open neighborhood of any local minimum of $f$ and $L$-gradient Lipschitz continuous with $L \geq 0$. Let $\mathcal{X}_*$ be the set of local minimum of $f$ and assume that $\mathcal{X}_*$ is compact\footnote{The assumption of compact $\mathcal{X}_*$ is justified while defining the class $\mathcal{S}^{loc}_{L}(\mathbb{R}^n) $ since we are only interested in the local strong convexity of $f$ in some compact set. } with isolated points. Then the function class $\mathcal{S}^{loc}_{L}(\mathbb{R}^n) $ is defined as follows:
  \begin{align*}
  \hspace{-0.5cm} \mathcal{S}^{loc}_{L}(\mathbb{R}^n)  = \left\{ f : \mathbb{R}^n \rightarrow \mathbb{R}; \hspace{0.1cm} f \in \mathcal{C}^1 \hspace{0.1cm} \ \middle\vert \begin{array}{l}
    \inf_{\substack{\x, \y \in \mathcal{B}_{\delta}({\x^*}) \\ {\x \neq \y} }}\frac{{\langle \nabla f(\x) - \nabla f(\y), \x -\y \rangle }}{\norm{\x -\y}^2} \geq 0 \\
   L =   \sup_{\substack{\x, \y \in \mathbb{R}^n\\ {\x \neq \y} }}\frac{{\norm{\nabla f(\x) - \nabla f(\y)} }}{\norm{\x -\y}} 
  \end{array}; \hspace{0.2cm} \begin{array}{l} \text{for some } \delta>0 \\  \text{ and any }  \x^* \in \mathcal{X}_{*}. \\   \end{array}  \right\}.
  \end{align*}
 Note that this class assumes only local convexity around a minima and hence nonconvex functions with local minimum can belong to this function class. Since $\mathcal{X}_*$ is compact with isolated points, $\mathcal{X}_*$ will have finitely many points, {therefore the choice of $\delta>0$ can be the same for every $\x^* \in \mathcal{X}_{*}$. This is a quite general class of functions; for example any $\mathcal{C}^1$ function that is strictly convex around its local minima will reside in $\mathcal{S}^{loc}_{L}(\mathbb{R}^n)$. Also, functions $f$ that admit a non-degenerate Hessian $\nabla^2 f(\x)$ (with a Hessian that has non-zero eigenvalues) around their local minima arise frequently in nonconvex optimization \cite{gao2020breaking} and belong to $\mathcal{S}^{loc}_{L}(\mathbb{R}^n)$. However, there are $\mathcal{C}^1$ functions that do not belong to this class, we provide an example in Section \ref{loc_counterex} of Appendix \ref{local minima analysis appendix}}. We now present a lemma that will be used to derive the convergence rate of \eqref{familyof momentum} to local minimum of any function in the class $\mathcal{S}^{loc}_{L}(\mathbb{R}^n) $.

 \begin{lemm}\label{lemsupnew123}
     Suppose $f: \mathbb{R}^n \to \mathbb{R}$ is $\mathcal{C}^1$ and is locally convex in the ball $\mathcal{B}_{R}(\x^*)$ where $\x^*$ is any local minimum of $f$. Then for any $R_1, R_2$ where $0 < R_1<R_2 < R$ we have that $$\sup_{{\norm{\x-\x^*}}=R_1} f(\x) \leq \sup_{{\norm{\x-\x^*}} = R_2} f(\x) \hspace{0.2cm}; \hspace{0.2cm} \inf_{{\norm{\x-\x^*}}=R_1} f(\x) \leq \inf_{{\norm{\x-\x^*}} = R_2} f(\x)$$ and the inequalities are strict if $f(\cdot)$ is locally strictly convex in the ball $\mathcal{B}_{R}(\x^*)$.
 \end{lemm}
 The proof of this lemma is given in Appendix \ref{local minima analysis appendix}. Using Lemma \ref{lemsupnew123} we show that the rates of convergence from Theorem \ref{thmconvexrate} are preserved under local strict convexity of $f$.

\begin{lemm}\label{convexextensionlemma}
 {Suppose $f(\cdot) \in \mathcal{S}^{loc}_{L}(\mathbb{R}^n) $ is some coercive function with parameter $\delta>0$ as defined in the class $ \mathcal{S}^{loc}_{L}(\mathbb{R}^n)$. Suppose the iterate sequence $\{\x_k\}$ generated by the family \eqref{familyof momentum}, after $K$ iterations, reaches some $\xi$-neighborhood of a local minimum $\x^*$ of $f$, i.e., $\x_K, \x_{K-1} \in \mathcal{B}_{\xi}(\x^*)$, where $\xi \ll \delta$ and for $\frac{\delta}{\xi} > C \gg 1$, $f \vert_{\mathcal{B}_{C\xi}(\x^*)}$ is\footnote{ {Here $ f \vert_{\mathcal{B}_{C\xi}(\x^*)}$ is the restriction of $f$ on the ball $\mathcal{B}_{C\xi}(\x^*) $.}} locally strictly convex, $L$-gradient Lipschitz continuous with $L\geq 0$. Then the sequence $\{f(\x_k)\}_{k \geq K}$ converges to $f(\x^*)$ with the rate given by Theorem \ref{thmconvexrate}. Also, for the gradient descent method, i.e. $\beta_k =0$ and constant $h$, the sequence $\{f(\x_k)\}_{k \geq K}$ converges to $f(\x^*)$ with $\mathcal{O}(1/k)$ rate.}
\end{lemm} 
 The proof of this lemma is given in Appendix \ref{local minima analysis appendix}. 

\begin{table}[H]
\centering
\renewcommand\thempfootnote{\arabic{mpfootnote}}
\begin{minipage}{1\textwidth}
\resizebox{1\columnwidth}{!}
{
\renewcommand{\arraystretch}{1}
\hspace{0cm}
\Rotatebox{0}{
\begin{tabular}{||c c c c ||}
 \hline
  \textbf{Dynamical system} &
 \textbf{Escape behavior} & \textbf{Lower bound on the }  &  \textbf{ Convergence rate in}
 \\ [0.5ex]
  &
 \textbf{ from strict saddle } & \textbf{divergence speed $\mathcal{M}^{\star}(f)$}  & \textbf{strictly convex }
 \\ [0.5ex]
  &
 \textbf{ neighborhood} & \textbf{at strict saddle point}  & \textbf{neighborhoods}
 \\ [0.5ex]
 &  (Corollary \ref{commentseccorr}) & (Theorem \ref{metricedivergethm}) & (Theorem \ref{thmconvexrate},  Lemma \ref{convexextensionlemma})\\ [0.5ex]
 \hline\hline
Gradient descent (GD) method  &  Worst among the listed & $  (1 + \mu h)^2$ &  $f(\x_k) - f(\x^*) = \mathcal{O} \bigg( \frac{1}{k}\bigg)$ \\
with step-size $h \in (0, \frac{1}{L})$ & & & \\
\hline
 \eqref{generaldsconst} with &  Better than GD & $    \frac{\sqrt{L} - \sqrt{\mu}}{\sqrt{L} + \sqrt{\mu}} \bigg((1 + \mu h)^2 -(1 + \mu h)\bigg)$ & $f(\Bar{\x}_k) - f(\x^*) = \mathcal{O} \bigg( \frac{1}{k}\bigg)$\footnote{The convergence rate for the particular constant momentum method over convex functions holds from Theorem 3 in \cite{ghadimi2015global}. However such rate is in terms of the Cesaro mean $ \frac{1}{k+1}\sum\limits_{l=0}^k \x_l$ of the iterate $\x_k$.} 
\\
 $\beta = \frac{\sqrt{L} - \sqrt{\mu}}{\sqrt{L} + \sqrt{\mu}}$  & & $+ (1 + \mu h)^2$ & where $ \Bar{\x}_k = \frac{1}{k+1}\sum\limits_{l=0}^k \x_l$
 \\
\hline
 \eqref{originalnesterov} with &  Better than \eqref{generaldsconst} & $   2(1 + \mu h)^2 - (1+ \mu h)$ &  $f(\x_k) - f(\x^*) = \mathcal{O} \bigg( \frac{1}{k^{2}}\bigg)$ 
\\
 $ \beta_k = \frac{k}{(k+3)} $  & &  &
 \\
\hline
 \eqref{familyof momentum} with  & Better than \eqref{originalnesterov}  & $   2(1 + \mu h)^2 - (1+ \mu h)$ &  $f(\x_k) - f(\x^*) = \mathcal{O} \bigg( \frac{1}{k^{(2-\frac{2r}{3})}}\bigg)$ \\
 $ \beta_k = \frac{k }{(k+3-r)} $ where $r \in [0,3)$ & for $r>0$ &  &
 \\
\hline
 \eqref{ds1} with & Best among the listed  & $   (1+ \mu h)\bigg((1+ r)(1+ \mu h)-  r \bigg) $ &  \xmark \\
 $ \beta_k = \frac{rk}{(k+3-r)}  $ where\footnote{The particular momentum scheme of $ \beta_k = \frac{rk}{(k+3-r)}  $ for $r>1$ is not a standard accelerated algorithm. However this scheme is shown to have far more superior escape behavior on phase retrieval problem when compared to other standard algorithms (see Figure \ref{fig11a} in Section \ref{numericalsection}).}  $r \in (1,3)$ &  &  &
 \\
\hline
\end{tabular}
}
\renewcommand{\arraystretch}{1}
} 
\caption{\small Summary of comparisons between \eqref{familyof momentum} and some related gradient methods on the class of $ \mathcal{C}^{\omega}_{\mu,L}(\mathbb{R}^n) \cap \mathcal{S}^{loc}_{L}(\mathbb{R}^n)$ locally Hessian Lipschitz continuous functions.}
\label{table:2}
\end{minipage}
\end{table}

Table \ref{table:2} provides a comparison between various first order methods based on the different key quantities discussed up to this point such as dominant escape behavior from strict saddle neighborhood (Corollary \ref{commentseccorr}), asymptotic divergence metric $\mathcal{M}^{\star}(f)$ (Theorem \ref{metricedivergethm}) and the convergence rate to a local minimum (Theorem \ref{thmconvexrate}, Lemma \ref{convexextensionlemma}). Since Corollary \ref{commentseccorr} assumes the class of functions $ \mathcal{C}^{2,1}_{\mu,L}(\mathbb{R}^n)$ which are locally Hessian Lipschitz continuous, Theorem \ref{metricedivergethm} assumes the class $ \mathcal{C}^{\omega}_{\mu,L}(\mathbb{R}^n)$ while Lemma \ref{convexextensionlemma} assumes the function class $\mathcal{S}^{loc}_{L}(\mathbb{R}^n) $, for sake of uniformity in comparisons, we assume in Table \ref{table:2} the function class of $ \mathcal{C}^{\omega}_{\mu,L}(\mathbb{R}^n) \cap \mathcal{S}^{loc}_{L}(\mathbb{R}^n)$ locally Hessian Lipschitz continuous functions.  

{We note that the limiting ODE for the choice of momentum $\beta_k$ in the accelerated scheme \eqref{familyof momentum} can be easily derived using the machinery of continuous time dynamical systems from \cite{su2014differential, attouch2019rate, vassilis2018differential}. The derived ODE lends us insights into the various parallels between the discrete time method \eqref{familyof momentum} and its continuous time counterpart (see Section \ref{ODEsection} in Appendix \ref{local minima analysis appendix}).} 

\vspace{-0.1in}
\section{Rates of convergence to second order stationarity}\label{globalsection}
We are now interested in identifying the class of accelerated gradient methods from \eqref{generalds} which offer convergence guarantees to second order stationary points of any nonconvex function $f(\cdot) \in \mathcal{C}^{2,1}_{L}(\mathbb{R}^n)$ and the rates of convergence associated with this algorithmic class. Recall that the general acceleration method from \eqref{generalds} on any $f(\cdot) \in \mathcal{C}^{2,1}_{L}(\mathbb{R}^n)$ is given by:
\begin{align}
    \y_{k} & = \x_{k} + \beta_k (\x_{k} - \x_{k-1}),  \label{gas1}\\
    \x_{k+1} & = \y_k - h \nabla f(\y_k), \label{gas2}
\end{align}
where $\beta_k$ is some momentum term, $h$ is some step size in $(0, \frac{1}{L}]$. We first need to find a condition on the sequence $\{\beta_k\}$ such that there exists a Lyapunov function which decreases monotonically over the sequence $ \{\x_k\}$ generated by the update \eqref{generalds}. Note that for establishing rate of convergence we need some form of monotonicity condition in order to invoke the monotone convergence theorem. The next lemma provides one such condition for a particular choice of the momentum sequence $\{\beta_k\}$.

\begin{lemm}\label{lemmalyapunov}
Suppose $f \in \mathcal{C}^1$ be any $L$-gradient Lipschitz function, the momentum sequence $\{\beta_k\}$ in \eqref{generalds} for any initialization $[\x_0;\x_{-1}]$ satisfies the condition $ \beta_k \leq \frac{1}{\sqrt{2}}$ for all $k \geq 0$, and we have $h \in (0, \frac{1}{L})$. Then there exists a Lyapunov function $\hat{f}(\cdot) $ given by $ \hat{f}(\x_k) = f(\x_k)+ \frac{\norm{\x_{k-1} -\x_{k} }^2}{2h}$ for all $k \geq 0$ which decreases monotonically with $k$, and under the initialization scheme of $\x_0 = \x_{-1}$ for \eqref{generalds}, we have the following bound on $ \inf_{0\leq k \leq K-1}\norm{ \nabla f(\y_k)}^2 $:
\begin{align}
    \inf_{0\leq k \leq K-1}\norm{ \nabla f(\y_k)}^2 & \leq \frac{{f}(\x_0) - {f}(\x_{K})}{K\bigg(\frac{h}{2}-\frac{L h^2}{2} \bigg)}. \label{gas4a}
\end{align}
If $f$ is also convex then the same Lyapunov function $ \hat{f}(\x_k) = f(\x_k)+ \frac{\norm{\x_{k-1} -\x_{k} }^2}{2h}$ decreases monotonically provided $ \beta_k \leq 1$ for all $k \geq 0$.
\end{lemm}
The proof of this lemma is given in Appendix \ref{local analysis appendix1}. Note that unlike all the previous results developed so far, Lemma \ref{lemmalyapunov} only requires $f$ to be $\mathcal{C}^1$ smooth and no twice continuous differentiability of $f$ is required.

\begin{rema}
    Observe that the Lyapunov function from Lemma \ref{lemmalyapunov} is not a standard Lyapunov function used in the literature \cite{fazlyab2018analysis, can2019accelerated, lessard2022analysis}. However, the more commonly used Lyapunov functions are either designed for convex and non-strongly convex objectives or they are used for continuous time dynamical systems. In contrast, since we are working with a more general $\mathcal{C}^1$ function class, our Lyapunov function is not directly related to those appearing in \cite{fazlyab2018analysis, can2019accelerated, lessard2022analysis}. {The same Lyapunov function  $ \hat{f}(\x_k) = f(\x_k)+ \frac{\norm{\x_{k-1} -\x_{k} }^2}{2h}$ was used to study a perturbed version of Nesterov's method for smooth and Hessian Lipschitz objectives. Here, we do not require Hessian Lipschitz condition and focus on \eqref{generalds} methods.}  
\end{rema}

Using Lemma \ref{lemmalyapunov} we can obtain convergence rate for the class of general accelerated methods \eqref{generalds} with $ \beta_k \leq \frac{1}{\sqrt{2}}$ for all $k \geq 0$ to any first order stationary points of smooth nonconvex functions. In particular under certain mild assumptions on the function class, such as assuming coercivity of $f(\cdot)$, i.e., $\lim_{\norm{\x}\to \infty}f(\x) = \infty$ and invertibility of the Hessian of $f(\cdot)$ at its critical points\footnote{By definition, Morse functions satisfy the invertibility assumption of the Hessian at their critical points and since Morse functions are dense in $\mathcal{C}^2$ functions \cite{matsumoto2002introduction} we are not giving up too much by working with this function class.}, we can even guarantee that the derived rate of convergence is for some local minimum of $f(\cdot)$ almost surely. The next section presents such second-order convergence guarantees. 

\subsection{Convergence guarantees for coercive Morse functions}\label{final_extension section}
The next two theorems establish the almost sure convergence to a local minimum and the rate of convergence to such local minimum respectively, under a given momentum scheme for the class of accelerated gradient methods \eqref{generalds}. Throughout this section we will assume that the iterates are always initialized in a compact set.
\begin{theo}\label{supptheoremnew}
Suppose $f(\cdot) \in \mathcal{C}^{2,1}_{L}(\mathbb{R}^n)$ is a coercive Morse function that is Hessian Lipschitz continuous in every compact set. Let $\mathcal{T}$ be the set of critical points of $f(\cdot)$. Also let $\{\x_k\}$ be the iterate sequence from the general accelerated method \eqref{generalds} with any initialization of $[\x_{0};\x_{-1}]$ where $\x_{-1},\x_0 \notin \mathcal{T}$ and we have $\beta_k < \infty$ for all $k $, $\beta_k \leq \frac{1}{\sqrt{2}}$ for all $k \geq \tilde{K}$ where $\tilde{K} \geq 0 $, $\beta_k \to \beta$ with $h \in (0, \frac{1}{L})$. Then as $k \to \infty$ we have $[\x_k; \x_{k-1} ] \to [\x^*;\x^*] $ $\probP$-almost surely where $\x^*$ is a local minimum of $f(\cdot)$. 
\end{theo}
The proof of this theorem is in Appendix \ref{local analysis appendix1}. Theorem \ref{supptheoremnew} in a way generalizes the result from Theorem~\ref{measurethm7} for those accelerated gradient methods within \eqref{generalds} where the momentum sequence $\{\beta_k\}$ is not always less than or equal to $\frac{1}{\sqrt{2}}$ and is allowed to take arbitrary large values. Also note that the proof of Theorem \ref{supptheoremnew} cannot be directly developed using the global convergence theorem (Theorem \ref{thmglob}) since in order to invoke Theorem \ref{thmglob} one requires a continuous function that always decreases on the complement of the fixed point set of the algorithm. However, obtaining such a continuous function, for instance the Lyapunov function from Lemma \ref{lemmalyapunov}, is not always possible when the momentum sequence $\{\beta_k\}$ is larger than $\frac{1}{\sqrt{2}}$ as monotonicity of the derived sequence $\{\hat{f}(\x_k)\}$ may not be preserved for arbitrary momentum. The next theorem derives the rate of convergence to a second order stationary point of \eqref{generalds} for $\beta_k \leq \frac{1}{\sqrt{2}}$.

\begin{theo}\label{thmlipschitzrate}
Suppose $f(\cdot) \in \mathcal{C}^{2,1}_{L}(\mathbb{R}^n)$ is a coercive Morse function that is Hessian Lipschitz continuous in every compact set. Let $\mathcal{T}$ be the set of critical points of $f(\cdot)$. Next, suppose in the general accelerated method \eqref{generalds} with the initialization scheme of $\x_{0}=\x_{-1} \notin \mathcal{T}$, we have $\beta_k \leq \frac{1}{\sqrt{2}}$ for all $k \geq 0$ where $\beta_k \to \beta$. Let $K_{max}$ be the iteration such that $\x_{K_{max}-1}$, $ \x_{K_{max}}$ belong to some $\xi$ neighborhood of a local minimum $\x^*$ of the function $f(\cdot)$ where we have $\inf_{0\leq k\leq K_{max} }\norm{ \nabla f(\y_k)} = \epsilon $ and $\epsilon \leq 3L\xi$. Then for $Lh \ll 1$ we have that:
\begin{align*}
     K_{max} & {=} \mathcal{O}\bigg(\frac{1}{\epsilon^2}\bigg)
\end{align*}
$\probP$-almost surely.
\end{theo}

\begin{proof}
From Theorem \ref{supptheoremnew} we know that the sequence $\{[\x_k;\x_{k-1}]\}$ for the given accelerated method \eqref{generalds} with $\beta_k \leq \frac{1}{\sqrt{2}}$ for all $k \geq 0$ and $h < \frac{1}{L}$ converges to $[\x^*;\x^*]$ $\probP$-almost surely where $\x^*$ is a local minimum of coercive Morse function $f(\cdot) \in \mathcal{C}^{2,1}_{L}(\mathbb{R}^n)$. Let $ K_{max}$ be the iteration such that $\x_{K_{max}-1}$, $ \x_{K_{max}}$ belong to some $\xi$ neighborhood of $\x^*$ where we have $ \norm{\x_{K_{max}-1} - \x^*} \leq \xi$, $ \norm{\x_{K_{max}} - \x^*} \leq \xi$. Then from gradient Lipschitz continuity of $f$ and \eqref{gas1} we get:
\begin{align}
    \norm{\nabla f(\y_{K_{max}})} & = \norm{\nabla f(\y_{K_{max}})- \nabla f(\x^*)} \leq  L \norm{\y_{K_{max}} - \x^*} \\
    & \leq (1 + \beta_k)L\norm{\x_{K_{max}} - \x^*} + \beta_k L\norm{\x_{K_{max}-1} - \x^*} \\
    & < 2L\norm{\x_{K_{max}} - \x^*} +L\norm{\x_{K_{max}-1} - \x^*} \leq 3L \xi
\end{align}
where we used $\beta_k  \leq \frac{1}{\sqrt{2}} <1$. Hence we have $\epsilon= \inf_{0\leq k\leq K_{max}}\norm{ \nabla f(\y_k)}\leq 3 L\xi $. Then using \eqref{gas4a} for $K= K_{max} + 1$ and $Lh \ll 1$ we get:
\begin{align}
    \inf_{0\leq k \leq K_{max}}\norm{ \nabla f(\y_k)}^2 & \leq \frac{{f}(\x_0) - {f}(\x_{K})}{(K_{max}+1)\bigg(\frac{h}{2}-\frac{L h^2}{2} \bigg)} \\
    \implies K_{max} & {=} \mathcal{O}\bigg(\frac{1}{\epsilon^2}\bigg) 
\end{align}
$\probP$-almost surely.
\end{proof}

{It is worth comparing Theorem~\ref{thmlipschitzrate} to a related result in~\cite{jin2017accelerated}, where a perturbed accelerated gradient method is proposed and shown to achieve a complexity of \( \tilde{\mathcal{O}}(\epsilon^{-7/4}) \), with \( \tilde{\mathcal{O}}(\cdot) \) hiding poly-logarithmic dependence on \( \epsilon^{-1} \). Although this rate is faster than the \( \mathcal{O}(\epsilon^{-2}) \) complexity established in Theorem~\ref{thmlipschitzrate} for G-AGM, the difference stems from fundamental distinctions between G-AGM and the perturbed method in~\cite{jin2017accelerated}. In regions of negative curvature and away from critical points, the perturbed method explicitly exploits directions of negative curvature to accelerate escape. In contrast, while our Lyapunov function---unlike the Hamiltonian-based Lyapunov function used in~\cite{jin2017accelerated}---is designed to accommodate a broader class of nonconvex landscapes, it does not, by itself, yield the improved complexity of \( \tilde{\mathcal{O}}(\epsilon^{-7/4}) \). Indeed, G-AGM, without any additional components or subroutines to exploit negative curvature, decreases monotonically according to our Lyapunov function at a rate of \( \mathcal{O}(\xi^{-2}) \), where \( \xi \) denotes the smallest gradient magnitude in the region of negative curvature. This behavior is evident from Lemma~\ref{lemmalyapunov}. Therefore, without the incorporation of explicit saddle-escaping subroutines, it is unlikely that G-AGM in its current form can match the improved complexity of \( \tilde{\mathcal{O}}(\epsilon^{-7/4}) \) achieved in~\cite{jin2017accelerated}.}

We now show that the convergence guarantees from Theorem \ref{thmlipschitzrate} can be extended to coercive Morse functions that are not globally gradient Lipschitz continuous. To do so we use a theorem from nonlinear functional analysis \cite{schwartz1969nonlinear}, namely, the Kirszbraun Theorem (\!\!\cite{kirszbraun1934contracting,schwartz1969nonlinear}), which is stated below. 
\begin{theo}[\!\!\cite{kirszbraun1934contracting}]\label{kirszbraunthm}
{
If $U$ is a subset of some Hilbert space $H_1$, and $H_2$ is another Hilbert space, and
    $g : U \rightarrow H_2$
is a Lipschitz-continuous map, then there is a Lipschitz-continuous map
    $G : H_1 \rightarrow H_2$ that extends $g$ and has the same Lipschitz constant as $G$.}
\end{theo}

  {We also need a supporting lemma for showing that the sequence of maps $\{N_k\}$ from Theorem \ref{diffeomorphthm} corresponding to \eqref{generalds} for $\beta_k \leq \frac{1}{\sqrt{2}}$ preserve certain domains for coercive functions.
\begin{lemm}\label{lemsublevelanalytic}
    For any coercive $\mathcal{C}^1$ function $f: \mathbb{R}^n \to \mathbb{R}$ that is locally gradient Lipschitz continuous in every compact set, the sequence of maps $\{N_k\}$ from Theorem \ref{diffeomorphthm} corresponding to \eqref{generalds} with $\beta_k \leq \frac{1}{\sqrt{2}}$ for all $k$ and initialization scheme of $\x_0=\x_{-1} \in \mathcal{D}$, where $\mathcal{D}$ is any sublevel set of $f$, satisfy the property:
    $$ N_k : \mathcal{D}  \rightarrow  \mathcal{D},$$
    for all $k \geq 0$ provided the step-size $h$ is sufficiently small and $h< \frac{1}{L}$ when $ f $ is $L$-gradient Lipschitz continuous on $\mathbb{R}^n$.
\end{lemm}
The proof of this lemma is in Appendix \ref{local analysis appendix1}. Using Lemma \ref{lemsublevelanalytic}, the sequence of maps $\{N_k\}$ can be restricted to any sublevel set of $f$ and then the non-convergence result from Theorem \ref{measuretheorem3} will directly hold in this compact set. The restriction of domains and codomains of all the maps in the sequence $\{N_k\}$ to a given compact set is not at all trivial. Notice that from the recursion $N_k(\x) = G \circ R_k(\x) = G(p_k \x - q_k N_{k-1}^{-1}(\x))$ in Theorem \ref{diffeomorphthm}, for a general nonconvex function $f$ that is not coercive, one could construct counterexamples where for a given compact set $S$, $N_k(\x)$ could be in the exterior of $S$ even if $\x$ lies in the interior or at the boundary of $S$. } 

{Using Theorem \ref{kirszbraunthm} and Lemma \ref{lemsublevelanalytic} we now show that the general accelerated method \eqref{generalds}, with the initialization scheme of $\x_{0} =\x_{-1}$, $\x_0 \notin \mathcal{T}$ and $\beta_k \leq \frac{1}{\sqrt{2}}$ for all $k \geq 0$, converges to a local minimum $\probP$-almost surely and has the same convergence rate as the one in Theorem \ref{thmlipschitzrate} without the gradient Lipschitz boundedness assumption on the function $f(\cdot)$, provided the step size is chosen small enough. In particular, Theorem~\ref{kirszbraunthm} will allow us to explicitly construct an extension of the function $f$ that is globally gradient Lipschitz continuous and is equal to $f$ on some appropriate sublevel set. Then the convergence guarantees from Theorem \ref{thmlipschitzrate} will hold for this extension of $f$.} 
\begin{theo}\label{kirszthmadapted}
{Suppose $f : \mathbb{R}^n \to \mathbb{R}$ is any coercive Morse function that is Hessian Lipschitz continuous on every compact set. Then there exists a finite constant $\tilde{L} >0 $ such that if $h < \frac{1}{\tilde{L}}$, the iterate sequence $\{\x_k\}$ generated by the general accelerated method \eqref{generalds} with the initialization scheme of $\x_{0} =\x_{-1}$, $\x_0 \notin \mathcal{T}$ where $\mathcal{T}$ is the set of critical points of $f(\cdot)$ and $\beta_k \leq \frac{1}{\sqrt{2}}$ for all $k \geq 0$ with $\beta_k \to \beta$, always stays within the compact set $\{\x \hspace{0.1cm} \vert \hspace{0.1cm} f(\x) \leq f(\x_{0})\} $ and the sequence $\{[\x_k;\x_{k-1}]\}$ converges to $[\x^*;\x^*]$ in this compact set $\probP$-almost surely where $\x^*$ is a local minimum. Furthermore, for $\beta_k \leq \frac{1}{\sqrt{2}}$ for all $k \geq 0$ where $\beta_k \to \beta$, we have that {the} convergence rate given by Theorem \ref{thmlipschitzrate} holds $\probP$-almost surely provided $ {\tilde{L}} h \ll 1$.   }
\end{theo}
The proof of this theorem is given in Appendix \ref{local analysis appendix1}. Theorem \ref{kirszthmadapted} is of particular significance because it generalizes the second order convergence guarantees from Theorem \ref{thmlipschitzrate} to nonconvex functions that are not globally gradient Lipschitz continuous such as the cost functions arising in low-rank matrix factorization \cite{dixit2021boundary}, phase retrieval problem \cite{ma2020implicit, dixit2020exit}, etc. 

\subsection{{Tradeoff between larger $\beta_k$ and convergence guarantees}
}
{Observe that in Theorems~\ref{thmlipschitzrate} and~\ref{kirszthmadapted}, the momentum parameter is constrained to satisfy $\beta_k \leq \frac{1}{\sqrt{2}}$. This restriction arises from the structure of the Lyapunov function introduced in Lemma~\ref{lemmalyapunov}, which is used to establish convergence guarantees. For general gradient Lipschitz functions, momentum methods are known to exhibit oscillatory behavior (see, e.g.,~\cite{JMLR:v17:15-084,7040179}). However, when the momentum is not too large, or when the objective function satisfies additional properties such as convexity, it becomes possible to construct a Lyapunov function, as in Lemma~\ref{lemmalyapunov}, that ensures a form of monotonicity across iterations. This monotonic behavior is crucial in deriving convergence rates.}

{From Lemma~\ref{lemmalyapunov}, we observe that for convex gradient Lipschitz functions, one can construct such a Lyapunov function that remains valid for values of $\beta_k$ up to $1$. In contrast, for general gradient Lipschitz functions, the same construction is valid only under the more conservative condition $\beta_k \leq \frac{1}{\sqrt{2}}$, thereby imposing a stricter upper bound on the momentum parameter to preserve the required monotonicity. Interestingly, when the goal is to escape strict saddle points, more aggressive momentum schemes, i.e., larger values of $\beta_k$, are beneficial, as demonstrated in Corollary~\ref{commentseccorr}. This highlights a fundamental tradeoff: while the less restrictive regime $\beta_k \in \left(\frac{1}{\sqrt{2}}, 1\right]$ can enhance saddle escape behavior, it may no longer support the construction of a Lyapunov function that yields the type of monotonicity needed to prove first-order convergence guarantees for general gradient Lipschitz functions. Since constructing such Lyapunov functions in this momentum regime and for this function class remains an open problem, we leave this direction for future work.}
\vspace{-0.1in}
\section{Numerical results}\label{numericalsection}
{We now provide some numerical results for the various classes of accelerated gradient methods discussed and proposed in this work, applied to well-known nonconvex functions (phase retrieval and low-rank matrix factorization problems), as well as a convex function (quadratic program). Note that the examples of phase retrieval and low-rank matrix factorization are chosen because, although their optimization behavior has been extensively studied in the literature (see references in the sections below), the question of saddle escape rate as a function of the momentum parameter in accelerated methods has remained unexplored for these problems. This, in our view, underscores the value of our work, which, although broadly applicable, provides new insights even for canonical problems when it comes to saddle escape behavior in the context of acceleration. These experiments, in particular, validate our theoretical findings that a higher limiting momentum parameter $\beta$ can enable faster escape from strict saddle neighborhoods, though potentially at the expense of slower convergence to local minima.}

\subsection{Phase retrieval problem}
To support the theoretical framework developed in this work and showcase the effectiveness of accelerated gradient trajectories with higher momentum in escaping from strict saddle neighborhoods, we evaluate the performance of the family of accelerated methods \eqref{generalds} on the phase retrieval problem \cite{candes2015phase}. Briefly, the phase retrieval problem formulation is given by
\begin{align}
    \min_{\x \in \mathbb{R}^n} f(\x) = \frac{1}{4m} \sum\limits_{j=1}^{m}\bigg[\langle\a_j,\x \rangle^2 - y_j \bigg]^2, \label{phaseretrieval}
\end{align}
where the $y_j$'s are known observations and the $\a_j$'s are independent and identically distributed (i.i.d.) random vectors whose entries are generated from a normal distribution. The formulation in \eqref{phaseretrieval} is the least-squares problem reformulation for the Short-Time Fourier Transform (STFT) of the actual phase retrieval problem (see \cite{jaganathan2016stft}). Moreover, the above least-squares reformulation of the original phase retrieval problem can also be found in recent works like \cite{chen2019gradient,ma2020implicit}, which highlight the efficacy of simple gradient descent method on structured nonconvex functions. Clearly, the function in \eqref{phaseretrieval} is locally gradient and locally Hessian Lipschitz continuous, and it also satisfies Assumption \textbf{A1} in every bounded domain. {Before presenting our numerical setup and results, we note that a recent work \cite{maunu2024acceleration} analyzes the convergence of the Nesterov accelerated gradient method on the Gaussian phase retrieval problem and shows that it achieves faster convergence than gradient descent. However, in contrast to our work, \cite{maunu2024acceleration} does not examine the role of the momentum parameter in saddle escape behavior, which is a central focus of our analysis.}

In the simulations, we set $y_j = 1$ for $ 1 \leq j \leq m-1 $ and $y_j = -1$ otherwise. Also, for the sake of simplicity we always set $m=n$ so that the system of equations $y_j = \langle \a_j, \x\rangle^2$ is neither under-determined nor over-determined and the Hessian of the function $f(\cdot)$ is full rank. The i.i.d. Gaussian nature of the $\a_j$'s thus implies that the parameter $\frac{\mu}{L} $ is bounded away from zero almost surely. The closed-form expressions for the gradient and the Hessian of the function in \eqref{phaseretrieval} are, respectively, as follows:
\begin{align}
    \nabla f(\x) &= \frac{1}{m} \sum\limits_{j= 1}^m \bigg(\langle\a_j,\x \rangle^2 - y_j \bigg) \langle\a_j,\x \rangle\a_j,  \\
    \nabla^2 f(\x) &=  \frac{1}{m} \sum\limits_{j= 1}^m \bigg(3\langle\a_j,\x \rangle^2 - y_j \bigg) \a_j\a_j^T.
\end{align}

For the particular choice of $y_j'$s it is observed that $\x^*= \mathbf{0} $ is a strict saddle point. We now initialize the family of accelerated gradient methods under the initialization scheme of $\x_0 =\x_{-1}$ in the $\epsilon$-neighborhood of $\x^*$, i.e., $\norm{\x_0} = \epsilon$ and with the same initial unstable subspace projection value so as to examine the escape behavior of its trajectories for given values of $n,m,\epsilon$. Note that the `projection' of the initial iterate on the unstable subspace corresponds to the quantity $ \sum_{j \in \mathcal{N}_{US}} (\theta_j^{us})^2$, where $ \x_0 - \x^* = \sum\limits_{j \in \mathcal{N}_{US}}\theta_j^{us} \e_j + \sum\limits_{i \in \mathcal{N}_{S}}\theta_i^{s}\e_i$, with $\mathcal{N}_{US}$ and $\mathcal{N}_{S}$ corresponding to the index sets for the negative and positive eigenvalues, respectively, of \( \nabla^2 f(\x^*) \), and $\e_j,  \e_i $ denoting the orthonormal eigenvectors of $\nabla^2 f({\x^*})$.

The results of the simulations are reported in Figure~\ref{fig11} for the step size of $h = 0.1/L$ and in Figure~\ref{fig12} for the step size of $h = 0.01/L$, with $L$ being the largest eigenvalue of $ \nabla^2 f(\mathbf{\x^*})$. Note that each subplot in both of the figures corresponds to different random $\a_j$'s. It is evident from the two figures that, as suggested by the theoretical developments in this paper, a larger momentum results in a faster escape from the strict saddle neighborhood even if we have a small step size $h$. More importantly, Figure \ref{fig12}$(b)$ and Figure \ref{fig11a} corroborate our findings that larger momentum assists in faster saddle escape even under very small initial unstable subspace projections. In particular Figure \ref{fig11a} simulates the momentum scheme \eqref{ds1} with parameters $\beta_k = \frac{rk}{(k+3-r)}$ for $r \in [0,3)$ which is slightly different from the scheme \eqref{familyof momentum}. Though this choice of $\beta_k$ does not offer convergence guarantees from Theorem \ref{thmlipschitzrate} when $r>1$ yet it escapes strict saddle neighborhoods remarkably fast.

\begin{rema}
It should be noted that the definition of projection used in this section {for the numerical experiments} differs from the one used in defining the exit time in \eqref{exitimeformal}, where the projections are taken with respect to the augmented $2n$-dimensional dynamics of $[\x_k - \x^*; \x_{k-1} - \x^*]$. {The reason for working in the $n$-dimensional space is that the augmented $2n$-dimensional space is introduced primarily for analytical purposes, while our goal is to observe how the theoretical guarantees translate into escape behavior in the ambient $n$-dimensional space, where the actual iterates reside.} Nonetheless, {it can be shown that for $\beta \leq 1$}, the projection of $\x_0 - \x^*$ onto the unstable subspace of $\nabla^2 f(\x^*)$ is equivalent to the projection of the vector $\V^{-1}[\x_0 - \x^*; \x_{-1} - \x^*]$ onto the unstable subspace of a matrix $\D$, where $\V \Lambda \V^{-1}$ is the eigendecomposition (which exists $\probP_1$-a.s.) of $\D$, and
$$\D=\begin{bmatrix}
 (1+\beta)(\mathbf{I}-h\nabla^2f(\x^*))\hspace{0.5cm} - \beta(\mathbf{I}-h\nabla^2f(\x^*)) \\  \mathbf{I} \hspace{3.5cm} \boldsymbol{0}
\end{bmatrix}, $$
{assuming that the matrix \( \V \), i.e., the eigenvectors of \( \D \) in the eigendecomposition \( \D = \V \Lambda \V^{-1} \), remain fixed and only the eigenvalues of \( \D \) are allowed to vary.} The equivalence of these projections {under this assumption} is proved in Section \ref{equivprojsec} of Appendix \ref{local analysis appendix1}. {We also note that in the experiments, while generating different functions by varying the ratio \( \mu/L \), the equivalence of the projections appeared to hold numerically even without assuming fixed eigenvectors. However, proving this more generally---when the eigenvectors of \( \mathbf{D} \) are also allowed to vary---is outside the scope of the present work.}
\end{rema}
\begin{figure}[H]
\centering
\begin{tabular}{c}
    \includegraphics[width=0.84\textwidth]{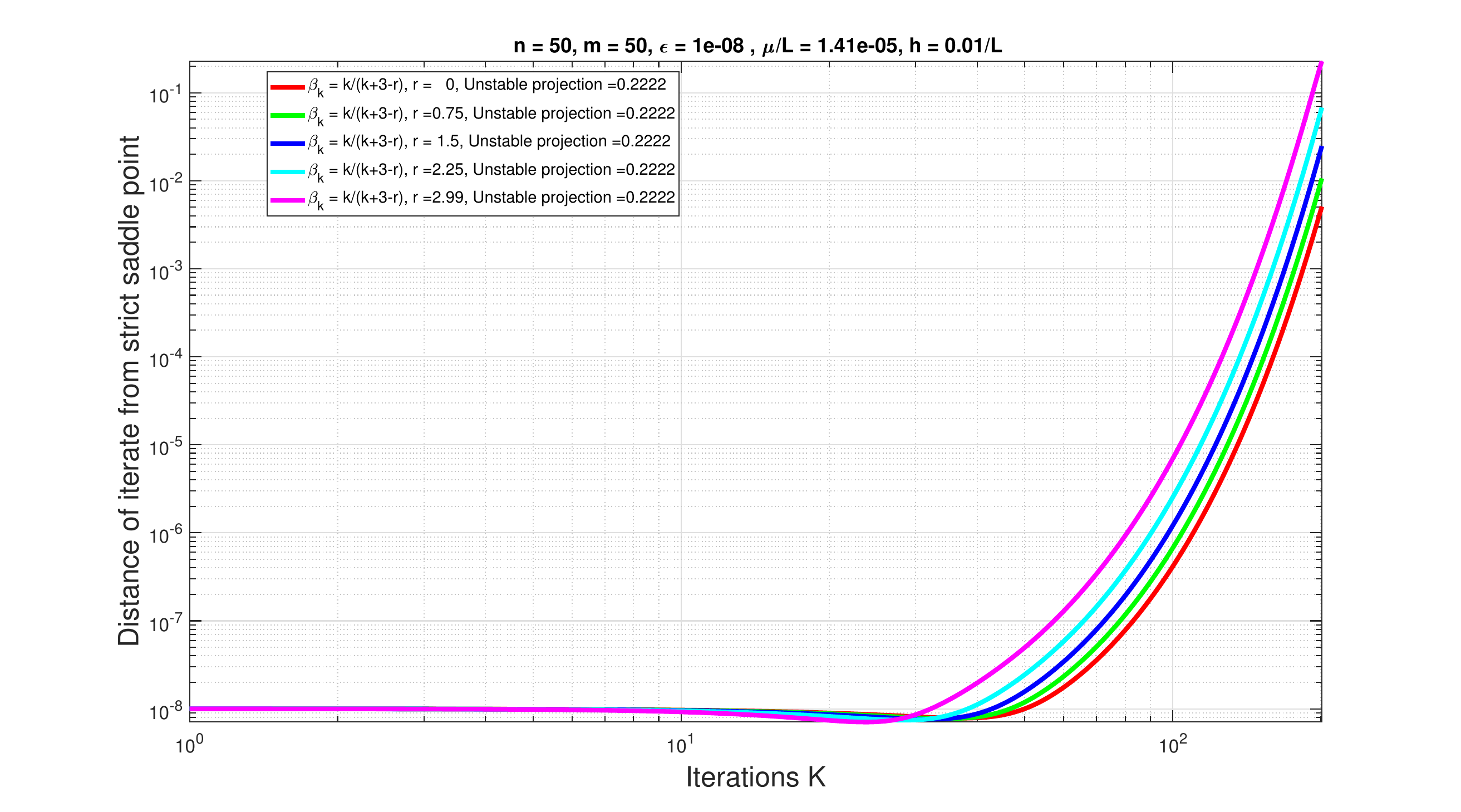} \\
    (a) \\
    \includegraphics[width=0.84\textwidth]{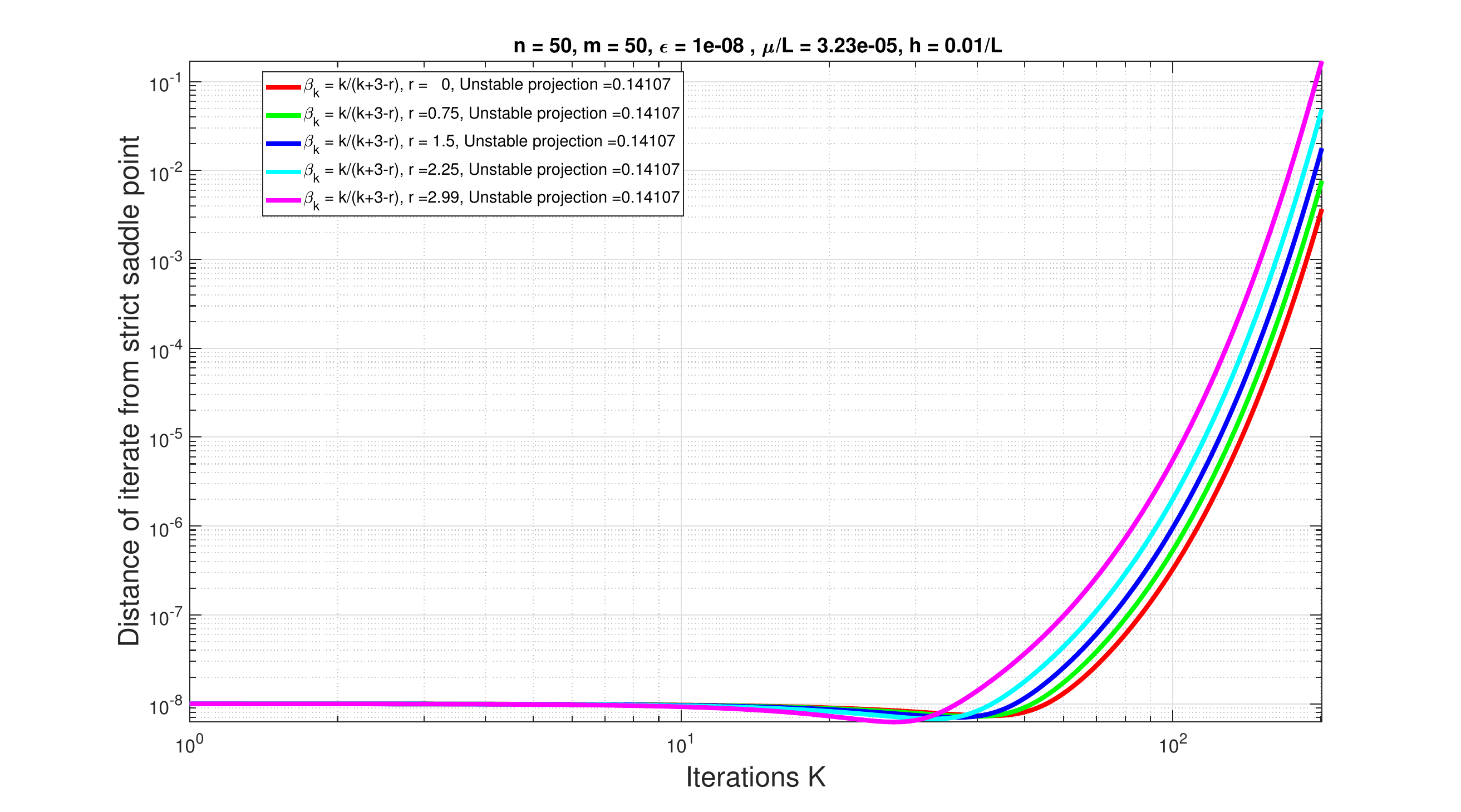} \\
    (b) 
\end{tabular}
\caption{Simulated trajectories of various accelerated gradient methods from the family \eqref{familyof momentum} on the phase retrieval problem with $h = 0.01/L$, under the same initial unstable projections for fixed values of $m$, $n$, and $\epsilon$. The parameter $r$ controls the momentum via $\beta_k = \frac{k}{k + 3 - r}$; $r = 0$ corresponds to the Nesterov accelerated method \eqref{originalnesterov}.}
\label{fig11}
\end{figure}

\begin{figure}[H]
\centering
\begin{tabular}{c}
    \includegraphics[width=0.95\textwidth]{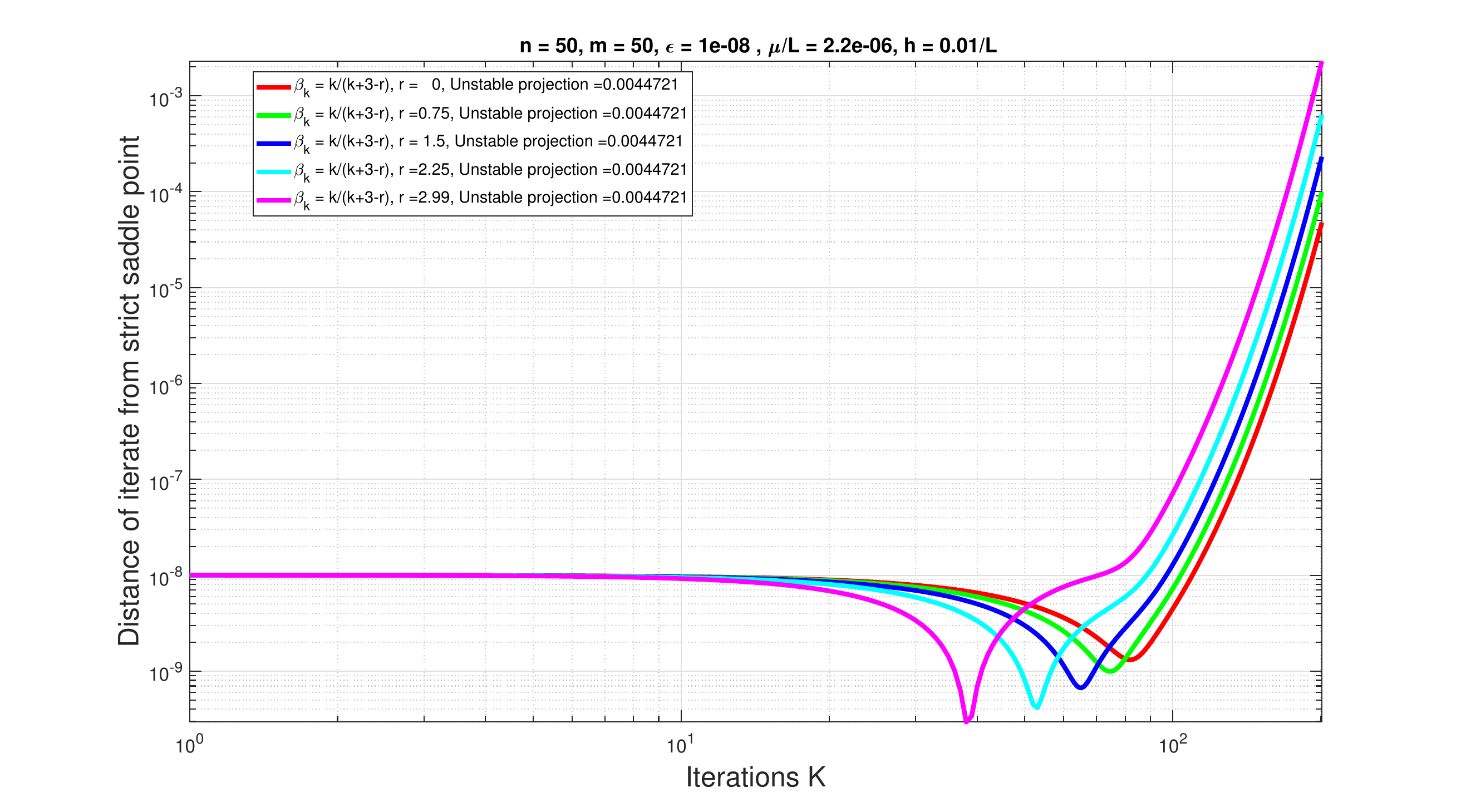} \\ (a) \\\includegraphics[width=0.95\textwidth]{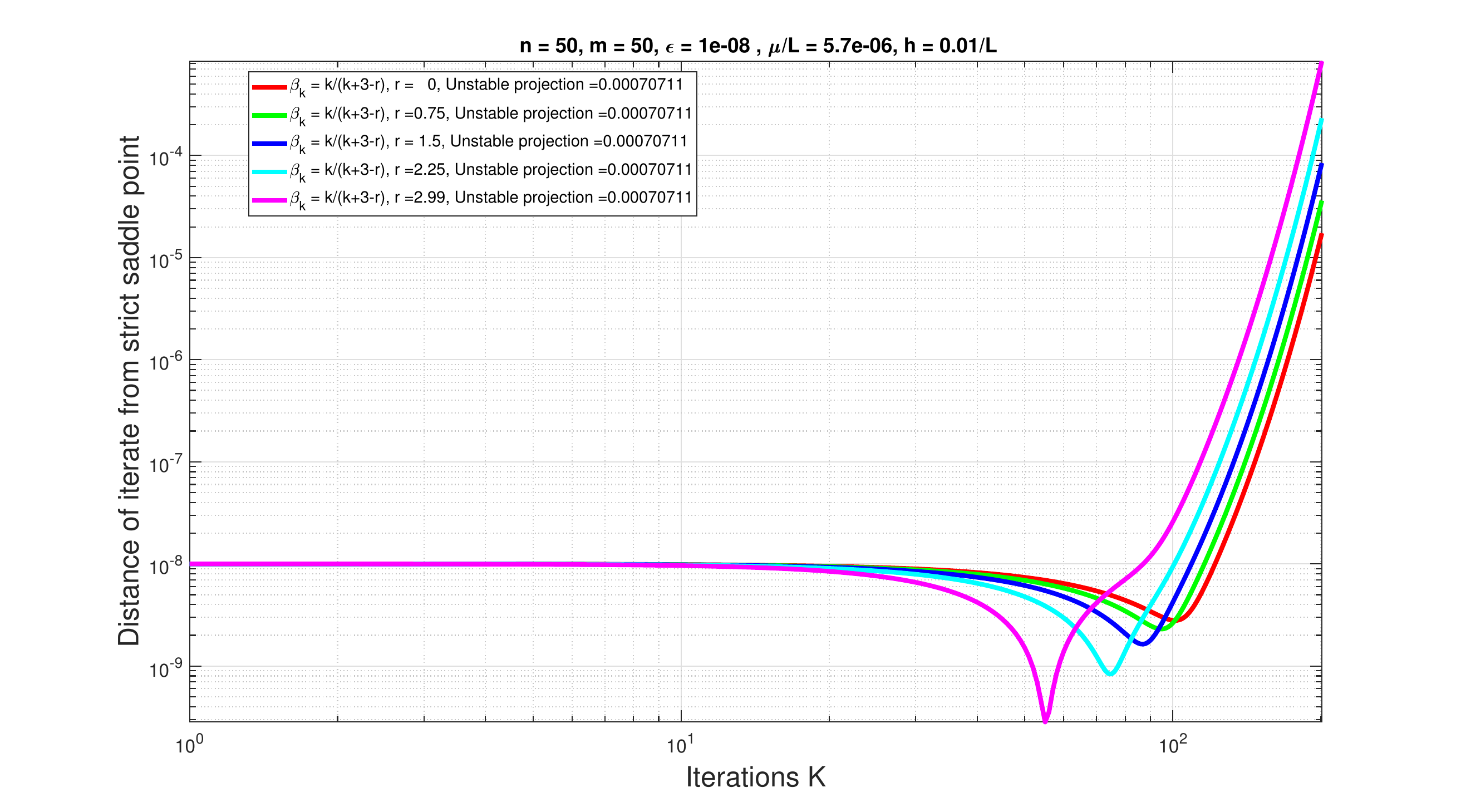} \\
     (b) 
\end{tabular}
\caption{Simulated trajectories of various accelerated gradient methods from the family \eqref{familyof momentum} on the phase retrieval problem with $h = 0.1/L$, under the same initial unstable projections for fixed values of $m$, $n$, and $\epsilon$. The parameter $r$ controls the momentum via $\beta_k = \frac{k}{k + 3 - r}$; $r = 0$ corresponds to the Nesterov accelerated method \eqref{originalnesterov}.
}
\label{fig12}
\end{figure}

\begin{figure}[H]
\centering
\begin{tabular}{c}
    \includegraphics[width=0.7\textwidth]{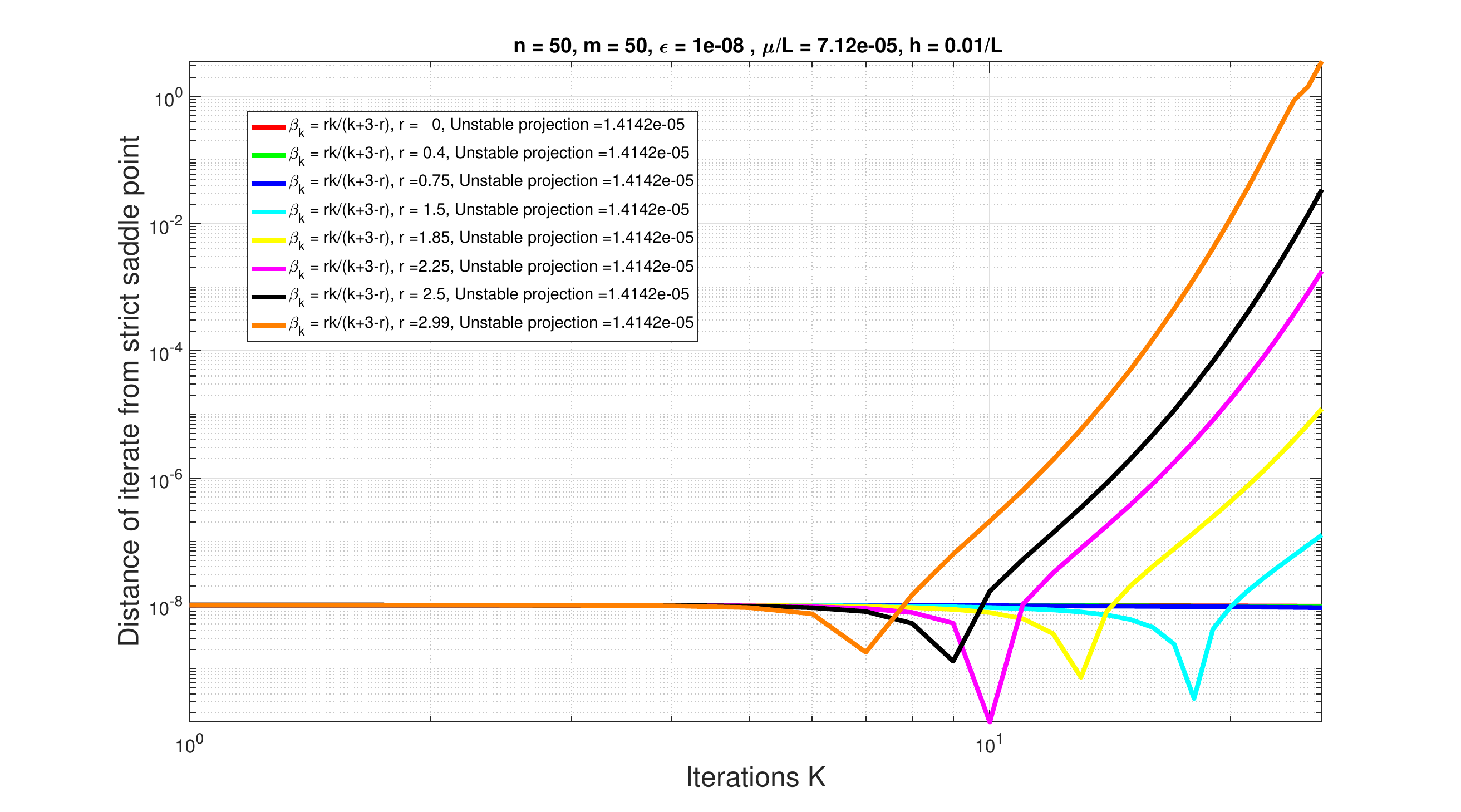} 
\end{tabular}
\caption{Simulated trajectories of various accelerated gradient methods from \eqref{ds1} with $\beta_k = \frac{r k}{k + 3 - r}$ on the phase retrieval problem with $h = 0.01/L$, under the same initial unstable projections for fixed values of $m$, $n$, and $\epsilon$. Note that $r = 0$ corresponds to the gradient descent method.
}
\label{fig11a}
\end{figure}

We next showcase the effectiveness of momentum scheme \eqref{ds1} with $\beta_k = \frac{rk}{(k+3-r)}$ for $r >1$ in a setting where the different trajectories generated by this scheme are initialized in a way such that their initial radial vectors $\x_0 - \x^*$ have different projections on the unstable subspace of $ \nabla^2 f(\mathbf{\x^*})$. In particular, methods with dominant momentum schemes (as per Corollary \ref{commentseccorr}) or equivalently methods with larger values of $r$ are initialized with a very low initial unstable subspace projection and vice versa. Figure \ref{fig11abc} then plots the trajectories of these methods for different values of $r$. From figure it is evident that even with extremely small initial unstable projection values of the orders $10^{-4}$ to $10^{-8}$, any acceleration scheme with $r> 1$ dominates all the other schemes with $r\leq 1$ in escaping the saddle neighborhood even if they have much larger order magnitudes of initial unstable projection values {(see also the discussion in Section~\ref{sssec:discussion.exit.time})}. In particular, from Figure \ref{fig11abc} we can clearly observe that the standard gradient descent with the given choice of step-size $h= \frac{0.01}{L}$ and a much larger initial unstable projection value of the order $10^{-1}$ performs significantly worse against the momentum schemes with $r>1$ in escaping the saddle neighborhood.

\begin{figure}[H]
\centering
\begin{tabular}{c}
\includegraphics[width=0.7\textwidth]{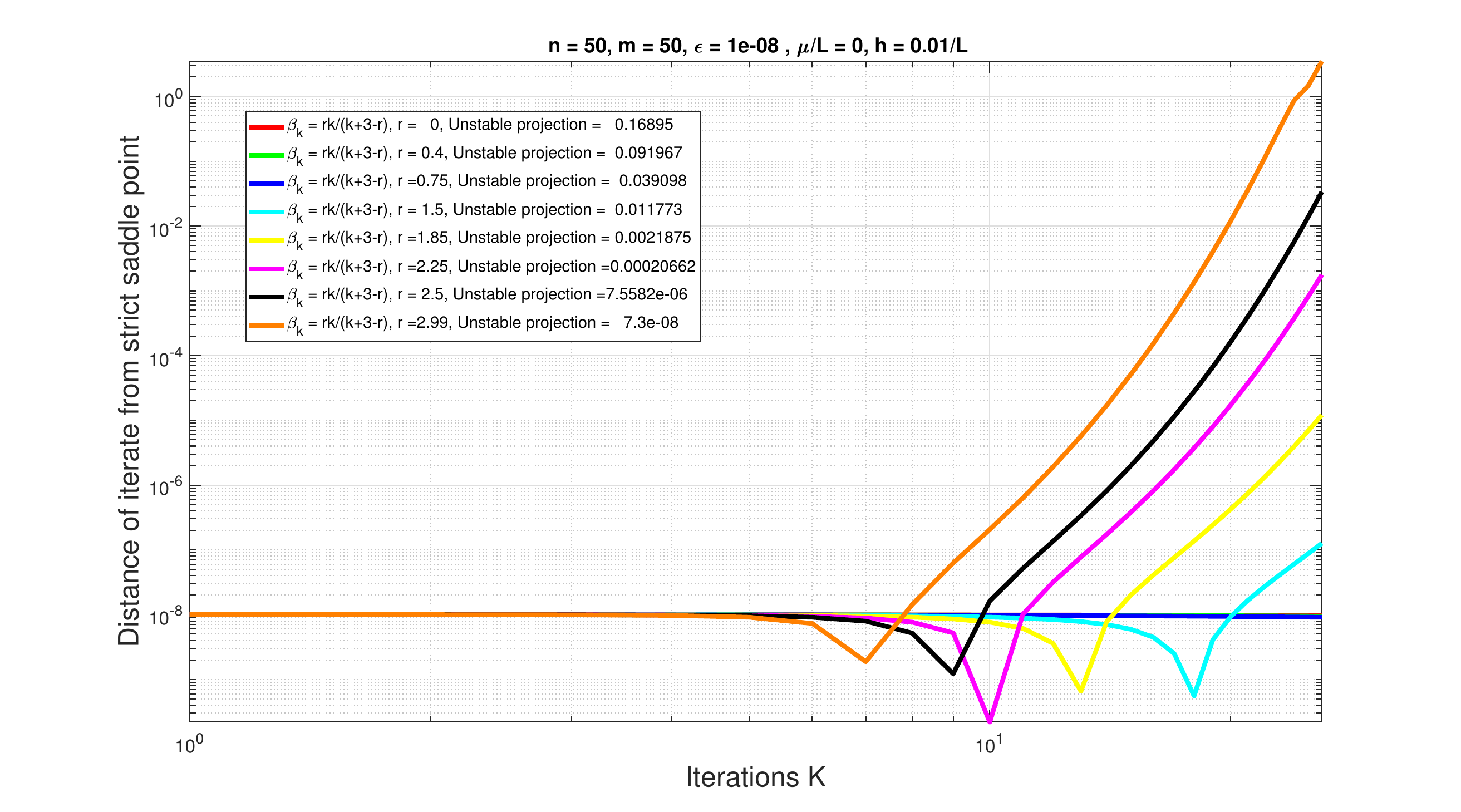} 
\end{tabular}
\caption{Simulated trajectories of various accelerated gradient methods from \eqref{ds1} with $\beta_k = \frac{r k}{k + 3 - r}$ on the phase retrieval problem with $h = 0.01/L$, under different initial unstable projections for fixed values of $m$, $n$, and $\epsilon$. Note that $r = 0$ corresponds to the gradient descent method.
}
\label{fig11abc}
\end{figure}
 
\subsection{{Low-rank matrix factorization problem}}
{In order to further validate the theoretical framework developed in this work and examine the behavior of accelerated gradient trajectories near strict saddle points, we next evaluate the performance of the family of accelerated methods \eqref{generalds} on the low-rank matrix factorization problem. The objective function is given by
\begin{align}
    f(\X_1,\X_2) = \frac{1}{4}\norm{\mathbf{M} - \X_1\X_2^T }^2_F + \varpi_1 \norm{\X_1}^2_F + \varpi_2 \norm{\X_2}^2_F, \label{simulate1}
\end{align}
where $\mathbf{M} \in \mathbb{R}^{n_1 \times n_2}$, $\X_1 \in \mathbb{R}^{n_1 \times d}$, and $\X_2 \in \mathbb{R}^{n_2 \times d}$, with $d \leq \min \{n_1, n_2\}$ denoting the rank of the target matrix. This formulation, as well as numerous variants of it, have been considered in prior work, including \cite{dixit2021boundary} for gradient descent and \cite{stoger2021small} for gradient descent with spectral initialization applied to a variant of this problem. However, to the best of our knowledge, there does not exist a study on low-rank matrix factorization that focuses on saddle escape rates of accelerated gradient methods as a function of the momentum parameters.}

{To simplify the problem structure and express \eqref{simulate1} as a function of a single matrix variable $\X$, we define $\X_1 = \B_1 \X$ and $\X_2 = \B_2 \X$, where
\[
\X = \begin{bmatrix}
\X_1 \\ \X_2
\end{bmatrix}, \quad
\B_1 = \begin{bmatrix}
\I_{n_1 \times n_1} & \mathbf{0}_{n_1 \times n_2}
\end{bmatrix}, \quad
\B_2 = \begin{bmatrix}
\mathbf{0}_{n_2 \times n_1} & \I_{n_2 \times n_2}
\end{bmatrix}.
\]
Here, $\I_{n_1 \times n_1}$ and $\I_{n_2 \times n_2}$ denote identity matrices, while $\mathbf{0}_{n_1 \times n_2}$ and $\mathbf{0}_{n_2 \times n_1}$ denote zero matrices of the indicated sizes. With this change of variables, the objective becomes
\begin{align}
    f(\X) = \frac{1}{4} \norm{\mathbf{M} - \B_1\X\X^T\B_2^T }^2_F + \varpi_1 \norm{\B_1\X}^2_F + \varpi_2 \norm{\B_2\X}^2_F. \label{simulate2}
\end{align}
The corresponding gradient, $\nabla f(\X)$, takes the form
\begin{align}
   &\nabla f(\X) \nonumber\\
   &= \frac{1}{2}(\B_1^T\B_1 \X \X^T \B_2^T\B_2 + \B_2^T\B_2 \X \X^T \B_1^T\B_1 )\X - \frac{1}{2}(\B_2^T \mathbf{M}^T \B_1 + \B_1^T \mathbf{M} \B_2)\X +  2 \varpi_1 \B_1^T\B_1 \X +  2 \varpi_2 \B_2^T\B_2 \X. \label{gradientsimulate}
\end{align}
Since the gradient in \eqref{gradientsimulate} is matrix-valued, its Hessian is a fourth-order tensor. To stay within our analytical framework, which assumes matrix-valued Hessians, we make use of \cite[Theorem 9]{magnus1985matrix} and vectorize $\X$ so that $\nabla^2 f(\mathrm{vec}(\X))$ becomes a Jacobian matrix. The closed-form expression for this Jacobian is as follows:
\begin{align}
  \nabla^2 f(\mathrm{vec}(\X)) &= \frac{1}{2}\bigg( ((\X^T\B_2^T\B_2) \otimes \I_{n \times n})( (\X \otimes \I_{n \times n}) (\I_{d \times d} \otimes (\B_1^T\B_1)) +\nonumber \\
    &\qquad\qquad (\I_{n \times n} \otimes (\B_1^T \B_1 \X)))  + (\I_{d \times d} \otimes (\B_1^T\B_1 \X \X^T) )(\I_{d \times d} \otimes (\B_2^T\B_2))  \bigg) \nonumber \\ 
    &\quad + \frac{1}{2}\bigg( ((\X^T\B_1^T\B_1) \otimes \I_{n \times n})( (\X \otimes \I_{n \times n}) (\I_{d \times d} \otimes (\B_2^T\B_2)) +\nonumber \\ 
    &\qquad\qquad (\I_{n \times n} \otimes (\B_2^T \B_2 \X))) + (\I_{d \times d} \otimes (\B_2^T\B_2 \X \X^T) )(\I_{d \times d} \otimes (\B_1^T\B_1))  \bigg)  \nonumber \\
    &\quad - \frac{1}{2}\bigg( \I_{d \times d} \otimes ( \B_2^T \mathbf{M}^T \B_1 + \B_1^T \mathbf{M} \B_2 ) \bigg)  +2\bigg( \I_{d \times d} \otimes (\varpi_1 \B_1^T\B_1 + \varpi_2 \B_2^T\B_2) \bigg), \label{simulatehessian}
\end{align}
where $n=n_1 + n_2$. For simulations, the matrix $\mathbf{M}$ is generated randomly using the relation
\[
\mathbf{M} = \mathbf{U}_1 \mathbf{U}_2^T + \varrho^2 \mathbf{N},
\]
where $\mathbf{U}_1 \in \mathbb{R}^{n_1 \times d}$ and $\mathbf{U}_2 \in \mathbb{R}^{n_2 \times d}$ have entries independently sampled from a standard normal distribution. The matrix $\mathbf{N} \in \mathbb{R}^{n_1 \times n_2}$ represents additive noise, also drawn from a standard normal distribution, with its variance scaled by $\varrho$. The formulation in \eqref{simulate1} is coercive and analytic, and the Hessian at the point $\X = \mathbf{0}$ is invertible. However, the function at $\X = \mathbf{0}$ has a poor condition number, as will be evident from the simulations. Since the function in \eqref{simulate1} is analytic and hence $\mathcal{C}^{\infty}$ smooth, it is locally gradient and Hessian Lipschitz continuous on every compact set.}

{For the experiments, we use $\varpi_1 = \varpi_2 = 0.5$, $\varrho = 0.15$, and a step size $h = \frac{0.01}{L}$, where $L = \lambda_{\max}(\nabla^2 f(\mathrm{vec}(\X_0)))$ is the local gradient Lipschitz constant of $f$ in a small neighborhood of the initialization $\X_0$. For this choice of parameters, $\X^* = \mathbf{0}$ is a strict saddle point. The family of accelerated gradient methods \eqref{generalds} is initialized using the scheme $\X_0 = \X_{-1}$ within the $\epsilon$-neighborhood of $\X^*$, i.e., $\|\X_0\|_F = \epsilon$, and with the same initial unstable subspace projection value, in order to examine the escape behavior of the resulting trajectories for given values of $n$, $d$, and $\epsilon$. Note that the `projection' of the initial iterate onto the unstable subspace corresponds to the quantity $\sum_{j \in \mathcal{N}_{\mathrm{US}}} (\theta_j^{\mathrm{us}})^2$, where
\[
\mathrm{vec}(\X_0) - \mathrm{vec}(\X^*) = \sum_{j \in \mathcal{N}_{\mathrm{US}}} \theta_j^{\mathrm{us}} \mathbf{e}_j + \sum_{i \in \mathcal{N}_{\mathrm{S}}} \theta_i^{\mathrm{s}} \mathbf{e}_i,
\]
and $\mathcal{N}_{\mathrm{US}}$ denotes the index set corresponding to the negative eigenvalues of $\nabla^2 f(\mathrm{vec}(\X^*))$.}

{The results of the simulations are reported in Figure~\ref{fig11cd} for $n_1 = 50$, $n_2 = 50$, $d = 5$, and in Figure~\ref{fig11ef} for $n_1 = 100$, $n_2 = 80$, $d = 9$. Note that in each subplot in both figures, different initial unstable projection values are used. It is evident from these results that, as suggested by the theoretical findings in this paper, larger momentum leads to faster escape from strict saddle neighborhoods, even with a small step size $h$. More importantly, Figure~\ref{fig11ef}$(b)$ and Figure~\ref{fig11y} corroborate our findings that higher momentum facilitates faster saddle escape even when the initial projection onto the unstable subspace is very small. In particular, Figure~\ref{fig11y} simulates the momentum scheme \eqref{ds1} with parameters $\beta_k = \frac{r k}{k + 3 - r}$ for $r \in [0, 3)$, which is slightly different from the scheme \eqref{familyof momentum}. Although this choice of $\beta_k$ does not satisfy the convergence guarantees of Theorem~\ref{thmlipschitzrate} when $r > 1$, it nonetheless escapes strict saddle neighborhoods remarkably quickly.}

\begin{figure}[H]
\centering
\begin{tabular}{cc}
    \includegraphics[width=0.47\textwidth]{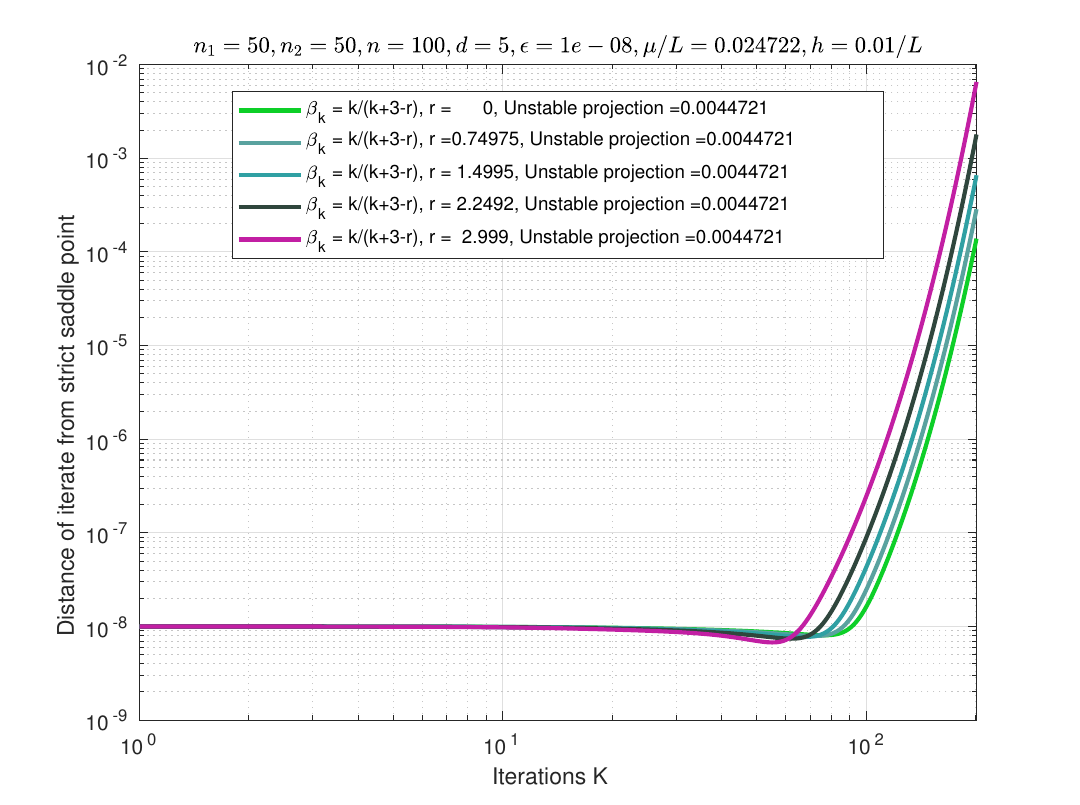} &
    \includegraphics[width=0.47\textwidth]{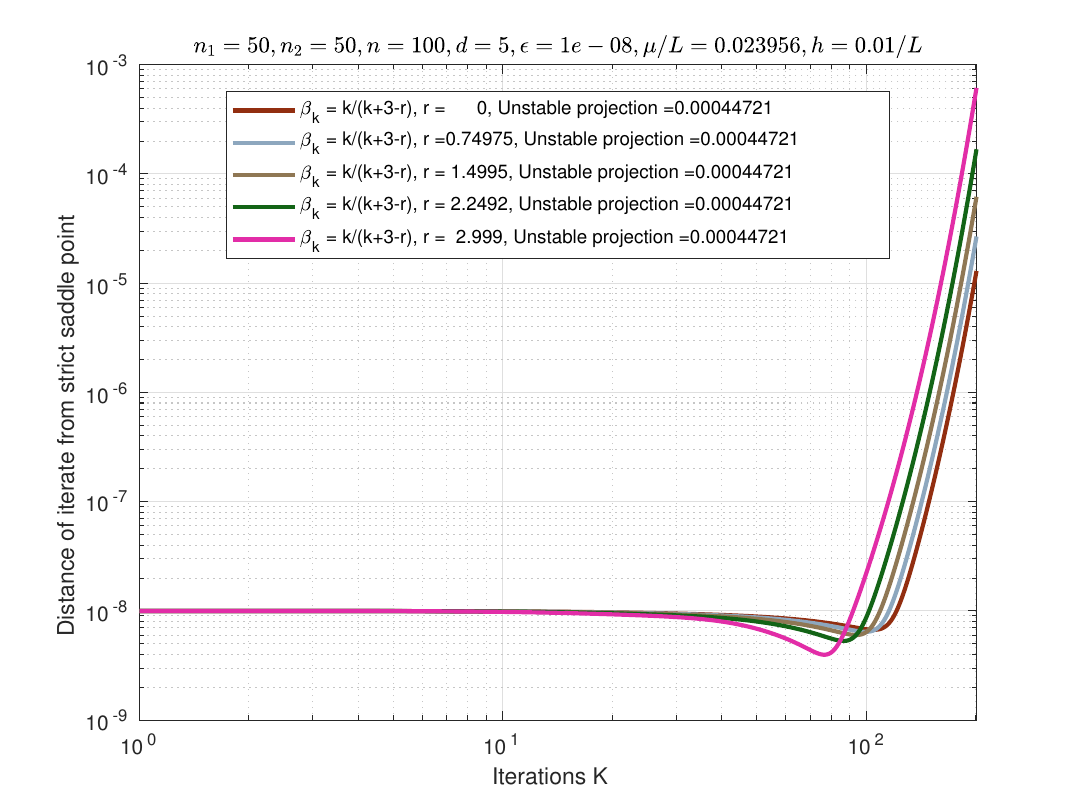} \\
    (a) & (b)
\end{tabular}
\caption{{Simulated trajectories of various accelerated gradient methods from the family \eqref{familyof momentum} on the low-rank matrix factorization problem with $h = 0.01/L$, under the same initial unstable projections for fixed values of $n_1$, $n_2$, $d$, and $\epsilon$. The parameter $r$ controls the momentum via $\beta_k = \frac{k}{k + 3 - r}$; $r = 0$ corresponds to the Nesterov accelerated method \eqref{originalnesterov}.}}
\label{fig11cd}
\end{figure}

\subsection{Positive definite quadratic minimization problem}
Having demonstrated the efficacy of higher momentum methods in escaping strict saddle neighborhoods in the phase retrieval {and low-rank matrix factorization problems}, we now turn to analyzing the convergence behavior of the novel momentum scheme \eqref{familyof momentum}, given by $\beta_k = \frac{k}{k + 3 - r}$, introduced in \cite{apidopoulos2020convergence}. Specifically, we show that this scheme exhibits convergence to a local minimum by studying its behavior on a strongly convex quadratic function $f(\x) = \langle \x, \A \x \rangle$, where $\A \in \mathbb{R}^{n \times n}$ is a symmetric positive definite matrix. This, in turn, would support the result from Theorem~\ref{thmconvexrate}. For simplicity, we set $\A$ to be a diagonal matrix of dimension $n$, with diagonal entries drawn i.i.d. from a uniform distribution on the interval $(0,1)$. Clearly, $\x^* = \mathbf{0}$ is the global minimum in this case. The family \eqref{familyof momentum}, with momentum parameter $\beta_k = \frac{k}{k + 3 - r}$ for various values of $r \in [0, 3)$, is initialized from the same point $\x_0$, chosen such that $\x_0$ is aligned with the eigenvector corresponding to the minimum eigenvalue of $\A$. In other words, all methods encounter extremely flat regions initially. The results are plotted in Figure~\ref{fig13a}, with each subplot corresponding to a different realization of the matrix $\A$. As evident from the simulations, higher momentum schemes do exhibit convergence, albeit at a slower rate compared to the Nesterov accelerated gradient method \eqref{originalnesterov} when applied to quadratic functions. {In particular, recall Theorem~\ref{thmconvexrate}, which states that for convex functions, the scheme \eqref{generalds} with $\beta_k = \frac{k}{k + 3 - r}$ for $r \in [0, 3)$ achieves a convergence rate of $f(\x_k) - f(\x^*) = \mathcal{O}\left(k^{-(2 - \frac{2r}{3})}\right)$. Since the quadratic function used in these numerical experiments is convex, Theorem~\ref{thmconvexrate} applies in this setting, which explains the slower convergence of higher momentum schemes compared to Nesterov's accelerated gradient method \eqref{originalnesterov}.}

\begin{figure}[t]
\centering
\begin{tabular}{cc}
    \includegraphics[width=0.47\textwidth]{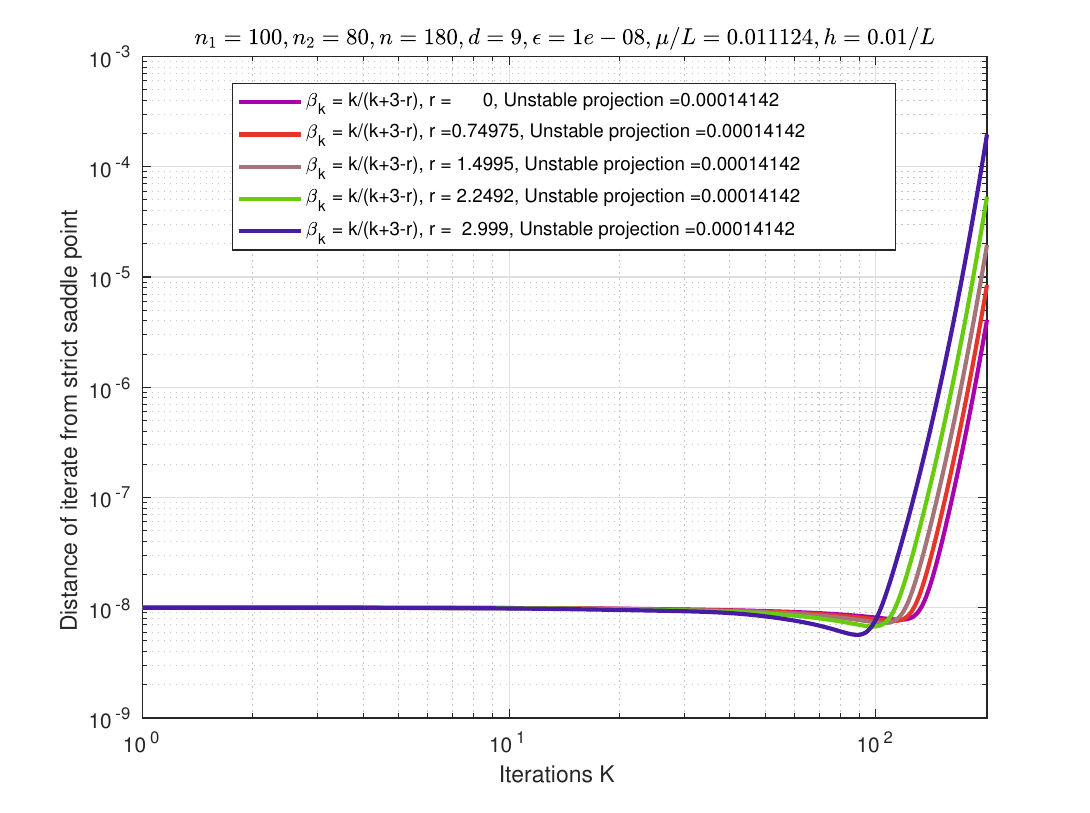} & 
    \includegraphics[width=0.47\textwidth]{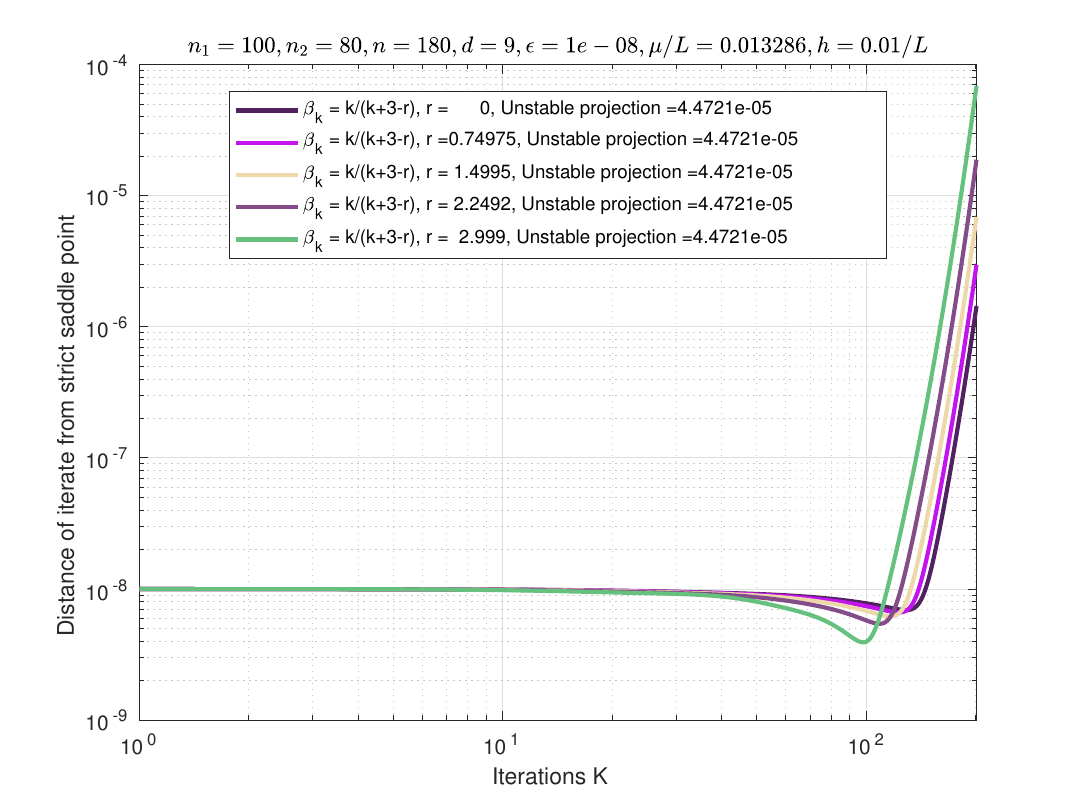} \\
    (a) & (b) 
\end{tabular}
\caption{{Simulated trajectories of various accelerated gradient methods from the family \eqref{familyof momentum} on the low-rank matrix factorization problem with $h = 0.01/L$, under the same initial unstable projections for fixed values of $n_1$, $n_2$, $d$, and $\epsilon$. The parameter $r$ controls the momentum via $\beta_k = \frac{k}{k + 3 - r}$; $r = 0$ corresponds to the Nesterov accelerated method \eqref{originalnesterov}.}}
\label{fig11ef}
\end{figure}

\begin{figure}[H]
    \centering
    \includegraphics[width=0.55\textwidth]{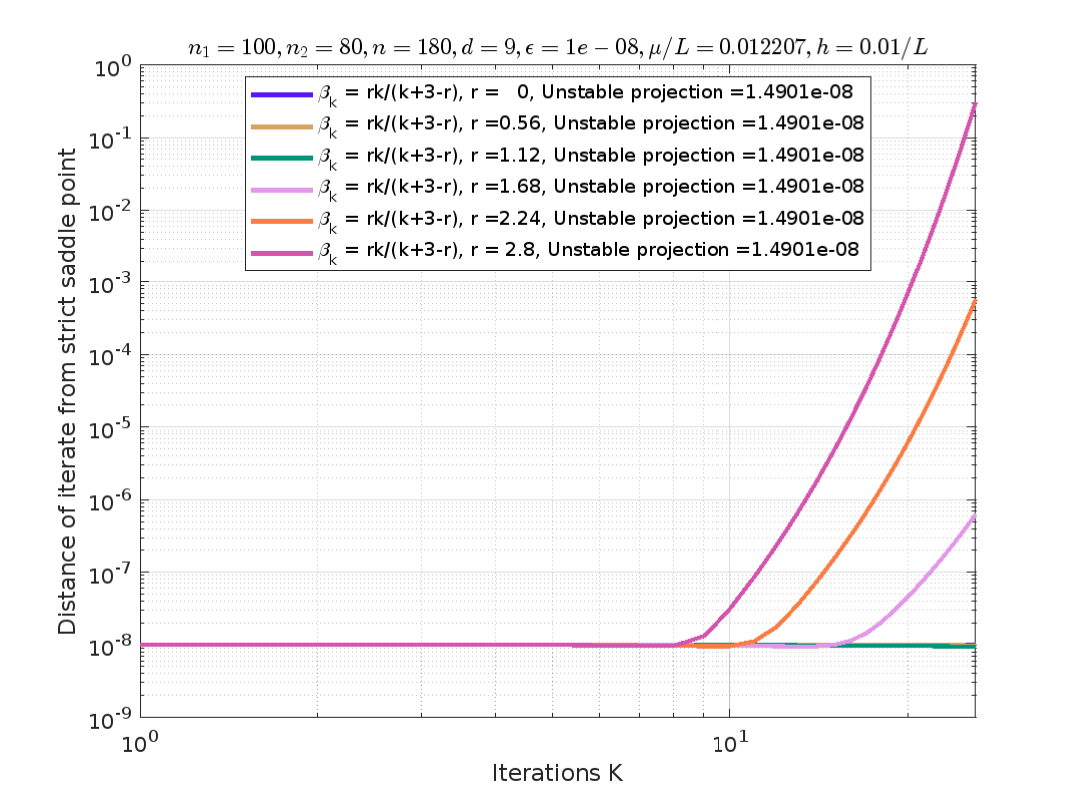}
\caption{{Simulated trajectories of various accelerated gradient methods from \eqref{ds1} with $\beta_k = \frac{r k}{k + 3 - r}$ on the low-rank matrix factorization problem with $h = 0.01/L$, under the same initial unstable projections for fixed values of $n_1$, $n_2$, $d$, and $\epsilon$. Note that $r = 0$ corresponds to the gradient descent method.}}
\label{fig11y}
\end{figure}

\begin{figure}[H]
\centering
\begin{tabular}{c}
    \includegraphics[width=0.95\textwidth]{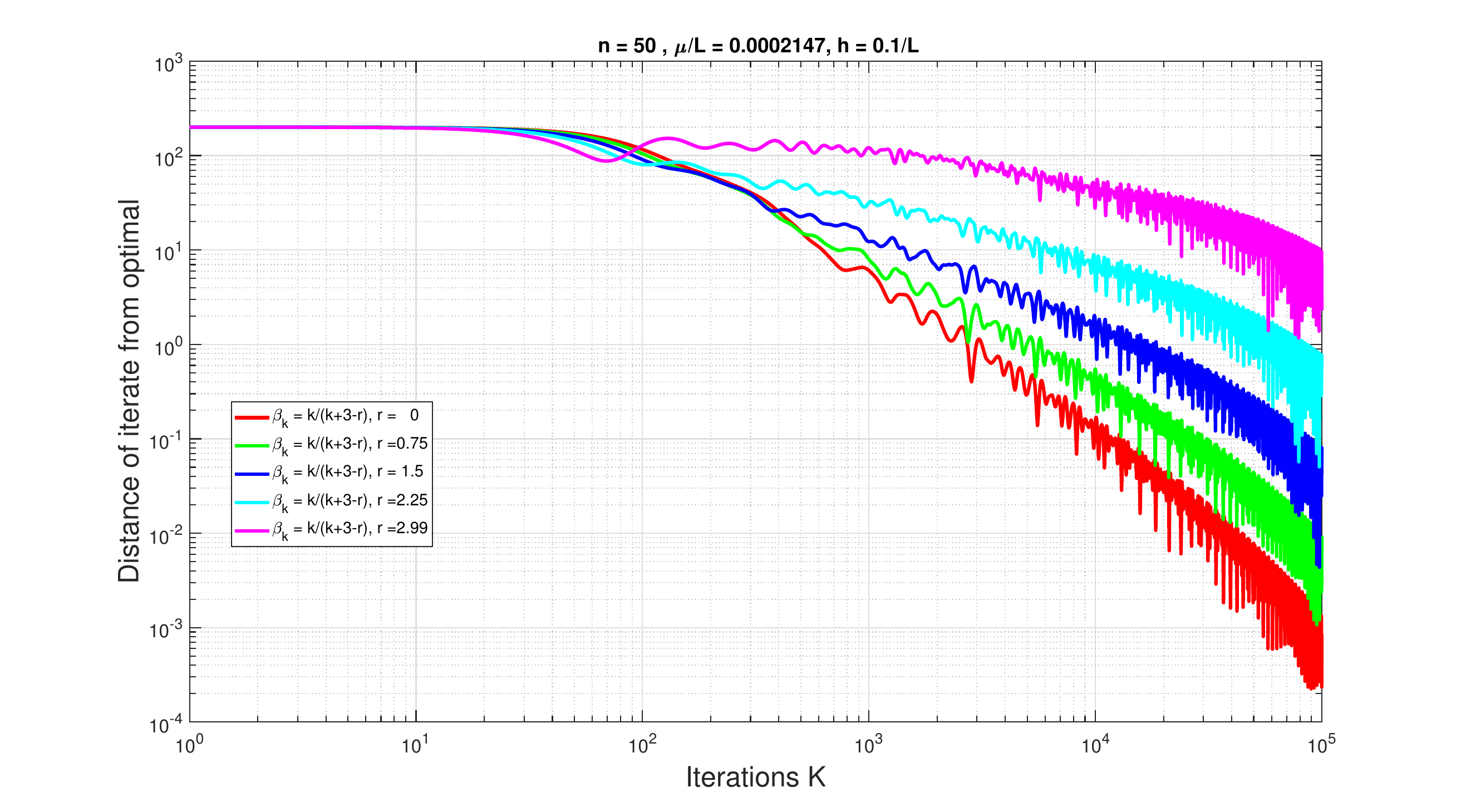} \\ (a) \\ \includegraphics[width=0.95\textwidth]{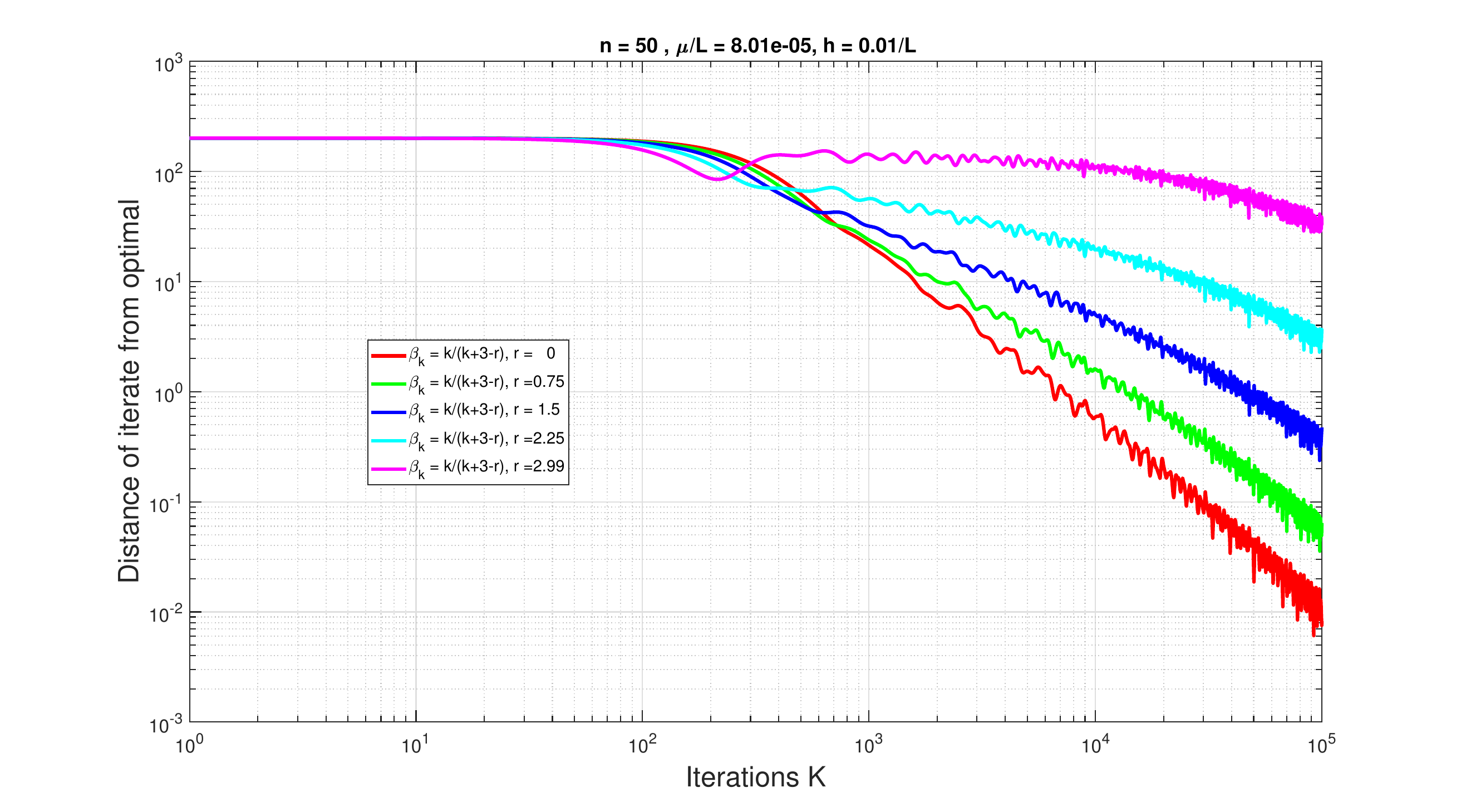} \\
     (b) 
\end{tabular}
\caption{Simulated trajectories of various accelerated gradient methods from the family \eqref{familyof momentum} on a positive definite quadratic function, under the same initialization scheme $\x_0 = \x_{-1}$ for fixed values of $n$ and $h$. The parameter $r$ controls the momentum via $\beta_k = \frac{k}{k + 3 - r}$; $r = 0$ corresponds to the Nesterov accelerated method \eqref{originalnesterov}.}
\label{fig13a}
\end{figure}

\section{Conclusion}
This work has focused on the analysis of a class of accelerated gradient methods on smooth nonconvex functions. The analysis in this work has subsumed the study of asymptotic, local and global behavior of the general accelerated methods \eqref{generalds}. In particular, within the asymptotic analysis, using tools from dynamical systems and Banach space theory, a proof technique has been developed that can be used to show the almost sure non-convergence of a class of accelerated methods to strict saddle points of coercive Morse functions. In addition, within the asymptotic analysis, this work has also proposed two metrics that measure best possible asymptotic speeds of convergence to and divergence from any critical point of analytic nonconvex functions. Three theorems have been presented that quantify these asymptotic speeds and describe the dependence of these speeds on the momentum parameters and step-size of the algorithms. 

Next, this work has also provided a local analysis of a complex valued dynamical system in the vicinity of weak hyperbolic fixed points. The exit time estimate for such a system from a sufficiently small neighborhood of the hyperbolic fixed point has been provided in a novel theorem. This result is then used to estimate the exit time of a class of accelerated methods \eqref{ds1} from strict saddle neighborhoods with respect to a weighted Euclidean metric which highlights the effect of large momentum on saddle escape behavior. Then a family of accelerated methods \eqref{familyof momentum} from \cite{apidopoulos2020convergence} is studied that achieves near optimal convergence rates in convex neighborhoods (in the sense of Nesterov convergence rate) and at the same time offers superior saddle escape behavior compared to that of Nesterov accelerated method \eqref{originalnesterov}. 

Finally, we presented some new results which provide almost sure convergence guarantees to local minimum of coercive Morse functions that may not be globally gradient Lipschitz continuous. The \eqref{generalds} scheme is simulated on the phase retrieval problem {and the low-rank matrix factorization problem,} where it is shown that larger momentum leads to faster escape from strict saddle points. In particular it is shown that though increasing the momentum parameter $\beta_k$ beyond $1$ does not offer convergence guarantees to local minimum, yet it remarkably improves the escape behavior from strict saddle neighborhoods, something which could be leveraged while developing hybrid accelerated algorithms that allow fast saddle escape and at the same time offer second order convergence guarantees.

{Beyond these analytical and algorithmic developments, this work also introduces several theoretical techniques that we believe are of independent interest and may extend beyond the analysis of the \eqref{generalds} scheme. These include: ($i$) the analysis of momentum methods in \( \mathbb{R}^n \) involving inverse maps, as developed in the proof of Theorem~\ref{diffeomorphthm}; in particular, the technique used in Theorem~\ref{diffeomorphthm} establishes that the forward maps \( N_k \) are almost surely diffeomorphisms by randomizing the step size \( h \) and removing the (measure-zero) set of values at which \( N_k \) fails to be a diffeomorphism; ($ii$) the ``switching technique'' introduced in Theorems~\ref{measuretheorem2} and~\ref{measuretheorem3}, which provides a framework for analyzing time-varying algorithms by alternating between maps with similar fixed points and stability properties; ($iii$) the proof of almost sure saddle avoidance in Theorem~\ref{measuretheorem3}, which leverages a dual-space analysis of trajectory sequences and may be applicable to other stochastic or deterministic iterative methods; ($iv$) the introduction of new metrics for evaluating asymptotic escape or convergence rates near critical points, discussed in Section~\ref{metricsection}, which can aid in characterizing the local dynamical behavior of iterative algorithms; and ($v$) the analysis of first-order methods near strict saddle points using a trajectory approximation approach, including quantitative estimates of exit times from small saddle neighborhoods, presented in Section~\ref{exittimesection}. In future work, we aim to investigate how these techniques can be applied to algorithms beyond \eqref{generalds}.}

\appendix
\section{Analysis of some properties of the algorithmic maps}\label{Appendix A}
\subsection{Theorem \ref{diffeomorphthm}}
\begin{proof}
{We recall that the maps $R_k$ and $G$ defined in \eqref{def-G} and \eqref{def-Rk} are such that} $\forall k$ we have:
\begin{align}
\y_k &= R_k(\x_k),\\
\x_{k+1} & = G(\y_k) = \y_k - h \nabla f(\y_k), 
\end{align}  
and 
\begin{align}
R_k(\x_k) &= \bar{R}(\x_k,\x_{k-1})= p_k\x_k -q_k\x_{k-1}, \label{affmap}
\end{align}
where $p_k$ and $q_k$ are some real numbers in sequences $\{p_k\}$ and $\{q_k\}$ with $p_k \neq 0$ for all $k$. Then $R_k$ from \eqref{affmap} is an affine map in $\x_k$ for some given $\x_{k-1}$. Therefore the map $R_k$ given by $ R_k(\x) =  p_k\x -q_k\x_{k-1}$ is a continuous bijection ({maps $\mathbb{R}^n$ to $\mathbb{R}^n$}) hence invertible with a continuous inverse and so a homeomorphism on compact sets. For smooth $f$ and $h < \frac{1}{L}$, the map $G \equiv \mathrm{id} - h \nabla f$
is a continuous bijection, is a closed map and hence invertible with a continuous inverse on compact sets (Closed map lemma \cite{lee2013smooth}).\footnote{The surjection of $G$ is {straightforward} so we only show that $G$ is injective. Let $\x \neq \y$ and $G(\x) = G(\y)$ then:
\begin{align}
    \x - h \nabla f(\x)  = \y - h \nabla f(\y) 
\implies \norm{\x - \y} = h\norm{\nabla f(\x) -\nabla f(\y)} \leq Lh \norm{\x - \y} < \norm{\x - \y},
\end{align}
which is a contradiction since $Lh < 1$. Hence, $\x = \y$ iff $G(\x) = G(\y)$ and $G$ is an injective map. The closedness of map $G$ is straightforward.} Therefore, the map $N_k \equiv G \circ R_k$ is a continuous bijection and also closed (composition of affine and closed map is a closed map) for all $k$ (hence invertible for all $k$) with a continuous inverse on compact sets by the Closed map lemma \cite{lee2013smooth} and so a homeomorphism on compact sets. Thus, we have shown that the map $N_k$ is invertible. Since we only used the continuity of $\nabla f$, the first part works with $f \in \mathcal{C}^1$ where $f$ is $L$-gradient Lipschitz continuous.   
 
 We now look at the smoothness of maps $R_k, N_k$ along any trajectory of $\{\x_k\}$ generated from the recursion $\x_{k+1} = N_k(\x_k)$. Assume now that $f$ is analytic along with the $L$-gradient Lipschitz continuity property. Since $N_k$ is invertible for all $k$ by which $ \x_{k-1} = N_{k-1}^{-1}(\x_{k})$, we can write the following recursion
\begin{align}
    R_k(\x_k)  = p_k\x_k -q_k\x_{k-1}  = p_k\x_k -q_k N_{k-1}^{-1}(\x_{k}). \label{newlogic1}
\end{align}
   {Also since $\{\x_k\}$ is the sequence for any possible trajectory generated from the recursion $\x_{k+1} = N_k(\x_k)$, the iterate $\x_k$ in \eqref{newlogic1} can be any point in $\mathbb{R}^n$. Thus replacing $\x_k$ with some general vector variable $\x \in \mathbb{R}^n$ in \eqref{newlogic1}, we get the recursive equation 
 \begin{align}
      R_k(\x) = p_k\x -q_k N_{k-1}^{-1}(\x) ; \hspace{0.1cm} N_k(\x) = G (p_k\x -q_k N_{k-1}^{-1}(\x)),\label{newlogic2}
 \end{align} 
  where the map $N_k \equiv G \circ R_k$ satisfies $ N_k \equiv G \circ (p_k\mathrm{id} -q_k N_{k-1}^{-1})$. Note that the map $R_k$ from \eqref{newlogic2} is a nonlinear map as opposed to the affine map given by $ R_k\vert_{\x_{k-1}}(\x) =  p_k\x -q_k\x_{k-1}$ along the trajectory of $\{\x_k\}$.} 

  To establish the almost sure smoothness of $N_k$ given by $ N_k \equiv G \circ (p_k\mathrm{id} -q_k N_{k-1}^{-1})$, we will use an inductive argument. Let us assume that for some $k$ and some given $\z \in \mathbb{R}^n$, the map $N_{r}$ is a diffeomorphism at $ N_{r}^{-1}\circ\dots \circ N_{k-1}^{-1}(\x)$ for all $r<k$ with $\x = N_{k-1} \circ \dots \circ N_{-1}(\z)$ and suppose that no eigenvalue of $ q_k [D N_{k-1}( N_{k-1}^{-1}(\x))]^{-1}$ is equal to ${p_k}$. Then $N_k$ is a diffeomorphism at $\x$. To see this we first differentiate \eqref{newlogic2} with respect to $\x$ to get 
  \begin{align}
      D R_k(\x)  = p_k\mathbf{I} -q_k D N_{k-1}^{-1}(\x) = p_k\mathbf{I} -q_k [D N_{k-1}(N_{k-1}^{-1}(\x))]^{-1}, \label{ftp1}
  \end{align}
   where $ D N_{k-1}^{-1}(\x) =  [D N_{k-1}(N_{k-1}^{-1}(\x))]^{-1}$ by the inverse function theorem. Next, the map $G$ is differentiable where $ D G(\x) = \mathbf{I} - h \nabla^2 f(\x)$ and we have that $det(D G (\x)) \neq 0 $ for any $\x$ since minimum eigenvalue of $DG(\x) = \mathbf{I}- h \nabla^2 f(\x)$ is {at least} $1- Lh$ which is positive for any $h < \frac{1}{L}$. Hence $[D G(\x)]^{-1} $ exists and therefore by the inverse function theorem $G^{-1}$ is also continuously differentiable, hence $G$ is a diffeomorphism. Therefore, the map $N_k = G \circ R_k$ is a bijection (hence invertible) and is a diffeomorphism at the given $\x $ whenever $ N_{k-1}$ is a diffeomorphism at $N_{k-1}^{-1}(\x)$. Then for the given $\x $, using \eqref{ftp1} we have:
\begin{align}
 D N_k(\x)  &= (\mathbf{I}- h \nabla^2 f(R_k(\x)) )\underbrace{(p_k \mathbf{I} - q_k D N_{k-1}^{-1}(\x))}_{=D R_k(\x)} =  (\mathbf{I}- h \nabla^2 f(R_k(\x)) ){(p_k \mathbf{I} - q_k [D N_{k-1}(N_{k-1}^{-1}(\x))]^{-1})}. \label{diffeoeqn}
\end{align}
{In the base case of $ k=-1$,} $ N_{-1}$ is a non-zero affine map from $\mathbb{R}^n$ to $\mathbb{R}^n$ given by $  N_{-1}(\x) = \x + \x_0 -\x_{-1}$ and for the initialization scheme of $[\x_0 ;\x_{-1}]$ we have $ \x_0 = N_{-1}(\x_{-1}) $. Notice that in \eqref{diffeoeqn} the Jacobian matrix $D N_k(\x) $ is defined at the given $\x $ only if $  N_{k-1}$ is a diffeomorphism at $N_{k-1}^{-1}(\x)$. Inducting this condition we get that $D N_k(\x) $ is defined at the given $\x $ only if the map $ N_{r}$ is a diffeomorphism at $ N_{r}^{-1}\circ\dots \circ N_{k-1}^{-1}(\x)$ for all $r<k$. From \eqref{diffeoeqn} observe that the matrix $DN_k(\x) $ is invertible for the given $\x$ if and only if $det\bigg(p_k \mathbf{I} - q_k [D N_{k-1}(N_{k-1}^{-1}(\x))]^{-1}\bigg) \neq 0 $ since $ \mathbf{I}- h \nabla^2 f(R_k(\x))$ for any $\x$ is invertible if $h < \frac{1}{L}$. Now, for $\x = N_{k-1} \circ \dots \circ N_{-1}(\z)$ where $\z$ is any point in $\mathbb{R}^n$, the entries of matrix $ D N_{k-1}(N_{k-1}^{-1}(\x))$ will be analytic functions of $h$. To see this, first note that by using the recursion \eqref{diffeoeqn} repeatedly up to $k=0$ and making repeated substitutions, the matrix $  D N_{k-1}(N_{k-1}^{-1}(\x))$ will be generated via a finite composition of affine, multiplication and inversion operations on the invertible matrix sequence $ \{ \mathbf{I}- h\nabla^2 f(M_j(\x)) \}_{j=0}^{k-1}$ where every map $M_j$ in the sequence $\{M_j\}_{j=0}^{k-1}$ will be a finite composition of the maps from the set $ \{\{N_j\}_{j=0}^{k-1}, \{R_j\}_{j=0}^{k-1}, G \}$ and also their inverses. In particular, we have that $M_j \equiv R_j \circ N_{j}^{-1} \circ \dots \circ N_{k-1}^{-1}$ for any $j< k$. 

Next, substituting $\x = N_{k-1} \circ \dots \circ N_{-1}(\z)$ in \eqref{newlogic2}, we get for any $k$ that:
 \begin{align}
      R_k(N_{k-1} \circ \dots \circ N_{-1}(\z)) = p_k N_{k-1} \circ \dots \circ N_{-1}(\z) -q_k N_{k-2} \circ \dots \circ N_{-1}(\z) .\label{newlogic2aba}
 \end{align} 
If for any $l \geq -1$, the points $N_{l} \circ \dots \circ N_{-1}(\z), N_{l+1} \circ \dots \circ N_{-1}(\z)$ are analytic functions of $h$, from \eqref{newlogic2aba} it must be that $ R_{l+2} \circ N_{l+1} \circ \dots \circ N_{-1}(\z)$ is an analytic function of $h$. Then $ N_{l+2} \circ \dots \circ N_{-1}(\z)$ is also an analytic function in $h$ from the facts that $N_{l+2}\equiv G \circ R_{l+2}$, $G$ is analytic by the analyticity of $f$ and finite compositions of analytic functions is analytic \cite{whittlesey1965analytic}. In the base case, $N_{-1}(\z)$ does not depend on $h$ so is analytic in $h$ with power of $0$ and $N_0(N_{-1}(\z))$ is an affine function of $h$ since $G \equiv \mathrm{id} - h \nabla f$ and is therefore analytic in $h$ with power of $1$. Then by induction we get that the point $N_{l} \circ \dots \circ N_{-1}(\z) $  for any $l$ is analytic in $h$. 

Now recall that we assumed the map $N_j$ to be a diffeomorphism at $ N_{j}^{-1}\circ\dots \circ N_{k-1}^{-1}(\x)$ for all $j<k$. Then using the substitution $\x = N_{k-1} \circ \dots \circ N_{-1}(\z)$, it is easy to see that the matrix $ D N_{k-1}(N_{k-1}^{-1}(\x))$ generated by the invertible matrix sequence $ \{ \mathbf{I}- h\nabla^2 f(M_j(\x)) \}_{j=0}^{k-1}$, where $M_j \equiv R_j \circ N_{j}^{-1} \circ \dots \circ N_{k-1}^{-1}$, will have entries that are ratios of analytic functions in $h$. This is because $M_j(\x) = R_j \circ N_{j-1} \circ \dots \circ N_{-1}(\z)$ which we have already shown is analytic function of $h$ and $\nabla^2 f$ is analytic. Since any finite composition of affine, multiplication and inversion operations on invertible matrices, whose entries are ratios of analytic functions of $h$, will result in a matrix whose entries are again ratios of analytic functions of $h$, we get that the matrix $  D N_{k-1}(N_{k-1}^{-1}(\x))$ will have entries that are ratios of analytic functions of $h$.

Hence, the left hand side of the equation $det\bigg(p_k \mathbf{I} - q_k [D N_{k-1}(N_{k-1}^{-1}(\x))]^{-1}\bigg)  = 0 $ will be a ratio of some analytic functions in $h$ and so the equation $det\bigg(p_k \mathbf{I} - q_k [D N_{k-1}(N_{k-1}^{-1}(\x))]^{-1}\bigg)  = 0 $ has at most countably many roots in $h$.\footnote{A simple induction gives that $ D N_{k-1}(N_{k-1}^{-1}(\x)) \to \mathbf{I}$ as $h \to 0$ and so $ det\bigg(p_k \mathbf{I} - q_k [D N_{k-1}(N_{k-1}^{-1}(\x))]^{-1}\bigg) \to 1$ as $h \to 0$ for $p_k -q_k =1$. \colb{Since this determinant is a ratio of analytic functions in $h$ and tends
to one as $h\to0$, its numerator is not identically zero. The zeros of
the numerator and of the denominators arising in the recursion therefore form
an at most countable exceptional set of stepsizes in $(0,1/L)$.}} Since $h \in \mathbb{R}$, removing these countably many roots from $\mathbb{R}$, we can guarantee that if $h$ is chosen from this pruned subset of $\mathbb{R}$ (by pruning we mean omitting measure zero choices of $h$), then the condition $det\bigg(p_k \mathbf{I} - q_k [D N_{k-1}(N_{k-1}^{-1}(\x))]^{-1}\bigg) \neq 0 $ is guaranteed. Recalling that $ \probP_1$ is an absolutely continuous probability measure with respect to the Lebesgue measure on the reals, we thus have that $det\bigg(p_k \mathbf{I} - q_k [D N_{k-1}(N_{k-1}^{-1}(\x))]^{-1}\bigg) \neq 0 $ for almost every choice of $h$ or equivalently $N_k$ is a diffeomorphism at the given $\x$ $ \probP_1$-almost surely provided $N_{k-1}$ is a diffeomorphism at $ N_{k-1}^{-1}(\x)$. Using the definition of conditional probability we get that
\begin{align}
     \probP_1\bigg( det (D N_k(\x)) \neq 0  \hspace{0.1cm} \bigg\vert  \hspace{0.1cm} det (D N_{k-1}(N_{k-1}^{-1}(\x))) \neq 0  \bigg) & =1 .
\end{align}
Inducting this result for any $r<k$ we get
\begin{align}
     \probP_1\bigg(  det (D N_{r+1}(N^{-1}_{r+1} \circ \dots \circ N_{k-1}^{-1}(\x))) \neq 0 \hspace{0.1cm} \bigg\vert  \hspace{0.1cm} det (D N_{r}(N^{-1}_r \circ \dots \circ N_{k-1}^{-1}(\x))) \neq 0  \bigg) & =1 .
\end{align}
Then by using the definition of marginal probability we get:
\begin{align}
    \probP_1\bigg( det (D N_k(\x)) \neq 0 \bigg) &= \nonumber \\ & \hspace{-1.5cm}\prod_{0\leq r<k} \probP_1\bigg(  det (D N_{r+1}(N^{-1}_{r+1} \circ \dots \circ N_{k-1}^{-1}(\x))) \neq 0 \hspace{0.1cm} \bigg\vert  \hspace{0.1cm} det (D N_{r}(N^{-1}_r \circ \dots \circ N_{k-1}^{-1}(\x))) \neq 0  \bigg)  \\
    & =1,
\end{align}
where we used the fact that countable product of $1$'s is $1$ and in the base case, $N_{-1}$ is an affine invertible map for any initialization $[\x_0;\x_{-1}]$ and thus a diffeomorphism. The fact that $N_k $ is a local diffeomorphism around $\x = N_{k-1} \circ \dots \circ N_{-1}(\z)$ $\probP_1$-almost surely then follows directly by inverse function theorem. This completes the proof.
\end{proof}

\subsection{Corollary \ref{fixedptcor}}
\begin{proof}
   If $p_k-q_k=1 $ for all $k$, then using the recursion \eqref{newlogic2} we get that any critical point of $f$ is a fixed point of the sequence of maps $\{N_k\}$. To see this we can use an inductive argument. Suppose $\x^*$ is a critical point of $f$ and a fixed point of $N_{k-1}$. Then from \eqref{newlogic2} we have
 \begin{align}
    N_k(\x^*) & = G (p_k \x^* - q_k N_{k-1}^{-1}(\x^*)) \\
    &=   G (p_k \x^* - q_k \x^*) = \x^*,
 \end{align}
 where in the last step we used the invertibility of $N_{k-1}$ to get $N_{k-1}^{-1}(\x^*) = \x^* $ and the fact that $G(\x^*) = \x^* - h \nabla f(\x^*) =\x^*$. But since $\x_0 =\x_{-1}$ we can set $N_{-1} \equiv \mathrm{id}$ and hence the induction is complete. Thus we have that any critical point of $f$ is a fixed point of the sequence of maps $\{N_k\}$. 
\end{proof}  

\subsection{Corollary \ref{corsup1}}
\begin{proof}
    Since $p_k -q_k =1$ for all $k$ with $p_k,q_k$ converging, using \eqref{diffeoeqn} from the proof of Theorem \ref{diffeomorphthm} and using the maps $N_k, R_k$ for any $k$, we get:
    \begin{align}
 D N_k(\x)  &=   (\mathbf{I}- h \nabla^2 f(R_k(\x)) )\underbrace{(p_k \mathbf{I} - q_k [D N_{k-1}(N_{k-1}^{-1}(\x))]^{-1})}_{=D R_k(\x)}. \label{diffeoeqna1}
\end{align}
Next, suppose $\x^*$ is any critical point of $f$. Then $\x^*$ is also a fixed point of the map $N_k$ for any $k$ by Corollary \ref{fixedptcor} and hence the fixed point of map $R_k \equiv p_k \mathrm{id} -q_k N_{k-1}^{-1}$ from the relation $R_k(\x^*) = p_k \x^* -q_k N_{k-1}^{-1}(\x^*) = (p_k-q_k) \x^* = \x^*$. The equivalence $R_k \equiv p_k \mathrm{id} -q N_{k-1}^{-1}$ follows from \eqref{newlogic2} where we use the maps $N_k, R_k$. For $\x = N_{k-1} \circ \dots \circ N_{-1}(\z)$ in \eqref{diffeoeqna1}, on substituting $\z = \x^*$ there we get:
\begin{align}
     D N_k(\x^*)  &=   (\mathbf{I}- h \nabla^2 f(\x^*)){(p_k \mathbf{I} - q_k [D N_{k-1}(\x^*)]^{-1})}, \label{diffeoeqna12}
\end{align}
$\probP_1$-almost surely provided $D N_{k-1}(\x^*)$ is invertible. Since $N_{-1}$ is affine and invertible, $ D N_{-1}(\x^*)$ is a constant invertible matrix. Using the recursion \eqref{diffeoeqna12} and substituting recursively $ D N_r(\x^*)$ for all $r<k$ in this recursion, we will get that $ D N_k(\x^*)$ will be a function of the finite sequences $\{p_j\}_{j=0}^k, \{q_j\}_{j=0}^k $ and the matrices $ (\mathbf{I}- h \nabla^2 f(\x^*)), D N_{-1}(\x^*)$. In particular, $ D N_k(\x^*)$ will be generated by a repeated combination of affine, multiplication and inversion operations on the matrix $(\mathbf{I}- h \nabla^2 f(\x^*)) $. Thus, $ det( D N_k(\x^*) )$ will be a ratio of polynomial functions in the sequences $\{p_j\}_{j=0}^k, \{q_j\}_{j=0}^k $ and the step-size $h$ with degree at most $k+1$ and hence $ det( D N_k(\x^*) )$ will be a finite degree polynomial in $h$. Now, $ det( D N_k(\x^*) ) = 0$ will hold for only finitely many $h$ and therefore for almost every choice of $h \in \mathbb{R}$ or equivalently $\probP_1$-almost surely $ det( D N_k(\x^*) ) \neq 0$. Then by inverse function theorem, $ N_k$ is a local diffeomorphism around $\x^*$ $\probP_1$-almost surely. 

Notice that we were able to get around the analyticity of $f$ from the fact that in the recursion \eqref{diffeoeqna12}, $D N_k(\x^*) $ for any $k$ depends on the same matrix $(\mathbf{I}- h \nabla^2 f(\x^*))$ whereas from the recursion \eqref{diffeoeqna1}, $D N_k(\x) $ depends on the $k$-dependent matrix $ (\mathbf{I}- h \nabla^2 f(R_k(\x)) )$. Now, the matrix $(\mathbf{I}- h \nabla^2 f(\x^*))$ is linear in $h$ because it gets rid of the dependency of the function $f$ on $h$, unlike the matrix $ (\mathbf{I}- h \nabla^2 f(R_k(\x)) )$, where $\nabla^2 f(R_k(\x)) $ will be a function (possibly nonlinear) of $h$ since $R_k$ depends on $h$. Therefore $ det( D N_k(\x^*) )$ can only be a polynomial in $h$ and not some nonlinear function of $h$ even when $ f \in \mathcal{C}^2$ class. This fact enables us to use $ \mathcal{C}^{2,1}_L(\mathbb{R}^n)$ function class as opposed to analytic function class and completes the proof.
\end{proof}  

\subsection{Lemma \ref{lemma_pk}}
\begin{proof}
    Using the update \eqref{generalds_adv} we can write $P_k : \begin{bmatrix}
\x_{k} \\  \x_{k-1}
\end{bmatrix} \mapsto \begin{bmatrix}
\x_{k+1} \\  \x_{k}
\end{bmatrix} $, where we have that:
\begin{equation}
\begin{bmatrix}
 \x_{k+1} \\  \x_{k}
\end{bmatrix} = P_k\bigg(\begin{bmatrix}
\x_k  \\  \x_{k-1}
\end{bmatrix}\bigg)= \begin{bmatrix}
p_k\x_k  -q_k\x_{k-1} - h \nabla f\bigg(p_k\x_k -q_k\x_{k-1}\bigg) \\  \x_{k}
\end{bmatrix}.
\end{equation}
Then extending the map $P_k$ from along the trajectory of $\{[\x_{k};\x_{k-1}]\}$ to any $[\x;\y] \in \mathbb{R}^n \times \mathbb{R}^n \cong \mathbb{R}^{2n}$ we get:
\begin{align}
    P_k\bigg(\begin{bmatrix}
\x  \\  \y
\end{bmatrix}\bigg)= \begin{bmatrix}
p_k\x  -q_k\y - h \nabla f\bigg(p_k\x -q_k\y)\bigg) \\  \x
\end{bmatrix}.
\end{align}
Since $f \in \mathcal{C}^2$ we can differentiate the above equation to get the Jacobian for extended map $P_k$ as follows:
$$ D P_k([\x;\y]) =  \begin{bmatrix}
p_k\bigg(\mathbf{I} - h \nabla^2 f\bigg(p_k\x - q_k\y\bigg)\bigg) \hspace{0.1cm} &  -q_k\bigg(\mathbf{I} - h \nabla^2 f\bigg(p_k\x - q_k\y\bigg)\bigg)\\  \mathbf{I} \hspace{0.2cm} & \mathbf{0}
\end{bmatrix}. $$ Now, if $p_k \to p$, $q_k \to q$ as $k \to \infty$ then $P_k \darrow P$ on compact sets of $\mathbb{R}^{2n}$ by the uniform continuity of $\nabla f$ on compact sets of $\mathbb{R}^{n}$ where the map $P$ satisfies
\begin{align}
    P\bigg(\begin{bmatrix}
\x  \\  \y
\end{bmatrix}\bigg)= \begin{bmatrix}
p\x  -q\y - h \nabla f\bigg(p\x -q\y\bigg) \\  \x
\end{bmatrix},
\end{align}
$$ D P([\x;\y]) =  \begin{bmatrix}
p\bigg(\mathbf{I} - h \nabla^2 f\bigg(p\x - q\y\bigg)\bigg) \hspace{0.1cm} &  -q\bigg(\mathbf{I} - h \nabla^2 f\bigg(p\x - q\y\bigg)\bigg)\\  \mathbf{I} \hspace{0.2cm} & \mathbf{0}
\end{bmatrix}, $$
$P$ is $\mathcal{C}^1$-smooth and $DP_k \darrow DP$ over compact sets of $\mathbb{R}^{2n}$ by the uniform continuity of $\nabla^2 f$ on compact sets of $\mathbb{R}^{n}$. It is easy to check that the maps $P_k, P$ are bijections from $\mathbb{R}^{2n}$ to itself provided $h < \frac{1}{L}$ and $q_k \neq 0$. To see this suppose $ P_k([\x_1;\y_1]) = P_k([\x_2;\y_2])$ and $[\x_1;\y_1] \neq [\x_2; \y_2] $ then from the definition of $P_k$ we get that $\x_1 = \x_2$ and $ p_k \x_1 -q_k \y_1 - h\nabla f(p_k\x_1 - q_k\y_1) = p_k \x_2 -q_k \y_2 - h\nabla f(p_k\x_2 - q_k\y_2)$. Simplifying the second equality and taking norm yields:
\begin{align}
    \norm{q_k(\y_1 - \y_2)} &= h\norm{\nabla f(p_k\x_1 - q_k\y_1) - \nabla f(p_k\x_2 - q_k\y_2)  } \\
    &\leq Lh \norm{q_k(\y_1 - \y_2)} < \norm{q_k(\y_1 - \y_2)}
\end{align}
which is a contradiction provided $q_k \neq 0$. The surjection of $P_k$ follows directly then and using the similar steps we get that $P$ is a bijection. Since the map $P_k$ for any $k$ and the map $P$ maps are bijective, continuous and closed (by invariance of the domain $\mathbb{R}^{2n}$), we get from the Closed map lemma \cite{lee2013smooth} that the map $P_k$ for any $k$ and the map $P$ are homeomorphisms on compact sets and therefore proper. Using the definition of $P_k$ and some simple algebra we get that the map $P_k^{-1}$ is defined as follows:
\begin{align}
    P_k^{-1}\bigg(\begin{bmatrix}
\x  \\  \y
\end{bmatrix}\bigg)= \begin{bmatrix}
\y \\ \frac{p_k}{q_k}\y  -\frac{1}{q_k}(\mathrm{id} - h \nabla f)^{-1}(\x) 
\end{bmatrix},
\end{align}
and then using the inverse function theorem for the map $\mathrm{id} - h \nabla f$, the corresponding Jacobian $DP_k^{-1}$ for the map $P_k^{-1}$ is given by
$$ D P_k^{-1}([\x;\y]) =  \begin{bmatrix} 
\mathbf{0} & \mathbf{I}\\   -\frac{1}{q_k} [DG (G^{-1}(\x))]^{-1}\hspace{0.2cm} & \frac{p_k}{q_k}\mathbf{I}
\end{bmatrix}= \begin{bmatrix} 
\mathbf{0} & \mathbf{I}\\   -\frac{1}{q_k} [\mathbf{I} - h \nabla^2 f (G^{-1}(\x))]^{-1} \hspace{0.2cm} & \frac{p_k}{q_k}\mathbf{I}
\end{bmatrix}, $$
where $G \equiv \mathrm{id} - h \nabla f$ and $ DG (G^{-1}(\x))$ is invertible for $h<\frac{1}{L}$. Then using the definition of map $P_k$, the Jacobian $DP_k^{-1}$ at the point $P_k([\x;\y])$ can be given by
\begin{align}
     D P_k^{-1}(P_k[\x;\y]) &=  \begin{bmatrix} 
\mathbf{0} &  \mathbf{I}\\  -\frac{1}{q_k} [\mathbf{I} - h \nabla^2 f (G^{-1}(G(p_k\x-q_k\y)))]^{-1}\hspace{0.2cm} & \frac{p_k}{q_k}\mathbf{I}
\end{bmatrix} \nonumber \\ &=  \begin{bmatrix} 
\mathbf{0} &  \mathbf{I}\\ -\frac{1}{q_k} [\mathbf{I} - h \nabla^2 f ((p_k\x-q_k\y))]^{-1} \hspace{0.2cm} & \frac{p_k}{q_k}\mathbf{I}
\end{bmatrix}. \label{invertjacobeval1}
\end{align}
Similarly, the map $P^{-1}$ and its Jacobian $DP^{-1}$ are defined.

$\clubsuit \hspace{0.3cm}$ Next, for any $[\x;\y] \in \mathbb{R}^{2n}$, the determinant of $DP_k([\x;\y])$ will be a polynomial function of step-size $h$ of finite degree and hence $det(DP_k([\x;\y]))$ can be zero only for finitely many choices of $h \in \mathbb{R}$. Thus, the map $P_k$ is a local diffeomorphism around $[\x;\y]$ $\probP_1$-almost surely by continuity of the map $DP_k$. As a consequence of the inverse function theorem, the map $P_k^{-1}$ is then a local diffeomorphism around $P_k([\x;\y])$ $\probP_1$-almost surely. {We now use a covering argument to show that $P_k^{-1}$ is a diffeomorphism on all of $\mathbb{R}^{2n}$ $\probP_1$-almost surely.}

{Let $\{\z_i\}$ where $\z_i = [\z_i^1; \z_i^2]$ be the set of points on $\mathbb{R}^{2n}$ such that $\z_i^1 \in \mathbb{R}^n $, $\z_i^2 \in \mathbb{R}^n $ and $\{P_k(\z_i)\}$ be the set of all points on $\mathbb{R}^{2n}$ with rational coordinates. Since $P_k$ is a bijection there is a one-to-one correspondence between the sets $\{P_k(\z_i)\}$ and $\{\z_i\}$. Let $\{O_j\}$ be a countable cover of bounded open sets covering $\mathbb{R}^{2n}$. Such a cover is guaranteed by the Lindel\"of's lemma \cite{kelley2017general}. Let $ \{P_k(\z_{i,j})\} \subset \{P_k(\z_i)\}$ be the collection of points in the compact set $ \bar{O}_j$ for any $j$. Now for any $P_k(\z_{i,j}) \in  \bar{O}_j$ the equation $det( D P_{k}^{-1}(P_k(\z_{i,j}))) = 0 $ will have at most countably many solutions in $h \in \mathbb{R}$ from \eqref{invertjacobeval1}. Hence, for all  $P_k(\z_{i,j}) \in  \bar{O}_j$, the countably many equations given by  $det( D P_{k}^{-1}(P_k(\z_{i,j}))) = 0 $, as $\z_{i,j}$ varies in $\bar{O}_j$, can have at most countably many solutions in $h$. Let that countable solution set of $h$ in the compact set $ \bar{O}_j$ be $S_{j} \subset \mathbb{R}$. By inverse function theorem, if for any $\w$ we have $ det( D P_{k}^{-1}(P_k(\w))) \neq 0$ then there will exist some neighborhood $V$ of $P_k(\w)$ and hence an open ball $\mathcal{B}(P_k(\w))$ centered at $P_k(\w)$ such that $det(DP_k^{-1}(\v) )\neq 0 $ for any $\v \in V \supset \mathcal{B}(P_k(\w))$. Let $r(\w)$ be the positive lower bound on the radius of this ball $\mathcal{B}(P_k(\w))$. From the literature on the domain of inverse function theorem \cite{papi2005domain, xinghua1999convergence} it is known that $r(\w)$ will depend on the inverse of the local Lipschitz constant of the map $ D P_k$ around $\w$ and also $\norm{[DP_k(\w)]^{-1}}^{-1}_2$, provided the map $DP_k$ is locally Lipschitz continuous. In particular, Theorem \ref{injectivitythm} gives an explicit expression for this function $r(\w)$. Since $f$ is Hessian Lipschitz continuous, we get that $DP_k$ will be Lipschitz continuous with some constant $\tilde{M}$. Then we have the following condition:
\begin{align}
  &\hspace{-1cm}\norm{[DP_k(\w)]^{-1}}_2 \norm{DP_k(\v) - DP_k(\w)}_2 \leq \tilde{M} \norm{[DP_k(\w)]^{-1}}_2\norm{\v - \w} \nonumber \\ &\hspace{7cm}, \forall \hspace{0.1cm} \v \in \mathbb{R}^{2n}  \\
\implies   & \norm{[DP_k(\w)]^{-1} DP_k(\v) - \mathbf{I}}_2  \leq \norm{[DP_k(\w)]^{-1}}_2 \norm{DP_k(\v) - DP_k(\w)}_2 \nonumber \\  & \hspace{5cm}\leq \tilde{M} \norm{[DP_k(\w)]^{-1}}_2  \norm{\v - \w} \nonumber \\ &\hspace{7cm}, \forall \hspace{0.1cm} \v \in \mathbb{R}^{2n} \label{injectdiffc1}
 \end{align}
 where $ \norm{[DP_k(\w)]^{-1}}_2$ is bounded and strictly positive from inverse function theorem provided we have $ det( D P_{k}^{-1}(P_k(\w))) \neq 0$. Then using \eqref{injectdiffc1}, we can apply Theorem \ref{injectivitythm} to get $r(\w)$ which is the smallest radius of the ball around the point $P_k(\w)$ in which $P_k^{-1}$ map is a diffeomorphism. Reading from Theorem~\ref{injectivitythm} we have $r(\w) = \bigg(2 \tilde{M}  \norm{[DP_k(\w)]^{-1}}^2_2 \bigg)^{-1} $ and so $r(\w)$ is a continuous function of $\w$ and thereby $r(\w)$ is also a continuous function of $P_k(\w)$ by the continuity of $P_k^{-1}$.} {Next, taking the Frobenius norm on both sides of \eqref{invertjacobeval1} with the substitution of $\w =[\x;\y]$, using the inverse function theorem and the fact that the Frobenius norm of $ \frac{1}{q_k} [\mathbf{I} - h \nabla^2 f ((p_k\x-q_k\y))]^{-1}$ is bounded for any $\w$ when $q_k \neq 0$ and $h < \frac{1}{L}$, we get that $ \norm{[DP_k(\w)]^{-1}}^2_2 \leq \norm{[DP_k(\w)]^{-1}}^2_F < \infty$. Hence $r(\w)$ is strictly positive for any $\w$.}

  {Since $r(\w)$ is a continuous positive function of $P_k(\w)$ and $\bar{O}_j$ is compact, we have $ \inf_{P_k(\w) \in  \bar{O}_j} r(\w) = \epsilon_j > 0$. Now, $ \bigcup_{P_k(\z_{i,j}) \in \bar{O}_j} \mathcal{B}_{\epsilon_j} (P_k(\z_{i,j})) $ is a cover for $ \bar{O}_j$ by the fact that $\{P_k(\z_{i,j})\}$ is the set of all points in $ \bar{O}_j$ with rational coordinates, rationals are dense in reals and $\epsilon_j$ is a fixed positive constant on the set $ \bar{O}_j$. Next, recall that for  $h \in \mathbb{R} \backslash S_j $ we will have $det( D P_{k}^{-1}(P_k(\z_{i,j})) ) \neq 0 $ for all $P_k(\z_{i,j}) \in \bar{O}_j $ which implies that for $h \in \mathbb{R} \backslash S_j $ we will have $det( D P_{k}^{-1}(\v) ) \neq 0 $ for all $\v \in \bigcup_{P_k(\z_{i,j}) \in \bar{O}_j} \mathcal{B}_{r(\z_{i,j})} (P_k(\z_{i,j}))$ by the definition of $ r(\z_{i,j})$ function. But since $ \epsilon_j =\inf_{P_k(\x) \in  \bar{O}_j} r(\x) $ we have $ \bigcup_{P_k(\z_{i,j}) \in \bar{O}_j} \mathcal{B}_{\epsilon_j} (P_k(\z_{i,j})) \subset \bigcup_{P_k(\z_{i,j}) \in \bar{O}_j} \mathcal{B}_{r(\z_{i,j})} (P_k(\z_{i,j}))$ and so for $h \in \mathbb{R} \backslash S_j $ we will have $det(DP_k^{-1}(\v)) \neq 0 $ for all $\v \in  \bar{O}_j $ because $ \bigcup_{P_k(\z_{i,j}) \in \bar{O}_j} \mathcal{B}_{\epsilon_j} (P_k(\z_{i,j})) $ covers $  \bar{O}_j$. Since the open set $ {O}_j$ was arbitrary and $ \bigcup_j {O}_j$ covers $\mathbb{R}^{2n}$, we have that for $h \in \mathbb{R} \backslash \bigcup_j S_j $ the following holds $$det(DP_k^{-1}(\v)) \neq 0, $$ for all $\v \in  \mathbb{R}^{2n}$ where the set $S_j$ is countable for any $j$. Hence we get that $P_k^{-1}$ is a diffeomorphism for $h \in \mathbb{R} \backslash \bigcup_j S_j$ or equivalently $P_k^{-1}$ is a diffeomorphism for almost every $h \in \mathbb{R}$ since the countable set $ \bigcup_j S_j$ has zero Lebesgue measure in $\mathbb{R}$. }

$\spadesuit \hspace{0.3cm}$   

{Recalling that $ \probP_1$ is an absolutely continuous probability measure with respect to the Lebesgue measure on the reals, we can write:
\begin{align}
    \probP_1\bigg( \exists h \in \mathbb{R} \hspace{0.1cm} \text{s.t.} \hspace{0.1cm} det (D P_k^{-1}(P_k(\w))) = 0  \hspace{0.1cm} \text{for some} \hspace{0.1cm} \w \in \mathbb{R}^{2n} \bigg) & =0.
\end{align}
}
Then by the inverse function theorem and bijection of $P_k$ we get
\begin{align}
    \probP_1\bigg( \exists h \in \mathbb{R} \hspace{0.1cm} \text{s.t.} \hspace{0.1cm} det (D P_k(\u)) = 0  \hspace{0.1cm} \text{for some} \hspace{0.1cm} \u \in \mathbb{R}^{2n} \bigg) & =0.
\end{align}
The same conclusion will hold for the map $P$ and this completes the proof.
\end{proof}

\subsection{Corollary \ref{corsup2}}
\begin{proof}
Since $f$ is Hessian Lipschitz continuous in any compact set $\mathcal{D}$ of $\mathbb{R}^n$ and $P_k$ is invertible from Lemma \ref{lemma_pk}, the part of Lemma \ref{lemma_pk}'s proof from `$\clubsuit$' symbol to `$\spadesuit$' symbol establishing the diffeomorphism of $P_k^{-1}$ map will now hold for the compact set $ P_k(\mathcal{D} \times \mathcal{D})$ instead of $\mathbb{R}^{2n}$. Hence we get  
    \begin{align}
    \probP_1\bigg( \exists h \in \mathbb{R} \hspace{0.1cm} \text{s.t.} \hspace{0.1cm} det (D P_k^{-1}(P_k(\w))) = 0  \hspace{0.1cm} \text{for some} \hspace{0.1cm} \w \in \mathcal{D} \times \mathcal{D} \bigg) & =0.
\end{align}
Then by the inverse function theorem and bijection of the map $P_k$ we get
\begin{align}
    \probP_1\bigg( \exists h \in \mathbb{R} \hspace{0.1cm} \text{s.t.} \hspace{0.1cm} det (D P_k(\u)) = 0  \hspace{0.1cm} \text{for some} \hspace{0.1cm} \u \in \mathcal{D} \times \mathcal{D} \bigg) & =0,
\end{align}
which proves that the map $P_k$ is a $\probP_1$-almost sure diffeomorphism from $\mathcal{D} \times \mathcal{D} $ to $P_k(\mathcal{D} \times \mathcal{D})$. Similarly, by inverse function theorem and bijection of the map $P$ we get
\begin{align}
    \probP_1\bigg( \exists h \in \mathbb{R} \hspace{0.1cm} \text{s.t.} \hspace{0.1cm} det (D P(\u)) = 0  \hspace{0.1cm} \text{for some} \hspace{0.1cm} \u \in \mathcal{D} \times \mathcal{D} \bigg) & =0,
\end{align}  
which proves that the map $P$ is a $\probP_1$-almost sure diffeomorphism from $\mathcal{D} \times \mathcal{D} $ to $P(\mathcal{D} \times \mathcal{D})$.

 Next, under the given conditions on $p_k,q_k$, the update from \eqref{generalds_adv} becomes \eqref{generalds} with $\beta_k \leq \frac{1}{\sqrt{2}}$. Since $f$ is coercive, its sublevel sets are compact. Then by Lemma \ref{lemsublevelanalytic} (proved later), we get that the algorithmic maps $ N_k : \x_k \mapsto \x_{k+1}$ for any $k$ from Theorem \ref{diffeomorphthm} will satisfy $ N_k : \mathcal{D} \to \mathcal{D}$ for any sublevel set $\mathcal{D}$ of coercive $f$ where $\mathcal{D}$ will be compact and hence we also have $ P_k \equiv [ N_k; N_{k-1}] : \mathcal{D} \times \mathcal{D}  \to \mathcal{D} \times \mathcal{D}  $ for all $k \geq 0$. The map $P$ by definition is the uniform limit of the sequence of maps $\{P_k\}$. Suppose $\u \in \mathcal{D} \times \mathcal{D}$ then $P_k(\u) \in \mathcal{D} \times \mathcal{D}$ for all $k \geq 0$. Since $ \mathcal{D} \times \mathcal{D}$ is compact and $P_k \darrow P$ in $ \mathcal{D} \times \mathcal{D}$, the sequence $\{P_k(\u)\}$ will converge in $ \mathcal{D} \times \mathcal{D}$ and hence $ P(\u) \in  \mathcal{D} \times \mathcal{D}$. Since $\u$ was arbitrary, we have $ P : \mathcal{D} \times \mathcal{D}  \rightarrow \mathcal{D} \times \mathcal{D}$.  
This completes the proof.
\end{proof}  

\subsection{Theorem \ref{measuretheorem2}}
\begin{proof}
Since $P$ is a diffeomorphism on every compact set and $D P([\x^*,\x^*])$ has at least one eigenvalue with magnitude greater than $1$, the map $P$ satisfies the criterion of the map $\phi$ from Theorem \ref{measuretheorem1} and there exists a $\mathcal{C}
^{{m}}$ embedded disc $\mathcal{W}^{CS}_{loc}$ with $m\geq 1$ that is tangent to $\mathcal{E}_{CS}$ at $[\x^*,\x^*]$, $\mathcal{E}_{CS}$ being the invariant subspace corresponding to the eigenvalues of $D P([\x^*,\x^*])$ whose magnitude is less than or equal to $1$. \footnote{Note that it may be the case that the Jacobian matrix $D P([\x^*,\x^*]) $ is a defective matrix and does not have a complete eigenbasis. Then we can always use the generalized eigenvectors in order to extend the incomplete basis of eigenvectors to a complete basis so that the eigenspace $ \mathcal{E}_{CS} \bigoplus \mathcal{E}_{US}$ of $D P([\x^*,\x^*])$ from Theorem \ref{measuretheorem1} spans $\mathbb{R}^{2n}$.}Hence, from Theorem \ref{measuretheorem1} there exists a neighborhood $\mathcal{B}$ of $[\x^*,\x^*]$ such
that $P(\mathcal{W}^{CS}_{loc}) \cap \mathcal{B} \subset \mathcal{W}^{CS}_{loc}$, and that if $\z$ is a point such that\footnote{ {Here $P^k $ denotes the composition of $P$ map $k$-times.}} $P^k
(\z) \in \mathcal{B}$ for all $k \geq 0$,
then $\z \in  \mathcal{W}^{CS}_{loc}$. Then for any converging trajectory generated from the recursion \begin{align}
    [\x_{k+1};\x_k] =  \begin{cases} 
      P_k([\x_k; \x_{k-1}]) & 0 \leq k\leq r \\
      P([\x_k; \x_{k-1}]) & k > r 
   \end{cases} \label{recurspfn1}
\end{align} for any $r \geq 0$ and initialized in any bounded neighborhood $\mathcal{U} $ of $[\x^*,\x^*]$, there will be some finite $l$ such that $P^k([\x_l; \x_{l-1}]) \in \mathcal{B} \bigcap \mathcal{W}^{CS}_{loc}$ for all $k \geq 0$.  

Let $\bigg \{[\x_0; \x_{-1}] \in \mathcal{U}\backslash [\x^*,\x^*] \hspace{0.1cm} \bigg\vert \hspace{0.1cm} [\x_k; \x_{k-1}] \to [\x^*,\x^*]  \bigg\}$ be the set of all possible initializations of \eqref{recurspfn1} for which $[\x_k; \x_{k-1}] \to [\x^*,\x^*] $. Then the map $P$ after iterating countable number of times pushes forward the set $$P_r \circ \dots \circ P_0\bigg(\bigg\{[\x_0; \x_{-1}] \in \mathcal{U}\backslash [\x^*,\x^*] \hspace{0.1cm} \bigg\vert \hspace{0.1cm} [\x_k; \x_{k-1}] \to [\x^*,\x^*]  \bigg\}\bigg)$$ to a subset of $ \mathcal{B} \bigcap \mathcal{W}^{CS}_{loc}$. Therefore we have the following containment: 
\begin{align}
P_r \circ \dots \circ P_0\bigg(\bigg\{[\x_0; \x_{-1}] \in \mathcal{U}\backslash [\x^*,\x^*] \hspace{0.1cm} \bigg\vert \hspace{0.1cm} &[\x_k; \x_{k-1}] \to [\x^*,\x^*]  \bigg\}\bigg) \subseteq \bigcup_{l \geq 0}^{\infty} P^{-l}\bigg(\mathcal{W}^{CS}_{loc} \cap \mathcal{B}\bigg). \label{containment1ab} 
\end{align}
 Next, note that the set $$\bigg\{[\x_0; \x_{-1}] \in \mathcal{U}\backslash [\x^*,\x^*] \hspace{0.1cm} \bigg\vert \hspace{0.1cm} [\x_k; \x_{k-1}] \to [\x^*,\x^*]  \bigg\}$$ is pre-compact as $\mathcal{U} $ is pre-compact and so the set $$P_r \circ \dots \circ P_0\bigg(\bigg\{[\x_0; \x_{-1}] \in \mathcal{U}\backslash [\x^*,\x^*] \hspace{0.1cm} \bigg\vert \hspace{0.1cm} [\x_k; \x_{k-1}] \to [\x^*,\x^*]  \bigg\}\bigg)  $$ will be pre-compact by continuity of $P_k$ for all $k$ where this set is covered by the countable open cover $\bigcup_{l \geq 0}^{\infty} P^{-l}\bigg(\mathcal{W}^{CS}_{loc} \cap \mathcal{B}\bigg)  $ ( the set $\mathcal{W}^{CS}_{loc} \cap \mathcal{B}$ is open in the relative topology and pre-image of open sets under any composition of continuous maps is open). Then by Heine-Borel theorem \cite{kirillov2012theorems}, there is a finite sub-cover of the countable cover $\bigcup_{l \geq 0}^{\infty} P^{-l}\bigg(\mathcal{W}^{CS}_{loc} \cap \mathcal{B}\bigg) $ given by
\begin{align}
   P_r \circ \dots \circ P_0\bigg(\bigg\{[\x_0; \x_{-1}] \in \mathcal{U}\backslash [\x^*,\x^*] \hspace{0.1cm} \bigg\vert \hspace{0.1cm} &[\x_k; \x_{k-1}] \to [\x^*,\x^*]  \bigg\}\bigg) \subseteq \bigcup_{l \in \mathcal{J}} P^{-l}\bigg(\mathcal{W}^{CS}_{loc} \cap \mathcal{B}\bigg) 
    \label{containment1}
\end{align}
where $\abs{\mathcal{J}} = J$ for some large enough finite $J$ by Heine-Borel theorem \cite{kirillov2012theorems}. This finite sub-cover will be pre-compact because the set $\mathcal{W}^{CS}_{loc} \cap \mathcal{B}$ is pre-compact, $P$ is invertible and proper map and thus $P^{-l}\bigg(\mathcal{W}^{CS}_{loc} \cap \mathcal{B}\bigg)  $ for any finite $l \in \mathcal{J}$ will be pre-compact by which $\bigcup_{l \in \mathcal{J}} P^{-l}\bigg(\mathcal{W}^{CS}_{loc} \cap \mathcal{B}\bigg)  $ is pre-compact since any finite union of pre-compact sets is a pre-compact set.   

Since $D P([\x^*,\x^*])$ has at least one eigenvalue with magnitude greater than $1$, the dimension of $ \mathcal{W}^{CS}_{loc}$ will be strictly less than $2n$ and therefore $ \mathcal{W}^{CS}_{loc}$ has zero Lebesgue measure in $\mathbb{R}^{2n}$. Then the set $ \bigcup_{l \in \mathcal{J}} P^{-l}\bigg(\mathcal{W}^{CS}_{loc} \cap \mathcal{B}\bigg) $ is Lebesgue measure zero since $P$ is a diffeomorphism on compact sets, the set $\mathcal{W}^{CS}_{loc} \cap \mathcal{B}$ is pre-compact, $P^{-l}\bigg(\mathcal{W}^{CS}_{loc} \cap \mathcal{B}\bigg)  $ for any $l$ will be pre-compact since $P$ is invertible and proper map, pull back of any pre-compact Lebesgue null set under the map $P$ is a pre-compact Lebesgue null set and countable union of null sets is again a null set. Since the initialization set will satisfy 
$$ \bigg\{[\x_0; \x_{-1}] \in \mathcal{U}\backslash [\x^*,\x^*] \hspace{0.1cm} \bigg\vert \hspace{0.1cm} [\x_k; \x_{k-1}] \to [\x^*,\x^*]  \bigg\} \subseteq P_0^{-1}\circ \dots \circ  P_r^{-1}\bigg(\bigcup_{l \in \mathcal{J}} P^{-l}\bigg(\mathcal{W}^{CS}_{loc} \cap \mathcal{B}\bigg) \bigg), $$ where $P_k$ are diffeomorphisms for any $k$ on any compact set, the set $\bigcup_{l \in \mathcal{J}} P^{-l}\bigg(\mathcal{W}^{CS}_{loc} \cap \mathcal{B}\bigg) $ is pre-compact with zero Lebesgue measure, $P_k$ are proper, invertible maps for any $k$, pre-image of a pre-compact set under proper map is pre-compact, and pull back of any pre-compact Lebesgue null set under the map $P_k$ for any $k$ is a pre-compact Lebesgue null set, we get that for any $r \geq 0$, the initialization set $$\bigg\{[\x_0; \x_{-1}] \in \mathcal{U}\backslash [\x^*,\x^*] \hspace{0.1cm} \bigg\vert \hspace{0.1cm} [\x_k; \x_{k-1}] \to [\x^*,\x^*]  \bigg\}$$  has zero Lebesgue measure.
From this we conclude that for any $r \geq 0$, $$\probP_2\bigg(\bigg\{[\x_k;\x_{k-1}] \to [\x^*;\x^*]\bigg\}\bigg) = 0 $$  {where the probability measure $\probP_2$ is absolutely continuous with respect to the Lebesgue measure on $\mathbb{R}^{2n}$.}
 This completes the proof.
\end{proof}   

\subsection{Supporting lemmas for Theorem \ref{measuretheorem3}}
Before presenting the proof for Theorem \ref{measuretheorem3} we need some supporting lemmas.
\begin{lemm}\label{measurelemmasup1}
The set of fixed points of the continuous map $P$ given by $ \{\w : \hspace{0.1cm} \w = P(\w) \}$ and the set of limit points of all convergent sequences $\{\lim_k\w_k  :  \hspace{0.1cm} \w_{k+1} = P(\w_k); \hspace{0.1cm} \w_0 \in \mathbb{R}^{2n} \} $ are the same. Furthermore, let $\{\lim_k\w_k  :  \hspace{0.1cm} \w_{k+1} = P_k(\w_k); \hspace{0.1cm} \w_0 \in \mathbb{R}^{2n} \} $ be the set of limit points of all convergent sequences generated by $\w_{k+1} = P_k(\w_k) $, then this set is contained in the set $ \{\w : \hspace{0.1cm} \w = P(\w) \} $ provided $P_k$ is continuous for all $k$ and $P_k \darrow P$ on any compact set.
\end{lemm}
\begin{proof}
Suppose $\w^*$ is a limit point of any convergent sequence $\{\w_k\} $ generated by the recursion $\w_{k+1} = P(\w_k)$, then we have
\begin{align}
    \norm{P(\w^*) - \w^*} & \leq \norm{P(\w_k) - \w^*} + \norm{P(\w_k) -P(\w^*)} \\
    \implies \lim_{k \to \infty} \norm{P(\w^*) - \w^*} & \leq \underbrace{\lim_{k \to \infty}\norm{\w_{k+1} - \w^*}}_{=0} + \underbrace{\lim_{k \to \infty}\norm{P(\w_k) -P(\w^*)}}_{=0 \hspace{0.1cm} \text{ by continuity of $P$}} \\
    \implies  \norm{P(\w^*) - \w^*} & = 0 ,
\end{align}
which proves $ \{\w : \hspace{0.1cm} \w = P(\w) \} \supseteq \{\lim_k\w_k  :  \hspace{0.1cm} \w_{k+1} = P(\w_k); \hspace{0.1cm} \w_0 \in \mathbb{R}^{2n} \}$. In the other direction let $\w$ be a fixed point of $P$ then the sequence $\{\w_k :  \hspace{0.1cm} \w_{k+1} = P(\w_k); \hspace{0.1cm} \w_0 = \w\}$ has limit equal to $\w$ which proves that  $ \{\w : \hspace{0.1cm} \w = P(\w) \} \subseteq \{\lim_k\w_k  :  \hspace{0.1cm} \w_{k+1} = P(\w_k); \hspace{0.1cm} \w_0 \in \mathbb{R}^{2n} \}$.

For the second part, suppose $\w^*$ be any fixed point of $P$ such that it is also a limit point of any convergent sequence $\{\w_k\} $ generated by the recursion $\w_{k+1} = P_k(\w_k)$, then we have
\begin{align}
    \norm{P(\w^*) - \w^*} & \leq \norm{P_k(\w_k) - \w^*} + \norm{P_k(\w_k) -P(\w_k)} + \norm{P(\w_k) -P(\w^*)}\\
    \implies \lim_{k \to \infty} \norm{P(\w^*) - \w^*} & \leq \underbrace{\lim_{k \to \infty}\norm{\w_{k+1} - \w^*}}_{=0} + \underbrace{\lim_{k \to \infty}\norm{P_k(\w_k) -P(\w_k)}}_{=0 \hspace{0.1cm} \text{by uniform continuity of $P_k$, $P_k \darrow P$ on compact sets}} \nonumber \\&+ \underbrace{\lim_{k \to \infty}\norm{P(\w_k) -P(\w^*)}}_{=0 \hspace{0.1cm} \text{ by continuity of $P$}} \\
    \implies  \norm{P(\w^*) - \w^*} & = 0 ,
\end{align}
which proves $ \{\w : \hspace{0.1cm} \w = P(\w) \} \supseteq \{\lim_k\w_k  :  \hspace{0.1cm} \w_{k+1} = P_k(\w_k); \hspace{0.1cm} \w_0 \in \mathbb{R}^{2n} \}$.
\end{proof}

\begin{lemm}\label{measurelemmasup2_a}
Suppose $P_k$ for all $k$ and $P$ are $\mathcal{C}^1$-smooth, $DP_k \darrow DP $ on compact sets, $\w^*$ is a fixed point of $P$, $ \norm{D P(\w^*)}_2 < 1$ and that $$\delta  \in \bigg(0,\inf\bigg\{ r > 0 \hspace{0.1cm} \bigg \vert \hspace{0.1cm} \sup_{\z \in \mathcal{B}_{r}(\w^*)}\norm{D P(\z)}_2 \geq 0.5( 1 + \norm{D P(\w^*)}_2) \bigg \}\bigg).$$ Then there exists an absolute constant $K_0 > 0$ satisfying
$$ K_0 = \inf \bigg\{ k > 0 \hspace{0.1cm} \bigg \vert \hspace{0.1cm} \sup_{\z \in \mathcal{B}_{10}(\w^*)}\norm{D P_k(\z) - DP(\z)}_2 \leq 0.25( 1 - \norm{D P(\w^*)}_2) \bigg \} ,$$
such that if any sequence generated by the recursion $\w_{k+1} = P_k(\w_k)$ or any sequence generated by
\begin{align}
    \w^r_{k+1} =  \begin{cases} 
      P_k(\w^r_k) & 0 \leq k\leq r \\
      P(\w^r_k) & k > r 
   \end{cases} 
\end{align}
for any $r \geq 0$, enters the ball $ \mathcal{B}_{\delta}(\w^*) $ at any $k \geq K_0$ where $\delta \ll 10$, it must thereafter converge to $\w^*$ at a geometric rate. In particular, the constant of geometric rate is $\upsilon < 1$ and it depends only on $ \norm{D P(\w^*)}_2 $.
\end{lemm}
\begin{proof}
    First observe that $g(r) := \sup_{\z \in \mathcal{B}_{r}(\w^*)}\norm{D P(\z)}_2 $ is continuous non-decreasing in $r$ and thus $$\infty \geq  \inf\bigg\{ r > 0 \hspace{0.1cm} \bigg \vert \hspace{0.1cm} g(r) \geq 0.5( 1 + \norm{D P(\w^*)}_2) \bigg \} > 0$$ from the fact that $ 0.5( 1 + \norm{D P(\w^*)}_2) > \norm{D P(\w^*)}_2 $. Since $g(r)$ is continuous, non-decreasing, we must have $$ g(\delta) =  \sup_{\z \in \mathcal{B}_{\delta}(\w^*)}\norm{D P(\z)}_2 \leq 0.5( 1 + \norm{D P(\w^*)}_2) $$ from the definition of $\delta$ and continuity of $DP$. 

    Next, from the integral form of mean value theorem, the sequence $ \{\w^r_{k}\}$ on entering the ball $ \mathcal{B}_{\delta}(\w^*) $ for any $k \geq K_0$ will satisfy:
    \begin{align}
        \norm{\w^r_{k+1} - \w^*} & \leq  \bigg(\mathbbm{1}_{k\leq r}\int_{t=0}^1 \norm{DP_k(\w^* + t(\w^r_k -\w^*))}_2 dt \bigg)\norm{\w^r_{k} - \w^*} + \nonumber \\ & \bigg(\mathbbm{1}_{k > r}\int_{t=0}^1 \norm{DP(\w^* + t(\w^r_k -\w^*))}_2 dt \bigg)\norm{\w^r_{k} - \w^*} \\
    \implies \norm{\w^r_{k+1} - \w^*} & \leq \bigg( \mathbbm{1}_{k\leq r}\sup_{\z \in {\mathcal{B}_{\delta}(\w^*)} }\norm{DP_k(\z)}_2 + \mathbbm{1}_{k > r} \sup_{\z \in {\mathcal{B}_{\delta}(\w^*)} }\norm{DP(\z)}_2 \bigg) \norm{\w^r_{k} - \w^*} \\
     \implies \norm{\w^r_{k+1} - \w^*} & \leq \bigg( \mathbbm{1}_{K_0 \leq k\leq r} \bigg(\sup_{\z \in {\mathcal{B}_{\delta}(\w^*)} }\norm{DP(\z)}_2 + \sup_{\z \in {\mathcal{B}_{\delta}(\w^*)} }\norm{DP_k(\z) - DP(\z)}_2  \bigg) \nonumber \\ & + \mathbbm{1}_{k > r} \sup_{\z \in {\mathcal{B}_{\delta}(\w^*)} }\norm{DP(\z)}_2 \bigg) \norm{\w^r_{k} - \w^*}.
    \end{align}
    Using $g(\delta)$ and $K_0$ definitions in the last step gives the bound:
        \begin{align}
         \norm{\w^r_{k+1} - \w^*} & \leq \bigg( \mathbbm{1}_{K_0 \leq k\leq r} \bigg( 0.5( 1 + \norm{D P(\w^*)}_2) +  0.25( 1 - \norm{D P(\w^*)}_2)  \bigg) \nonumber \\ & + \mathbbm{1}_{k > r}  0.5( 1 + \norm{D P(\w^*)}_2) \bigg) \norm{\w^r_{k} - \w^*} \\
         & \leq \bigg( \mathbbm{1}_{K_0 \leq k\leq r} ( 0.75 + 0.25\norm{D P(\w^*)}_2)   + 0.5\mathbbm{1}_{k > r}  ( 1 + \norm{D P(\w^*)}_2) \bigg) \norm{\w^r_{k} - \w^*} \\
         & \leq ( 0.75 + 0.25\norm{D P(\w^*)}_2) \norm{\w^r_{k} - \w^*}
    \end{align}
    by which we have a geometric convergence rate to $\w^*$ with a geometric constant $ \upsilon = ( 0.75 + 0.25\norm{D P(\w^*)}_2) < 1 $. Similarly, for the case of sequence $ \{\w_{k}\}$ generated by the recursion $\w_{k+1} = P_k(\w_k)$ we will have
    \begin{align}
         \norm{\w_{k+1} - \w^*} & \leq ( 0.75 + 0.25\norm{D P(\w^*)}_2) \norm{\w_{k} - \w^*}
    \end{align}
     by which we again have a geometric convergence rate to $\w^*$ with a geometric constant $ \upsilon = ( 0.75 + 0.25\norm{D P(\w^*)}_2) < 1 $.
\end{proof}

\begin{lemm}\label{measurelemmasup2}
 Suppose $P_k$ for all $k$ and $P$ are $\mathcal{C}^1$-smooth, $DP_k \darrow DP $ on compact sets, $\w^*$ is a fixed point of the map $P$, $\norm{D P(\w^*)}_2 < 1$ and $\delta, K_0$ are defined from Lemma \ref{measurelemmasup2_a}. Let the sequence $\{\w^r_k\} $ for any $r \geq 0$ is generated by the following recursion
\begin{align}
    \w^r_{k+1} =  \begin{cases} 
      P_k(\w^r_k) & 0 \leq k\leq r \\
      P(\w^r_k) & k > r 
   \end{cases} .\label{recur1abc}
\end{align} Then if the sequence $\{\w_k\}$ generated by the recursion $\w_{k+1} = P_k(\w_k)$ for any initialization $\w_0$, hits the ball $\mathcal{B}_{\delta}(\w^*)$ in the iteration interval $[K_0, N_0]$, then, for the same initialization $\w_0$, the sequences $ \{\w^r_k\}$ generated by \eqref{recur1abc} for all $r \geq N_0$ and the sequence $\{\w_k\}$ generated by the recursion $\w_{k+1} = P_k(\w_k)$, converge to $\w^* $ uniformly\footnote{By uniform convergence, we mean that for every $\delta >0$, there exists a $K > 0$ such that $ \norm{\w_k^r - \w^*} < \delta$ for all $k > K$ and all $r \geq r_0$.}. Also, if for any initialization $\w_0$ and some $r_0 \geq N_0$, the sequence $ \{\w^{r_0}_k\}$ generated by \eqref{recur1abc} hits the ball $\mathcal{B}_{\delta}(\w^*)$ in the iteration interval $[K_0, N_0]$, then, for the same initialization $\w_0$, the sequences $ \{\w^r_k\}$ generated by \eqref{recur1abc} for all $r \geq r_0$ and the sequence $\{\w_k\}$ generated by the recursion $\w_{k+1} = P_k(\w_k)$, converge to $\w^* $ uniformly.
\end{lemm}
\begin{proof}
For the first part, observe that for the same initialization $\w_0$, the sequences $ \{\w^r_k\}$ generated by \eqref{recur1abc} for all $r \geq N_0$ and the sequence $\{\w_k\}$ generated by the recursion $\w_{k+1} = P_k(\w_k)$ will have at least the first $N_0$ elements identical and thus all these sequences will also hit the ball $\mathcal{B}_{\delta}(\w^*)$ in the iteration interval $[K_0, N_0]$. Then by Lemma \ref{measurelemmasup2_a}, all these sequences will exhibit a uniform geometric convergence for $k \geq N_0$ where the geometric decay constant for all the sequences depends only on $\norm{D P(\w^*)}_2 < 1$. Similar argument holds for the second part. This completes the proof.
\end{proof}

\begin{defi}
    Let $(X,d)$ be a metric space. For $\epsilon >0$ write $$A^{(\epsilon)} = \{x \in X : d(x,a)< \epsilon \text{ for some } a \in A\}.$$
    Let $2^X$ denote the space of compact, non-empty subsets of $X$, then for $A,B \subseteq X$ the Hausdorff distance $d_H$ on $2^X$ is given by
$$d_{H}(A,B) = \inf\{ \epsilon >0 :  A \subseteq B^{(\epsilon)} \text{ and } B \subseteq A^{(\epsilon)}\}.$$
\end{defi}

\begin{lemm}\label{measurelemmasup3}
    Let $X$ be a compact metric space, the power set $2^X =
\{A \subseteq X : \emptyset  \neq A, A \text{ is compact}\}$ be equipped with the Hausdorff metric $d_H(\cdot,\cdot)$ and $\{A_n\}$ be a sequence of events in $X$. Then $ A_n \to A$ in $2^X$ (convergence with respect to the Hausdorff metric and the limit defined in the set theoretic sense) iff the following two conditions hold:
\begin{itemize}
    \item[(i)] If $ x \in A$ then there exists $ x_r \in A_r$ with $x_r \to x$.
    \item[(ii)] If $ x_{r_j} \in A_{r_j}$ and $x_{r_j} \to x$ then $x \in A$. 
\end{itemize}
\end{lemm}

The proof of Lemma \ref{measurelemmasup3} is in \cite{convergenceofsets} and its background details can be found in the lectures \cite{convergenceofsets}. The next lemma will help us in characterizing the limits of the sequence of union of disjoint events. Though seemingly trivial, we present its proof for sake of completeness. 

\begin{lemm}\label{measurelemmasup4}
   Suppose $\{A_k\}, \{B_k\}$ are sequence of events such that $A_i \bigcap B_j = \emptyset$ for all $i,j \geq 0$ and we have $\lim_k (A_k \bigcup B_k) = A \bigcup B$ where $A \bigcap B = \emptyset$ and for all $k\geq 0$ $  A_k \bigcap B = \emptyset$, $  B_k \bigcap A = \emptyset$. Then we have $\lim_k A_k  = A $ and $\lim_k  B_k =  B$ where the limits of the events are in the set theoretic sense. 
\end{lemm}
\begin{proof}
    Let $\overline{A}_k = \bigcap_{j=0}^k \bigcup_{r \geq j} A_r $, $\underbar{A}_k = \bigcup_{j=0}^k \bigcap_{r \geq j} A_r $, $\overline{B}_k = \bigcap_{j=0}^k \bigcup_{r \geq j} B_r $ and $\underbar{B}_k = \bigcup_{j=0}^k \bigcap_{r \geq j} B_r $. Next, let $\overline{A} = \limsup_k A_k =\bigcap_{j=0}^{\infty} \bigcup_{r \geq j} A_r $, $\underbar{A} = \liminf_k A_k = \bigcup_{j=0}^{\infty} \bigcap_{r \geq j} A_r $, $\overline{B} = \limsup_k B_k = \bigcap_{j=0}^{\infty} \bigcup_{r \geq j} B_r $ and $\underbar{B} =  \liminf_k B_k = \bigcup_{j=0}^{\infty} \bigcap_{r \geq j} B_r $. Clearly, we have $ \overline{A}_k \bigcap \overline{B}_k =  \underbar{A}_k \bigcap \underbar{B}_k = \emptyset $ for all $k \geq 0$ since 
    \begin{align}
        \overline{A}_k \bigcap \overline{B}_k &= \bigg(\bigcap_{j=0}^k \bigcup_{r \geq j} A_r\bigg)  \bigcap \bigg( \bigcap_{j=0}^k \bigcup_{r \geq j} B_r\bigg) \\
        &= \bigcap_{j=0}^k \bigg( \bigg( \bigcup_{r \geq j} A_r\bigg)  \bigcap \bigg(  \bigcup_{r \geq j} B_r\bigg) \bigg) \\
        & = \bigcap_{j=0}^k \bigcup_{r,l \geq j} \bigg(   A_r  \bigcap  B_l \bigg)
    \end{align}
   and countable unions and intersections of empty sets is an empty set. Similar argument holds for $ \underbar{A}_k \bigcap \underbar{B}_k$. Hence, $ \lim_k \bigg(\overline{A}_k \bigcap \overline{B}_k \bigg) =  \lim_k \bigg(\underbar{A}_k \bigcap \underbar{B}_k \bigg)= \emptyset  $ and thus $\overline{A} \bigcap \overline{B} = \underbar{A} \bigcap \underbar{B} = \emptyset$.
   We are given $\lim_k (A_k \bigcup B_k) = A \bigcup B$ which implies 
   $$ \liminf_k \bigg(A_k \bigcup B_k \bigg) = \bigcup_{j=0}^{\infty} \bigcap_{r \geq j} \bigg(A_k \bigcup B_k \bigg) =  \bigcap_{j=0}^{\infty} \bigcup_{r \geq j} \bigg(A_r \bigcup B_r \bigg) =  \limsup_k \bigg(A_k \bigcup B_k \bigg).$$
   But $ \limsup_k \bigg(A_k \bigcup B_k \bigg) = \limsup_k A_k \bigcup \limsup_k  B_k =  \overline{A} \bigcup  \overline{B}$ and since $ \underbar{A} \subseteq \overline{A} $, $ \underbar{B} \subseteq \overline{B} $ as well as $\lim_k (A_k \bigcup B_k) = A \bigcup B$, we get that $\underbar{A} \bigcup \underbar{B} \subseteq  \overline{A} \bigcup  \overline{B}  = A \bigcup B$. Since $\overline{A} \bigcap \overline{B} = \emptyset$, ${A} \bigcap {B} = \emptyset$ it must be that either
   \begin{align}
       \overline{A} = A, \hspace{0.2cm} \overline{B} = B, \label{case1cond1} 
   \end{align}
   or
   \begin{align}
       \overline{A} = B, \hspace{0.2cm} \overline{B} = A. \label{case2cond2}
   \end{align}
   Since for all $k\geq 0$ we have $  A_k \bigcap B = \emptyset$, $  B_k \bigcap A = \emptyset$, we get that $ \lim_k (A_k \bigcap B) = \emptyset$, $ \lim_k (B_k \bigcap A) = \emptyset$. Therefore we get
   \begin{align}
     \overline{A} \bigcap B = (\limsup_k A_k) \bigcap B =  \bigg(\bigcap_{j=0}^{\infty} \bigcup_{r \geq j} A_r \bigg)\bigcap B = \bigcap_{j=0}^{\infty}\bigg(\bigg( \bigcup_{r \geq j} A_r\bigg) \bigcap B\bigg) &= \bigcap_{j=0}^{\infty}\bigcup_{r \geq j}\bigg(  A_r \bigcap B\bigg)\nonumber\\ &=\limsup_k (A_k \bigcap B)\nonumber \\ &= \emptyset 
   \end{align}
   since $\lim_k (A_k \bigcap B) = \emptyset $. Similarly, we get that $\overline{B} \bigcap A = \emptyset $ using $\lim_k (B_k \bigcap A) = \emptyset $. Therefore, the second case from \eqref{case2cond2} is not possible and so we have $$  \overline{A} = A, \hspace{0.2cm} \overline{B} = B$$ from \eqref{case1cond1}. Next, we will have $ \lim_k(A_k \bigcup B)= A \bigcup B$ from the following step:
   \begin{align}
       \lim_k\bigg(A_k \bigcup B\bigg) & = \limsup_k \bigg(A_k \bigcup B\bigg) =  \limsup_k A_k \bigcup \limsup_k B = \overline{A} \bigcup B. \label{case3cond1}
   \end{align}
   Also, we will have that $ \underbar{A}  \bigcup B = \lim_k\bigg(A_k \bigcup B\bigg)$ from the following steps:
     \begin{align}
    \underbar{A}  \bigcup B = (\liminf_k A_k) \bigcup B =  \bigg(\bigcup_{j=0}^{\infty} \bigcap_{r \geq j} A_r \bigg)\bigcup B = \bigcup_{j=0}^{\infty}\bigg(\bigg( \bigcap_{r \geq j} A_r\bigg) \bigcup B\bigg) &= \bigcup_{j=0}^{\infty}\bigcap_{r \geq j}\bigg(  A_r \bigcup B\bigg)\nonumber\\ &=\liminf_k (A_k \bigcup B). \label{case4cond1}
   \end{align}
   Combining \eqref{case3cond1}, \eqref{case4cond1} we get:
   \begin{align}
       \underbar{A}  \bigcup B = \liminf_k \bigg(A_k \bigcup B\bigg) =  \lim_k\bigg(A_k \bigcup B\bigg) = \overline{A} \bigcup B , \label{case5cond1}
   \end{align}
   and since $\underbar{A}  \subseteq \overline{A} $, $ \overline{A}\bigcap B = \emptyset   $, from \eqref{case5cond1} it must be that $\underbar{A}  = \overline{A} =A$ and similarly $\underbar{B}  = \overline{B} = B  $ which completes the proof.
\end{proof}  

We are now ready to present the proof of Theorem \ref{measuretheorem3}.
\subsection{Theorem \ref{measuretheorem3}}
\begin{proof}
  {Let $X$ be the space of convergent $2n$-dimensional vector sequences on the real field denoted by $\w_{\bullet} = \{\w_k\}_{k=0}^{\infty}$ equipped with the norm $ \norm{\w_{\bullet}}_{\ell^{\infty}(\mathbb{R}^{2n})} = \sup_{1\leq i \leq 2n} \norm{[\w_{\bullet}]_i}_{\ell^{\infty}}$ and $ [\w_{\bullet}]_i$ is the scalar sequence corresponding to the $i^{th}$ entry of the vector $ \w_k$ in $\w_{\bullet}$. We first show that the dual $X^*$ of $X$ is the space of $2n$-dimensional vector sequences $\y_{\bullet} = \{\y_k\}_{k=0}^{\infty} \in X^* $ equipped with $\ell^{1}(\mathbb{R}^{2n})$ norm $\norm{\y_{\bullet}}_{\ell^{1}(\mathbb{R}^{2n})} =  \sum\limits_{k=0}^{\infty}\norm{{\y}_k}_1 = \sum\limits_{i=1}^{2n} \norm{[\y_{\bullet}]_i}_{\ell^{1}}$, where $ [\y_{\bullet}]_i$ is the scalar sequence corresponding to the $i^{th}$ entry of the vector $ \y_k$ in $\y_{\bullet}$. For any $ \w_k \in \w_{\bullet} \in  X$, the sequence of the $i^{th}$ entry of vector $\w_k$ for any $i \in \{1,2,\dots,2n\}$ given by $\{\w_{k,i}\}_{k}=[\w_{\bullet}]_i$ converges and belongs to the Banach space $X_i$ of real convergent sequences equipped with $\ell^{\infty}$ norm where the dual of $X_i$ is the space $X_i^*$ equipped with $ \ell^{1}$ norm (dual of the space of convergent sequences in the $\ell^{\infty}$ space, i.e., the $\textit{c}$ space, is the $\ell^{1}$ space \cite{dunford1988linear, megginson2012introduction}). Now $X = \bigoplus_{i=1}^{2n} X_i$ where each $X_i$ is a Banach space and therefore $X$ is a Banach space and so $X^{*} = \bigoplus_{i=1}^{2n} X_i^{*} $ since dual of direct sum of Banach spaces is a direct sum of dual spaces. Next, by the equivalence of norms on the direct sum of Banach spaces we have that $ \norm{\cdot}_{\ell^{\infty}(\mathbb{R}^{2n})} $ is equivalent to $ \norm{\cdot}_{\ell^{\infty}} $ and similarly $ \norm{\cdot}_{\ell^{1}(\mathbb{R}^{2n})} $ is equivalent to $ \norm{\cdot}_{\ell^{1}} $. Hence the dual of $X$ is $X^*$ which is the space of sequences equipped with $\ell^{1}(\mathbb{R}^{2n})$ norm. }

  {Next for any $X_i$ with dual $X_i^*$, the bilinear map $g_i$ from the product space $X_i \times X_i^*$ to $\mathbb{R}$ is given by \cite{dunford1988linear}: $$ g_i([\w_{\bullet}]_i,[\y_{\bullet}]_i ) =\langle [\w_{\bullet}]_i, [{\y}_{\bullet}]_i \rangle = \lim_{k \to \infty}\w_{k,i}\y_{0,i}  +\sum\limits_{k=1}^{\infty}  \w_{k,i}{\y}_{k,i}$$ with $\w_{\bullet} \in X$, ${\y}_{\bullet} \in X^*$ and $ [\w_{\bullet}]_i \in X_i$, $ [\y_{\bullet}]_i \in X_i^*$. Now the bilinear map $g_i : X_i \times X_i^* \rightarrow \mathbb{R} $ has a unique extension to a bilinear map $g : X \times X^* \rightarrow \mathbb{R}  $ where $g(\w_{\bullet},\y_{\bullet} ) = \sum\limits_{i=1}^{2n}  g_i([\w_{\bullet}]_i,[\y_{\bullet}]_i ) = \sum\limits_{i=1}^{2n} \langle [\w_{\bullet}]_i, [{\y}_{\bullet}]_i \rangle $. Therefore the bilinear map $g: X \times X^* \rightarrow \mathbb{R}$ is given by $ g( \w_{\bullet},{\y}_{\bullet} )=\langle \w_{\bullet},{\y}_{\bullet} \rangle = \langle\lim_{k \to \infty}\w_{k},\y_{0} \rangle+\sum\limits_{k=1}^{\infty} \langle \w_k, {\y}_k\rangle$ with $\w_{\bullet} \in X$, ${\y}_{\bullet} \in X^*$ where the inner product $\langle \w_k, {\y}_k\rangle$ is the usual inner product between $2n$-dimensional vectors. Let $\sigma(X,X^*)$ be the weak topology from $X$ to $X^*$ in which the bilinear maps $ X \times X^* \rightarrow \mathbb{R}$ are continuous in $X $.}

Let $\mathcal{I} = \{[\x^*;\x^*] : \nabla f(\x^*) = \mathbf{0}\} $ and the sets $ \mathcal{I}_+ , \mathcal{I}_- , \mathcal{I}_0 $ are defined as
\begin{align*}
    \mathcal{I}_+ &= \mathcal{I} \cap \{\w : \hspace{0.1cm} \norm{D P(\w)}_2 < 1\}   \\
    \mathcal{I}_- &= \mathcal{I} \cap \{\w : \hspace{0.1cm} \norm{D P(\w)}_2 > 1\} \\
    \mathcal{I}_0 &= \mathcal{I} \cap \{\w : \hspace{0.1cm} \norm{D P(\w)}_2 = 1 \},
\end{align*}
where $\mathcal{I} = \mathcal{I}_+ \cup \mathcal{I}_- \cup \mathcal{I}_0 $ and the sets $ \mathcal{I}_+ , \mathcal{I}_- , \mathcal{I}_0 $ are pairwise disjoint. Clearly $\mathcal{I}_0 $ is $\probP_1$ null by the following reasoning. Recall that from Lemma \ref{lemma_pk} we have $DP_k \darrow DP$ by which 
 $$ D P([\x^*;\x^*]) =  \begin{bmatrix}
p\bigg(\mathbf{I} - h \nabla^2 f(\x^*)\bigg) \hspace{0.1cm} &  -q\bigg(\mathbf{I} - h \nabla^2 f(\x^*)\bigg)\\  \mathbf{I} \hspace{0.2cm} & \mathbf{0}
\end{bmatrix}. $$
Then, the eigenvalues of $DP(\w^*)$ are ratios of polynomials in step-size $h$ and so the unit magnitude condition of $  \norm{D P(\w^*)}_2 = 1$ can be satisfied by at most finitely many $h$. Since the critical points of $f$ are isolated, we get that the fixed points of $P$ are isolated and hence countable. Then the condition $  \norm{D P(\w^*)}_2 = 1$ will only have to be satisfied for countably many $\w^* \in \mathcal{I}$ or equivalently countably many $h$ and thus the set $\mathcal{I}_0 $ is $\probP_1$ null. 
 
 Next, let $\mathcal{H} $ be the set of limit points of all convergent sequences generated by the recursion 
\begin{align}
    \w^r_{k+1} =  \begin{cases} 
      P_k(\w^r_k) & 0 \leq k\leq r \\
      P(\w^r_k) & k > r  \label{switchdyn}
   \end{cases}; \hspace{0.1cm} \w_0 \in \mathcal{U}
\end{align}
for any $r \geq 0$. Also let $\mathcal{G} $ be the set of limit points of all convergent sequences generated by the recursion 
\begin{align}
    \w_{k+1} = P_k(\w_k) ; \hspace{0.1cm} \forall  \hspace{0.1cm} k \geq 0 ; \hspace{0.1cm} \w_0 \in \mathcal{U}  \label{switchdyn1}
\end{align}
for the compact set $\mathcal{U} $. Since $\inf_{\x} f(\x) > -\infty $ and $f \in \mathcal{C}^2$, $f$ has a global minimum and hence at least one critical point. Thus, we can use Lemma \ref{measurelemmasup1}. Then $\mathcal{G} \subseteq \mathcal{H} $ from Lemma \ref{measurelemmasup1} and using the fact that $p_k - q_k =1$ for all $k$, $\mathcal{I} = \mathcal{H} $ from Lemma \ref{measurelemmasup1} and Lemma \ref{lemma_pk}.  

Since the elements of the sequence $\{\w^r_k\}$ from \eqref{switchdyn} for any $r$, and the elements of the sequence $\{\w_k\}$ from \eqref{switchdyn1}, when initialized in any compact $\mathcal{U} \subsetneq \mathbb{R}^{2n}$, stay bounded in some compact set $\mathcal{V} \subsetneq \mathbb{R}^{2n}$, it must be that $ \{\w_k\}  \in \ell^{\infty}(\mathbb{R}^{2n}) $, $ \{\w^r_k\}  \in \ell^{\infty}(\mathbb{R}^{2n}) $ for any $r \geq 0$.
 Let us then define the following sets $\{S_r\}$ for any $r \geq 0$ as follows:
\begin{align}
 S_r = \Bigg\{ \begin{array}{l}
     \{\w^r_k\}  \in \ell^{\infty}(\mathbb{R}^{2n})  
\end{array}\hspace{0.1cm}\bigg\vert \hspace{0.1cm} \w^r_{k+1} =  \begin{cases} 
      P_k(\w^r_k) & 0 \leq k\leq r \\
      P(\w^r_k) & k > r 
   \end{cases}; \hspace{0.1cm} \w_0 \in \mathcal{U} \subsetneq \mathbb{R}^{2n}
\Bigg\},
\end{align}
and similarly the set $S$ is given by:
\begin{align}
 S &=  \Bigg\{ \begin{array}{l}
     \{\w_k\}  \in \ell^{\infty}(\mathbb{R}^{2n})  
\end{array} \hspace{0.1cm}\bigg\vert \hspace{0.1cm} \w_{k+1} = P_k(\w_k) \hspace{0.1cm} \forall \hspace{0.1cm} k \geq 0 ; \hspace{0.1cm} \w_0 \in \mathcal{U} \subsetneq \mathbb{R}^{2n}
\Bigg\}.
\end{align}
Next, suppose for any $\w^*  \in \mathcal{I}_+$, let $ \mathcal{B}_{\delta}(\w^*)$ be a ball such that if any of the sequences from set $S_{r}$ for any $r$ or any of the sequences from the set $S$ converge to $\w^*$, then there exists a $\delta >0$ such that all these sequences enter this ball $ \mathcal{B}_{\delta}(\w^*)$ only after some sufficiently large iteration $K_0$. In particular, we can choose $\delta, K_0$ from Lemma \ref{measurelemmasup2_a}. To see this first observe that since the initialization set $\mathcal{U}$ is compact, it must be that the sets of the form
$$ \underbrace{P \circ \cdots \circ P}_{K_0-1 \text{ times}}\circ P_0(\mathcal{U}) \quad;\quad  \underbrace{P \circ \cdots \circ P}_{K_0-2 \text{ times}}\circ P_1 \circ P_0(\mathcal{U}) \quad ;\cdots ; \quad  P_{K_0-1} \circ \cdots \circ P_1\circ P_0(\mathcal{U})  $$
are compact by continuity of maps $\{P_k\}$, $P$. These sets represent the $K_0 $ times push forward of the compact set $\mathcal{U}$ via the maps generating trajectories in set $S_r$ for any $r$ and the maps generating trajectories in set $S$. From Lemma \ref{lemma_pk}, the map $P_k$ for any $k$ is a homeomorphism on $ \mathcal{U}$ and thus $\w^*$ will not be in any of these $K_0 $ push forward sets from the facts that $\w^* \notin \mathcal{U}$ and $\w^*$ is a fixed point of maps $\{P_k\}$ and $P$ from Lemma \ref{lemma_pk}. Then we can always find a ball $ \mathcal{B}_{\delta_1}(\w^*)$ such that its distance from any of the $K_0$ push forward sets is lower bounded by $\delta_1$ or equivalently\footnote{Note that $d_H(\cdot, \cdot)$ is the Hausdorff distance.}
$$ \min \bigg\{d_H\bigg(\mathcal{B}_{\delta_1}(\w^*), \underbrace{P \circ \cdots \circ P}_{K_0-1 \text{ times}}\circ P_0(\mathcal{U})\bigg)  \quad ; \cdots; \quad d_H\bigg(\mathcal{B}_{\delta_1}(\w^*), P_{K_0-1} \circ \cdots \circ P_1\circ P_0(\mathcal{U})\bigg) \bigg\} > \delta_1.  $$
 Hence, if any of the sequences from set $S_{r}$ for any $r$ or any of the sequences from the set $S$ converge to $\w^*$, then these sequences enter this ball $ \mathcal{B}_{\delta_1}(\w^*)$ only after iteration $K_0$. Finally, if required we can shrink the $\delta_1$ ball further to a $\delta$ ball so as to accommodate the range of $\delta$ from Lemma \ref{measurelemmasup2_a}. Since $f$ has isolated critical points, using a non-intersecting cover for this set of critical points and by Heine-Borel theorem, we get that this set of critical points on any compact set must have finite cardinality. Then the set $\mathcal{I} \cap \mathcal{V}$ will have finitely many points and thus one can choose a uniform $\delta$ for $\mathcal{I}_{+} \cap \mathcal{V}$. After these convergent sequences hit $ \mathcal{B}_{\delta}(\w^*)$, these sequences will converge uniformly to $\w^*$ from Lemma \ref{measurelemmasup2_a}.

Let us define, for any $N_0 \geq K_0$, the following sequence of sets $\{S^{'}_r(N_0)\}_r$ and their corresponding random events $\{E_r(N_0)\}_r$  :
\begin{align}
 & S^{'}_r(N_0) = \Bigg \{  \{\w^r_k\} \in S_r   \hspace{0.1cm}\bigg\vert \hspace{0.1cm} \lim_{k\to \infty}\w^r_k \in \mathcal{I}_+ \hspace{0.1cm} ; \hspace{0.1cm} \inf\bigg\{k>0 ; \hspace{0.1cm} \w^r_{k} \in (\mathcal{I}_+\cap \mathcal{V}) + \mathcal{B}_{\delta}(\mathbf{0}) \bigg\} \in [K_0, N_0] 
\Bigg \}, \\
   & E_r(N_0) = \Bigg(    \{\w^r_k\} \in S^{'}_r(N_0) \hspace{0.1cm}; \hspace{0.1cm} \w_0 \overset{\text{Unif}}{\sim} \mathcal{U} \hspace{0.1cm}; \hspace{0.1cm} h \overset{\text{Unif}}{\sim} [0,1/L] \Bigg), 
\end{align}
\footnote{Here $ (\mathcal{I}_+\cap \mathcal{V}) + \mathcal{B}_{\delta}(\mathbf{0})$ represents the Minkowski sum of sets $ \mathcal{I}_+\cap \mathcal{V}$, $ \mathcal{B}_{\delta}(\mathbf{0})$.} and similarly the set $S^{'}(N_0)$ and the corresponding random event $E(N_0)$ as:
\begin{align}
 & S^{'}(N_0) = \Bigg \{  \{\w_k\} \in S  \hspace{0.1cm}\bigg\vert \hspace{0.1cm} \lim_{k\to \infty}\w_k \in \mathcal{I}_+ \hspace{0.1cm} ; \hspace{0.1cm}  \inf\bigg\{k>0 ; \hspace{0.1cm} \w_{k} \in (\mathcal{I}_+ \cap \mathcal{V})+ \mathcal{B}_{\delta}(\mathbf{0}) \bigg\} \in [K_0, N_0]  
\Bigg \}, \\
   & E(N_0) = \Bigg(    \{\w_k\} \in S^{'}(N_0) \hspace{0.1cm}; \hspace{0.1cm} \w_0 \overset{\text{Unif}}{\sim} \mathcal{U} \hspace{0.1cm}; \hspace{0.1cm} h \overset{\text{Unif}}{\sim} [0,1/L] \Bigg).
\end{align}
\footnote{From here onwards in this proof we will drop the expression $ \w_0 \overset{\text{Unif}}{\sim} \mathcal{U} \hspace{0.1cm}; \hspace{0.1cm} h \overset{\text{Unif}}{\sim} [0,1/L] $ while defining random events for sake of brevity.}Observe that the sequences in the sets $ S'(N_0), S'_r(N_0)$ for any $r \geq 0$ converge uniformly. This is because these sequences hit the set $ (\mathcal{I}_+\cap \mathcal{V}) + \mathcal{B}_{\delta}(\mathbf{0}) $ in the iteration interval $ k \in [K_0, N_0]$ and then for all $k \geq N_0$, these sequences converge uniformly geometrically with the same geometric rate from Lemma \ref{measurelemmasup2_a}. Since the set of uniformly convergent, bounded sequences in the $c_0$ space (i.e., the space of sequences converging to $0$) is totally bounded, we get that 
$$ S'(N_0) \bigcup  \bigg(\bigcup_{r \geq 0} S'_r(N_0) \bigg)\subset \mathcal{K}(N_0)  $$ 
for any $ N_0 \geq K_0$ where $\mathcal{K}(N_0) $ is compact in the $\ell^{\infty}(\mathbb{R}^{2n})$ topology by completeness of $ \ell^{\infty}$ space and total boundedness of the set $ S'(N_0) \bigcup  \bigg(\bigcup_{r \geq 0} S'_r(N_0) \bigg)$.
By the definition of $S'_r(N_0)$ we have the following monotone non-decreasing property on the sequence of sets $ \{S'_r(N)\}_{N = N_0}^{\infty}$ and the sequence of events $ \{E_r(N)\}_{N = N_0}^{\infty}$ :
$$ S'_r(N_0) \subseteq S'_r(N_0+ 1) \subseteq \cdots \subseteq S'_r(N + N_0) \subseteq \cdots ,$$
$$ E_r(N_0) \subseteq E_r(N_0+ 1) \subseteq \cdots \subseteq E_r(N + N_0) \subseteq \cdots .$$
Then $  E_r(N_0) \xrightarrow{N_0 \to \infty} \bigcup_{N_0 \geq K_0}^{\infty} E_r(N_0) $.
Similarly, by the definition of $S'(N_0)$ we have the following monotone non-decreasing property on the sequence of sets $ \{S'(N)\}_{N = N_0}^{\infty}$ and the sequence of events $ \{E(N)\}_{N = N_0}^{\infty}$ :
$$ S'(N_0) \subseteq S'(N_0+ 1) \subseteq \cdots \subseteq S'(N + N_0) \subseteq \cdots ,$$
$$ E(N_0) \subseteq E(N_0+ 1) \subseteq \cdots \subseteq E(N + N_0) \subseteq \cdots .$$
Then $  E(N_0) \xrightarrow{N_0 \to \infty} \bigcup_{N_0 \geq K_0}^{\infty} E(N_0) $.
 Let $$ E_r  = \bigcup_{N_0 \geq K_0}^{\infty} E_r(N_0), \hspace{0.1cm} E  = \bigcup_{N_0 \geq K_0}^{\infty} E(N_0).$$
Using the definition of sets $ S^{'}_r(N_0), S^{'}(N_0)$ and the choice of $\delta$ from Lemma \ref{measurelemmasup2_a}, it can be readily deduced that
$$ E_r  = \Bigg(\lim_{k\to \infty}\w^r_k \in \mathcal{I}_+ \hspace{0.1cm};\hspace{0.1cm} \{ \w^r_k\} \in S_r 
\Bigg), $$
$$ E  = \Bigg(\lim_{k\to \infty}\w_k \in \mathcal{I}_+ \hspace{0.1cm}; \hspace{0.1cm} \{ \w_k\} \in S 
\Bigg),$$
and we have the following convergence from below
\begin{align}
     E_r(N_0) \uparrow E_r, \hspace{0.2cm} E(N_0) \uparrow E \hspace{0.1cm} \text{as } N_0 \uparrow \infty  \label{set_theory_converge*prob}
\end{align}
in the set-theoretic sense. The complementary events $E^c_r$ and $E^c$ are given by
\begin{align}
 & E_r^c = \Bigg(\lim_{k\to \infty}\w^r_k \in  \mathcal{H}\backslash \mathcal{I}_+ \hspace{0.1cm} ; \hspace{0.1cm} \{ \w^r_k\} \in S_r 
\Bigg) \bigcup \Bigg(\lim_{k\to \infty}\w^r_k \text{ does not exist} \hspace{0.1cm}; \hspace{0.1cm}  \{ \w^r_k\} \in S_r 
\Bigg)  \nonumber \\
 & \hspace{0.3cm} = \Bigg(\lim_{k\to \infty}\w^r_k \in  \mathcal{I}\backslash \mathcal{I}_+ \hspace{0.1cm}; \hspace{0.1cm} \{ \w^r_k\} \in S_r 
\Bigg) \bigcup \Bigg(\lim_{k\to \infty}\w^r_k \text{ does not exist} \hspace{0.1cm}; \hspace{0.1cm}  \{ \w^r_k\} \in S_r 
\Bigg) ,\label{complementEr}\\
  &  E^c = \Bigg(\lim_{k\to \infty}\w_k \in   \mathcal{H}\backslash \mathcal{I}_+  \bigcup \bigg(\bigg(\mathcal{G}\backslash \mathcal{I}_+\bigg) \bigg\backslash \bigg(\mathcal{H}\backslash \mathcal{I}_+\bigg)\bigg)\hspace{0.1cm}; \hspace{0.1cm} \{ \w_k\} \in S 
\Bigg) \nonumber \\ & \hspace{0.3cm}\bigcup  \Bigg(\lim_{k\to \infty}\w_k \text{ does not exist} \hspace{0.1cm}; \hspace{0.1cm}  \{ \w_k\} \in S 
\Bigg) \nonumber\\
&=\Bigg(\lim_{k\to \infty}\w_k \in   \mathcal{I}\backslash \mathcal{I}_+ \hspace{0.1cm}; \hspace{0.1cm} \{ \w_k\} \in S 
\Bigg) \bigcup \Bigg(\lim_{k\to \infty}\w_k \text{ does not exist} \hspace{0.1cm}; \hspace{0.1cm} \{ \w_k\} \in S
\Bigg), \label{complementE}
\end{align}
where we used the fact that $\mathcal{G} $ is the set of limit points of all convergent sequences generated by the recursion $ \w_{k+1} = P_k(\w_k)$ for any $\w_0 \in \mathbb{R}^{2n} $ and the set $\bigg(\mathcal{G}\backslash \mathcal{I}_+\bigg) \bigg\backslash \bigg(\mathcal{H}\backslash \mathcal{I}_+\bigg)$ is an empty set from Lemma~\ref{measurelemmasup1}. 

  {{We first prove that $S_r'(N_0) \to S'(N_0)$ as $r\to \infty$ in the Hausdorff metric for any $N_0 \geq K_0$ using Lemma \ref{measurelemmasup3}. Thereafter, using the definition of $\delta$ and uniform convergence from Lemma~\ref{measurelemmasup2_a} we will prove $E_r(N_0) \to E(N_0)$ as $r\to \infty$. Note that since the sets $S_r'(N_0), S'(N_0) \subset \mathcal{K}(N_0)$, $\mathcal{K}(N_0)$ are compact\footnote{{The sets $S_r'(N_0), S'(N_0) \subset \mathcal{K}(N_0)$ are closed by the fact that for any sequence $\{\y_{\bullet}^j\}_j$ in a set $ S_r'(N_0)$, if $ \norm{\y_{\bullet}^j - \y_{\bullet}}_{\ell^{\infty}(\mathbb{R}^{2n})} \xrightarrow[]{j \to \infty} 0$ then the sequences in $\{\y_{\bullet}^j\}_j$ have the same limit point for $j$ large enough because $\mathcal{I}$ has isolated points. Hence, $\y_{\bullet} \in S_r'(N_0) $.}} in $\ell^{\infty}(\mathbb{R}^{2n})$, Lemma \ref{measurelemmasup3} can be applied.} In order to satisfy the hypotheses of Lemma \ref{measurelemmasup3} we prove the following two statements:
\begin{itemize}
    \item[\textbf{(s1.)}]\label{newlogic5a} For every sequence $\w_{\bullet} \in S'(N_0)$ that has some initialization $\w_0 \in \mathcal{U}$, there exists a sequence $ \{\w_{\bullet}^r\}_r$ with $ \w_{\bullet}^r \in S_r'(N_0)$ such that $ \w_{\bullet}^r \to \w_{\bullet}$ as $r \to \infty$ in $\mathcal{K}(N_0)$.
    
          {We first show that for any common initialization $\w_0 \in \mathcal{U}$ of the sequences $  \w_{\bullet}, \w_{\bullet}^r$ for any $r \geq N_0$, where $\w_{\bullet}^r = \{\w^r_k\}_{k=0}^{\infty} \in S_r'(N_0) $ and $\w_{\bullet} = \{\w_k\}_{k=0}^{\infty}  \in S'(N_0) $, the limit point of any sequence $\w_{\bullet}^r$ given by $ \lim_{k \to \infty} \w_{k}^r$ converges to the limit point $\lim_{k \to \infty} \w_k $ of the sequence $\w_{\bullet}  $ as $r \to \infty$, i.e., $\lim_{r \to \infty}\lim_{k \to \infty} \w_{k}^r = \lim_{k \to \infty} \w_k$. Suppose $ \lim_{k \to \infty} \w_k = \q$ where $\q  \in \mathcal{I}_+$, then from definition of $\mathcal{I}_+ $ we have $\norm{D P(\q)}_2 <1 $. From Lemma~\ref{measurelemmasup2}, for all $r \geq N_0$, the sequence $\w_{\bullet}^r = \{\w^r_k\}_{k=0}^{\infty} \in S_r'(N_0) $ with initialization $\w_0$ then satisfies $ \lim_{k \to \infty}\w^r_k = \q$.  }
Hence we will have:
\begin{align}
    \lim_{r \to \infty}\lim_{k \to \infty} \w_{k}^r = \q . \label{topologicalproof2}
\end{align}
  {We are now ready to prove that for any $\w_{\bullet}  \in S'(N_0) $ there exists a sequence $\{ \w_{\bullet}^r\}_r$, with $ \w_{\bullet}^r \in S_r'(N_0)$, such that $ \w_{\bullet}^r \to \w_{\bullet}$ as $r \to \infty$ in the topology of $X$. For the sequence $\{ \w_{\bullet}^r\}_{r=0}^{\infty} $ in $\mathcal{K}(N_0)\subsetneq X$ with every $\w_{\bullet}^r$ having the same initialization $\w_0 \in \mathcal{U}$ for all $r\geq 0$, we have that $ \w_{\bullet}^r \rightharpoonup \w_{\bullet}$ as $r \to \infty$ where $\w_{\bullet} $ also has the initialization $\w_0 \in \mathcal{U}$ and the weak convergence ($\rightharpoonup$) is in the $\sigma(X,X^*)$ topology.\footnote{By weak convergence we mean the following convergence $ \langle \w_{\bullet}^r, \y_{\bullet} \rangle  \to \langle \w_{\bullet}, \y_{\bullet} \rangle $ as $r \to \infty$ for all $\y_{\bullet} \in X^* $ (since $ \norm{\y_{\bullet}}_{\ell^1(\mathbb{R}^{2n})} < \infty$ we have for $ \y_{\bullet} = \{\y_k\}_{k=0}^{\infty} $ that $ \sum\limits_{k=K}^{\infty}\norm{{\y}_k}_1 \to 0 $ as $K \to \infty$).} This weak convergence holds because by definition $ \langle \w_{\bullet}^r, \y_{\bullet} \rangle =  \langle\lim_{k \to \infty}\w^r_{k},\y_{0} \rangle+\sum\limits_{k=1}^{\infty} \langle \w^r_k, {\y}_k\rangle$, the first $r$ terms of the sequences  $\w_{\bullet}^r, \w_{\bullet}$ are identical and so $\sum\limits_{k=1}^{\infty} \langle \w^r_k, {\y}_k\rangle \to \sum\limits_{k=1}^{\infty} \langle \w_k, {\y}_k\rangle $ for any $\y_{\bullet} = \{\y_k\}_k \in X^* $ since 
  \begin{align*}
     \limsup_{r \to \infty} \sum\limits_{k\geq  r}^{\infty} \langle (\w_k - \w_k^r), {\y}_k\rangle &\leq  \limsup_{r \to \infty}\sum\limits_{k\geq  r}^{\infty} \norm{\w_k - \w_k^r}_1 \norm{\y_k}_1 \\
      &\leq  \limsup_{r \to \infty}\bigg(\sum\limits_{k\geq  r}^{\infty} \norm{\w_k - \w_k^r}_1  \bigg) \bigg(\sum\limits_{k\geq  r}^{\infty}\norm{\y_k}_1 \bigg) \\
       &\leq  \limsup_{r \to \infty}\underbrace{\bigg(\sum\limits_{k\geq  r}^{\infty} (\norm{\w_k -\q}_1+ \norm{\w_k^r-\q}_1)  \bigg)}_{\to 0 \text{ by convergence of } \w_k, \w_k^r} \limsup_{r \to \infty}\underbrace{\bigg(\sum\limits_{k\geq  r}^{\infty}\norm{\y_k}_1 \bigg)}_{\to 0 \text{ since } \y_{\bullet} \in X^*}  = 0,
  \end{align*}
   and $ \langle\lim_{k \to \infty}\w^r_{k},\y_{0} \rangle \to \langle\lim_{k \to \infty}\w_{k},\y_{0} \rangle$ as $r \to \infty$ by \eqref{topologicalproof2}.} 
  {Since $ S_r'(N_0) \subset \mathcal{K}(N_0) $  for all $r\geq 0$ and $ S'(N_0) \subset \mathcal{K}(N_0) $, $\mathcal{K}(N_0) $ is compact in $X$ and a weakly convergent sequence on a strongly compact set is strongly convergent (section 3.2 in \cite{brezis2011functional}), we have that $ \w_{\bullet}^r \to \w_{\bullet}$ as $r \to \infty$ in the topology of $X$ (strong topology) for a common initialization $\w_0$ of the sequences $\w_{\bullet}^r, \w_{\bullet}$. Since the initialization $\w_0 \in \mathcal{U}$ of the sequence $\w_{\bullet}$ was arbitrary, our first statement $\textbf{(s1.)}$ is proved.  }
    
    \item[\textbf{(s2.)}]\label{newlogic5b} For any subsequence $ \{\w_{\bullet}^{r_j}\}_j$ of $\{\w_{\bullet}^{r}\}_r$ where $ \w_{\bullet}^{r_j} \in S_{r_j}'(N_0)$, if $ \w_{\bullet}^{r_j} \to \w_{\bullet}$ in $\mathcal{K}(N_0)$ then $ \w_{\bullet} \in S'(N_0)$. 
    
    If $\w_{\bullet}^{r_j} \to \w_{\bullet}$ as $j \to \infty$ in the topology of $X$, where $\w_{\bullet}^{r_j} =\{\w^{r_j}_{k}\}_{k=0}^{\infty} $,  $\w_{\bullet} =\{\w_{k}\}_{k=0}^{\infty} $, then we have $\norm{\w_{\bullet}^{r_j} - \w_{\bullet}}_{\ell^{\infty}(\mathbb{R}^{2n})} \to 0$ as $j \to \infty$. Since $\mathcal{I}_+$ has isolated points, it must be that for all $j$ large enough, $ \lim_{k\to \infty}\w_k^{r_j} $ is a unique point in $ \mathcal{I}_+$ where we call that point $\q$. Then $\w_{\bullet}$ must have $\q$ as an accumulation point otherwise $\norm{\w_{\bullet}^{r_j} - \w_{\bullet}}_{\ell^{\infty}(\mathbb{R}^{2n})} $ cannot be arbitrary small for any large $j$. Next, if $\w_{\bullet}$ has $\q_0$ as another accumulation point then $\w_{\bullet}$ must visit every $\epsilon'>0$ neighborhoods of $\q, \q_0$ infinitely often. Suppose we choose $\epsilon' < \frac{1}{4}\norm{\q -\q_0} $ and $j$ large enough such that $ \norm{\w_{\bullet}^{r_j} - \w_{\bullet}}_{\ell^{\infty}(\mathbb{R}^{2n})}  < \epsilon'$. Since $\w_{\bullet}^{r_j} $ converges to $\q$, there exists $K>0$ such that $ \norm{ \w^{r_j}_{k} - \q} < \epsilon'$ for all $k>K$. But then there is an index $K'>K$ such that $ \norm{\w^{r_j}_{K'} - \w_{K'} } > \frac{1}{2}\norm{\q -\q_0}$, and so $ \norm{\w_{\bullet}^{r_j} - \w_{\bullet}}_{\ell^{\infty}(\mathbb{R}^{2n})}  > \frac{1}{2}\norm{\q -\q_0}$, a contradiction. Hence, $\w_{\bullet}$ must converge to $\q$. Since $ \w_{\bullet}$ is generated from the recursion $\w_{k+1} =P_k(\w_k)$, we only need to prove that $ \w_{\bullet}$ hits the set $ (\mathcal{I}_+\cap \mathcal{V}) + \mathcal{B}_{\delta}(\mathbf{0})$ in the interval $[K_0, N_0]$, i.e., $$ \inf\bigg\{k>0 ; \hspace{0.1cm} \w_{k} \in (\mathcal{I}_+ \cap \mathcal{V})+ \mathcal{B}_{\delta}(\mathbf{0}) \bigg\} \in [K_0, N_0].  $$
    Since $\w_{\bullet}^{r_j} \in S'_{r_j}(N_0) $, it must be $ \norm{\w^{r_j}_{K_0-1} - \q} \geq \delta $ and $ \norm{\w^{r_j}_{N_0} - \q} \leq \delta (1- \epsilon'') $ for some $\epsilon'' \ll 1$. Then by triangle inequalities and using convergence in $\ell^{\infty}(\mathbb{R}^{2n}) $ we get:
    \begin{align}
        \norm{\w_{K_0-1} - \q} &\geq \norm{\w^{r_j}_{K_0-1} - \q} - \norm{\w^{r_j}_{K_0-1} - \w_{K_0-1}} \nonumber\\
        \implies \norm{\w_{K_0-1} - \q} &\geq \delta - \lim_{j \to \infty}\norm{\w^{r_j}_{K_0-1} - \w_{K_0-1}} = \delta \label{newproofthesis*1}
    \end{align}
    and
    \begin{align}
        \norm{\w_{N_0} - \q} & \leq \norm{\w^{r_j}_{N_0} - \q} + \norm{\w^{r_j}_{N_0} - \w_{N_0}} \nonumber\\
        \implies \norm{\w_{N_0} - \q} & \leq \delta (1- \epsilon'') + \lim_{j \to \infty}\norm{\w^{r_j}_{N_0} - \w_{N_0}} <\delta. \label{newproofthesis*2}
    \end{align}
    Since $\w_k$ cannot leave the $ \delta$ ball around $\q$ once it hits this ball (Lemma \ref{measurelemmasup2}), from \eqref{newproofthesis*1}, \eqref{newproofthesis*2} it must be that $ \w_{\bullet}$ hits the set $ (\mathcal{I}_+ \cap \mathcal{V})+ \mathcal{B}_{\delta}(\mathbf{0})$ in the interval $[K_0, N_0]$. Thus, $\{\w_{k}\} =  \w_{\bullet} \in S'(N_0) $ which proves the second statement $\textbf{(s2.)}$.
\end{itemize}
}

{ Using $\textbf{(s1.)}, \textbf{(s2.)}$ in Lemma \ref{measurelemmasup3} we have that $ S_r'(N_0)\to S'(N_0)$ in the Hausdorff metric as $r \to \infty$. Next, we need to show that $ E_r(N_0) \to E(N_0)$ as $r \to \infty$ in the set-theoretic sense for any $N_0 \geq K_0$. Note that $ S_r'(N_0)\to S'(N_0)$ in the Hausdorff metric is not enough to prove convergence in measure $ \probP(E_r(N_0))\to \probP(E(N_0))$ and there are counterexamples where the Hausdorff limit of a sequence of Lebesgue null sets is not a zero Lebesgue measure set.}

{Let $ \{I_r(N_0)\}_r, I(N_0)$ be the initialization sets of sequences in the sets $ \{S_r'(N_0)\}_r, S'(N_0)$, respectively. Observe that for any $r \geq N_0$, $I(N_0) \subseteq I_r(N_0)  $ and also $ I_r(N_0) \subseteq I(N_0)$ from the uniform convergence and definition of $\delta$ (Lemmas \ref{measurelemmasup2_a}, \ref{measurelemmasup2}). Thus, $ E_r(N_0) = E(N_0)$ for all $r \geq N_0$ and hence $ E_r(N_0) \to E(N_0)$ as $r \to \infty$ in the set-theoretic sense.}
 
We now prove part $b.$ of Theorem \ref{measuretheorem3} and part $a.$ is proved thereafter. Note that from the part $b.$ of theorem statement we are given that for any $r\geq 0$ that $$ \probP_1\Bigg(\lim_{k\to \infty}\w^r_k \text{ does not exist} \hspace{0.1cm}; \hspace{0.1cm}\{ \w^r_k\} \in S_r
\Bigg) = 0.$$
 {We will first prove that $\probP(E_r) = 1 $ for all $r \geq 0$. Recall that the complement of $E_r$ is given by $$ E_r^c = E_r^c \vert_{\mathcal{I}_-} \bigcup E_r^c \vert_{\mathcal{I}_0} \bigcup \Bigg(\lim_{k\to \infty}\w^r_k \text{ does not exist  }\hspace{0.1cm}; \hspace{0.1cm} \{ \w^r_k\} \in S_r
\Bigg), $$ where we define $$E_r^c \vert_{\mathcal{I}_-}  = \Bigg(\lim_{k\to \infty}\w^r_k \in  \mathcal{I}_- \hspace{0.1cm}; \hspace{0.1cm} \{\w^r_k\} \in S_r
\Bigg), \hspace{0.3cm} E_r^c \vert_{\mathcal{I}_0}  = \Bigg(\lim_{k\to \infty}\w^r_k \in  \mathcal{I}_0 \hspace{0.1cm}; \hspace{0.1cm} \{\w^r_k\} \in S_r
\Bigg).$$ Also, we define the following event $$E_r^c \vert_{\w^* \in \mathcal{I}_-}  = \Bigg(\lim_{k\to \infty}\w^r_k = \w^*  \in  \mathcal{I}_- \hspace{0.1cm}; \hspace{0.1cm} \{\w^r_k\} \in S_r
\Bigg). $$} Let $T_r$ be the $\probP_1 \bigotimes \probP_2$ measurable event in which the maps $P_k$ for all $k$ and the map $P$ are diffeomorphisms on any compact set. Since $f$ is Hessian Lipschitz continuous in every compact set, using Corollary \ref{corsup2} we get that the maps $P_k$ for all $k $ and the map $P$ are $ \probP_1$-almost sure diffeomorphisms on any compact set of $\mathbb{R}^{2n}$ and so $\probP_1(T_r) =1 $.\footnote{Since the sequence of maps $\{P_k\}$ and their uniform limit $P$ are diffeomorphisms on any given compact set, say $\mathcal{W}$, of $\mathbb{R}^{2n}$ $ \probP_1$-a.s., these maps will be jointly diffeomorphic on $\mathcal{W}$ $ \probP_1$-a.s. from the fact that countable intersections of almost sure events is an almost sure event. Then $\probP_1(T_r) =1 $.} Then $\probP(T_r) = \probP_1 \bigotimes \probP_2(T_r) = \probP_1(T_r) =1 $ since the event $T_r$ is independent of the $ \probP_2$-measurable random variable $\w_0$. Also, from Lemma \ref{lemma_pk}, the maps $P_k$ for all $k$ and the map $P$ are proper, invertible.

Then using Theorem~\ref{measuretheorem2} and the fact that the compact initialization set $\mathcal{U}$ is contained within some sufficiently large compact ball around $\w^*_j$, we will have $$\probP\bigg(E_r^c \vert_{\w_j^* \in \mathcal{I}_-} \hspace{0.1cm} \bigg\vert \hspace{0.1cm} T_r\bigg) =\probP_2\bigg(E_r^c \vert_{\w_j^* \in \mathcal{I}_-} \hspace{0.1cm} \bigg \vert \hspace{0.1cm} T_r\bigg)= 0$$ for any $j$ and $r$ by conditional probability. Since the critical points of $f$ are isolated and therefore countable, we get that $\mathcal{I}$ is countable and so $$ \probP\bigg(E_{{r}}^c \vert_{\mathcal{I}_- } \hspace{0.1cm} \bigg \vert \hspace{0.1cm} T_{{r}}\bigg)=\sum_{j} \probP\bigg(E_{{r}}^c \vert_{\w_j^* \in \mathcal{I}_-} \hspace{0.1cm} \bigg\vert \hspace{0.1cm} T_{{r}}\bigg)=0.$$ Since $\probP_1(T_r) = 1 $ from Lemma \ref{lemma_pk} and thus $\probP(T_r) = 1 $, we will have $\probP(T^c_r) = 0 $ which implies $$  \probP(E_r^c \vert_{\mathcal{I}_-}) = {\probP(T_r)}\probP\bigg(E_r^c \vert_{\mathcal{I}_- } \hspace{0.1cm} \bigg\vert \hspace{0.1cm} T_r\bigg)  + {\probP(T^c_r)}\probP\bigg(E_r^c \vert_{\mathcal{I}_-} \hspace{0.1cm} \bigg\vert \hspace{0.1cm} T^c_r\bigg)= \probP\bigg(E_r \hspace{0.1cm} \bigg\vert \hspace{0.1cm} T_r\bigg) = 0$$ from total probability for all $r \geq 0$. Since $\probP_1(E_r^c \vert_{\mathcal{I}_0})=0 $ and $$\probP_1\Bigg(\lim_{k\to \infty}\w^r_k \text{ does not exist}\hspace{0.1cm}; \hspace{0.1cm} \{ \w^r_k\} \in S_r
\Bigg) =0,$$ we get that $$ \probP(E_r^c) =0$$ for all $r \geq 0$ and thus $\probP(E_r) =1 $ for all $r \geq 0$.

We are now ready to prove the first main result of $ \probP(E) =   1$. We first make the following observation: \\ For any $N_0 \geq K_0$, the sequence of events $\{E_r(N_0)\}_{r=r_0}^{\infty}$ for any $r_0 \geq N_0$ satisfy the following monotonicity property:
\begin{align}
    \probP(E_{r_0}(N_0)) \leq \probP( E_{r_0+1}(N_0)) \leq \cdots \leq \probP( E_r(N_0)) \leq \cdots \leq \probP( E(N_0)). \label{probabilitymontonic*}
\end{align}
To see this first observe that $ \probP = \probP_1 \bigotimes \probP_2$ where $\probP_1$ measures the random variable $h \in [0, 1/L]$ and $\probP_2$ measures the random variable $\w_0 \in \mathcal{U}$, the sequence of events $\{E_r(N_0)\}_{r=0}^{\infty}$ occur $\probP_1$ a.s., and hence $ \probP( E_r(N_0))$ effectively measures the Lebesgue volume of the set $\mathcal{U}$ over which event $E_r(N_0)$ occurs, modulo some $\probP_1$ null sets on the real line. Next, if for any random $\w_0 \in \mathcal{U}$ we have the event $E_{r_0}(N_0)$, then for the same $\w_0$ we will also have the entire sequence of events $\{E_r(N_0)\}_{r=r_0}^{\infty} $ provided $r_0 \geq N_0$ from Lemma \ref{measurelemmasup2}. Hence, the probabilities over these sequence of events will be monotonically non-decreasing. Finally, using the fact that $E_r(N_0) \xrightarrow[]{r \to \infty} E(N_0)$ which was proved before, \eqref{probabilitymontonic*} is established.

Since $\probP(E_r)=1$ for all $r \geq 0$, measure $ \probP$ is continuous and $ E_r(N_0) \uparrow E_r$ from \eqref{set_theory_converge*prob} we get the following convergence from below 
$$ \probP(E_r(N_0)) \uparrow \probP(E_r) =1 \text{ as } N_0 \uparrow \infty ,$$ 
and thus for any $N_0 \geq K_0$, any $r \geq 0$, we have $ \probP(E_r(N_0)) = 1 -\epsilon_r(N_0)$ where $\epsilon_r(N_0) \downarrow 0 $ as $N_0 \uparrow \infty$. Moreover, from the non-decreasing monotonicity of probabilities \eqref{probabilitymontonic*}, we have that the sequence $ \{\epsilon_r(N_0)\}_{r= r_0}^{\infty}$ for any $r_0 \geq N_0$ is non-increasing, non-negative and hence by monotone convergence this sequence converges to some $\epsilon(N_0)$ where $\epsilon(N_0) \downarrow 0 $ as $N_0 \uparrow \infty$. Then taking $r \to \infty$ in $\probP(E_r(N_0)) $, using continuity of probability measure and the fact that $E_r(N_0) \xrightarrow[]{r \to \infty} E(N_0)$ for any $N_0 \geq K_0$ yields:
\begin{align}
 \probP( E(N_0)) = \probP(\lim_{r \to \infty} E_r(N_0))  = \lim_{r \to \infty} \probP(E_r(N_0)) = 1 - \lim_{r \to \infty}\epsilon_r(N_0) = 1- \epsilon(N_0). 
\end{align}
Finally, taking $N_0 \to\infty$ above, using continuity of probability measure and the convergence $ E(N_0) \uparrow E$ from \eqref{set_theory_converge*prob} we get 
\begin{align}
 \probP(E)=\probP(  \lim_{N_0 \to \infty}  E(N_0))  = \lim_{N_0 \to \infty} \probP( E(N_0))  = 1- \lim_{N_0 \to \infty} \epsilon(N_0) = 1
\end{align}
where we used the fact that $ \epsilon(N_0) \xrightarrow[]{N_0 \to \infty} 0 $.
This proves part $b.$ of the theorem statement. 

Next, for part $a.$ note that we do not have the following condition for any $r\geq 0$:
$$ \probP_1\Bigg(\lim_{k\to \infty}\w^r_k \text{ does not exist} \hspace{0.1cm};\hspace{0.1cm}\{ \w^r_k\} \in S_r
\Bigg) = 0$$
and therefore we need a different approach to prove part $a$. Recall that we have already proved that $E_r(N_0) \to E(N_0)$ as $r \to \infty$ for any $N_0 \geq K_0$. Let $\mathcal{S}$ be the event space of sequences generated by recursions \eqref{switchdyn}, \eqref{switchdyn1} when initialized in $\mathcal{U}$. Clearly, $\mathcal{S}$ is $\probP_1 \bigotimes \probP_2$ measurable. Then $E_r^c(N_0) = \mathcal{S} \backslash E_r(N_0)$, $E^c(N_0) = \mathcal{S} \backslash E(N_0)$ and since $E_r(N_0) \to E(N_0)$ as $r \to \infty$, we get that $ \mathcal{S} \backslash E_r(N_0) \to \mathcal{S} \backslash E (N_0)$ as $r \to \infty$ or equivalently $ E_r^c(N_0) \to E^c(N_0)$ as $r \to \infty$. Let 
$$  A_r = \Bigg(\lim_{k\to \infty}\w^r_k \in  \mathcal{I}\backslash \mathcal{I}_+ \hspace{0.1cm}; \hspace{0.1cm} \{ \w^r_k\} \in S_r 
\Bigg),$$ 
$$ B_r = \Bigg(\lim_{k\to \infty}\w^r_k \text{ does not exist} \hspace{0.1cm}; \hspace{0.1cm}  \{ \w^r_k\} \in S_r
\Bigg) \bigcup \bigg(E_r \backslash E_r(N_0)\bigg),$$ 
then $ E_r^c(N_0) = A_r \bigcup B_r$ from \eqref{complementEr}. Let 
$$  A = \Bigg(\lim_{k\to \infty}\w_k \in  \mathcal{I}\backslash \mathcal{I}_+ \hspace{0.1cm}; \hspace{0.1cm} \{ \w_k\} \in S 
\Bigg),$$ 
$$ B = \Bigg(\lim_{k\to \infty}\w_k \text{ does not exist} \hspace{0.1cm}; \hspace{0.1cm}  \{ \w_k\} \in S
\Bigg) \bigcup \bigg(E \backslash E(N_0)\bigg),$$ 
then $ E^c(N_0) = A \bigcup B$ from \eqref{complementE}. Now $A_i \bigcap B_j = \emptyset$ for all $i,j \geq 0$ and we have $\lim_r (A_r \bigcup B_r) = A \bigcup B$ where $A \bigcap B = \emptyset$ and for all $r\geq 0$ we have $  A_r \bigcap B = \emptyset$, $  B_r \bigcap A = \emptyset$. Then from Lemma~\ref{measurelemmasup4} we get that $\lim_r A_r = A$, $\lim_r B_r = B$. Also, recall that 
$$\probP_2 \Bigg(\lim_{k\to \infty}\w^r_k \in  \mathcal{I}_- \hspace{0.1cm}; \hspace{0.1cm} \{ \w^r_k\} \in S_r 
\Bigg) =0$$ for any $r \geq 0$ from Theorem \ref{measuretheorem2}, compactness of $\mathcal{U}$ and 
$$\probP_1 \Bigg(\lim_{k\to \infty}\w^r_k \in  \mathcal{I}_0 \hspace{0.1cm}; \hspace{0.1cm} \{ \w^r_k\} \in S_r 
\Bigg) =0$$ for any $r \geq 0$ since $\mathcal{I}_0 $ is $\probP_1$ null. Therefore we have 
\begin{align}
    \probP\bigg( A_r \bigg) &= \probP \Bigg(\lim_{k\to \infty}\w^r_k \in  \mathcal{I}\backslash \mathcal{I}_+ \hspace{0.1cm}; \hspace{0.1cm} \{ \w^r_k\} \in S_r 
\Bigg) \nonumber \\ &=\probP \Bigg(\lim_{k\to \infty}\w^r_k \in  \mathcal{I}_- \hspace{0.1cm}; \hspace{0.1cm} \{ \w^r_k\} \in S_r 
\Bigg) + \probP \Bigg(\lim_{k\to \infty}\w^r_k \in  \mathcal{I}_0 \hspace{0.1cm}; \hspace{0.1cm} \{ \w^r_k\} \in S_r 
\Bigg) = 0
\end{align}
 for any $r \geq 0$. Then using the convergence of $A_r$ to $A$ and continuity of probability measure we get that:
\begin{align}
    \probP(A) = \probP(\lim_{r \to \infty} A_r ) = \lim_{r \to \infty} \probP( A_r ) = 0
\end{align} 
 which completes the proof of part $a$.
\end{proof} 

\section{Asymptotic eigenvalue analysis at critical points} \label{Appendix B}

\subsection{Theorem \ref{generalacclimiteigen}}
\begin{proof}

Using Lemma \ref{lemma_pk}, the asymptotic Jacobian map $DP$ is evaluated at $[\x^*;\x^*]$ to obtain:
 \begin{equation}
 DP([\x^*;\x^*]) = \begin{bmatrix}
 (1+\beta)\bigg(\mathbf{I} - h \nabla^2 f(\x^*)\bigg) \hspace{0.1cm} &  -\beta\bigg(\mathbf{I} - h \nabla^2 f(\x^*)\bigg)\\  \mathbf{I} \hspace{0.2cm} & \mathbf{0}
 \end{bmatrix}.
 \end{equation}
The eigenvalues for this matrix can be obtained by applying a permutation operation followed by using the Jordan normal form and solving a quadratic equation. In particular, after some trivial steps and simplification, the quadratic eigenvalue equation obtained is as follows:
\begin{align}
\lambda( DP([\x^*;\x^*]))^2 - (1+\beta)\lambda(\M)\lambda( DP([\x^*;\x^*]))  & = -\beta \lambda(\M)  
\end{align} 
where $\M = \mathbf{I} - h \nabla^2 f(\x^*)$, $\lambda(.)$ operator gives any general eigenvalue and the roots of the above quadratic can be given by:
\begin{align}
2\lambda( DP([\x^*;\x^*])) &= (1+\beta)\lambda(\M) \pm C \nonumber\\
\text{where} \hspace{0.2cm} C^2 &= (1+\beta)^2\lambda(\M)^2 - 4\beta\lambda(\M). \nonumber  
\end{align} 
For $h \in (0,\frac{1}{L})$, the $i^{th}$ eigenvalue of matrix $\M= \mathbf{I} - h\nabla^2 f(\x^*)$ is given by $\lambda_i(\M) = 1 - h \lambda_i(\nabla^2 f(\x^*))$. Now, $\lambda_i(\M)$ has the following property: 
$\lambda_i(\M) \in
 (0,1] $ if $\lambda_i(\nabla^2 f(\x^*)) \geq 0$ whereas  
 $\lambda_i(\M) \in (1,2)$ when $\lambda_i(\nabla^2 f(\x^*)) < 0$. 
Then simplifying the solution of $\lambda( DP(\x^*,\x^*))$, we get:
\begin{align}
    \lambda_i(D P([\x^*;\x^*])) = \begin{cases}
   \frac{1}{2}\bigg((1+\beta)\lambda_i(\M) \pm \sqrt{(1+\beta)^2\lambda_i(\M)^2 - 4\beta\lambda_i(\M)}\bigg)  &  ; \hspace{0.1cm}\lambda_i(\M) > \frac{4 \beta}{(1+ \beta)^2} \\
      \frac{1}{2}\bigg((1+\beta)\lambda_i(\M) \pm \textbf{\textit{i}}\sqrt{4\beta\lambda_i(\M)-(1+\beta)^2\lambda_i(\M)^2 }\bigg)  &  ; \hspace{0.1cm}\lambda_i(\M) \in (0, \frac{4 \beta}{(1+ \beta)^2}]
    \end{cases},\label{quad1}
\end{align} 
with the convention that $(0, \frac{4 \beta}{(1+ \beta)^2}] $ is an empty set for $\beta =0$.
For the eigenvalues of $\nabla^2 f(\x^*)$ corresponding to its unstable subspace $\mathcal{E}_{US}$, we have $\lambda_i(\M) \in
(1,2) $. Then $(1+\beta)^2\lambda_i(\M)^2 - 4\beta\lambda_i(\M)>0$ for $\beta\leq 1$ (since $ \frac{4\beta}{(1+\beta)^2}\leq 1$ for $\beta \leq 1$) and we get real roots where the larger root satisfies the condition: 
\begin{align}
\abs{\lambda_i(DP([\x^*;\x^*]))} & = \frac{1}{2}\bigg((1+\beta)\lambda_i(\M) + \sqrt{(1+\beta)^2\lambda_i(\M)^2 - 4\beta\lambda_i(\M)}\bigg) > 1 \label{largerroot1}
\end{align}
and the smaller root satisfies:
\begin{align}
    \abs{\lambda_i(DP([\x^*;\x^*]))} & = \frac{1}{2}\bigg((1+\beta)\lambda_i(\M) - \sqrt{(1+\beta)^2\lambda_i(\M)^2 - 4\beta\lambda_i(\M)}\bigg) \leq 1. \label{smallerroot}
\end{align}
Hence, for $\beta\leq 1$ the real roots of $\lambda_i(DP([\x^*;\x^*]))$ contribute to both expansive and contractive dynamics for $ \lambda_i(\M) \in (1,2)$.

Next, if $ (1+\beta)^2\lambda_i(\M)^2 - 4\beta\lambda_i(\M) \leq 0$ then we have complex roots whose magnitude is given by:
\begin{align}
    \abs{\lambda_i(DP([\x^*;\x^*]))} & = \sqrt{\beta \lambda_i(\M)} \leq \frac{2\beta}{1+\beta} \label{complexmag}
\end{align}
which could be expansive or contractive depending upon $\beta$. These complex roots occur when $ \lambda_i(\M) \leq \frac{4 \beta}{(1+\beta)^2}$ which implies that for large momentum not every eigenvalue from the stable subspace of $\M$ will yield complex eigenvalues of $DP([\x^*;\x^*])$. However, when we have complex eigenvalues for $DP([\x^*;\x^*])$ we get spiralling of trajectories. 

For the case when $\x^*$ is a local minimum and $\beta_k \leq 1$ for all $k$, from \eqref{quad1} the eigenvalues of $DP([\x^*;\x^*])$ will be complex with magnitude equal to $\sqrt{\beta\lambda_i(\M) }$ which is less than equal to $1$ since $\beta \leq 1$ and $  \lambda_i(\M) \in (0,1]$. These eigenvalues are given by:
\begin{align}
    \lambda_i(D P([\x^*;\x^*])) =  
      \frac{1}{2}\bigg((1+\beta)\lambda_i(\M) \pm \textbf{\textit{i}}\sqrt{4\beta\lambda_i(\M)-(1+\beta)^2\lambda_i(\M)^2 }\bigg)  &  ; \hspace{0.1cm}\lambda_i(\M) \in (0, \frac{4 \beta}{(1+ \beta)^2}].
\end{align}
This completes the proof. 
\end{proof}   

\section{Almost sure non-convergence guarantees}\label{Appendix C_a}

\subsection{Lemma \ref{lemmaasymnec}}
\begin{proof}
     {Let $\w_k = [\x_{k}; \x_{k-1}]$ for any $k$ and $\w^* =[\x^*;\x^*]$ for any critical point $\x^*$ of $f$. Since $f \in \mathcal{C}^{2,1}_{L}(\mathbb{R}^n)$, $h< \frac{1}{L}$ and the choice of $p_k, q_k$ satisfy \eqref{generalds} update with $\beta_k=q_k$, we have that $$\{k \hspace{0.1cm} \vert \hspace{0.1cm}\w_k \in \mathcal{B}_{\delta}(\w^*); \hspace{0.1cm} \w_0 \in \mathbb{R}^{2n} \backslash \w^*\} = \mathbb{Z}^* $$ from Lemma \ref{lemmaSdelta} (proved later) for any $\delta >0$ . Also, we have $DP_k \darrow DP$ on compact sets, $P_k$ for any $k$ and the map $P$ are $\probP_1$-a.s. diffeomorphisms on compact sets from Lemma \ref{lemma_pk}, Corollary \ref{corsup2}. Using these facts we get that 
     \begin{align}
          \inf_{\w_k \in \mathcal{B}_{\delta}(\w^*); \hspace{0.1cm} \w_0 \in \mathbb{R}^{2n} \backslash \w^*}  \norm{[D P_k(\w_k)]^{-1}}^{-1}_2 & \leq  \lim_{\delta \to 0}\inf_{\w_k \in \mathcal{B}_{\delta}(\w^*); \hspace{0.1cm} \w_0 \in \mathbb{R}^{2n} \backslash \w^*}  \norm{[D P_k(\w_k)]^{-1}}^{-1}_2  \nonumber \\ 
          &\underbrace{\leq}_{I_1} \lim_{k \to \infty} \norm{[D P_k(\w_k)]^{-1}}^{-1}_2  = \norm{[D P(\w^*)]^{-1}}^{-1}_2  \hspace{0.1cm} \probP_1-a.s. ,\label{tempiq1a}
     \end{align}
     where the inequality $I_1$ follows from the fact that $$ \lim_{\delta \to 0}\inf_{\w_k \in \mathcal{B}_{\delta}(\w^*); \hspace{0.1cm} \w_0 \in \mathbb{R}^{2n} \backslash \w^*}  \norm{[D P_k(\w_k)]^{-1}}^{-1}_2 \leq \lim_{\delta \to 0}\inf_{\w_k \in \mathcal{B}_{\delta}(\w^*); \hspace{0.1cm} \w_0 \in \mathbb{R}^{2n} \backslash \w^*; \w_k \to \w^*}  \norm{[D P_k(\w_k)]^{-1}}^{-1}_2.$$
     In particular, the inequality $I_1$ is obtained by evaluation of $\norm{[D P_k(\w_k)]^{-1}}^{-1}_2 $ along a converging trajectory of $\w_k$ and taking $k \to \infty$.
     
     Next, \eqref{necessaryconvergencestability} is a necessary condition from the following steps:
\begin{align}
   &\w_{k+1} =  P_k(\w_k)  \\
  \implies  &\w_{k+1}= \underbrace{P_k(\w^*)}_{=\w^* }  + DP_k(\w^*)(\w_{k}-\w^*) + o(\norm{\w_{k}-\w^*})\\
\implies & [DP_k(\w^*)]^{-1}(\w_{k+1} - \w^*)  = (\w_{k} - \w^*) +  o(\norm{\w_{k}-\w^*}) \\
   \implies & {\norm{ [DP(\w^*)]^{-1}}^{-1}_{2}} \bigg(1+ \norm{ [DP(\w^*)]^{-1}}^{-1}_{2}\norm{ [DP_k(\w^*)]^{-1} - [DP(\w^*)]^{-1}}_{2}\bigg)^{-1}\norm{\w_{k} -\w^*}   \leq \nonumber\\ & \hspace{6.5cm} \norm{\w_{k+1}-\w^*} + o(\norm{\w_{k}-\w^*}) \label{stepdiv1} \\ 
      \implies   &  \hspace{0.5cm} \frac{\inf_{\w_k \in \mathcal{B}_{\delta}(\w^*); \hspace{0.1cm} \w_0 \in \mathbb{R}^{2n} \backslash \w^*}\norm{[D P_k(\w_k)]^{-1}}^{-1}_2}{ \bigg(1+ \norm{ [DP(\w^*)]^{-1}}^{-1}_{2}\underbrace{\norm{ [DP_k(\w^*)]^{-1} - [DP(\w^*)]^{-1}}_{2}}_{\to 0 \text{ as } k \to \infty}\bigg)} \leq   \inf_{\w_k \in \mathcal{B}_{\delta}(\w^*); \hspace{0.1cm} \w_0 \in \mathbb{R}^{2n} \backslash \w^*}\frac{\norm{\w_{k+1}-\w^*}}{\norm{\w_{k}-\w^*}} \nonumber \\ & \hspace{10cm}  +  o(1) \hspace{0.2cm} \probP_1-\text{a.s.}
\end{align}
where\footnote{ {Note that in the second last step \eqref{stepdiv1} we can divide by $\norm{ [DP_k(\w^*)]^{-1}}_2 $ on both sides because $P_k$ is a diffeomorphism around $\w^*$ for all $k$ $\probP_1$-almost surely and $DP_k \darrow DP$ on compact sets.}} we used \eqref{tempiq1a} in the last step. Since the above inequality holds for any $k$, taking $k \to \infty$ on the denominator of left hand side and using $DP_k \darrow DP$ on compact sets we get that
\begin{align}
    {\inf_{\w_k \in \mathcal{B}_{\delta}(\w^*); \hspace{0.1cm} \w_0 \in \mathbb{R}^{2n} \backslash \w^*}\norm{[D P_k(\w_k)]^{-1}}^{-1}_2} \leq   \inf_{\w_k \in \mathcal{B}_{\delta}(\w^*); \hspace{0.1cm} \w_0 \in \mathbb{R}^{2n} \backslash \w^*}\frac{\norm{\w_{k+1}-\w^*}}{\norm{\w_{k}-\w^*}}   +  o(1) \hspace{0.2cm} \probP_1-\text{a.s.}
\end{align}
Then in the above inequality after taking the limit $\delta \to 0$ or equivalently $\norm{\w_{k}-\w^*} \to 0 $ (since $\norm{\w_{k}-\w^*} < \delta $), using $DP_k \darrow DP$ and hence local equicontinuity of $\{DP_k\}$, we can get rid of $o(1)$ term\footnote{Using the fundamental theorem of calculus, it can be readily deduced that the $ o(1)$ term is upper bounded by $ \norm{DP_k(\z) - DP_k(\w^*)}$ for some $\z \in \mathcal{B}_{{\delta}}(\w^*)$ and therefore is $k$ dependent. But since $\{DP_k\}$ are locally equicontinuous, we can replace $ \norm{DP_k(\z) - DP_k(\w^*)}$ with a uniform modulus of continuity depending only on $\delta$. The same principle of equicontinuity will be used in subsequent proofs to get rid of $o(1)$ terms.} and the first term on right hand side after taking limit will simplify to the following expression 
$$\lim_{\delta \to 0} \inf_{\w_k \in \mathcal{B}_{\delta}(\w^*); \hspace{0.1cm} \w_0 \in \mathbb{R}^{2n} \backslash \w^*}\frac{\norm{\w_{k+1}-\w^*}}{\norm{\w_{k}-\w^*}} =\sup_{\delta>0} \inf_{\w_k \in \mathcal{B}_{\delta}(\w^*); \hspace{0.1cm} \w_0 \in \mathbb{R}^{2n} \backslash \w^*} \frac{\norm{\w_{k+1}-\w^*}}{\norm{\w_{k}-\w^*}}.$$ 
Also, we have that $$ \liminf_{k \to \infty} \frac{\norm{\w_{k+1}-\w^*}}{\norm{\w_{k}-\w^*}}=  \sup_{K \geq 0} \inf_{k \geq K} \frac{\norm{\w_{k+1}-\w^*}}{\norm{\w_{k}-\w^*}} \geq \sup_{\delta>0} \inf_{\w_k \in \mathcal{B}_{\delta}(\w^*); \hspace{0.1cm} \w_0 \in \mathbb{R}^{2n} \backslash \w^*} \frac{\norm{\w_{k+1}-\w^*}}{\norm{\w_{k}-\w^*}} $$ from the fact the set $\{k \hspace{0.1cm} \vert \hspace{0.1cm}\w_k \in \mathcal{B}_{\delta}(\w^*); \hspace{0.1cm} \w_0 \in \mathbb{R}^{2n} \backslash \w^*\} = \mathbb{Z}^* $ from Lemma \ref{lemmaSdelta} for any $\delta >0$ and thus inf over the set $\{k: k \geq K\}$ for all $K > 0$ will be greater than or equal to inf over the set $\mathbb{Z}^*$.
 Now for $\w_k \to \w^*$ to hold, we must necessarily have the condition $ \limsup_{k \to \infty} \frac{\norm{\w_{k+1}-\w^*}}{\norm{\w_{k}-\w^*}} \leq 1$. This would imply that 
 $$ \sup_{\delta>0} \inf_{\w_k \in \mathcal{B}_{\delta}(\w^*); \hspace{0.1cm} \w_0 \in \mathbb{R}^{2n} \backslash \w^*} \frac{\norm{\w_{k+1}-\w^*}}{\norm{\w_{k}-\w^*}}  \leq \liminf_{k \to \infty} \frac{\norm{\w_{k+1}-\w^*}}{\norm{\w_{k}-\w^*}} \leq \limsup_{k \to \infty} \frac{\norm{\w_{k+1}-\w^*}}{\norm{\w_{k}-\w^*}}\leq 1 ,$$
 thereby proving that \eqref{necessaryconvergencestability} is a necessary condition.  
}

Next, if the Jacobian $D P(\w^*) $ satisfies
\begin{align}
     \min_i\abs{\lambda_i(D P(\w^*) )} < 1, 
\end{align}
 then the necessary condition \eqref{necessaryconvergencestability} for the case of converging trajectory of $\{\w_k\}$ is automatically satisfied. This can be readily checked from the following steps:
 \begin{align}
     \inf_{\w_k \in \mathcal{B}_{\delta}(\w^*); \hspace{0.1cm} \w_0 \in \mathbb{R}^{2n} \backslash \w^*}  \norm{[D P_k(\w_k)]^{-1}}^{-1}_2 & \leq \lim_{k \to \infty} \norm{[D P_k(\w_k)]^{-1}}^{-1}_2 \nonumber \\&=\norm{[D P(\w^*)]^{-1}}^{-1}_2 \hspace{0.2cm} \probP_1-\text{a.s.} \nonumber \\ & \leq \min_i\abs{\lambda_i(D P(\w^*))}  \\
   \implies  \lim_{\delta \downarrow 0} \inf_{\w_k \in \mathcal{B}_{\delta}(\w^*); \hspace{0.1cm} \w_0 \in \mathbb{R}^{2n} \backslash \w^*}  \norm{[D P_k(\w_k)]^{-1}}^{-1}_2 & \leq \min_i\abs{\lambda_i(D P(\w^*))}, \hspace{0.2cm} \probP_1-\text{a.s.}
 \end{align}
 where we used \eqref{tempiq1a} in the first step. This completes the proof.
\end{proof}

\subsection{Lemma \ref{suplem_3}}
\begin{proof}
Let $\mathcal{A}$ be the set of accumulation points of $\{\w_k\}$ and suppose $\mathcal{A}$ is not connected. We know that the set of all subsequential limits of a sequence in a metric space is closed \cite{rudin1976principles}. Hence $\mathcal{A}$ is disconnected and closed and so is separated by closed compact sets $C$ and $D$. Hence there exits an $\epsilon >0$ such that for all $\x \in C$ and $\y \in D$ we have $ d(\x, \y)> 2 \epsilon $ where $d(\cdot,\cdot)$ is the metric.  

Let $W_C$ be the union of all the open $\epsilon$-balls about elements of $C$ and define $W_D$ similarly. By the definition of accumulation point, $\{\w_k\}$ is frequently in $W_C$ and frequently in $W_D$. Suppose for the sake of contradiction that there exists $\x \in W_C \cap W_D $. Then there must be $\p \in C$ and $\q \in D$ such that $d(\x,\p)< \epsilon$ and $d(\x,\q)< \epsilon$. By the triangle inequality, $d(\q,\p)< 2\epsilon$, a contradiction. Thus we conclude that $W_C \cap W_D = \emptyset$.

Next for any $\x,\y$ in $X$ and any $\epsilon>0$ we have for all $k\geq 0 $ that $d(A_k(\x),A_k(\y)) <  \epsilon $ whenever $ d(\x,\y) < \delta$. Here we used $d(A_k(\x), A_k(\y)) < \epsilon $ for all $k $ and for all $\x, \y \in X$ by uniform equicontinuity of $A_k$. 

$ \bLozenge$ Thus there is a $\delta$ such that $0< \delta < \epsilon$ such that for all $\x,\y \in X$ and some $K$, $d(\x,\y)<\delta$ implies $ d(A_k(\x), A_k(\y)) < \epsilon$ for all $k \geq K$.

Let $U$ be the union of all the open $\delta$-balls about elements of $\mathcal{A}$. Then $\{\w_k\}$ is eventually in $U$: if not, then there is a subsequence in $X \backslash U $, which is closed, hence compact, so there is an accumulation point in $X \backslash U $, a contradiction.

Let $K_0> K$, where $K$ is sufficiently large, be such that $k \geq K_0$ implies $\w_k \in U$. Let $k > K_0$. Since $\delta < \epsilon$, $\w_k \in U \subseteq W_C \cup W_D$. Assume without loss of generality that $\w_k\in U \cap W_C$. By the definition of $U$, there is an element $\a \in \mathcal{A}$ such that $d(\w_k,\a)<\delta<\epsilon$. Since $W_C$ and $W_D$ are disjoint, $\w_k \notin W_D$, so $\a \notin D$, so $\a \in C$. 

By the choice of $\delta$, we get $d(A_k(\w_k),A_k(\a))< \epsilon$. Since $\a \in  \mathcal{A}$, $\a$ is a fixed point of $A_k$ for any $k$ implying $A_k(\a) = \a$ for all $k \geq 0$. Thus $d(\w_{k+1},\a)<\epsilon$, so $\w_{k+1} \in W_C$. 

Thus we have shown that $\{\w_k\}$ is eventually in $W_C$ and thus is not infinitely often in $W_D$, contradicting the fact that $W_D$ is an open set containing elements of $\mathcal{A}$.

As this contradiction arose from the assumption that $\mathcal{A}$ is disconnected, we conclude that $\mathcal{A}$ is connected.

For the other part suppose the family of maps $\{A_k\}$ is uniformly equicontinuous in $X$ $\probP_1$-almost surely and all other assumptions on $\{A_k\}$ remain unchanged. Then for every $\epsilon >0$ there exists some $\delta>0$ such that $$ \probP_1\bigg(\sup_{k}\sup_{\substack{d({\x,\y})< \delta\\ \x, \y \in X}} d({A_k(\x),A_k(\y)})< \epsilon \bigg) = 1$$  and all the statements from the symbol `$ \bLozenge$' onward hold $\probP_1$-almost surely. Hence the statement of the lemma stands proved. It should be noted that here we used the definition of uniform stochastic equicontinuity instead of the more commonly used asymptotic uniform stochastic equicontinuity which is defined as follows:

A sequence of random functions $\{Q_n\}$ on Euclidean space is asymptotically uniformly stochastic equicontinuous if $\forall \epsilon>0$, $\eta > 0$, there exists $\delta>0$ such that 
$$ \limsup_{n \to \infty} \probP_1\bigg(\sup_{\substack{\norm{\x-\y}< \delta}} \norm{Q_n(\x) -Q_n(\y) }> \epsilon \bigg) < \eta. $$ Then if $\eta = 0$ we have almost sure asymptotic uniform stochastic equicontinuity. For details see chapter 21 from \cite{davidson1994stochastic}.
\end{proof}  

\subsection{Theorem \ref{measurethm7}}
\begin{proof}
From Lemma \ref{lemmalyapunov}, for \eqref{generalds} with $\beta_k \leq \frac{1}{\sqrt{2}}$ for all $k$, we have a monotonically decreasing Lyapunov function given by $ \hat{f}([\x_k; \x_{k-1}]) = f(\x_k)+ \frac{\norm{\x_{k-1} -\x_{k} }^2}{2h}$ that decreases along the sequence $\{\x_k\}$ of \eqref{generalds}. Since $f$ is coercive, its sublevel sets are compact \cite{kinderlehrer2000introduction} and since $ \hat{f}([\x_k; \x_{k-1}])$ is decreasing, from Lemma \ref{lemmalyapunov}, we get that the sequence $\{\x_k\}$ always stays within the compact set $ \mathcal{D}=\bigg\{\x : f(\x) \leq f(\x_0) + \frac{\norm{\x_{0} -\x_{-1} }^2}{2h}\bigg\}$. The Lyapunov function $\hat{f}(\cdot) $ is continuous in $[\x_k;\x_{k-1}]$ and $[\x^*;\x^*]$ are the fixed points of $\hat{f}(\cdot) $ where $\x^*$ is any critical point of $f$. Next, it is easy to check that the map $P_k : [\x_k; \x_{k-1}] \mapsto [\x_{k+1};\x_k]$ from Lemma \ref{lemma_pk} for any $k$, corresponding the update $\w_{k+1} = P_k(\w_k)$ from Theorem~\ref{measuretheorem3}, is closed. Let $S$ be the set of critical points of $f$. Notice that the map $P_k$ is continuous by Lemma \ref{lemma_pk}. Since $\x_k \in \mathcal{D} $ for all $k$, the map $P_k  : [\x_k; \x_{k-1}] \mapsto [\x_{k+1};\x_k]$ takes the product set $\mathcal{D} \times \mathcal{D}$ to itself, i.e., $ P_k : \mathcal{D} \times\mathcal{D}  \rightarrow \mathcal{D} \times\mathcal{D} $, where $\mathcal{D}$ is compact and Hausdorff \footnote{A Hausdorff space is a topological space with a separation property: any two distinct points can be separated by disjoint open sets.} and so is $\mathcal{D}\times\mathcal{D}$. Then by the closed map lemma (Lemma A.52 in \cite{lee2013smooth}), $P_k$ is a closed map in $\mathcal{D} \times\mathcal{D}$ and hence closed in $\mathcal{D}\times\mathcal{D} \backslash S \times S$. The same conclusion holds for the map $P$. Then the two algorithms generating the following two sequences $\{\w^r_k\}$, $\{\w_k\}$ where:
\begin{align}
    \w^r_{k+1} =  \begin{cases} 
      P_k(\w^r_k) & 0 \leq k\leq r \\
     P(\w^r_k) & k > r 
   \end{cases} \label{recuro1}
\end{align} for any $r\geq 0$ and
\begin{align}
    \w_{k+1} = P_k(\w_k) \hspace{0.1cm} \forall \hspace{0.1cm} k \geq 0, \label{recuro2}
\end{align}
satisfy all three hypotheses of Theorem \ref{thmglob}. Then these sequences $\{\w^r_k\}$, $\{\w_k\}$ have convergent subsequences that converge in compact sets to $\mathcal{I} = \{[\x;\x]: \nabla f(\x)= \mathbf{0}\}$ from Theorem \ref{thmglob}. These subsequential limits cannot be outside $ \mathcal{I}$ because it follows directly from Lemma \ref{lemma_pk} that the fixed points of the map $P_k$ for any $k$ and the map $P$ are the points in $\mathcal{I}$.

$ \bLozenge \hspace{0.2cm}$  We now show that the sequences generated from these recursions also converge. First, since $f$ is Morse, its critical points are isolated \cite{matsumoto2002introduction} and thus the set $ \mathcal{I}$ has isolated points. Next, for any sequence $\{\w^r_k\}$ that is generated from \eqref{recuro1} for any $r \geq 0$, the set of accumulation points for this sequence $\{\w^r_k\}$ will be connected from Lemma \ref{suplem_3} due to the uniform continuity of the map $P$ on compact $\mathcal{D} \times\mathcal{D}$. But the only connected isolated set of points is singleton and so this sequence $\{\w_k\}$ generated from \eqref{recuro1} for any $r\geq 0$ has only a single accumulation point, thus implying that the sequence $\{\w^r_k\}$ from \eqref{recuro1} converges in $ \mathcal{I}$. Now for the sequence $\{\w_k\}$ generated from \eqref{recuro2}, we do not have a constant map $P$. From Lemma \ref{lemma_pk}, Corollary \ref{corsup2} we have that $DP_k \darrow DP$ on $\mathcal{D} \times\mathcal{D}$ and so $\sup_{\x \in \mathcal{D} \times\mathcal{D}} \norm{DP_k(\x)}_2 \to \sup_{\x \in \mathcal{D}\times\mathcal{D}} \norm{DP(\x)}_2$ as $k \to \infty$ implying that the derivative of the family of maps $\{P_k\}$ is uniformly bounded on $\mathcal{D} \times\mathcal{D}$, i.e. $\sup_k \sup_{\x \in \mathcal{D}\times\mathcal{D}} \norm{DP_k(\x)}_2 < \infty$. Hence, the family of maps $\{P_k\}$ is uniformly equicontinuous on $\mathcal{D}\times\mathcal{D}$. Then using Lemma \ref{suplem_3} we get that the set of accumulation points for the sequence $\{\w_k\}$ generated by \eqref{recuro2} will be connected on $\mathcal{D}\times\mathcal{D}$. Since these points are contained in $\mathcal{I}$ they are isolated and so the sequence $\{\w_k\}$ has only a single accumulation point, thus implying that the sequence $\{\w_k\}$ from the recursion $\w_{k+1} = P_k(\x_k)$ converges in $\mathcal{I}$ and hence we have that for \eqref{generalds} with $\beta_k \leq \frac{1}{\sqrt{2}}$, $\beta_k \to \beta$ and $h<\frac{1}{L}$,
\begin{align}
     \probP(\{\lim_{k \to \infty}[\x_k;\x_{k-1}] \in  \mathcal{I}\}) = 1. \label{prob1ab}
\end{align}

Since $f$ is Hessian Lipschitz on any compact set and $\mathcal{D}$ is compact, the maps $P_k$ for all $k$ and the map $P$ are $\probP_1$-a.s. diffeomorphisms on $\mathcal{D} \times \mathcal{D}$ by Corollary \ref{corsup2}. Since $f$ is Morse, its first order critical points will be local maxima, local minima and strict saddle points and also these critical points will be isolated and hence countable. From Theorem~\ref{generalacclimiteigen}, we get that for \eqref{generalds} with $\beta_k \leq \frac{1}{\sqrt{2}}$ for all $k$, the largest magnitude eigenvalue or the spectral radius of $DP([\x^*;\x^*])$, where $\x^*$ is any local maxima or strict saddle point of $f$, will be strictly greater than $1$. Thus, $ \norm{DP([\x^*;\x^*])}_2 >1 $ since operator norm is greater than or equal to the spectral radius of any square matrix. Thus, from part $a.$ of Theorem~\ref{measuretheorem3}, for any initialization in compact $ \mathcal{U}_1' \times  \mathcal{U}_1' $ we get that 
\begin{align}
    \probP(\{[\x_k;\x_{k-1}] \to [\x^*;\x^*]\}) &= 0, \label{prob2ab}
\end{align}
where $\{\x_k\}$ is the \eqref{generalds} sequence with $\beta_k \leq \frac{1}{\sqrt{2}}$ for all $k$, $\beta_k \to \beta$, $h<\frac{1}{L}$ and $\x^*$ is any local maxima or strict saddle point of $f$. 

Let $\mathcal{I}_* = \mathcal{I} \bigcap  \{[\x;\x]: \nabla^2 f(\x) \succ \mathbf{0}\}  $ and since $f$ is Morse, the set $\mathcal{I} \bigcap  \{[\x;\x]: det(\nabla^2 f(\x)) = 0\} $ is empty. Since $\mathcal{I} $ has countable elements, for the complement $\mathcal{I}_*^c = \mathcal{I} \bigcap  \{[\x;\x]: \nabla^2 f(\x) \succ \mathbf{0}\}^c$ of the set $\mathcal{I}_*$ in $\mathcal{I}$ we then have 
\begin{align}
     \probP(\{\lim_{k \to \infty}[\x_k;\x_{k-1}] \in  \mathcal{I}^c_*\}) = 0, \label{prob3ab}
\end{align}
 by \eqref{prob2ab}, the fact that for any $[\x^*;\x^*] \in \mathcal{I}^c_* $ the point $\x^*$ must either be a local maxima or a strict saddle point of $f$, $\mathcal{I}^c_* $ has countable elements and sum of countable $0$'s is $0$. Then using \eqref{prob1ab} and \eqref{prob3ab} for \eqref{generalds} with $\beta_k \leq \frac{1}{\sqrt{2}}$, $\beta_k \to \beta$ and $h<\frac{1}{L}$, it must be that $ \probP(\{\lim_{k \to \infty}[\x_k;\x_{k-1}] \in \mathcal{I}_*\}) = 1$, which proves almost sure convergence of \eqref{generalds} sequence with $\beta_k \leq \frac{1}{\sqrt{2}}$ to some local minimum.  $ \clubsuit$

  {Finally, let $f \in \mathcal{C}^{2}$ be a coercive, Morse function that is Hessian Lipschitz continuous in every compact set. Let $\phi(\x)$ be a $\mathcal{C}^{\infty}$ smooth bump function that is equal to $1$ on the compact sublevel set $\mathcal{V}_1$ (compactness of $\mathcal{V}_1$ follows from coercivity of $f$), $\phi \equiv 0$ on $\mathbb{R}^n \backslash\mathcal{V}_2$ for some compact $\mathcal{V}_2 \supsetneq \mathcal{V}_1$ and $0< \phi <1$ on $\mathcal{V}_2 \backslash \mathcal{V}_1$. Also, $\norm{\nabla \phi}$ on $\mathcal{V}_2 \backslash \mathcal{V}_1$ is controlled\footnote{The explicit construction of the globally gradient Lipschitz extension of $f$ can be done using Lipschitz extension results such as the Theorem \ref{kirszbraunthm} from \cite{schwartz1969nonlinear}. In a later section of this work (Section \ref{final_extension section}) we construct one such extension while proving Theorem \ref{kirszthmadapted}.} so that the function $f \phi$ is at most $\tilde{L}$-smooth on the compact set $\mathcal{V}_2 \backslash \mathcal{V}_1$ for some $\tilde{L}>L $. Then $F \equiv f \phi \in \mathcal{C}^{2,1}_{\tilde{L}}(\mathbb{R}^n) $ and $F \equiv f$ on $\mathcal{V}_1$. Now for the function $f$, it is given that the sequence $\{\x_k\}$ from \eqref{generalds} for any momentum sequence $\{\beta_k\}$, where $\beta_k \leq 1$ for all $k$ and $h \in (0, \frac{1}{\tilde{L}})$, stays within the sublevel set $ \mathcal{V}_1 \supsetneq \mathcal{U}_1$ of $f$. Since $F \equiv f$ on the compact set $\mathcal{V}_1$, for a common initialization in $\mathcal{U}_1$, the sequence $\{\x_k\}$ from \eqref{generalds} iterated on the function $F$ must be equal to the sequence $\{\x_k\}$ from \eqref{generalds} iterated on the function $f$. Hence we can now work with the sequence $\{\x_k\}$ from \eqref{generalds} iterated on the $\tilde{L}$-gradient Lipschitz function $F$. Then the corresponding sequence $\{[\x_k;\x_{k-1}]\}$ always stays bounded in the compact set $\mathcal{V} = \mathcal{V}_1 \times \mathcal{V}_1$ by compactness of $\mathcal{V}_1$.} From Theorem~\ref{generalacclimiteigen} for \eqref{generalds}, with $\beta_k \to \beta$ where $\beta \leq 1$, for any strict saddle point $\x^*$ of $F \equiv f$ in $\mathcal{V}_1$, we get that the spectral radius of $DP([\x^*;\x^*])$ is larger than $1$ and so $\norm{DP([\x^*;\x^*])}_2 >1$. Then for the function $F $, using part $a.$ of Theorem~\ref{measuretheorem3} on the compact set $\mathcal{V}$ and any initialization in the compact set $\mathcal{U}_1' \times \mathcal{U}_1' \subset \mathcal{U}_1 \times \mathcal{U}_1 = \mathcal{U}$ where $ \mathcal{U} \subsetneq \mathcal{V}$, we get that $\probP(\{[\x_k;\x_{k-1}] \to [\x^*;\x^*]\})=0$ for \eqref{generalds} with $\beta_k \to \beta$, $\beta_k \leq 1$ for all $k$ and $h \in (0, \frac{1}{\tilde{L}})$. This completes the proof. 
\end{proof}

\subsection{Note on the boundedness of \eqref{generalds} using dissipative property of $f$}\label{sub-disp}
Suppose $f \in \mathcal{C}^{2}$ is a coercive function that satisfies the following $(\rho,a,b)$ dissipative property $$\langle \nabla f(\x), \x \rangle \geq a\norm{\x}^{2+\rho} - b$$ for any positive $\rho$ and sufficiently large $a,b$. Also, let us assume that $\norm{\nabla f(\x)} = o(\norm{\x}^{1+2\rho})$ \footnote{The little-o notation in this section is used in the sense of $\frac{1}{\norm{\x}} \to 0$ or equivalently ${\norm{\x}} \to \infty$.}for any large $\norm{\x}$ so as to control the gradient growth from above.\footnote{Note that this gradient growth boundedness does not contradict the dissipative property of $f$.} {We note that the dissipative property and the gradient growth assumption are naturally satisfied by functions that include an $\ell_p$ penalty with any exponent $p > 2$, i.e., functions of the form
\[
f(\mathbf{x}) = \tilde{f}(\mathbf{x}) + \lVert \mathbf{x} \rVert_{p}^{\,p},
\]
where the base term $\tilde{f}(\mathbf{x})$ grows at most quadratically. In such cases, the composite function satisfies both the required dissipative property and the gradient growth assumption with parameter $\rho = p - 2$. For concrete examples of functions of this form arising in applications, see~\cite{lacroix2018canonical, pmlr-v89-hainline19a} and the references therein.}

We now claim that the trajectories $\{[\x_k;\x_{k-1}]\}$ of \eqref{generalds} with $\beta_k \leq 1$, when initialized in any compact $\mathcal{U} $, always stay within a compact $\mathcal{V} \supset \mathcal{U}$ provided $h < \frac{1}{\tilde{L}}$ where $\tilde{L}$ is the local gradient Lipschitz constant of $f$ on some sufficiently large sublevel set. To prove this claim, we will first derive a maximum principle for the function $f$ and then establish the boundedness guarantee using Lemma \ref{lemmalyapunov}. Suppose that $ S_R = \bigg\{\x \hspace{0.1cm}: \hspace{0.1cm} f(\x) \leq \sup_{\norm{\x}\leq R} f(\x) \bigg\}$ is some sublevel set of $f$ and $R$ is sufficiently large. Then $ S_R \subset \mathcal{B}_{C_1 R}(\mathbf{0}) $ for some $C_1 \geq 1$ by the compactness of $S_R$ from coercivity of $f$. 

Next, we derive a maximum principle for $f$ and using that maximum principle we show that the constant $C_1$ will be independent of $R$ for any sufficiently large $R$. Evaluating $\frac{\partial f(\x)}{\partial \norm{\x}}$ for some fixed unit direction vector $ \hat{\x}$ and using the chain rule we get: 
\begin{align}
    \frac{\partial f(\x)}{\partial \norm{\x}} & = \bigg\langle \nabla f(\x), \frac{\x}{\norm{\x}} \bigg\rangle \nonumber \\
    & \geq a\norm{\x}^{1+\rho} - \frac{b}{\norm{\x}} \label{comparisonineq1a1}
\end{align}
where we used the dissipative property in the last step\footnote{Note that since $\x = \norm{\x} \hat{\x}$ where $\hat{\x}$ is a fixed unit direction vector, we get $\frac{\partial \x}{\partial \norm{\x}} = \hat{\x} $. Then $\frac{\partial f(\x)}{\partial \norm{\x}} $ can be interpreted as the derivative of $f$ w.r.t. the radial distance along the direction $\hat{\x} $.}. Since $R$ is sufficiently large, for any $\norm{\x} \geq R$ we have $ a\norm{\x}^{1+\rho} - \frac{b}{\norm{\x}} > 0$ and thus $ \frac{\partial f(\x)}{\partial \norm{\x}}>0$, i.e., $f$ is increasing with $\norm{\x}$ along any given direction. This implies the following maximum principle
\begin{align}
    \sup_{\x \in \partial \mathcal{B}_{R_1}(\mathbf{0})} f(\x) > \sup_{\x \in \partial \mathcal{B}_{R_2}(\mathbf{0})} f(\x) \hspace{0.2cm} ; \inf_{\x \in \partial \mathcal{B}_{R_1}(\mathbf{0})} f(\x) > \inf_{\x \in \partial \mathcal{B}_{R_2}(\mathbf{0})} f(\x) \hspace{0.2cm} \forall \hspace{0.2cm} R_1,R_2 \hspace{0.2cm} \text{s.t.} \hspace{0.2cm} R_1 > R_2 \geq R, \label{comparisonineq1}
\end{align} 
where $\partial  \mathcal{B}$ is the boundary of the ball $ \mathcal{B} $. 

Using \eqref{comparisonineq1} and the continuity of $f$, we now want to show that for any sufficiently large $R$, $f$ on the ball $ \mathcal{B}_{R}(\mathbf{0})$ will attain maximum on its boundary $ \partial \mathcal{B}_{R}(\mathbf{0})$, i.e., $ \sup_{\norm{\x} \leq R } f(\x) = \sup_{\norm{\x} = R } f(\x) $. To prove this fact we use a contradiction argument. Suppose that $f$, on the ball $ \mathcal{B}_{R}(\mathbf{0})$, does not attain {its} maximum on the boundary $ \partial \mathcal{B}_{R}(\mathbf{0})$. Then $f$ will attain maximum in the interior of the ball $ \mathcal{B}_{R}(\mathbf{0})$ at some point, say $\u$ where $\norm{\u} <R$. Now two cases are possible, the first case where $\norm{\u}$ is sufficiently large with $ a\norm{\u}^{1+\rho} - \frac{b}{\norm{\u}} > 0 $, and the second case where $\norm{\u}$ is upper bounded with $  a\norm{\u}^{1+\rho} - \frac{b}{\norm{\u}} \leq 0$. In the first case since \eqref{comparisonineq1a1} is satisfied we can use the maximum principle from \eqref{comparisonineq1} to get $f(\u) \leq \sup_{\norm{\x}=\norm{\u} } f(\x) < \sup_{\x \in \partial \mathcal{B}_{R}(\mathbf{0})} f(\x)$, a contradiction. In the second case, since $ \norm{\u}$ is upper bounded, $f(\u)$ is bounded by continuity of $f$, then from the maximum principle \eqref{comparisonineq1} we get that there exists some $R_0$ sufficiently large such that $f(\x)$ increases as $\norm{\x}$ increases for all $\norm{\x} > R_0$. Then $f$ will eventually be larger than $f(\u)$ and hence taking $R \gg R_0$ in the definition of $S_R$ contradicts the second case.

Using the fact that $f$, on the ball $ \mathcal{B}_{R}(\mathbf{0})$ for any sufficiently large $R$, will attain maximum on its boundary $ \partial \mathcal{B}_{R}(\mathbf{0})$, we show that $C_1 =o(R^{\rho})$ for any sufficiently large $R$ where $ S_R \subset \mathcal{B}_{C_1 R}(\mathbf{0}) $. Let $\z \in S_R$, then we have $f(\z) \leq \sup_{\norm{\x} \leq R } f(\x) = \sup_{\norm{\x} = R } f(\x) < \sup_{\norm{\x} = (1+ \epsilon)R } f(\x)$ for any $\epsilon>0$ from maximum principle \eqref{comparisonineq1}. Since $\partial \mathcal{B}_{R}(\mathbf{0}) $ is compact and $\nabla f$ continuous, we get that $$\sup_{\norm{\x} = R } f(\x) - \inf_{\norm{\x} = R } f(\x) \leq \sup_{\x \in \mathcal{B}_{R}(\mathbf{0}) } \norm{\nabla f(\x)} \textbf{diam}(\mathcal{B}_{R}(\mathbf{0}) ) = C_0(R) R \ll \infty$$ where $C_0(R)$ depends on $R$ and $C_0(R) = o(R^{1+2\rho})$. Since \eqref{comparisonineq1} holds for any $R_1 > R$, integrating \eqref{comparisonineq1a1} from $R$ to $R_1 = C_1 R$ along any fixed direction yields $$ f(\x)\bigg\vert_{\norm{\x}=C_1 R } - f(\x)\bigg\vert_{\norm{\x}=R } \geq \frac{a}{2+\rho}((C_1R)^{2+\rho}-R^{2+\rho}) - b \log \frac{C_1 R}{R} ,$$ which implies 
$$ \inf_{\norm{\x} = C_1 R } f(\x)  \geq \frac{a}{2+\rho}((C_1R)^{2+\rho}-R^{2+\rho}) - b \log C_1 +\inf_{\norm{\x} = R } f(\x).$$ Then if $C_1$ satisfies the condition $$ \frac{a}{2+\rho}((C_1R)^{2+\rho}-R^{2+\rho}) - b \log C_1 = (1+\epsilon) C_0(R)R $$ for any $\epsilon>0$, we will have $ \inf_{\norm{\x} = C_1 R } f(\x)  > \sup_{\norm{\x} =  R } f(\x)  $ and hence it must be that if $\z \in S_R$ then $ \z \in \mathcal{B}_{C_1 R}(\mathbf{0})$.
If not then suppose $\z \in S_R$, $\norm{\z} >{C_1 R}$, and so by maximum principle \eqref{comparisonineq1} we have $f(\z) \geq \inf_{\norm{\x} = \norm{\z}} f(\x) >  \inf_{\norm{\x} = C_1 R } f(\x) > \sup_{\norm{\x} = R } f(\x)$, a contradiction from the definition of $S_R$. Hence, if $\z \in S_R$ then $\norm{\z} \leq {C_1R}$ which implies $ S_R \subset \mathcal{B}_{C_1 R}(\mathbf{0}) $ for some $C_1>1$. From the condition $ \frac{a}{2+\rho}((C_1R)^{2+\rho}-R^{2+\rho}) - b \log C_1 = (1+\epsilon) C_0(R)R $ and the fact that $C_0(R)=o(R^{1+2\rho})$, we get that $C_1 =o(R^{\rho})$ for any sufficiently large $R$.

We are now ready to show the boundedness of iterates when initialized in the set $S_R$ for any sufficiently large $R$. Suppose that the sequence $\{\x_k\}$ is initialized in $S_R$ with $\x_0,\x_{-1} \in S_R$ so that $\mathcal{U} = S_R \times S_R$ and this sequence $\{\x_k\}$ first exits $S_R$ at some $k=K+1$, i.e., $\x_k \in S_R$ for all $k \leq K$ and $\x_{K+1} \notin S_R$ where $\norm{\x_{K+1}} \leq C_2 R$ for some $C_2 > C_1$. Suppose $f$ is $\tilde{L}$-gradient Lipschitz continuous on $ \mathcal{B}_{C_2 R}(\mathbf{0}) \supset S_R$, i.e., $ f \in \mathcal{C}^{2,1}_{\tilde{L}}(\mathcal{B}_{C_2 R}(\mathbf{0}))$. Since $\x_{K+1} \notin S_R$, we have $\norm{\x_{K+1}}>R$ or $\x_{K+1} \notin \mathcal{B}_R(\mathbf{0})$ from the definition of $S_R$. Since $\x_{k+1} =\y_k - h \nabla f(\y_k) $ from \eqref{generalds} for any $k$, using coercivity of $f$ and $h < \frac{1}{\tilde{L}}$ we get that if $\x_{k+1} \notin S_R$ then $\y_k \notin S_R$ from the fact that $S_R$ is a sublevel set of $f$ and the descent property, i.e. $f(\x_{k+1})\leq f(\y_k)$. The descent property follows from \eqref{generalds} update and local gradient Lipschitz continuity of $f$ as follows:
\begin{align*}
    f(\x_{k+1} ) &\leq f(\y_k) + \langle \nabla f(\y_k), \x_{k+1}-\y_k\rangle + \frac{\tilde{L}}{2}\norm{\x_{k+1}-\y_k}^2 \\
     f(\x_{k+1} ) &\leq f(\y_k) - \frac{h}{2}\bigg(1 - \tilde{L}h \bigg)\norm{\nabla f(\y_k)}^2
\end{align*}
where we require $ h< \frac{1}{\tilde{L}}$. Hence, if $\x_{K+1} \notin S_R$ then that implies $\y_{K} \notin S_R$ or $\norm{\y_{K}}>R$ from the definition of $S_R$. Now, for any sufficiently large $R$ and $h < \frac{1}{\tilde{L}} $, any positive constants $a,b$ will satisfy the following condition 
\begin{align}
\bigg(\frac{h}{2}-\frac{\tilde{L} h^2}{2} \bigg)\bigg(a^2 R^{2+2\rho} - 2ab R^{\rho}  \bigg) - \frac{(2C_1 R)^2 }{h} > 0 \label{disscond1}
\end{align}
from the fact that $\rho > 0$ and $C_1=o(R^{\rho})$. Since $f$ is $\tilde{L}$-gradient Lipschitz continuous on $ \mathcal{B}_{C_2 R}(\mathbf{0})$, we can now invoke the following Lyapunov function inequality (\eqref{liapunovdecrease} from Lemma \ref{lemmalyapunov}'s proof)\footnote{Though Lemma \ref{lemmalyapunov} requires $f$ to be globally gradient Lipschitz continuous, we can always use some $\mathcal{C}^{\infty}$ smooth bump function $\phi$ that is $1$ on $ \mathcal{B}_{C_1 R}(\mathbf{0})$, $0$ outside $ \mathcal{B}_{C_2 R}(\mathbf{0})$ and $0<\phi<1$ on $  \mathcal{B}_{C_2 R}(\mathbf{0})\backslash\mathcal{B}_{C_1 R}(\mathbf{0}) $ with $\norm{\nabla\phi}$ bounded so that $f \phi$ becomes globally gradient Lipschitz continuous.} so as to establish some form of monotonic decrease and thus the boundedness of iterates:
\begin{align*}
    f(\x_k)+ \frac{\norm{\x_{k-1} -\x_{k} }^2}{2h} - f(\x_{k+1}) - \frac{\norm{\x_k -\x_{k+1} }^2}{2h} &\geq \bigg(\frac{h}{2}-\frac{\tilde{L} h^2}{2} \bigg)\norm{ \nabla f(\y_k)}^2 + \nonumber \\ & {\beta_k^2\bigg(\frac{1}{2h\beta_k^2}- \frac{1}{h}\bigg)} \norm{\x_k -\x_{k-1}}^2.
\end{align*}
 Then using the above inequality for $k=K$ and $\beta_k \leq 1$ followed by the inequality $ \norm{ \nabla f(\y_K)} \geq \bigg(a\norm{\y_K}^{1+\rho} - \frac{b}{\norm{\y_K}} \bigg)$ from the dissipative property of $f$, we get:
\begin{align}
    {f}(\x_{K}) + \frac{\norm{\x_{K-1} -\x_{K}}^2}{2h} - {f}(\x_{K+1}) - \frac{\norm{\x_K -\x_{K+1}}^2}{2h} &  \geq \bigg(\frac{h}{2}-\frac{\tilde{L} h^2}{2} \bigg)\norm{ \nabla f(\y_K)}^2 + \nonumber \\ & {\beta_K^2\bigg(\frac{1}{2h\beta_K^2}- \frac{1}{h}\bigg)} \norm{\x_K -\x_{K-1}}^2 \\
    & \hspace{-2.4cm}\geq  \bigg(\frac{h}{2}-\frac{\tilde{L} h^2}{2} \bigg)\bigg(a\norm{\y_K}^{1+\rho} - \frac{b}{\norm{\y_K}} \bigg)^2 - {\frac{1}{2h}\norm{\x_K -\x_{K-1}}^2} \\
     & \hspace{-2.4cm} \geq \bigg(\frac{h}{2}-\frac{\tilde{L} h^2}{2} \bigg)a^2\norm{\y_K}^{2+2\rho}\bigg(1 - \frac{2b}{a \norm{\y_K}^{2+\rho}} \bigg) - \frac{1}{2h}(2C_1 R)^2 \\
    & \hspace{-2.4cm} \geq \bigg(\frac{h}{2}-\frac{\tilde{L} h^2}{2} \bigg)\bigg(a^2 R^{2+2\rho} - 2ab R^{\rho}  \bigg) - \frac{1}{2h}(2C_1 R)^2, 
    \end{align}
    which after simplification gives
    \begin{align}
    {f}(\x_{K}) + \frac{(2C_1 R)^2 }{2h} - {f}(\x_{K+1}) & \geq\bigg(\frac{h}{2}-\frac{\tilde{L} h^2}{2} \bigg)\bigg(a^2 R^{2+2\rho} - 2ab R^{\rho} \bigg) - \frac{(2C_1 R)^2 }{2h} \\
    \implies {f}(\x_{K}) - {f}(\x_{K+1}) & \geq \bigg(\frac{h}{2}-\frac{\tilde{L} h^2}{2} \bigg)\bigg(a^2 R^{2+2\rho} - 2ab R^{\rho} \bigg) - \frac{(2C_1 R)^2 }{h}.
\end{align}
But $\bigg(\frac{h}{2}-\frac{\tilde{L} h^2}{2} \bigg)\bigg(a^2 R^{2+2\rho} - 2ab R^{\rho}  \bigg) - \frac{(2C_1 R)^2 }{h} > 0 $ from \eqref{disscond1} for any sufficiently large $R$ and so $ {f}(\x_{K}) > {f}(\x_{K+1}) $ by which $\x_{K+1} \in S_R$, a contradiction. Hence the sequence $\{\x_k\}$ can never escape the compact set $S_R$ for sufficiently large $R$ provided $f$ satisfies the dissipative property with positive constants $\rho, a,b$ along with the gradient growth boundedness assumption. Thus, for the initialization set $\mathcal{U} = S_R \times S_R$ we have identified the compact set $\mathcal{V} \supsetneq \mathcal{U}$ which satisfies $\mathcal{V} = \mathcal{B}_{C_1 R}(\mathbf{0})  \times \mathcal{B}_{C_1 R}(\mathbf{0})  $ from the fact that $ S_R \subset \mathcal{B}_{C_1 R}(\mathbf{0}) $. This proves our claim.

\section{Metrics for asymptotic convergence and divergence} \label{Appendix C}

\subsection{Relation between the singular values of two step Jacobian and the local convergence/ divergence rate}\label{relcontractexpand}
The metrics defined in \eqref{asymptot1} and \eqref{asymptot2} represent two-step asymptotic convergence rate to and divergence rate from the critical point $\x^*$. Let $\x^*$ be any fixed point of the map $N_k$ for all $k$ where $N_k$ are local diffeomorphisms around $\x^*$. Also, for sake of simplicity we may assume that the family of maps $\{N_k\}$ are locally diffeomorphic on a ball $\mathcal{B}_{\Delta}(\x^*)$ for some $\Delta>0$, the maps $\{DN_k\}$, $\{DN_k^{-1}\}$ are equicontinuous on the ball $\mathcal{B}_{\Delta}(\x^*)$, the sequences $\{DN_k(\x^*)\}$, $\{DN_k^{-1}(\x^*)\}$ are uniformly bounded and the variable $\delta$ that will be used in the subsequent analysis satisfies $\delta < \Delta$.\footnote{Note that the analysis can also be carried out without the uniformly diffeomorphic assumption on the maps $\{N_k\}$ but will be much more tedious. Since the goal here is to just show the relation between the singular values of two step Jacobian and the local convergence/ divergence rate, we steer away from such tedious analysis. Also, the condition of equicontinuity for the Jacobian maps is not vacuous and it holds for quadratic functions (see Lemma \ref{lemsupab}).}Then writing $\x_{k+1}$ as series expansion about $\x^*$ we get:
\begin{align}
    &\x_{k+1}  = N_k (\x_k) = N_k\circ N_{k-1}(\x_{k-1}) \\
  \implies & \x_{k+1} =  \underbrace{N_k\circ N_{k-1}(\x^*)}_{= \x^*} +  {DN_k(\x^*) DN_{k-1}(\x^*)}(\x_{k-1}-\x^*) + o(\norm{\x_{k-1}-\x^*})\\
  \implies & \norm{\x_{k+1} -\x^*}  \leq \norm{{DN_k(\x^*) DN_{k-1}(\x^*)}}_2 \norm{\x_{k-1} -\x^*}  + o(\norm{\x_{k-1}-\x^*})
  \end{align}
  \begin{align}
  \implies & \frac{\norm{\x_{k+1} -\x^*}}{\norm{\x_{k-1} -\x^*}}  \leq \sup_{\{\x_k \hspace{0.1cm}\vert\hspace{0.1cm} \norm{\x_k-\x^*}\leq \delta\}}\bigg(\norm{{DN_k(\x_k) DN_{k-1}(\x_{k-1})}}_2   + \nonumber \\ & \hspace{4cm}\norm{{DN_k(\x^*) DN_{k-1}(\x^*)} - {DN_k(\x_k) DN_{k-1}(\x_{k-1})} }_2\bigg)+ o(1), 
  \end{align}
  which implies
  \begin{align}
  & \hspace{-0.5cm} \inf_{\delta>0} \sup_{\{\x_k \hspace{0.1cm}\vert\hspace{0.1cm} \norm{\x_k-\x^*}\leq \delta\}}\frac{\norm{\x_{k+1} -\x^*}}{\norm{\x_{k-1} -\x^*}}  \leq \inf_{\delta>0} \sup_{\{\x_k \hspace{0.1cm}\vert\hspace{0.1cm} \norm{\x_k-\x^*}\leq \delta\}}\norm{{DN_k(\x_k) DN_{k-1}(\x_{k-1})}}_2  \nonumber \\ & \hspace{5cm} + \underbrace{\inf_{\delta>0} \sup_{\{\x_k \hspace{0.1cm}\vert\hspace{0.1cm} \norm{\x_k-\x^*}\leq \delta\}}o(1)}_{=0 \text{ by local equicontinuity of } \{DN_k\} }  \nonumber \\ & \hspace{3cm}+\underbrace{ \inf_{\delta>0} \sup_{\{\x_k \hspace{0.1cm}\vert\hspace{0.1cm} \norm{\x_k-\x^*}\leq \delta\}}\norm{{DN_k(\x^*) DN_{k-1}(\x^*)} - {DN_k(\x_k) DN_{k-1}(\x_{k-1})} }_2}_{=0 \text{ by local equicontinuity of } \{DN_k\} \text{ and uniform boundedness of } \{DN_k(\x^*)\} } \label{tempeq1}\\ 
 \implies & \inf_{\delta>0} \sup_{\{\x_k \hspace{0.1cm}\vert\hspace{0.1cm} \norm{\x_k-\x^*}\leq \delta\}}\frac{\norm{\x_{k+1} -\x^*}}{\norm{\x_{k-1} -\x^*}} \leq \inf_{\delta>0} \sup_{\{\x_k \hspace{0.1cm}\vert\hspace{0.1cm} \norm{\x_k-\x^*}\leq \delta\}}\norm{\frac{\partial \x_{k+1}}{\partial \x_{k-1} }}_2 \label{deffasymptot1}
\end{align}
where in the second last step we used the fact that $\norm{\x_{k}-\x^*}\leq \delta \iff \norm{\x_{k-1}-\x^*}\leq C(k)\delta $ for some finite positive $C(k)$. {To see this notice that using the fact that $\x^*$ is a fixed point of $N_k$ for all $k$ we can write $ \norm{\x_{k}-\x^*} = \norm{N_{k-1}(\x_{k-1})-N_{k-1}(\x^*)} \leq  C_1(k)\norm{\x_{k-1}-\x^*} $ where $C_1(k)>0$ is the local Lipschitz constant for $N_{k-1}$ (the map $N_{k-1}$ will be locally Lipschitz continuous around $\x^*$ since $N_k$ is a local diffeomorphism for all $k$). Similarly in the other direction we can write $ \norm{\x_{k-1}-\x^*} = \norm{N^{-1}_{k-1}(\x_{k})-N^{-1}_{k-1}(\x^*)} \leq  C_2(k)\norm{\x_{k}-\x^*} $ where $C_2(k)< \infty$ is the local Lipschitz constant for $N^{-1}_{k-1}$. Then for $C(k) =\min \{C_2(k), \frac{1}{C_1(k)}\}$ we will have $\norm{\x_{k}-\x^*}\leq \delta \iff \norm{\x_{k-1}-\x^*}\leq C(k)\delta $} and therefore\footnote{Since the $o(1)$ term in \eqref{tempeq1} is with respect to $\norm{\x_{k-1}-\x^*}$, the $\inf \sup $ needs to be evaluated with respect to $ \x_{k-1}$ and not $\x_k$.} $$\inf_{\delta>0} \sup_{\{\x_k \hspace{0.1cm}\vert\hspace{0.1cm} \norm{\x_k-\x^*}\leq \delta\}}o(1) = \inf_{\delta>0} \sup_{ \{\x_{k-1} \hspace{0.1cm}\vert\hspace{0.1cm} \norm{\x_{k-1}-\x^*}\leq C(k)\delta\}}o(1) = 0, $$ provided $\lim\sup_k \abs{C(k)} < \infty$. Now the constant $C(k)$ in the condition $$\norm{\x_{k}-\x^*}\leq \delta \iff \norm{\x_{k-1}-\x^*}\leq C(k)\delta, $$ will be upper bounded by some constant $C$ which is independent of $k$ from \eqref{checkineq1a} in the proof of Lemma \ref{lemmasupport}. Thus, $\lim\sup_k \abs{C(k)} < \infty$ can be easily satisfied {and we get $\norm{\x_{k}-\x^*}\leq \delta \iff \norm{\x_{k-1}-\x^*}\leq C\delta $ where $C$ is independent of $k$}. 

{Next, to see that the term \( \inf_{\delta>0} \sup_{\{\x_k \,|\, \norm{\x_k - \x^*} \leq \delta\}} \norm{DN_k(\x^*) DN_{k-1}(\x^*) - DN_k(\x_k) DN_{k-1}(\x_{k-1})}_2 \) in \eqref{tempeq1} is indeed zero, note that it can be simplified by applying the triangle inequality and bounding the term \( \norm{DN_k(\x^*) DN_{k-1}(\x^*) - DN_k(\x_k) DN_{k-1}(\x_{k-1})}_2 \) as follows:
\begin{align}
    &\norm{DN_k(\x^*) DN_{k-1}(\x^*) - DN_k(\x_k) DN_{k-1}(\x_{k-1})}_2 \nonumber \\
        &\quad \leq \norm{DN_k(\x^*) DN_{k-1}(\x^*) - DN_k(\x_k) DN_{k-1}(\x^*)}_2 + \norm{DN_k(\x_k) DN_{k-1}(\x^*) - DN_k(\x_k) DN_{k-1}(\x_{k-1})}_2 \\
        &\quad\leq \norm{DN_{k-1}(\x^*)}_2 \norm{DN_k(\x^*) - DN_k(\x_k)}_2 + \norm{DN_k(\x_k)}_2 \norm{DN_{k-1}(\x^*) - DN_{k-1}(\x_{k-1})}_2 \\
        &\quad\leq \norm{DN_{k-1}(\x^*)}_2 \norm{DN_k(\x^*) - DN_k(\x_k)}_2 + \norm{DN_k(\x_k) - DN_k(\x^*)}_2 \norm{DN_{k-1}(\x^*) - DN_{k-1}(\x_{k-1})}_2 \nonumber \\
        &\qquad\qquad + \norm{DN_k(\x^*)}_2 \norm{DN_{k-1}(\x^*) - DN_{k-1}(\x_{k-1})}_2 \\
        &\quad\leq \norm{DN_{k-1}(\x^*)}_2 \omega(\delta) + \omega(\delta)\omega(C\delta) + \norm{DN_k(\x^*)}_2 \omega(C\delta), \label{revjac1a}
\end{align}
where in the last step, for \(\|\mathbf{x}_k - \mathbf{x}^*\| \leq \delta\) and \(\|\mathbf{x}_{k-1} - \mathbf{x}^*\| \leq C\delta\), we have used the bounds 
\[
\|DN_k(\mathbf{x}^*) - DN_k(\mathbf{x}_k)\|_2 \leq \omega(\delta), \quad \|DN_{k-1}(\mathbf{x}^*) - DN_{k-1}(\mathbf{x}_{k-1})\|_2 \leq \omega(C\delta)
\]
for some uniform modulus of continuity \(\omega : \mathbb{R}_+ \to \mathbb{R}_+\), where \(\omega(0) = 0\) and \(\omega\) is continuous, since the maps \(\{DN_k\}_k\) are uniformly equicontinuous. Then taking \(\inf_{\delta > 0} \sup_{\{\mathbf{x}_k \mid \|\mathbf{x}_k - \mathbf{x}^*\| \leq \delta\}}\) on both sides of \eqref{revjac1a} and using the uniform bound \(\|DN_k(\mathbf{x}^*)\|_2 \leq M_1\) for all \(k\), we obtain:
\begin{align}
    &\inf_{\delta > 0} \sup_{\{\mathbf{x}_k \mid \|\mathbf{x}_k - \mathbf{x}^*\| \leq \delta\}} \|DN_k(\mathbf{x}^*) DN_{k-1}(\mathbf{x}^*) - DN_k(\mathbf{x}_k) DN_{k-1}(\mathbf{x}_{k-1})\|_2 \nonumber \\
        &\qquad \leq \inf_{\delta > 0} \sup_{\{\mathbf{x}_k \mid \|\mathbf{x}_k - \mathbf{x}^*\| \leq \delta\}} \norm{DN_{k-1}(\x^*)}_2 \omega(\delta) + \inf_{\delta > 0} \sup_{\{\mathbf{x}_k \mid \|\mathbf{x}_k - \mathbf{x}^*\| \leq \delta\}} \omega(\delta)\omega(C\delta) \nonumber \\
            &\qquad\qquad\qquad + \inf_{\delta > 0} \sup_{\{\mathbf{x}_k \mid \|\mathbf{x}_k - \mathbf{x}^*\| \leq \delta\}} \norm{DN_k(\x^*)}_2 \omega(C\delta) \\
        &\qquad \leq M_1 \inf_{\delta > 0} \sup_{\{\mathbf{x}_k \mid \|\mathbf{x}_k - \mathbf{x}^*\| \leq \delta\}} \omega(\delta) + \inf_{\delta > 0} \sup_{\{\mathbf{x}_k \mid \|\mathbf{x}_k - \mathbf{x}^*\| \leq \delta\}} \omega(\delta) \omega(C\delta) + M_1 \inf_{\delta > 0} \sup_{\{\mathbf{x}_k \mid \|\mathbf{x}_k - \mathbf{x}^*\| \leq \delta\}} \omega(C\delta) \\
        &\qquad = M_1 \lim_{\delta \to 0} \omega(\delta) + \lim_{\delta \to 0} \omega(\delta) \omega(C\delta) + M_1 \lim_{\delta \to 0} \omega(C\delta) = 0,
\end{align}
where the last step follows from the continuity of \(\omega\) at \(\delta = 0\).
}

Hence from \eqref{deffasymptot1}, $\inf_{\delta>0} \sup_{\{\x_k \hspace{0.1cm}\vert\hspace{0.1cm} \norm{\x_k-\x^*}\leq \delta\}}\norm{\frac{\partial \x_{k+1}}{\partial \x_{k-1} }}_2 $ is the largest two-step asymptotic divergence rate from $\x^*$. 

Similarly, repeating the entire argument for the smallest singular value $ \norm{ \bigg(\frac{\partial \x_{k+1}}{\partial \x_{k-1} }\bigg)^{-1}}_{2}^{-1}$ we can show that:
\begin{align}
   & \sup_{\delta>0} \inf_{\{\x_k \hspace{0.1cm}\vert\hspace{0.1cm} \norm{\x_k-\x^*}\leq \delta\}}\norm{ \bigg(\frac{\partial \x_{k+1}}{\partial \x_{k-1} }\bigg)^{-1}}_{2}^{-1}   \leq \sup_{\delta>0} \inf_{\{\x_k \hspace{0.1cm}\vert\hspace{0.1cm} \norm{\x_k-\x^*}\leq \delta\}} \frac{\norm{\x_{k+1}-\x^*}}{\norm{\x_{k-1}-\x^*}} . \label{deffasymptot2}
\end{align}
Then from \eqref{deffasymptot2}, $\sup_{\delta>0} \inf_{\{\x_k \hspace{0.1cm}\vert\hspace{0.1cm} \norm{\x_k-\x^*}\leq \delta\}}\norm{ \bigg(\frac{\partial \x_{k+1}}{\partial \x_{k-1} }\bigg)^{-1}}_{2}^{-1}  $ is the smallest two step asymptotic convergence rate to $\x^*$.
\\

\subsection{Lemma \ref{lemmaSdelta}}
\begin{proof}
Since the iteration index $k$ is agnostic of the map $N_k$ from Theorem \ref{diffeomorphthm} or the map $P_k$ defined in Lemma \ref{lemma_pk}, it suffices to prove the claim using the map $P_k \equiv [N_k; N_{k-1}]$. Let $\w_k = [\x_k; \x_{k-1}]$ for any $k$ and $\w^* = [\x^*; \x^*]$.
 In order to show that the set $  S_{\delta}$ is $\mathbb{Z}^*$, i.e., the set of all non-negative integers first observe that $ \norm{D P_k(\w_k)}_2 $ is always bounded from above. In particular, for some positive constant $C$, using Lemma \ref{lemma_pk} we have $$ \norm{D P_k(\w^*)}_2  \leq C n\beta_k\norm{\mathbf{I} - h \nabla^2 f(\x^*)}_2 $$ which is bounded since $\beta_k \to \beta$ and $f \in \mathcal{C}^{2,1}_L(\mathbb{R}^n)$. Hence, $ \norm{D P_k(\w^*)}_2  $ is bounded from above for all $k$ and is greater than $1$. Now using the uniform bound $\sup_k\norm{D P_k(\w^*)}_2 \leq D < \infty$ for some $D>1$ and writing $\w_{k+1}$ as series expansion about $\w^*$ we get:
 \begin{align}
        \w_{k+1} &=  P_k(\w_k)  \\
  \implies  \w_{k+1}&= \underbrace{P_k(\w^*)}_{=\w^*}  +  DP_k(\w^*)(\w_{k}-\w^*) + o(\norm{\w_{k}-\w^*})\\
  \implies \norm{\w_{k+1} -\w^*} & \leq \underbrace{\norm{ DP_k(\w^*)}_2}_{\leq \sup_k\norm{D P_k(\w^*)}_2} \norm{\w_{k} -\w^*}  + o(\norm{\w_{k}-\w^*})\\
      \implies \norm{\w_{k+1}-\w^*} & \leq  (D + o(1))\norm{\w_{k}-\w^*} \\
     \implies \norm{\w_{K}-\w^*} & \leq  (D + o(1))^K\norm{\w_{0}-\w^*} \\
     \implies K & \geq \log_{(D+ o(1))} \bigg( \frac{\norm{\w_{K}-\w^*}}{\norm{\w_{0}-\w^*}} \bigg)\geq \log_{2D} \bigg( \frac{\delta}{\epsilon} \bigg) \label{suptrajec}
 \end{align}
 where we substituted\footnote{Here we assume that $\{\w_k\}$ is initialized inside $\mathcal{B}_{\delta}(\w^*)$ with $ \norm{\w_{0}-\w^*} = \epsilon < \delta$ and $\{\w_k\}$ exits $\mathcal{B}_{\delta}(\w^*)$ at $k=K$.} $\norm{\w_{K}-\w^*} \geq \delta $, $ \norm{\w_{0}-\w^*} = \epsilon$ and $\epsilon < \delta$, $D > o(1)$ for sufficiently small $\delta$. Note that the $o(1)$ term will be upper bounded by $\norm{D P_k(\w^*) - DP_k(\z)}_2  $ for some $\z \in \mathcal{B}_{\norm{\w_k - \w^*}}(\w^*)$. Then, by local equicontinuity of $\{DP_k\}$ from $DP_k \darrow DP$, we can choose $\delta$ small enough for which $D  > o(1)$ for any $k$. From the above relation \eqref{suptrajec} and the definition of the set $ {S_{\delta}}$ it is clear that $ \bigg\{K \bigg \vert 0 \leq K  < \log_{2D} \bigg( \frac{\delta}{\epsilon} \bigg) \bigg\} \subset  S_{\delta}$. In particular we have the following containment
 \begin{align}
     \bigcup_{0< \epsilon< \delta}\bigg\{K \bigg \vert 0 \leq K  <  \log_{2D} \bigg( \frac{\delta}{\epsilon} \bigg) \bigg\} \subset S_{\delta} \label{hittingtime1}
 \end{align}
   using which we obtain:
 \begin{align}
  \sup \bigcup_{0< \epsilon< \delta}\bigg\{K \bigg \vert 0 \leq K  <  \log_{2D} \bigg( \frac{\delta}{\epsilon} \bigg) \bigg\} &\leq  \sup \{k \hspace{0.1cm} \vert \hspace{0.1cm} k \in S_{\delta}\}  \\
  \implies \infty = \sup_{0< \epsilon< \delta} \bigg\{K \bigg \vert 0 \leq K  < \log_{2D} \bigg( \frac{\delta}{\epsilon} \bigg) \bigg\} &\leq  \sup \{k \hspace{0.1cm} \vert \hspace{0.1cm} k \in S_{\delta}\} \\
  \implies\sup \{k \hspace{0.1cm} \vert \hspace{0.1cm} k \in S_{\delta}\} & = \infty. \label{sdeltaproof}
 \end{align}
  Since every finite non-negative integer will be contained in the set $S_{\delta}$ from \eqref{hittingtime1} and $\sup \{k \hspace{0.1cm} \vert \hspace{0.1cm} k \in S_{\delta}\}  = \infty $ from \eqref{sdeltaproof}, we conclude that the set $ S_{\delta}$ is $\mathbb{Z}^*$. In the case of $\beta_k=0$ for all $k$, i.e. the gradient descent method, the above conclusion holds trivially, which completes the proof.
\end{proof}  

Before presenting the proofs of Theorems \ref{metricconvergethm2}, \ref{metricedivergethm} we provide some supporting lemmas.

 \begin{lemm}\label{lemmasupport}
 Suppose $\{\x_k\} \in J_{\tau}(f)$ be any trajectory generated by the sequence $\x_{k+1} = N_k(\x_k)$ for $h< \frac{1}{L}$ and $\x^*$ is a critical point of  locally Hessian Lipschitz function $f(\cdot) \in \mathcal{C}^{2,1}_L(\mathbb{R}^n)$. Then the trajectory $\{\x_k\}$ can indeed approach $\x^*$ only when $k \to \infty$ $\probP_1$-almost surely and not in some finite $k$. As a result we have $ \infty \in \overline{G_{\delta}}$, i.e., $k = \infty$ belongs to the closure of set $ G_{\delta}$ $\probP_1$-almost surely where $G_{\delta} = \bigg\{k \hspace{0.1cm}\bigg\vert\hspace{0.1cm} \x_k \in \mathcal{B}_{\delta}(\x^*); \hspace{0.1cm}\{\x_k\}_{k=0}^{\infty} \in J_{\tau}(f); \hspace{0.1cm} \x_k \to \x^* \bigg\} $.
 \end{lemm}
 \begin{proof}
 Since the iteration index $k$ is agnostic of the map $N_k$ from Theorem \ref{diffeomorphthm} or the map $P_k$ defined in Lemma \ref{lemma_pk}, it suffices to prove the claim using map $P_k \equiv [N_k; N_{k-1}]$. Let $\w_k = [\x_k; \x_{k-1}]$ for any $k$ and $\w^* = [\x^*; \x^*]$. Using the facts that the Jacobian map $D P_k(\cdot)$ for all $k \geq 0$ as well as the asymptotic Jacobian map $ D P(\cdot)$ remain full rank on compact sets $\probP_1$-almost surely from Corollary \ref{corsup2} which implies $ \norm{[D P_k(\w^*)]^{-1}}^{-1}_2>0$ for all $k$ and $\norm{[D P(\w^*)]^{-1}}^{-1}_2 > 0$ $\probP_1$-almost surely, $DP_k \darrow DP$ on compact sets from Lemma \ref{lemma_pk} and so $ \norm{[D P_k(\w^*)]^{-1}}^{-1}_2 \to \norm{[D P(\w^*)]^{-1}}^{-1}_2$ as $k \to \infty$, then putting everything together we have $ \inf_k\norm{[D P_k(\w^*)]^{-1}}^{-1}_2 \geq C > 0$  $\probP_1$-almost surely. Since $\w^*$ is a fixed point of the map $P_k $ for all $k$, writing $\w_{k+1}$ as series expansion about $\w^*$ we get:
 \begin{align}
     \w_{k+1} &=  P_k(\w_k)  \\
  \implies  \w_{k+1}&= \underbrace{P_k(\w^*)}_{=\w^*}  +  DP_k(\w^*)(\w_{k}-\w^*) + o(\norm{\w_{k}-\w^*})\\
\implies [ DP_k(\w^*)]^{-1}(\w_{k+1} - \w^*) & = (\w_{k} - \w^*) + o(\norm{\w_{k}-\w^*}) \\
  \implies \norm{\w_{k} -\w^*}  & \leq \norm{ [ DP_k(\w^*)]^{-1}}_{2}\norm{\w_{k+1}-\w^*} + o(\norm{\w_{k}-\w^*}) \\ 
   \implies  \underbrace{\norm{ [ DP_k(\w^*)]^{-1}}^{-1}_{2}}_{\geq\inf_{k}\norm{[D P_k(\w^*)]^{-1}}^{-1}_2  }\norm{\w_{k} -\w^*}  & \leq \norm{\w_{k+1}-\w^*} + o(\norm{\w_{k}-\w^*}) \\ 
     \implies (C - o(1))\norm{\w_{k}-\w^*} & \leq \norm{\w_{k+1}-\w^*} \label{checkineq1a}\\ 
     \implies (C - o(1))^K \norm{\w_{0}-\w^*} & \leq \norm{\w_{K}-\w^*}
 \end{align}
 $\probP_1$-almost surely and hence $\w_k$ cannot converge to $\w^*$ in any finite $k$ $\probP_1$-almost surely since $C  > o(1)$ (choosing $\delta$ sufficiently small in the definition of $G_{\delta}$ will make sure that $C  > o(1)$ holds). Note that the $o(1)$ term will be upper bounded by $\norm{D P_k(\w^*) - DP_k(\z)}_2 $ for some $\z \in \mathcal{B}_{\norm{\w_k -\w^*}}(\w^*)$. Then by the local equicontinuity of $\{DP_k\}$ from $DP_k \darrow DP$, we can choose $\delta$ small enough for which $C  > o(1)$ for any $k$. Therefore from the definition of $G_{\delta}$ we get $ \infty \in \overline{G_{\delta}}$  $\probP_1$-almost surely where $\overline{G_{\delta}} $ is the closure of set ${G_{\delta}}$. In the case of $\beta_k=0$ for all $k$, i.e. the gradient descent method, the above conclusion holds trivially and without the need of $\probP_1$-almost surely condition. This completes the proof.
  \end{proof}

\begin{lemm}\label{lemsupab}
    Let $g \in \mathcal{C}^{\omega}_{\mu, L}(\mathbb{R}^n)$ and suppose the sequence of maps $\{N_{k;g}\}$ from Theorem \ref{diffeomorphthm} for the function $g$, that correspond to \eqref{generalds} with any initialization $\x_0=\x_{-1}$, are $\probP_1$-a.s. diffeomorphisms on the ball $\mathcal{B}_{\Delta}(\x^*)$ for some $\Delta > 0$ where $\x^*$ is a critical point of $g$. If $g$ is a quadratic function then $\Delta = \infty$ and the sequence of maps $\{DN_{k;g}\}$, $\{DN_{k;g}^{-1}\}$ are $\probP_1$-a.s. constant maps and hence $\probP_1$-a.s. equicontinuous. Next, let $f \in \mathcal{C}^{\omega}_{\mu, L}(\mathbb{R}^n)$ be a quadratic function with critical point $\x^*$ and suppose $ \nabla^2 f(\x^*) = \nabla^2 g(\x^*)$ where $g \in \mathcal{C}^{\omega}_{\mu, L}(\mathbb{R}^n)$ is not a quadratic. Then if the maps $\{DN_{k;g}\}$ corresponding to the function $g$ are equicontinuous on $ \mathcal{B}_{\Delta}(\x^*)$ $\probP_1$-a.s. and satisfy the growth condition for any $\y \in \mathcal{B}_{\Delta}(\x^*)$:
   $$ \lim_{\norm{\y -\x^*} \to 0 }\sup_{k \geq 0} \norm{DN_{k;g}(\x^*) - DN_{k;g}(\y)}_2 \norm{DN_{k;g}(\x^*)}_2 = 0 \quad \probP_1 \text{ a.s.},$$ we have $ \mathcal{M}^{\star}(g) =  \mathcal{M}^{\star}(f)$ $\probP_1$-a.s.. Similarly, if the maps $\{DN_{k;g}^{-1}\}$ are equicontinuous on $ \mathcal{B}_{\Delta}(\x^*)$ $\probP_1$-a.s. and satisfy the growth condition for any $\y \in \mathcal{B}_{\Delta}(\x^*)$:
   $$ \lim_{\norm{\y -\x^*} \to 0 }\sup_{k \geq 0} \norm{[DN_{k;g}(\x^*)]^{-1} - [DN_{k;g}(\y)]^{-1}}_2 \norm{[DN_{k;g}(\x^*)]^{-1}}_2 = 0 \quad \probP_1 \text{ a.s.},$$  then $ \mathcal{M}_{\star}(g) =  \mathcal{M}_{\star}(f)$ $\probP_1$-a.s..
\end{lemm}
\begin{proof}
    Observe that in the metric \eqref{asymptot2}, $f$ can be any function in the class $ \mathcal{C}^{\omega}_{\mu, L}(\mathbb{R}^n)$. Taking $f$ to be a quadratic function (constant hessian) and iterating \eqref{generalds} for any initialization $\x_0=\x_{-1}$, the relation on $DN_k$ from Theorem \ref{diffeomorphthm} for any $k$ is as follows:
\begin{align}
   D N_k(\x_k) & = \bigg(\mathbf{I} -   h \nabla^2 f( \x^*)\bigg)\bigg((1+ \beta_k)\mathbf{I} -  \beta_k[D N_{k-1}(\x_{k-1})]^{-1}\bigg),  \label{Asymptot3abc1}
\end{align}
$\probP_1$-a.s.. Using the fact that the initialization is $\x_0 =\x_{-1}$, we can set $N_{-1} \equiv \mathrm{id}$ which gives $D N_{-1}(\x_{-1})  = \mathbf{I}$. Now it can be readily checked from \eqref{Asymptot3abc1} that for quadratic $f$, the eigenbasis of the matrix $ D N_k(\x_{k}) $ is equal to the eigenbasis of $ \bigg(\mathbf{I} -   h \nabla^2 f( \x^*)\bigg)$ for all $k \geq 0$ and the matrix $ D N_k(\x_{k}) $ has real eigenvalues for all $k \geq 0$. Hence, for quadratic $f$, from \eqref{Asymptot3abc1} the maps $\{N_k\}$ are $\probP_1$-a.s. diffeomorphisms everywhere and the sequence of maps $\{DN_k\}$, $\{DN_k^{-1}\}$ are also $\probP_1$-a.s. equicontinuous everywhere due to the fact that $ DN_k$ is a $\probP_1$-a.s. constant function for any possible initialization $\x_0$ and so is $ DN_k^{-1}$ by inverse function theorem. This proves the first part.

From Corollary \ref{corsup1} we have that for any $g \in \mathcal{C}^{2,1}_{ L}(\mathbb{R}^n)$ with the critical point $\x^*$, the map $DN_{k;g}$ \footnote{To avoid confusion between the maps used in the recursions \eqref{Asymptot3abc1} and \eqref{Asymptot3abc2}, instead of using the notation $DN_k$ again we use the notation $DN_{k;g}$ when we have function $g$.} satisfies the following relation at $\x^*$ for any $k$ $\probP_1$-a.s.:
\begin{align}
   D N_{k;g}(\x^*) & = \bigg(\mathbf{I} -   h \nabla^2 g( \x^*)\bigg)\bigg((1+ \beta_k)\mathbf{I} -  \beta_k[DN_{k-1;g}(\x^*)]^{-1}\bigg).  \label{Asymptot3abc2}
\end{align}
Now, observe that in the recursions \eqref{Asymptot3abc1}, \eqref{Asymptot3abc2} we will have $D N_{-1}(\x_{-1}) = D N_{-1;g}(\x^*) = \mathbf{I}$. Then if the function $g$ satisfies $\nabla^2 g( \x^*) =\nabla^2 f( \x^*)  $ we get that the sequence $\{DN_k(\x_k)\}$ generated from \eqref{Asymptot3abc1} and the sequence $\{DN_{k;g}(\x^*)\}$ generated from \eqref{Asymptot3abc2} are identical. Then if $g \in \mathcal{C}^{\omega}_{ \mu, L}(\mathbb{R}^n) \subset \mathcal{C}^{2,1}_{ L}(\mathbb{R}^n)$, from triangle inequality we have for any $k$ $\probP_1$-a.s. that:
\begin{align}
    \norm{D N_{k;g}(\x_{k;g}) D N_{k-1;g}(\x_{k-1;g}) }_2 &\geq \norm{D N_{k;g}(\x^*) D N_{k-1;g}(\x^*)}_2   \nonumber \\ & -\norm{D N_{k;g}(\x_{k;g})-D N_{k;g}(\x^*)}_2\underbrace{\norm{ D N_{k-1;g}(\x_{k-1;g})}_2}_{=T_1} \nonumber \\ & - \norm{D N_{k-1;g}(\x_{k-1;g})-D N_{k-1;g}(\x^*)}_2\norm{ D N_{k;g}(\x^*)}_2 \label{supineq1a} \\
    \norm{D N_{k;g}(\x^*) D N_{k-1;g}(\x^*)}_2 &\geq  \norm{D N_{k;g}(\x_{k;g}) D N_{k-1;g}(\x_{k-1;g}) }_2   \nonumber \\ & -\norm{D N_{k;g}(\x_{k;g})-D N_{k;g}(\x^*)}_2\underbrace{\norm{ D N_{k-1;g}(\x_{k-1;g})}_2}_{=T_1} \nonumber \\ & - \norm{D N_{k-1;g}(\x_{k-1;g})-D N_{k-1;g}(\x^*)}_2\norm{ D N_{k;g}(\x^*)}_2 \label{supineq1a1} 
    \end{align}
    and also 
    \begin{align}
    & \norm{[D N_{k;g}(\x_{k;g}) D N_{k-1;g}(\x_{k-1;g})]^{-1} }_2 \geq \norm{[D N_{k;g}(\x^*) D N_{k-1;g}(\x^*)]^{-1}}_2  \nonumber \\ & \hspace{3cm}-\norm{[D N_{k;g}(\x_{k;g})]^{-1}-[D N_{k;g}(\x^*)]^{-1}}_2\underbrace{\norm{ [D N_{k-1;g}(\x_{k-1;g})]^{-1}}_2}_{=T_2} \nonumber \\ & \hspace{3cm} - \norm{[D N_{k-1;g}(\x_{k-1;g})]^{-1}-[D N_{k-1;g}(\x^*)]^{-1}}_2\norm{ [D N_{k;g}(\x^*)]^{-1}}_2 .\label{supineq1b} \\
  &\norm{[D N_{k;g}(\x^*) D N_{k-1;g}(\x^*)]^{-1}}_2     \geq \norm{[D N_{k;g}(\x_{k;g}) D N_{k-1;g}(\x_{k-1;g})]^{-1} }_2 \nonumber \\ & \hspace{3cm}-\norm{[D N_{k;g}(\x_{k;g})]^{-1}-[D N_{k;g}(\x^*)]^{-1}}_2\underbrace{\norm{ [D N_{k-1;g}(\x_{k-1;g})]^{-1}}_2}_{=T_2} \nonumber \\ & \hspace{3cm} - \norm{[D N_{k-1;g}(\x_{k-1;g})]^{-1}-[D N_{k-1;g}(\x^*)]^{-1}}_2\norm{ [D N_{k;g}(\x^*)]^{-1}}_2 .\label{supineq1b1}
\end{align}
\footnote{The terms $T_1,T_2$ can be simplified further via triangle inequality by introducing the terms $D N_{k-1;g}(\x^*), [D N_{k-1;g}(\x^*)]^{-1}$ respectively. Then the growth condition and local equicontinuity can be applied.}Then taking $\lim_{\delta \to 0}  \sup_{\substack{\x_{k;g} \in \mathcal{B}_{\delta}(\x^*) \\  {\{\x_{k;g}\}_{k=0}^{\infty} \in J_{\tau}(g)}}} $ on both sides of the inequalities \eqref{supineq1a}, \eqref{supineq1a1} for $\delta < \Delta$ we get:
\begin{align}
   & \hspace{-0.5cm}\lim_{\delta \to 0}  \sup_{\substack{\x_{k;g} \in \mathcal{B}_{\delta}(\x^*) \\  {\{\x_{k;g}\}_{k=0}^{\infty} \in J_{\tau}(g)}}} \norm{D N_{k;g}(\x_{k;g}) D N_{k-1;g}(\x_{k-1;g}) }_2 = \lim_{\delta \to 0} \sup_{\substack{\x_{k;g} \in \mathcal{B}_{\delta}(\x^*) \\  {\{\x_{k;g}\}_{k=0}^{\infty} \in J_{\tau}(g)}}}\norm{D N_{k;g}(\x^*) D N_{k-1;g}(\x^*)}_2  , \label{supineq2a}
\end{align}
$\probP_1$-a.s.\footnote{{We are allowed to take limsup in \eqref{supineq1a}, \eqref{supineq1a1} because of the $\probP_1$-a.s. local equicontinuity of the sequence of maps $\{DN_k\}$, the $\probP_1$-a.s. diffeomorphism of the sequence of maps $\{N_k\}$ on $\Delta$ ball around $\x^*$ and the fact that countable intersection of almost sure events is also almost sure event.}} where in the residual terms of \eqref{supineq1a}, \eqref{supineq1a1}, using the definition of the set $S_{\delta}(g) =\mathbb{Z}^*$ for the function $g$ from Lemma \ref{lemmaSdelta}, we substituted $\lim_{\delta \to 0}  \sup_{\substack{\x_{k;g} \in \mathcal{B}_{\delta}(\x^*) \\  {\{\x_{k;g}\}_{k=0}^{\infty} \in J_{\tau}(g)}}}  $ with $ \lim_{\delta \to 0} \sup_{k \in S_{\delta}(g) =\mathbb{Z}^*}$  to get
$$ \lim_{\delta \to 0} \sup_{\substack{\x_{k;g} \in \mathcal{B}_{\delta}(\x^*) \\  {\{\x_{k;g}\}_{k=0}^{\infty} \in J_{\tau}(g)}}}\norm{D N_{k;g}(\x_{k;g})-D N_{k;g}(\x^*)}_2\norm{ D N_{k-1;g}(\x_{k-1;g})}_2 =0,  \hspace{0.1cm}\probP_1-\text{a.s.} $$
$$ \lim_{\delta \to 0}  \sup_{\substack{\x_{k;g} \in \mathcal{B}_{\delta}(\x^*) \\  {\{\x_{k;g}\}_{k=0}^{\infty} \in J_{\tau}(g)}}} \norm{D N_{k-1;g}(\x_{k-1;g})-D N_{k-1;g}(\x^*)}_2\norm{ D N_{k;g}(\x^*)}_2=0, \hspace{0.1cm}\probP_1-\text{a.s.}$$ 
by the $\probP_1$-a.s. local equicontinuity of $\{DN_k\}$ maps, the growth condition on $\norm{ D N_{k;g}(\x^*)}_2 $ and the fact that if $ \x_{k;g} \in \mathcal{B}_{\delta}(\x^*)$ then $\x_{k-1;g} \in \mathcal{B}_{C\delta}(\x^*)$ for some constant $C$ independent of $k$. The constant $C$ is independent of $k$ from \eqref{checkineq1a} in the proof of Lemma \ref{lemmasupport}. Using the definition of the set $S_{\delta}(g) =\mathbb{Z}^*$ for the function $g$ from Lemma \ref{lemmaSdelta} followed by the equivalence of sequences $\{DN_k(\x_k)\}$ generated from \eqref{Asymptot3abc1} and $\{DN_{k;g}(\x^*)\}$ generated from \eqref{Asymptot3abc2}, the equality \eqref{supineq2a} can be further simplified as:
\begin{align}
    &\hspace{-0.5cm}\lim_{\delta \to 0}  \sup_{\substack{\x_{k;g} \in \mathcal{B}_{\delta}(\x^*) \\  {\{\x_{k;g}\}_{k=0}^{\infty} \in J_{\tau}(g)}}} \norm{D N_{k;g}(\x_{k;g}) D N_{k-1;g}(\x_{k-1;g}) }_2 = \lim_{\delta \to 0} \sup_{\substack{\x_{k;g} \in \mathcal{B}_{\delta}(\x^*) \\  {\{\x_{k;g}\}_{k=0}^{\infty} \in J_{\tau}(g)}}}\norm{D N_{k;g}(\x^*) D N_{k-1;g}(\x^*)}_2 \\
    & \hspace{5cm}= \lim_{\delta \to 0} \sup_{k \in S_{\delta}(g) =\mathbb{Z}^*}\norm{D N_{k;g}(\x^*) D N_{k-1;g}(\x^*)}_2 \\
    & \hspace{5cm}= \lim_{\delta \to 0} \sup_{k \in S_{\delta}(f) =\mathbb{Z}^*}\norm{D N_{k}(\x_k) D N_{k-1}(\x_{k-1})}_2 \\
     & \hspace{5cm}= \lim_{\delta \to 0} \sup_{\substack{\x_{k} \in \mathcal{B}_{\delta}(\x^*) \\  {\{\x_{k}\}_{k=0}^{\infty} \in J_{\tau}(f)}}}\norm{D N_{k}(\x_k) D N_{k-1}(\x_{k-1})}_2 = \mathcal{M}^{\star}(f)
\end{align}
$\probP_1$-a.s., where in the second last step we used the fact that the set $S_{\delta} =\mathbb{Z}^*$ for both functions $g$ and $f$ from Lemma \ref{lemmaSdelta}.

Similarly, rearranging \eqref{supineq1b}, \eqref{supineq1b1}, taking the inverse on both sides of \eqref{supineq1b}, \eqref{supineq1b1} followed by taking $\lim_{\delta \to 0}  \inf_{\substack{\x_{k;g} \in \mathcal{B}_{\delta}(\x^*) \\  {\{\x_{k;g}\}_{k=0}^{\infty} \in J_{\tau}(g)}}} $  for $\delta < \Delta$ we get:
\begin{align}
     \lim_{\delta \to 0}  \inf_{\substack{\x_{k;g} \in \mathcal{B}_{\delta}(\x^*) \\  {\{\x_{k;g}\}_{k=0}^{\infty} \in J_{\tau}(g)}}} \norm{[D N_{k;g}(\x_{k;g}) D N_{k-1;g}(\x_{k-1;g})]^{-1} }^{-1}_2 &=\nonumber \\&\hspace{-2cm}\lim_{\delta \to 0}  \inf_{\substack{\x_{k;g} \in \mathcal{B}_{\delta}(\x^*) \\  {\{\x_{k;g}\}_{k=0}^{\infty} \in J_{\tau}(g)}}}\norm{[D N_{k;g}(\x^*) D N_{k-1;g}(\x^*)]^{-1}}^{-1}_2 ,  \label{supineq2b}
\end{align}
$\probP_1$-a.s. where we used the fact that for any positive sequence $\{a_n\}$, $\limsup a_n^{-1} = (\liminf a_n)^{-1}$ provided $\liminf a_n \neq 0 $ and $\limsup a_n^{-1} = \infty$ when $\liminf a_n=0 $ followed by the equalities
$$ \lim_{\delta \to 0} \sup_{\substack{\x_{k;g} \in \mathcal{B}_{\delta}(\x^*) \\  {\{\x_{k;g}\}_{k=0}^{\infty} \in J_{\tau}(g)}}}\norm{[D N_{k;g}(\x_{k;g})]^{-1}-[D N_{k;g}(\x^*)]^{-1}}_2\norm{ [D N_{k-1;g}(\x_{k-1;g})]^{-1}}_2 =0,  \hspace{0.1cm}\probP_1-\text{a.s.}$$
$$ \lim_{\delta \to 0} \sup_{\substack{\x_{k;g} \in \mathcal{B}_{\delta}(\x^*) \\  {\{\x_{k;g}\}_{k=0}^{\infty} \in J_{\tau}(g)}}} \norm{[D N_{k-1;g}(\x_{k-1;g})]^{-1}-[D N_{k-1;g}(\x^*)]^{-1}}_2\norm{ [D N_{k;g}(\x^*)]^{-1}}_2=0, \hspace{0.1cm}\probP_1-\text{a.s.}$$ obtained by replacing $\lim_{\delta \to 0}  \sup_{\substack{\x_{k;g} \in \mathcal{B}_{\delta}(\x^*) \\  {\{\x_{k;g}\}_{k=0}^{\infty} \in J_{\tau}(g)}}}  $ with $\lim_{\delta \to 0} \sup_{k \in S_{\delta}(g) =\mathbb{Z}^*} $ from Lemma \ref{lemmaSdelta}, the $\probP_1$-a.s. local equicontinuity of $\{DN_k^{-1}\}$ maps, the growth condition on $\norm{ [D N_{k;g}(\x^*)]^{-1}}_2 $ and the fact that if $ \x_{k;g} \in \mathcal{B}_{\delta}(\x^*)$ then $\x_{k-1;g} \in \mathcal{B}_{C\delta}(\x^*)$ for some constant $C$ independent of $k$. Using the definition of the set $S_{\delta}(g) =\mathbb{Z}^*$ for the function $g$ from Lemma \ref{lemmaSdelta} followed by the equivalence of sequences $\{DN_k(\x_k)\}$ generated from \eqref{Asymptot3abc1} and $\{DN_{k;g}(\x^*)\}$ generated from \eqref{Asymptot3abc2}, the equality \eqref{supineq2b} can be further simplified as:
\begin{align}
    \lim_{\delta \to 0}  \inf_{\substack{\x_{k;g} \in \mathcal{B}_{\delta}(\x^*) \\  {\{\x_{k;g}\}_{k=0}^{\infty} \in J_{\tau}(g)}}} & \norm{[D N_{k;g}(\x_{k;g}) D N_{k-1;g}(\x_{k-1;g}) ]^{-1}}^{-1}_2 \nonumber \\ 
        &\qquad\qquad = \lim_{\delta \to 0} \inf_{\substack{\x_{k;g} \in \mathcal{B}_{\delta}(\x^*) \\  {\{\x_{k;g}\}_{k=0}^{\infty} \in J_{\tau}(g)}}}\norm{[D N_{k;g}(\x^*) D N_{k-1;g}(\x^*)]^{-1}}^{-1}_2 \\
        &\qquad\qquad  = \lim_{\delta \to 0} \inf_{k \in S_{\delta}(g) =\mathbb{Z}^*}\norm{[D N_{k;g}(\x^*) D N_{k-1;g}(\x^*)]^{-1}}^{-1}_2 \\
        &\qquad\qquad  = \lim_{\delta \to 0} \inf_{k \in S_{\delta}(f) =\mathbb{Z}^*}\norm{[D N_{k}(\x_k) D N_{k-1}(\x_{k-1})]}^{-1}_2 \\
        &\qquad\qquad  = \lim_{\delta \to 0} \inf_{\substack{\x_{k} \in \mathcal{B}_{\delta}(\x^*) \\  {\{\x_{k}\}_{k=0}^{\infty} \in J_{\tau}(f)}}}\norm{[D N_{k}(\x_k) D N_{k-1}(\x_{k-1})]^{-1}}^{-1}_2 = \mathcal{M}_{\star}(f)
\end{align}
$\probP_1$-a.s., where in the second last step we used the fact that the set $S_{\delta} =\mathbb{Z}^*$ for both functions $g$ and $f$ from Lemma \ref{lemmaSdelta}.
\end{proof}

\subsection{Theorem \ref{metricconvergethm2}}
\begin{proof}
To compute the metric $\mathcal{M}_{\star}(f)$ we use quadratic function $g \in \mathcal{C}^{\omega}_{\mu,L}(\mathbb{R}^n)$ in Theorem \ref{diffeomorphthm} such that $\x^*$ is a critical point of both $f,g$ and $\nabla^2 f(\x^*) = \nabla^2 g(\x^*)$. Recall that since $N_k : \x_k \mapsto \x_{k+1}$, we have $\frac{\partial \x_{k+1}}{\partial \x_{k} } = D N_k(\x_k)$ where $D N_k(\x_k) $ satisfies the following relation from Theorem \ref{diffeomorphthm} $\probP_1$-a.s.:
\begin{align}
    D N_k(\x_{k}) & =\bigg(\mathbf{I} -   h \nabla^2 g( R_k(\x_{k}))\bigg)D R_k(\x_k) \label{Asymptot3}
\end{align}
with $R_k(\x_k) = (1+ \beta_k)\x_k - \beta_k \x_{k-1}$, $ D R_k(\x_k) = (1+ \beta_k)\mathbf{I} -  \beta_k[D N_{k-1}(\x_{k-1})]^{-1}$. Further simplifying \eqref{Asymptot3} we get $\probP_1$-a.s. that:
\begin{align}
     D N_k(\x_{k}) & = \bigg(\mathbf{I} -   h \nabla^2 g( R_k(\x_{k}))\bigg)\bigg((1+ \beta_k)\mathbf{I} -  \beta_k[D N_{k-1}(\x_{k-1})]^{-1}\bigg) \label{Asymptot3abc}\\
  \implies \bigg(\mathbf{I} -   h \nabla^2 g( R_k(\x_{k}))\bigg)^{-1}    & = \bigg((1+ \beta_k)\mathbf{I} -  \beta_k[D N_{k-1}(\x_{k-1})]^{-1}\bigg)[D N_k(\x_{k})]^{-1}
     \label{Asymptot4a}\\
     \implies [D N_{k-1}(\x_{k-1})]^{-1} [D N_k(\x_{k})]^{-1} &= \bigg(\frac{\partial \x_{k+1}}{\partial \x_{k-1} }\bigg)^{-1} \nonumber \\ & = \frac{1}{\beta_k}\bigg((1+ \beta_k)[D N_{k}(\x_{k})]^{-1}-  \bigg(\mathbf{I} -   h \nabla^2 g( R_k(\x_{k}))\bigg)^{-1} \bigg) \label{Asymptot4}.
\end{align}
Since $g$ is a quadratic function (constant hessian), \eqref{Asymptot3abc} for any $k$ is as follows:
\begin{align}
   D N_k(\x_k) & = \bigg(\mathbf{I} -   h \nabla^2 g( \x^*)\bigg)\bigg((1+ \beta_k)\mathbf{I} -  \beta_k[D N_{k-1}(\x_{k-1})]^{-1}\bigg),  \label{Asymptot3abc11}
\end{align}
$\probP_1$-a.s.. Using the fact that the initialization is $\x_0 =\x_{-1}$, we can set $N_{-1} \equiv \mathrm{id}$ which gives $D N_{-1}(\x_{-1})  = \mathbf{I}$. Then from \eqref{Asymptot3abc11}, the eigenbasis of the matrix $ D N_k(\x_{k}) $ is equal to the eigenbasis of $ \bigg(\mathbf{I} -   h \nabla^2 g( \x^*)\bigg)$ for all $k \geq 0$ and the matrix $ D N_k(\x_{k}) $ has real eigenvalues for all $k \geq 0$.
Let $\lambda_l$ be the eigenvalue of $D N_k(\x_k) $ corresponding to the eigenvector $\e_l$ where $l$ is some fixed index between $1$ and $n$ and $ \lambda_l([\mathbf{I} -   h \nabla^2 g( \x^*)]^{-1}) = \norm{[\mathbf{I} -   h \nabla^2 g( \x^*)]^{-1}}_2$. Then applying $\lambda_l(.)$ operator on both sides of \eqref{Asymptot4a} for quadratic $f$ we get $\probP_1$-a.s.:
\begin{align}
 \lambda_l({(1+ \beta_k) \mathbf{I}-  \beta_k[D N_{k-1}(\x_{k-1})]^{-1}}) \lambda_l({[{D N_k(\x_k)}]^{-1}})
    & = \lambda_l({[\mathbf{I} -   h \nabla^2 g( \x^*)]^{-1}}) \\
 \implies \lambda_l({[{D N_k(\x_k)}]^{-1}})   & =  \frac{\norm{[\mathbf{I} -   h \nabla^2 g( \x^*)]^{-1}}_2}{(1+ \beta_k)- \beta_k\lambda_l({[D N_{k-1}(\x_{k-1})]^{-1}})} \label{Recursionmetric1} 
\end{align}
where we used the fact that $ \lambda_l([\mathbf{I} -   h \nabla^2 g( \x^*)]^{-1}) = \norm{[\mathbf{I} -   h \nabla^2 g( \x^*)]^{-1}}_2$ in the last step since $\lambda_l(.)$ operator gives the eigenvalue of $[\mathbf{I} -   h \nabla^2 g( \x^*)]^{-1} $ corresponding to the eigenvector $\e_l$. For all $k \geq 0$ let $ t_k = \frac{\beta_k}{(1+ \beta_k)}\lambda_l({[D N_{k-1}(\x_{k-1})]^{-1}})$ where $t_k \neq 1$ $\probP_1$-almost surely for all $k$ and also suppose that $c_k =  \frac{\beta_{k+1}}{(1+ \beta_k)(1+ \beta_{k+1})}\norm{[\mathbf{I} -   h \nabla^2 g( \x^*)]^{-1}}_2 $ where $c_k>0$ for all $k$ with $c_k \to c$ as $k \to \infty$. Then writing \eqref{Recursionmetric1} in terms of $t_k$ and $c_k$ yields $\probP_1$-a.s.:
\begin{align}
    t_{k+1} & = \frac{c_k}{1- t_k} \label{Recursionmetric2}
\end{align}
where we require a non-zero lower bound on $ \limsup_{k \to \infty} \abs{t_k}$. From \eqref{Recursionmetric2} it is evident that the sequence $\{t_k\}$ is chaotic due to the particular recursion expression. Hence evaluating a bound on $t_k$ for every $k$ is not feasible. However one can estimate the asymptotics of $\abs{t_k}$. In particular, the sequence $ \{\sup_{k\geq K} \abs{t_k}\}_K$ will $\probP_1$-a.s. converge\footnote{Here supremum is allowed since \eqref{Recursionmetric2} holds $\probP_1$-a.s. for any $k$ and so all such equalities will hold simultaneously $\probP_1$-a.s. by the fact that countable intersections of almost sure events is also an almost sure event.} even if $\abs{t_k}$ $\probP_1$-a.s. fails to converge and therefore a non-zero lower bound on $ \limsup_{k \to \infty} \abs{t_k}$ can be evaluated very easily. Let $ \overline{t} = \limsup_{k \to \infty} \abs{t_k}$ then evaluating $\limsup$ after taking absolute value in \eqref{Recursionmetric2} and using the fact that $c_k \to c$ as $k \to \infty$ where $c>0$ yields the following $\probP_1$-a.s.:
\begin{align}
   \limsup_{k \to \infty} \abs{t_{k+1}(1-t_k)} & = \limsup_{k \to \infty} c_k = \lim_{k \to \infty} c_k=c\\
   \implies \limsup_{k \to \infty} \abs{t_{k+1}} (1+ \limsup_{k \to \infty}\abs{t_k}) & \geq c \\
   \implies \overline{t} (1 + \overline{t}) & \geq c \\
   \implies \overline{t} &\in \bigg[ \frac{\sqrt{4c+1}-1}{2}, \infty\bigg) \label{supmetrica}
\end{align}
where in the first step we used the fact that $ t_k \neq 1$ $\probP_1$-almost surely for all $k$ and in the last step we disregarded the negative interval due to the fact that $\overline{t} = \limsup_{k \to \infty}\abs{t_k} \geq 0 $. Now from the defintion of $c_k$ we have that $c = \lim_{k \to \infty}  \frac{\beta_{k+1}}{(1+ \beta_k)(1+ \beta_{k+1})}\norm{[\mathbf{I} -   h \nabla^2 g( \x^*)]^{-1}}_2 = \frac{\beta}{(1+ \beta)^2}\norm{\M^{-1}}_2$ where we used $\lim_{k \to \infty}\beta_k = \beta$ and substituted $ \M = \mathbf{I} -   h \nabla^2 g( \x^*)$. Then from \eqref{supmetrica} we have the following condition on $\overline{t} $ $\probP_1$-a.s.:
\begin{align}
   \overline{t} \geq \frac{\sqrt{\frac{4\beta}{(1+ \beta)^2}\norm{\M^{-1}}_2+1}-1}{2}. \label{supmetricb}
\end{align}
Next, applying $\lambda_l(.)$ on both sides of \eqref{Asymptot4}, using $\norm{.}_2 \geq \abs{\lambda_l(.)}$, the equality \eqref{Recursionmetric1} and the definition of $t_k$ yields the following $\probP_1$-a.s.:
\begin{align}
   \norm{\bigg(\frac{\partial \x_{k+1}}{\partial \x_{k-1} }\bigg)^{-1}}_2  &\geq  \frac{1}{\beta_k}\bigg\vert (1+ \beta_k)\lambda_l({[D N_{k}(\x_{k})]^{-1}})-  \lambda_l({[\mathbf{I} -   h \nabla^2 g( \x^*)]^{-1}})\bigg\vert \\
   &=  \frac{1}{\beta_k}\bigg\vert (1+ \beta_k)\lambda_l({[D N_{k}(\x_{k})]^{-1}})-  \norm{[\mathbf{I} -   h \nabla^2 g( \x^*)]^{-1}}_2\bigg\vert \\
    & = \frac{1}{\beta_k}\bigg\vert \bigg(\frac{(1+ \beta_k)\norm{[\mathbf{I} -   h \nabla^2 g( \x^*)]^{-1}}_2}{(1+ \beta_k) - \beta_k\norm{[D N_{k-1}(\x_{k-1})]^{-1}}_2}-  \norm{[\mathbf{I} -   h \nabla^2 g( \x^*)]^{-1}}_2 \bigg)\bigg\vert \\
    & = \frac{\norm{[\mathbf{I} -   h \nabla^2 g( \x^*)]^{-1}}_2}{\beta_k}\bigg\vert \frac{\beta_k\lambda_l([D N_{k-1}(\x_{k-1})]^{-1})}{(1+ \beta_k) -  \beta_k\lambda_l([D N_{k-1}(\x_{k-1})]^{-1})} \bigg\vert \\
    & \geq \frac{\norm{[\mathbf{I} -   h \nabla^2 g( \x^*)]^{-1}}_2}{\beta_k} \frac{\beta_k\abs{\lambda_l([D N_{k-1}(\x_{k-1})]^{-1})}}{(1+ \beta_k) +  \beta_k\abs{\lambda_l([D N_{k-1}(\x_{k-1})]^{-1})}}  \\
     & = \frac{\norm{[\mathbf{I} -   h \nabla^2 g( \x^*)]^{-1}}_2}{\beta_k}\bigg( 1-\frac{(1+ \beta_k)}{(1+ \beta_k) +  \beta_k\abs{\lambda_l([D N_{k-1}(\x_{k-1})]^{-1})}} \bigg) \\
      & = \frac{\norm{\M^{-1}}_2}{\beta_k}\bigg( 1-\frac{1}{1+\abs{t_k}} \bigg) \\
      \implies \norm{\bigg(\frac{\partial \x_{k+1}}{\partial \x_{k-1} }\bigg)^{-1}}^{-1}_2  &\leq \beta_k \norm{\M^{-1}}^{-1}_2\bigg( 1-\frac{1}{1+\abs{t_k}} \bigg)^{-1}
    \label{Asymptot6}
\end{align}
where we substituted $\mathbf{I} -  h \nabla^2 g( \x^*) =  \M$ in the second last step. Taking infimum in \eqref{Asymptot6} over the set $ k \in S_{\delta}$ where the set $S_{\delta}$ is given by $ S_{\delta} =   \bigg\{k \hspace{0.1cm}\bigg\vert\hspace{0.1cm} \x_k \in \mathcal{B}_{\delta}(\x^*); \hspace{0.1cm}\{\x_k\}_{k=0}^{\infty} \in J_{\tau}(g)\bigg\}$ from the definition of constraint set in \eqref{asymptot1} yields:\footnote{Observe that $ \bigg\{\x_k \hspace{0.1cm}\bigg\vert\hspace{0.1cm} \x_k \in \mathcal{B}_{\delta}(\x^*); \hspace{0.1cm}\{\x_k\}_{k=0}^{\infty} \in J_{\tau}(g)\bigg\}  = \bigg\{ \x_k \hspace{0.1cm} \bigg\vert \hspace{0.1cm} k \in S_{\delta}\bigg\}$.}
\begin{align}
     \inf_{\substack{ k \in S_{\delta} }}  \norm{\bigg(\frac{\partial \x_{k+1}}{\partial \x_{k-1} }\bigg)^{-1}}^{-1}_2  &\leq  \inf_{k \in S_{\delta}}  \beta_k \norm{\M^{-1}}^{-1}_2\bigg( 1-\frac{1}{1+\abs{t_k}} \bigg)^{-1} \\
     &= \inf_{k \geq 0}  \beta_k \norm{\M^{-1}}^{-1}_2\bigg( 1-\frac{1}{1+\abs{t_k}} \bigg)^{-1} \\
     & \leq \liminf_{k \to \infty}  \beta_k \norm{\M^{-1}}^{-1}_2\bigg( 1-\frac{1}{1+\abs{t_k}} \bigg)^{-1} \\
     & \underbrace{=}_{\boldsymbol{*}}\liminf_{k \to \infty}  \beta_k \liminf_{k \to \infty}\norm{\M^{-1}}^{-1}_2\bigg( 1-\frac{1}{1+\abs{t_k}} \bigg)^{-1} ,
     \end{align}
     $\probP_1$-a.s. where we used the fact that $S_{\delta} = \mathbb{Z}^{*}$, i.e., $ \{k \hspace{0.1cm} \vert \hspace{0.1cm} k \in S_{\delta}\} = \{k \hspace{0.1cm} \vert \hspace{0.1cm} k \geq 0 \}$ in the second step from Lemma \ref{lemmaSdelta}. The equality $\boldsymbol{*}$ holds from the facts that $\liminf_{k \to \infty}\beta_k = \lim_{k \to \infty}\beta_k = \beta$ and $ \liminf_{k \to \infty} a_k b_k = \liminf_{k \to \infty} a_k\liminf_{k \to \infty} b_k$ whenever $ \lim_{k \to \infty} a_k$ or $ \lim_{k \to \infty} b_k$ exists. 
     Further simplifying the last step from above we get that:
     \begin{align}
      \inf_{\substack{ k \in S_{\delta} }}  \norm{\bigg(\frac{\partial \x_{k+1}}{\partial \x_{k-1} }\bigg)^{-1}}^{-1}_2  &\leq  \beta \norm{\M^{-1}}^{-1}_2\bigg( 1-\frac{1}{1+\limsup_{k \to \infty}\abs{t_k}} \bigg)^{-1} \\
      & =  \beta \norm{\M^{-1}}^{-1}_2\bigg( 1-\frac{1}{1+\overline{t}} \bigg)^{-1} \\ &\leq \beta \norm{\M^{-1}}^{-1}_2\bigg( 1-\frac{2}{1+{\sqrt{\frac{4\beta}{(1+ \beta)^2}\norm{\M^{-1}}_2+1}}} \bigg)^{-1} \\
      &= \beta \norm{\M^{-1}}^{-1}_2\bigg( \frac{{\sqrt{\frac{4\beta}{(1+ \beta)^2}\norm{\M^{-1}}_2+1}} -1}{{\sqrt{\frac{4\beta}{(1+ \beta)^2}\norm{\M^{-1}}_2+1}}+1} \bigg)^{-1},
      \label{Asymptotquad1}
\end{align}
$\probP_1$-a.s. where we substituted lower bound on $\overline{t} $ from \eqref{supmetricb} in the last step. 
Taking limit of $\delta \downarrow 0$ on both sides of \eqref{Asymptotquad1} and using the substitution $ \bigg\{\x_k \hspace{0.1cm}\bigg\vert\hspace{0.1cm} \x_k \in \mathcal{B}_{\delta}(\x^*); \hspace{0.1cm}\{\x_k\}_{k=0}^{\infty} \in J_{\tau}(g)\bigg\}  = \bigg\{ \x_k \hspace{0.1cm} \bigg\vert \hspace{0.1cm} k \in S_{\delta}\bigg\}$ gives the following $\probP_1$-a.s.:
\begin{align}
   \mathcal{M}_{\star}(g)= \lim_{\delta \downarrow 0}\inf_{\substack{   \bigg\{\x_k \hspace{0.1cm}\bigg\vert\hspace{0.1cm} \x_k \in \mathcal{B}_{\delta}(\x^*); \hspace{0.1cm}\{\x_k\}_{k=0}^{\infty} \in J_{\tau}(g)\bigg\}}} & \norm{\bigg(\frac{\partial \x_{k+1}}{\partial \x_{k-1} }\bigg)^{-1}}^{-1}_2 \nonumber \\ &\leq   \lim_{\delta \downarrow 0} \beta \norm{\M^{-1}}^{-1}_2\bigg( \frac{{\sqrt{\frac{4\beta}{(1+ \beta)^2}\norm{\M^{-1}}_2+1}} -1}{{\sqrt{\frac{4\beta}{(1+ \beta)^2}\norm{\M^{-1}}_2+1}}+1} \bigg)^{-1} \\
 &= \beta \norm{\M^{-1}}^{-1}_2\bigg( \frac{{\sqrt{\frac{4\beta}{(1+ \beta)^2}\norm{\M^{-1}}_2+1}} -1}{{\sqrt{\frac{4\beta}{(1+ \beta)^2}\norm{\M^{-1}}_2+1}}+1} \bigg)^{-1}.  \label{Asymptotquad2}
\end{align}
Using Lemma \ref{lemsupab} for\footnote{Note that any analytic function is locally Hessian Lipschitz continuous and so Lemma \ref{lemsupab} can be applied here.} $f \in \mathcal{C}^{\omega}_{\mu,L}(\mathbb{R}^n) $ where $\nabla^2 f(\x^*)= \nabla^2 g(\x^*)$ along with \eqref{Asymptotquad2} gives:
\begin{align}
 \mathcal{M}_{\star}(f)=  \mathcal{M}_{\star}(g) &\leq  \beta \norm{\M^{-1}}^{-1}_2\bigg( \frac{{\sqrt{\frac{4\beta}{(1+ \beta)^2}\norm{\M^{-1}}_2+1}} -1}{{\sqrt{\frac{4\beta}{(1+ \beta)^2}\norm{\M^{-1}}_2+1}}+1} \bigg)^{-1} \label{Contractionmetric12}
\end{align}
 $\probP_1$-almost surely.
\end{proof}  

\subsection{Theorem \ref{metricedivergethm}}
\begin{proof}
To compute the metric $\mathcal{M}^{\star}(f)$ we use quadratic function $g \in \mathcal{C}^{\omega}_{\mu,L}(\mathbb{R}^n)$ in Theorem \ref{diffeomorphthm} such that $\x^*$ is a critical point of both $f,g$ and $\nabla^2 f(\x^*) = \nabla^2 g(\x^*)$. Recall that since $N_k : \x_k \mapsto \x_{k+1}$, we have $\frac{\partial \x_{k+1}}{\partial \x_{k} } = D N_k(\x_k)$ where $D N_k(\x_k) $ satisfies the following relation from Theorem \ref{diffeomorphthm} $\probP_1$-a.s.:
\begin{align}
    D N_k(\x_{k}) & =\bigg(\mathbf{I} -   h \nabla^2 g( R_k(\x_{k}))\bigg)D R_k(\x_k) \label{asymptot3}
\end{align}
with $R_k(\x_k) = (1+ \beta_k)\x_k - \beta_k \x_{k-1}$, $ D R_k(\x_k) = (1+ \beta_k)\mathbf{I} -  \beta_k[D N_{k-1}(\x_{k-1})]^{-1}$. Further simplifying \eqref{asymptot3} we get that $\probP_1$-a.s.:
\begin{align}
     D N_k(\x_{k}) & = \bigg(\mathbf{I} -   h \nabla^2 g( R_k(\x_{k}))\bigg)\bigg((1+ \beta_k)\mathbf{I} -  \beta_k[D N_{k-1}(\x_{k-1})]^{-1}\bigg) \label{asymptot4a}\\
     \implies D N_k(\x_{k})D N_{k-1}(\x_{k-1})= \frac{\partial \x_{k+1}}{\partial \x_{k-1} } & = \bigg(\mathbf{I} -   h \nabla^2 g( R_k(\x_{k}))\bigg)\bigg((1+ \beta_k)D N_{k-1}(\x_{k-1})-  \beta_k\mathbf{I} \bigg) \label{asymptot4}.
\end{align}
Since $g$ is a quadratic function (constant hessian), the equation \eqref{asymptot3} reduces to the update $ D N_k(\x_{k})  =\bigg(\mathbf{I} -   h \nabla^2 g( \x^*)\bigg)D R_k(\x_k) $. Using the fact that the initialization is $D N_{-1}(\x_{-1})  = \mathbf{I}$, it can be readily checked from \eqref{asymptot4a} that the eigenbasis of the matrix $ D N_k(\x_{k}) $ is equal to the eigenbasis of $ \bigg(\mathbf{I} -   h \nabla^2 g( \x^*)\bigg)$ for all $k \geq 0$ and the matrix $ D N_k(\x_{k}) $ has real eigenvalues for all $k \geq 0$.

Let $\lambda_l$ be the eigenvalue of $D N_k(\x_k) $ corresponding to eigenvector $\e_l$ where $l$ is some fixed constant between $1$ and $n$ and also $ \lambda_l(\mathbf{I} -   h \nabla^2 g( \x^*)) = \norm{\mathbf{I} -   h \nabla^2 g( \x^*)}_2$. Then using the fact that $D N_k(\x_k) $ and $\bigg(\mathbf{I} -   h \nabla^2 g( \x^*)\bigg)$ have the same eigenbasis we get $\probP_1$-a.s.:
\begin{align}
     \lambda_l({D N_k(\x_k)})
    & = \lambda_l\bigg(\mathbf{I} -   h \nabla^2 g( \x^*)\bigg)\bigg((1+ \beta_k) -  \frac{\beta_k}{\lambda_l(D N_{k-1}(\x_{k-1}))}\bigg) \\
     & = \norm{\mathbf{I} -   h \nabla^2 g( \x^*)}_2\bigg((1+ \beta_k) -  \frac{\beta_k}{\lambda_l(D N_{k-1}(\x_{k-1}))}\bigg) \label{recursionmetric1} 
\end{align}
where we used the fact that $\lambda_l\bigg(\mathbf{I} -   h \nabla^2 g( \x^*)\bigg)$ is $ \norm{\mathbf{I} -   h \nabla^2 g( \x^*)}_2$. Next we claim that $ \lambda_l({D N_k(\x_k)}) \geq \norm{\mathbf{I} -   h \nabla^2 g( \x^*)}_2$ for all $k \geq 0$. Notice that $D N_0(\x_{0}) = \bigg(\mathbf{I} -   h \nabla^2 g( \x^*)\bigg)$ because $D N_{-1}(\x_{-1})  = \mathbf{I}$ so the claim holds true for $ k = 0$. Let the claim be true for some $k-1$, i.e., $ \lambda_l({D N_{k-1}(\x_{k-1})}) \geq \norm{\mathbf{I} -   h \nabla^2 g( \x^*)}_2$ where we have $ \norm{\mathbf{I} -   h \nabla^2 g( \x^*)}_2 > 1$, then from \eqref{recursionmetric1} we get $\probP_1$-a.s.:
\begin{align}
    \lambda_l({D N_k(\x_k)})
    & =  \norm{\mathbf{I} -   h \nabla^2 g( \x^*)}_2\bigg((1+ \beta_k) -  \frac{\beta_k}{\lambda_l(D N_{k-1}(\x_{k-1}))}\bigg) \\
    & \geq \norm{\mathbf{I} -   h \nabla^2 g( \x^*)}_2\underbrace{\bigg(1 + \beta_k \bigg(1 -  \frac{1}{\norm{\mathbf{I} -   h \nabla^2 g( \x^*)}_2}\bigg)\bigg)}_{>1} \\
    & > \norm{\mathbf{I} -   h \nabla^2 g( \x^*)}_2
\end{align}
and therefore by induction we have $ \lambda_l({D N_k(\x_k)}) \geq \norm{\mathbf{I} -   h \nabla^2 g( \x^*)}_2$ for all $k \geq 0$. Using this fact in \eqref{asymptot4} by applying $\lambda_l(\cdot)$ on both sides of \eqref{asymptot4} for quadratic $f$ yields:
\begin{align}
   \norm{\frac{\partial \x_{k+1}}{\partial \x_{k-1} }}_2  \geq \lambda_l\bigg(\frac{\partial \x_{k+1}}{\partial \x_{k-1} } \bigg) & = \lambda_l\bigg({\mathbf{I} -   h \nabla^2 g( \x^*)}\bigg)\lambda_l{\bigg((1+ \beta_k)D N_{k-1}(\x_{k-1})-  \beta_k\mathbf{I} \bigg)} \\
   & = \norm{\mathbf{I} -   h \nabla^2 g( \x^*)}_2{\bigg((1+ \beta_k)\lambda_l(D N_{k-1}(\x_{k-1}))-  \beta_k\mathbf{I} \bigg)}\\
   & \geq \norm{\mathbf{I} -   h \nabla^2 g( \x^*)}_2\bigg((1+ \beta_k)\norm{\mathbf{I} -   h \nabla^2 g( \x^*)}_2-  \beta_k \bigg)\\
   & =\norm{\M}_2\bigg((1+ \beta_k)\norm{\M}_2-  \beta_k \bigg), \label{asymptot6}
\end{align}
$\probP_1$-a.s. where we substituted $\mathbf{I} -  h \nabla^2 g( \x^*) =  \M$ in the last step.

 Taking supremum in \eqref{asymptot6} over the set $ k \in {S_{\delta} =  \bigg\{k \hspace{0.1cm}\bigg\vert\hspace{0.1cm} \x_k \in \mathcal{B}_{\delta}(\x^*); \hspace{0.1cm}\{\x_k\}_{k=0}^{\infty} \in J_{\tau}(g)\bigg\} } $ from the definition of \eqref{asymptot1} \footnote{Observe that $\bigg\{\x_k \hspace{0.1cm}\bigg\vert\hspace{0.1cm} \x_k \in \mathcal{B}_{\delta}(\x^*); \hspace{0.1cm}\{\x_k\}_{k=0}^{\infty} \in J_{\tau}(g)\bigg\} = \bigg\{ \x_k \hspace{0.1cm} \bigg\vert \hspace{0.1cm} k \in S_{\delta}\bigg\}$.} followed by taking limit of $\delta \downarrow 0$ yields $\probP_1$-a.s.:
     \begin{align}
     \sup_{\substack{  k \in S_{\delta} }} \norm{\frac{\partial \x_{k+1}}{\partial \x_{k-1} }}_2 &\geq \sup_{\substack{  k \in S_{\delta} }} \norm{\M}_2\bigg((1+ \beta_k)\norm{\M}_2-  \beta_k \bigg) \\
     &=  \norm{\M}_2\bigg(\norm{\M}_2+ \sup_{k \in S_{\delta}}\beta_k \underbrace{(\norm{\M}_2- 1)}_{>0}  \bigg) \\
         \implies  \lim_{\delta \downarrow 0}\sup_{\substack{  \bigg\{\x_k \hspace{0.1cm}\bigg\vert\hspace{0.1cm} \x_k \in \mathcal{B}_{\delta}(\x^*); \hspace{0.1cm}\{\x_k\}_{k=0}^{\infty} \in J_{\tau}(g)\bigg\}}} \norm{\frac{\partial \x_{k+1}}{\partial \x_{k-1} }}_2 & \geq  \lim_{\delta \downarrow 0}  \norm{\M}_2\bigg((1+ \beta)\norm{\M}_2-  \beta \bigg)\\ & = \norm{\M}_2\bigg((1+ \beta)\norm{\M}_2-  \beta \bigg) \label{metricdivergence1}
     \end{align}
where we used the fact that $\beta_k \to \beta$, $\beta_k$ is non decreasing with $k$, $S_{\delta} = \mathbb{Z}^* $ from Lemma \ref{lemmaSdelta}  and hence $ \sup_{k \in S_{\delta}} \beta_k = \beta $. \\
Hence from the definition of \eqref{asymptot1} and the bound in \eqref{metricdivergence1} we will have that:
\begin{align}
    \mathcal{{M}}^{\star}{}(g) & \geq \lim_{\delta \downarrow 0}\sup_{\substack{  \bigg\{\x_k \hspace{0.1cm}\bigg\vert\hspace{0.1cm} \x_k \in \mathcal{B}_{\delta}(\x^*); \hspace{0.1cm}\{\x_k\}_{k=0}^{\infty} \in J_{\tau}(g)\bigg\}}} \norm{\frac{\partial \x_{k+1}}{\partial \x_{k-1} }}_2 \geq \norm{\M}_2\bigg((1+ \beta)\norm{\M}_2-  \beta \bigg) \label{expansionmetric12}
\end{align}
$\probP_1$-a.s.. Using Lemma \ref{lemsupab} for $f \in \mathcal{C}^{\omega}_{\mu,L}(\mathbb{R}^n) $ where $\nabla^2 f(\x^*)= \nabla^2 g(\x^*)$ along with \eqref{expansionmetric12} gives:
\begin{align}
  \mathcal{{M}}^{\star}{}(f) = \mathcal{{M}}^{\star}{}(g) & \geq \norm{\M}_2\bigg((1+ \beta)\norm{\M}_2-  \beta \bigg)  
\end{align}
$\probP_1$-almost surely.
\end{proof}  

\subsection{Lemma \ref{lemmaasympevalm}}
\begin{proof}

    \subsubsection{Gradient descent (GD) method}
For GD method we have $\beta = 0$ and so the map $N_k$ from Theorem \ref{diffeomorphthm} satisfies $N_k \equiv G$ for all $k$ which is a diffeomorphism for $h<\frac{1}{L}$ and hence a diffeomorphism $\probP_1$-a.s. as well. Using Theorem \ref{metricconvergethm2} we obtain the upper bound on $ \mathcal{{M}}_{\star}(f)$ as follows by taking $\beta \downarrow 0$:
\begin{align}
    \mathcal{{M}}_{\star}{}(f)  &\leq \lim_{\beta \downarrow 0} \beta (1-Lh)\bigg( \frac{{\sqrt{\frac{4\beta}{(1+ \beta)^2(1-Lh)}+1}} -1}{{\sqrt{\frac{4\beta}{(1+ \beta)^2(1-Lh)}+1}}+1} \bigg)^{-1} \\
    & = \lim_{\beta \downarrow 0} \beta (1-Lh)\bigg( \frac{\lim_{\beta \downarrow 0}{\sqrt{\frac{4\beta}{(1+ \beta)^2(1-Lh)}+1}} -1}{\lim_{\beta \downarrow 0}{\sqrt{\frac{4\beta}{(1+ \beta)^2(1-Lh)}+1}}+1} \bigg)^{-1} \\
     & = \lim_{\beta \downarrow 0} \beta (1-Lh)\bigg( \frac{\lim_{\beta \downarrow 0}{{\bigg(\frac{2\beta}{(1+ \beta)^2(1-Lh)}+o(\beta)\bigg)}} }{2} \bigg)^{-1} \\
     & =  (1-Lh)^2
\end{align}
 $\probP_1$-almost surely. From Theorem \ref{metricedivergethm}, the divergence metric $\mathcal{M}^{\star}(f)$ is lower bounded as:
\begin{align}
    \mathcal{{M}}^{\star}{}(f)  &\geq  (1+ \mu h)\bigg((1+ \beta)(1+ \mu h)-  \beta \bigg)  = (1 + \mu h)^2
\end{align}
 $\probP_1$-almost surely. Combining these bounds together we get:
\begin{align}
  \mathcal{{M}}_{\star}{}(f) =  (1- L h)^2  < 1 < (1 + \mu h)^2 \leq \mathcal{{M}}^{\star}{}(f) 
\end{align}
for GD method $\probP_1$-almost surely. 

\subsubsection{Nesterov constant momentum method \eqref{generaldsconst}}
For the case of constant momentum we have $\beta_k = \beta$ for all $k$ in \eqref{generalds} where $\beta \in (0,1)$. In particular for $\beta = \frac{1 - \sqrt{Lh}}{1 + \sqrt{Lh}}$, the metric $\mathcal{M}_{\star}(f)$ is bounded as
\begin{align}
    \mathcal{{M}}_{\star}{}(f)  &\leq  \beta (1-Lh)\bigg( \frac{{\sqrt{\frac{4\beta}{(1+ \beta)^2(1-Lh)}+1}} -1}{{\sqrt{\frac{4\beta}{(1+ \beta)^2(1-Lh)}+1}}+1} \bigg)^{-1} \\
    & =  (1-\sqrt{Lh})^2\frac{\sqrt{2} +1}{\sqrt{2}-1}
\end{align}
 $\probP_1$-almost surely. Since $ \min_i \abs{\lambda_i{DP([\x^*;\x^*])}} < (1 - \sqrt{Lh}) < 1$ from Theorem \ref{generalacclimiteigen}, the necessary condition \eqref{sufficientconvergencestability} from Lemma \ref{lemmaasymnec} is satisfied provided $h$ is not too small.
 
From Theorem \ref{metricedivergethm}, the divergence metric $\mathcal{M}^{\star}(f)$ is lower bounded as:
\begin{align}
    \mathcal{{M}}^{\star}{}(f)  &\geq   (1+ \mu h)\bigg((1+ \beta)(1+ \mu h)-  \beta \bigg) \\ &= (1 + \mu h)^2 + \frac{1 - \sqrt{Lh}}{1 + \sqrt{Lh}} \bigg((1 + \mu h)^2 -(1 + \mu h)\bigg)
\end{align}
 $\probP_1$-almost surely. 
 
\subsubsection{Nesterov acceleration method \eqref{originalnesterov}}
For the case of Nesterov acceleration, in \eqref{generalds}, we have the momentum $\beta_k = \frac{k}{k+3}$ and $\beta_k \to \beta$ with $\beta=1$. Therefore, the bound on convergence metric from Theorem \ref{metricconvergethm2} gives
$$ \mathcal{M}_{\star}(f) \leq \beta \norm{\M^{-1}}^{-1}_2\bigg( \frac{{\sqrt{\frac{4\beta}{(1+ \beta)^2}\norm{\M^{-1}}_2+1}} -1}{{\sqrt{\frac{4\beta}{(1+ \beta)^2}\norm{\M^{-1}}_2+1}}+1} \bigg)^{-1} = 
(1-Lh) \bigg( \frac{{\sqrt{(1-Lh)^{-1}+1}} -1}{{\sqrt{(1-Lh)^{-1}+1}}+1} \bigg)^{-1}$$ $\probP_1$-almost surely
 for $\beta =1$.
 
From Theorem \ref{metricedivergethm}, the divergence metric $\mathcal{M}^{\star}(f)$ is lower bounded as:
\begin{align}
    \mathcal{{M}}^{\star}{} (f) &\geq  (1+ \mu h)\bigg((1+ \beta)(1+ \mu h)-  \beta \bigg) \\ &= 2(1 + \mu h)^2 - (1+ \mu h)
\end{align}
 $\probP_1$-almost surely. Combining these bounds together we get:
\begin{align}
  \mathcal{{M}}_{\star}{}(f) &\leq  (1-Lh) \bigg( \frac{{\sqrt{(1-Lh)^{-1}+1}} -1}{{\sqrt{(1-Lh)^{-1}+1}}+1} \bigg)^{-1},\\  
  \mathcal{{M}}^{\star}{}(f) &\geq 2(1 + \mu h)^2 - (1+ \mu h)
\end{align}
 $\probP_1$-almost surely.
\end{proof}  

\section{Local analysis of accelerated methods around strict saddle points}\label{local analysis appendix}

\subsection{Theorem \ref{exittimethm1}}
\begin{proof}
Evaluating $ \norm{\tilde{\u}_K}^2 =\tilde{\u}_K^{{H}} \tilde{\u}_K $ with $(.)^{{H}}$ being the Hermitian operator, we get:
\begin{align}
 \norm{\tilde{\u}_K}^2 &= \tilde{\u}_K^{{H}} \tilde{\u}_K \\
  & = \u_0^{{H}}\bigg( \bigg(\Lambda^K \bigg)^H \Lambda^K \bigg) \u_0 + \underbrace{\epsilon\bigg( \u_0^{{H}} (\A + \A^{{H}}) \u_0 \bigg)}_{T_1} + \underbrace{\epsilon^2\bigg( \u_0^{{H}} \C^{{H}}\C \u_0 \bigg)}_{T_2}\\
 & = \epsilon^2\sum\limits_{i=1}^{d} \abs{z_i}^{2K}\abs{\theta_i}^2 + {\epsilon\bigg( \u_0^{{H}} (\A + \A^{{H}}) \u_0 \bigg)} + \mathcal{O}(K^2  \norm{\Lambda}_2^{2K-2}\epsilon^4) \label{exittime3}
\end{align}
where $ \A = \bigg( \Lambda^K \bigg)^H\sum\limits_{{r=1}}^{K}\Lambda^{r-1} (o(1)+\mathbf{R}(\u_{K-r})) \Lambda^{K-r}$, $\C = \sum\limits_{{r=1}}^{K}\Lambda^{r-1} (o(1)+\mathbf{R}(\u_{K-r})) \Lambda^{K-r}$ and $\u_0 = \epsilon\sum\limits_{i=1}^{d} \theta_i \e_i $ with $\e_i$ forming the canonical basis of $\mathbb{R}^d$ Euclidean space and $\epsilon\theta_i = \langle \u_0, \e_i\rangle$ for all $i$.

Next, the absolute value of term $T_1$ is bounded as:
\begin{align}
    \abs{T_1} & \leq  K \epsilon \norm{\u_0}^2 \norm{\Lambda}_2^{2K-1}  \sup_{{1}\leq r\leq K}\norm{o(1)+\mathbf{R}(\u_{K-r})} \\
      & \leq K \epsilon^3  \norm{\Lambda}_2^{2K-1}  (o(1)+\Gamma) \label{exittime4}
\end{align}
where we used the fact that $\norm{\u_K} \leq \epsilon$ for all $0 \leq K < K_{exit}$ and substituted the bound $$\sup_{{1} \leq r \leq K <  K_{exit}}\norm{\mathbf{R}(\u_{K-r})} \leq \sup_{\norm{\u_{K-r}}\leq \epsilon}\norm{\mathbf{R}(\u_{K-r})}\leq \Gamma$$ in the last step with the constant $\Gamma$ left unspecified at this stage. 
Similarly, in order notation it is easy to check that the absolute value of term $T_2$ in \eqref{exittime4} is upper bounded as $\abs{T_2} = \mathcal{O}(K^2 \norm{\Lambda}_2^{2K-2}\epsilon^4)  $. We can now find the upper bound on the exit time of our approximation of trajectory using \eqref{exittime3}.

\subsubsection*{Upper bound on the exit time}

Using the orthogonal splitting of subspaces followed by the bound \eqref{exittime4}, equation \eqref{exittime3} can be lower bounded as:
\begin{align}
   \norm{\tilde{\u}_K}^2   = & \epsilon^2\bigg(\sum\limits_{i \in \mathcal{N}_{S}} \abs{z_i}^{2K}\abs{\theta_i}^2  + \sum\limits_{i \in \mathcal{N}_{US}} \abs{z_i}^{2K}\abs{\theta_i}^2 + \sum\limits_{i \in \mathcal{N}_{C}} \abs{z_i}^{2K}\abs{\theta_i}^2\bigg) +  \underbrace{\epsilon\bigg( \u_0^{{H}} (\A + \A^{{H}}) \u_0 \bigg)}_{T_1} + \\ \nonumber &  \mathcal{O}(K^2  \norm{\Lambda}_2^{2K-2}\epsilon^4) \\
   \geq & \epsilon^2\bigg(\inf_{i  \in \mathcal{N}_{S}}\abs{z_i}^{2K} \sum\limits_{i \in \mathcal{N}_{S}} \abs{\theta_i}^2  + \inf_{i  \in \mathcal{N}_{US}}\abs{z_i}^{2K}\sum\limits_{i \in \mathcal{N}_{US}} \abs{\theta_i}^2 + \sum\limits_{i \in \mathcal{N}_{C}} \abs{\theta_i}^2\bigg) -  K \epsilon^3  \norm{\Lambda}_2^{2K-1}  (o(1)+\Gamma) \nonumber \\ & \hspace{1cm}  -\mathcal{O}(K^2 \norm{\Lambda}_2^{2K}\epsilon^4) \\
      &  = \epsilon^2\bigg(\inf_{i  \in \mathcal{N}_{S}}\abs{z_i}^{2K} \sum\limits_{i \in \mathcal{N}_{S}} \abs{\theta_i}^2  + \inf_{i  \in \mathcal{N}_{US}}\abs{z_i}^{2K}\sum\limits_{i \in \mathcal{N}_{US}} \abs{\theta_i}^2 + \sum\limits_{i \in \mathcal{N}_{C}} \abs{\theta_i}^2\bigg) \nonumber \\ & -  K \epsilon^3\norm{\Lambda}_2^{2K-1} (o(1)+\Gamma) \bigg(  1 + \mathcal{O}(K \epsilon)\bigg) \\
     & \hspace{-1cm} \geq \epsilon^2\bigg(\inf_{i  \in \mathcal{N}_{S}}\abs{z_i}^{2K} \sum\limits_{i \notin \mathcal{N}_{US}  } \abs{\theta_i}^2  + \inf_{i  \in \mathcal{N}_{US}}\abs{z_i}^{2K}\sum\limits_{i \in \mathcal{N}_{US}} \abs{\theta_i}^2 \bigg) -  K \epsilon^3\norm{\Lambda}_2^{2K-1} (o(1)+\Gamma) \bigg(  1 + \mathcal{O}(K \epsilon)\bigg)  \label{exittime5} \\
   =  &  \epsilon^2 \Psi(K)
\end{align}
with 
\begin{align}
    \Psi(K) = \bigg(\inf_{i  \in \mathcal{N}_{S}}\abs{z_i}^{2K} \sum\limits_{i \notin \mathcal{N}_{US}  } \abs{\theta_i}^2  + \inf_{i  \in \mathcal{N}_{US}}\abs{z_i}^{2K}\sum\limits_{i \in \mathcal{N}_{US}} \abs{\theta_i}^2 \bigg) -  K \epsilon\norm{\Lambda}_2^{2K-1} (o(1)+\Gamma) \bigg(  1 + \mathcal{O}(K \epsilon)\bigg).
\end{align}
Now, in order to obtain the exit time, we need to obtain the smallest $K$ where $\norm{\tilde{\u}_K} > \epsilon $. The condition can be relaxed by finding the smallest $K$ such that $\Psi(K) > 1 $.  Let $ K^{\iota}  = \inf_{K\geq 1} \{K \hspace{0.2cm} | \hspace{0.2cm} \Psi(K) > 1\}$, then the following holds:
\begin{align}
    K^{\iota}  \geq \inf_{K\geq 1} \{K \hspace{0.2cm} | \hspace{0.2cm} \norm{\tilde{\u}_K} > \epsilon\}.
\end{align}
The above inequality holds from the fact that $ \{K \hspace{0.2cm} | \hspace{0.2cm} \Psi(K) > 1\} \subset \{K \hspace{0.2cm} | \hspace{0.2cm} \norm{\tilde{\u}_K} > \epsilon\} $ since $ \norm{\tilde{\u}_K} \geq \epsilon \sqrt{\Psi(K)} $ and infimum over any set $S$  is smaller than the infimum over any subset of $S$.

Hence solving the condition $ \Psi(K) > 1$, we obtain:
\begin{align}
    \bigg(\inf_{i  \in \mathcal{N}_{S}}\abs{z_i}^{2K} \sum\limits_{i \notin \mathcal{N}_{US}  } \abs{\theta_i}^2  + \inf_{i  \in \mathcal{N}_{US}}\abs{z_i}^{2K}\sum\limits_{i \in \mathcal{N}_{US}} \abs{\theta_i}^2 \bigg) -  K \epsilon\norm{\Lambda}_2^{2K-1}(o(1)+\Gamma) \bigg(  1 + \mathcal{O}(K \epsilon)\bigg)  &> 1 \\
    \bigg(\underbrace{\frac{\inf_{i  \in \mathcal{N}_{S}}\abs{z_i}^{2K}}{\norm{\Lambda}_2^{2K}} \sum\limits_{i \notin \mathcal{N}_{US}  } \abs{\theta_i}^2}_{F_1}  + \frac{\inf_{i  \in \mathcal{N}_{US}}\abs{z_i}^{2K}}{\norm{\Lambda}_2^{2K}}\sum\limits_{i \in \mathcal{N}_{US}} \abs{\theta_i}^2 \bigg) -  K \epsilon\norm{\Lambda}^{-1}_2 (o(1)+\Gamma) \bigg(  1 + \mathcal{O}(K \epsilon)\bigg) & \nonumber \\ & \hspace{-1.5cm}> \frac{1}{\norm{\Lambda}_2^{2K}}. \label{exittime6}
\end{align}
We now have two cases corresponding to the value of $ \inf_{i  \in \mathcal{N}_{US}}\abs{z_i}$.

\subsubsection*{Case 1 of ${\inf_{i  \in \mathcal{N}_{US}}\abs{z_i}} < {\norm{\Lambda}_2} $:\\}
For obtaining linear exit time solutions, i.e. $ K = \mathcal{O}(\log(\epsilon^{-1}))$, we set $\frac{1}{\norm{\Lambda}_2^{2K}} = a_1 \epsilon^{b_1} $ for some $a_1>0, b_1>0$ (since $ \norm{\Lambda}_2>1$) and $ \inf_{i  \in \mathcal{N}_{S}}\abs{z_i}^{2K} = a_2 \epsilon^{b_2}$ for some $a_2>0, b_2>0$ (since $ \inf_{i  \in \mathcal{N}_{S}}\abs{z_i}< 1$ ) in \eqref{exittime6} and get:
\begin{align}
    & \bigg(\underbrace{\mathcal{O}(\epsilon^{b_1+b_2})}_{F_1}  + \frac{\inf_{i  \in \mathcal{N}_{US}}\abs{z_i}^{2K}}{\norm{\Lambda}_2^{2K}}\sum\limits_{i \in \mathcal{N}_{US}} \abs{\theta_i}^2 \bigg) -  K \epsilon\norm{\Lambda}^{-1}_2 (o(1)+\Gamma) \bigg(  1 + \mathcal{O}(K \epsilon)\bigg)  >  a_1 \epsilon^{b_1} \\
    & \implies  \frac{\inf_{i  \in \mathcal{N}_{US}}\abs{z_i}^{2K}}{\norm{\Lambda}_2^{2K}}\sum\limits_{i \in \mathcal{N}_{US}} \abs{\theta_i}^2  -  K \epsilon\norm{\Lambda}^{-1}_2 (o(1)+\Gamma)   > a_1 \epsilon^{b_1} - \mathcal{O}(\epsilon^{b_1+b_2}) \\ &\implies  \frac{\inf_{i  \in \mathcal{N}_{US}}\abs{z_i}^{2K}}{\norm{\Lambda}_2^{2K}}\sum\limits_{i \in \mathcal{N}_{US}} \abs{\theta_i}^2    >    K \epsilon\norm{\Lambda}^{-1}_2 \Gamma +\mathcal{O}(\epsilon^{b_1}) \label{exittime7}
\end{align}
where we dropped terms $F_1$ (since it is dominated by order $ \mathcal{O}(\epsilon^{b_1})$ term on right-hand-side) and also dropped the term $\mathcal{O}(K \epsilon)$ and the term $o(1)$ relative to $\Gamma$ for $K\epsilon \ll 1$ on the left hand side. Now \eqref{exittime7} is a transcendental inequality which can be solved from the solution of transcendental equation $ q^x = ax + b$ given by $ x = -\frac{W(-\frac{\log q}{a}q^{-\frac{b}{a}})}{\log q} - \frac{b}{a} \leq \frac{\log\bigg(\frac{\log (q^{-1})}{a}\bigg)}{\log(q^{-1})}$ for $q<1$ where $W(.)$ is the Lambert W function and $W(y) \leq \log y$ for large $y$ (for details see \cite{corless1996lambertw}). Using this result in \eqref{exittime7} for $q = \frac{\inf_{i  \in \mathcal{N}_{US}}\abs{z_i}}{\norm{\Lambda}_2}$, $x = 2K_{exit}$, $a = \frac{ \epsilon \norm{\Lambda}^{-1}_2 \Gamma}{2\sum\limits_{i \in \mathcal{N}_{US}} \abs{\theta_i}^2}$ and $b = \mathcal{O}(\epsilon^{b_1})$ we get that:
\begin{align}
2 K_{exit} \lessapprox \frac{\log\bigg( \frac{2\sum\limits_{i \in \mathcal{N}_{US}} \abs{\theta_i}^2 \log(q^{-1})}{\epsilon \norm{\Lambda}^{-1}_2 \Gamma}\bigg)}{\log(q^{-1})} \leq      \frac{\log\bigg( \frac{2 \log\bigg(\frac{\norm{\Lambda}_2}{\inf_{i  \in \mathcal{N}_{US}}\abs{z_i}}\bigg)}{\epsilon \norm{\Lambda}^{-1}_2 \Gamma}\bigg)}{\log\bigg(\frac{\norm{\Lambda}_2}{\inf_{i  \in \mathcal{N}_{US}}\abs{z_i}}\bigg)} \label{exittime8}
\end{align}
where in last step we used that $\sum\limits_{i \in \mathcal{N}_{US}} \abs{\theta_i}^2 \leq 1 $. The approximate inequality results from dropping Big-O terms in the steps preceding \eqref{exittime7}.

\subsubsection*{Case 2 of ${\inf_{i  \in \mathcal{N}_{US}}\abs{z_i}} = {\norm{\Lambda}_2} $:\\}
Since $ \norm{\Lambda}_2 = {\sup_{i  \in \mathcal{N}_{US}}\abs{z_i}}$, the case of ${\inf_{i  \in \mathcal{N}_{US}}\abs{z_i}} = {\norm{\Lambda}_2} $ can occur when $dim(\mathcal{E}_{US}) =1$ or the eigenvalues in $ \mathcal{E}_{US}$ have the same magnitude. For obtaining linear exit time solutions, i.e. $ K = \mathcal{O}(\log(\epsilon^{-1}))$, similar to the previous case we can drop the term $F_1$ with respect to to the term $ \frac{1}{\norm{\Lambda}_2^{2K}}$ on the right hand side of \eqref{exittime6} we get:
\begin{align}
  \sum\limits_{i \in \mathcal{N}_{US}} \abs{\theta_i}^2  -  K \epsilon\norm{\Lambda}^{-1}_2 (o(1)+\Gamma) \bigg(  1 + \mathcal{O}(K \epsilon)\bigg)  &> \frac{1}{\norm{\Lambda}_2^{2K}} \\
  \implies \sum\limits_{i \in \mathcal{N}_{US}} \abs{\theta_i}^2  -  K \epsilon\norm{\Lambda}^{-1}_2 \Gamma  &> \frac{1}{\norm{\Lambda}_2^{2K}} \label{exittime7abc}
\end{align}
where we dropped $\mathcal{O}(K \epsilon)$ and $o(1)$ terms in the last step. Further manipulating the last step and using the fact that $1 - \frac{1}{\norm{\Lambda}_2^{2K}} >  \bigg(1-\frac{1}{\norm{\Lambda}_2}\bigg)^{2K}  $ for $K= \mathcal{O}(\log(\epsilon^{-1}))$, $\epsilon \ll 1$ yields:
\begin{align}
1 - \frac{1}{\norm{\Lambda}_2^{2K}}    &> 1-\sum\limits_{i \in \mathcal{N}_{US}} \abs{\theta_i}^2  +  K \epsilon\norm{\Lambda}^{-1}_2 \Gamma  \label{exittime7ab}\\
\impliedby \bigg(1-\frac{1}{\norm{\Lambda}_2}\bigg)^{2K}   &> 1-\sum\limits_{i \in \mathcal{N}_{US}} \abs{\theta_i}^2  +  K \epsilon\norm{\Lambda}^{-1}_2 \Gamma 
   \label{exittime7a}
\end{align}
where if $K=K_{exit}$ satisfies \eqref{exittime7a} then $K=K_{exit}$ will satisfy \eqref{exittime7ab}. Next observe that \eqref{exittime7a} is a transcendental inequality of the form $ q^x > ax + b$ whose solution is $ x < -\frac{W(-\frac{\log q}{a}q^{-\frac{b}{a}})}{\log q} - \frac{b}{a} \leq \frac{\log\bigg(\frac{\log (q^{-1})}{a}\bigg)}{\log(q^{-1})}$ for $q<1$ where $W(.)$ is the Lambert W function. Then for $x =2K_{exit}$, $q = \bigg(1-\frac{1}{\norm{\Lambda}_2}\bigg)$, $a = \frac{ \epsilon\norm{\Lambda}^{-1}_2 \Gamma}{2}$, $b= 1-\sum\limits_{i \in \mathcal{N}_{US}} \abs{\theta_i}^2$ we get:
\begin{align}
  2 K_{exit} & \lessapprox   \frac{\log\bigg(\frac{2\log \bigg(\bigg(1-\frac{1}{\norm{\Lambda}_2}\bigg)^{-1}\bigg)}{\epsilon\norm{\Lambda}^{-1}_2 \Gamma}\bigg)}{\log\bigg(\bigg(1-\frac{1}{\norm{\Lambda}_2}\bigg)^{-1}\bigg)}. \label{exittime8a}
\end{align}
\textbf{A sufficient condition for the two cases:\\}
In particular, the linear exit time bound of \eqref{exittime8} or \eqref{exittime8a} will always hold true whenever the initial projection value satisfies $\sum\limits_{i \in \mathcal{N}_{US}} \abs{\theta_i}^2 \geq \sigma$ for some constant order term\footnote{The term $\sigma$ is of constant order with respect to $\epsilon$ and also independent of $\Lambda$.} $0 \ll \sigma < 1$. This can be readily checked by substituting $\sum\limits_{i \in \mathcal{N}_{US}} \abs{\theta_i}^2 \geq \sigma $ in \eqref{exittime7}, \eqref{exittime7ab}. In other words, if $\u_0 \in \mathcal{K}_{\sigma} \cap \mathcal{B}_{\epsilon}(\boldsymbol{0})$ where $\mathcal{K}_{\sigma} = \{ \z \in \mathbb{C}^d \hspace{0.1cm}\vert \hspace{0.1cm} \frac{\pi_{\mathcal{E}_{US}}(\z)}{\norm{\z}} \geq \sigma^{\frac{1}{2}}; \hspace{0.1cm} \z \neq \boldsymbol{0} \} $ is the double cone containing the unstable subspace $\mathcal{E}_{US} $, then the approximate trajectory $\{\tilde{\u}_k\}$ exits $ \mathcal{B}_{\epsilon}(\boldsymbol{0})$ in at most linear time. 

Finally if the relative error condition \eqref{relerror} is satisfied by the exit time bounds \eqref{exittime8} and \eqref{exittime8a}, we can write $\tilde{\u}_K = {\u}_K + \v$ for some vector $\v$. Then $\norm{\v} \leq  \mathcal{R} \norm{\u_K}$ from \eqref{relerror} and therefore $\v = o(\norm{\u_K})$ because as $\epsilon \to 0$ we have $ \norm{\u_K} \to 0$ but then $\frac{\norm{\v}}{\norm{\u_K}} \leq  \mathcal{R}\xrightarrow{\epsilon \to 0} 0 $. Hence, {we have that $$ K_{exit} =  \inf_{k>0}\{k \vert \norm{\tilde{\u}_k} > \epsilon\}\geq  \inf_{k>0}\bigg\{k \bigg\vert \norm{{\u}_k} > \frac{\epsilon}{1 + {\tilde{\gamma}(\epsilon)}}\bigg\}$$ {for some scalar ${\tilde{\gamma}(\epsilon)}\geq 0$ with ${\tilde{\gamma}(\epsilon)}=o(1)$. Here, we used} the definition of infimum and the fact that $\norm{\tilde{\u}_{K_{exit}}} > \epsilon $ implies $ \norm{{\u}_{K_{exit}}} + o(\norm{{\u}_{K_{exit}}}) > \epsilon$ or equivalently $\norm{{\u}_{K_{exit}}} > \frac{\epsilon}{1 + {\tilde{\gamma}(\epsilon)}}$ for some ${\tilde{\gamma}(\epsilon)}=o(1)$.}  
\end{proof}  

\subsection{Relative error bound}
\subsubsection{Lemma \ref{lemmarelerrorexact}}
\begin{proof}
We split the proof of Lemma \ref{lemmarelerrorexact} in following two cases.\\
\textbf{Case 1 : } $\inf_{i  \in \mathcal{N}_{US}}\abs{z_i} <  \norm{\Lambda}_2 $\\
Recall that $\u_0 = \epsilon\sum\limits_{i=1}^{d} \theta_i \e_i $, then from \eqref{exittimetailerror} we have:
\begin{align}
   \u_K &=  {\bigg( \Lambda^K  +  \underbrace{\epsilon\sum\limits_{{r=1}}^{K}\Lambda^{r-1} (o(1) +\mathbf{R}(\u_{K-r})) \Lambda^{K-r}}_{\mathcal{O}(\norm{\Lambda}_2^{K-1} (K\epsilon))} + {\mathcal{O}(\norm{\Lambda}_2^{K-2} (K\epsilon)^2)}\bigg)\u_0} \\
   \implies \u_K &= { \Lambda^K\u_0  +   {\mathcal{O}(\norm{\Lambda}_2^{K-1} (K\epsilon^2)})} \\
    \implies \norm{\u_K} &\geq { \norm{\Lambda^K\u_0} -   {\mathcal{O}(\norm{\Lambda}_2^{K-1} (K\epsilon^2)})} \\
    & = { \norm{\epsilon\sum\limits_{i=1}^{N} z_i^K\theta_i \e_i} -   {\mathcal{O}(\norm{\Lambda}_2^{K-1} (K\epsilon^2)})} \\
    & \geq \epsilon \bigg(\inf_{i  \in \mathcal{N}_{US}}\abs{z_i}\bigg)^K\sqrt{\sum\limits_{i \in \mathcal{N}_{US}} \abs{\theta_i}^2} -   {\mathcal{O}(\norm{\Lambda}_2^{K} (K\epsilon^2)}). \label{relerr1}
\end{align}
Also, from \eqref{exittimetailerror} and \eqref{exittime2} we have:
\begin{align}
    \norm{{\u}_K - \tilde{\u}_K } & =  {\mathcal{O}(\norm{\Lambda}_2^{K-2} (K\epsilon)^2)}\norm{\u_0} \leq  {\mathcal{O}(\norm{\Lambda}_2^{K} (K\epsilon)^2\epsilon)}. \label{relerr2}
\end{align}
Then using \eqref{relerr1} and \eqref{relerr2} we obtain:
\begin{align}
    \frac{\norm{{\u}_K - \tilde{\u}_K }}{\norm{{\u}_K  }} & \leq \frac{{\mathcal{O}(\norm{\Lambda}_2^{K} (K\epsilon)^2\epsilon)}}{\epsilon \bigg(\inf_{i  \in \mathcal{N}_{US}}\abs{z_i}\bigg)^K\sqrt{\sum\limits_{i \in \mathcal{N}_{US}} \abs{\theta_i}^2} -   {\mathcal{O}(\norm{\Lambda}_2^{K} (K\epsilon^2)})} \\
    & = \frac{{\mathcal{O}\bigg(\frac{\norm{\Lambda}_2^{K}}{\bigg(\inf_{i  \in \mathcal{N}_{US}}\abs{z_i}\bigg)^K} (K\epsilon)^2\bigg)}}{ \sqrt{\sum\limits_{i \in \mathcal{N}_{US}} \abs{\theta_i}^2} -   {\mathcal{O}\bigg(\frac{\norm{\Lambda}_2^{K}}{\bigg(\inf_{i  \in \mathcal{N}_{US}}\abs{z_i}\bigg)^K} (K\epsilon)}\bigg)} \label{relerr3a}\\
    & \leq \frac{{\mathcal{O}\bigg(\frac{1}{\sqrt{\epsilon}} (\log(\epsilon^{-1})\epsilon)^2\bigg)}}{ \sqrt{\sum\limits_{i \in \mathcal{N}_{US}} \abs{\theta_i}^2} -   {\mathcal{O}\bigg(\frac{1}{\sqrt{\epsilon}} (\log(\epsilon^{-1})\epsilon)}\bigg)} \label{relerr3}
\end{align}
where we used the fact that $K \leq K_{exit}$ and substituted the bound on $ K_{exit}$ (the exit time for the approximate trajectory) from Theorem \ref{exittimethm1} for the case $\inf_{i  \in \mathcal{N}_{US}}\abs{z_i} <  \norm{\Lambda}_2$ in the last step, i.e., ${\mathcal{O}\bigg(\frac{\norm{\Lambda}_2^{K_{exit}}}{\bigg(\inf_{i  \in \mathcal{N}_{US}}\abs{z_i}\bigg)^{K_{exit}}} \bigg)}  = \mathcal{O}\bigg(\frac{1}{\sqrt{\epsilon}}\bigg)$.  Now if we have $\sqrt{\sum\limits_{i \in \mathcal{N}_{US}} \abs{\theta_i}^2} >   {\mathcal{O}\bigg(\frac{1}{\sqrt{\epsilon}} (\log(\epsilon^{-1})\epsilon)}\bigg) $ then from \eqref{relerr3} and \eqref{relerror}:
\begin{align}
  \mathcal{R} =  \sup_{0 \leq K \leq K_{exit}} \frac{\norm{{\u}_K - \tilde{\u}_K }}{\norm{{\u}_K  }}  & \leq \frac{{\mathcal{O}\bigg(\frac{1}{\sqrt{\epsilon}} (\log(\epsilon^{-1})\epsilon)^2\bigg)}}{ \sqrt{\sum\limits_{i \in \mathcal{N}_{US}} \abs{\theta_i}^2} -   {\mathcal{O}\bigg(\frac{1}{\sqrt{\epsilon}} (\log(\epsilon^{-1})\epsilon)}\bigg)} \xrightarrow{\epsilon \to 0} 0.
\end{align}
From Theorem \ref{exittimethm1} since $\u_0 \in \mathcal{K}_{\sigma} \cap \mathcal{B}_{\epsilon}(\boldsymbol{0})$ where $\mathcal{K}_{\sigma} = \{ \z \in \mathbb{C}^d \hspace{0.1cm}\vert \hspace{0.1cm} \frac{\pi_{\mathcal{E}_{US}}(\z)}{\norm{\z}} \geq \sigma^{\frac{1}{2}}; \hspace{0.1cm} \z \neq \boldsymbol{0} \} $ is the double cone containing the unstable subspace $\mathcal{E}_{US} $ and $0 \ll \sigma < 1$, the condition $\frac{\pi_{\mathcal{E}_{US}}(\u_0)}{\norm{\u_0}}  = \sqrt{\sum\limits_{i \in \mathcal{N}_{US}} \abs{\theta_i}^2} >   {\mathcal{O}\bigg(\frac{1}{\sqrt{\epsilon}} (\log(\epsilon^{-1})\epsilon)}\bigg) $ will be trivially satisfied for any sufficiently small $\epsilon$.\\
\textbf{Case 2 : } $\inf_{i  \in \mathcal{N}_{US}}\abs{z_i} =  \norm{\Lambda}_2 $\\
For this case the relative error bound in \eqref{relerr3a} becomes:
\begin{align}
     \frac{\norm{{\u}_K - \tilde{\u}_K }}{\norm{{\u}_K  }} & \leq  \frac{{\mathcal{O}\bigg(\frac{\norm{\Lambda}_2^{K}}{\bigg(\inf_{i  \in \mathcal{N}_{US}}\abs{z_i}\bigg)^K} (K\epsilon)^2\bigg)}}{ \sqrt{\sum\limits_{i \in \mathcal{N}_{US}} \abs{\theta_i}^2} -   {\mathcal{O}\bigg(\frac{\norm{\Lambda}_2^{K}}{\bigg(\inf_{i  \in \mathcal{N}_{US}}\abs{z_i}\bigg)^K} (K\epsilon)}\bigg)} = \frac{\mathcal{O}((K\epsilon)^2) }{ \sqrt{\sum\limits_{i \in \mathcal{N}_{US}} \abs{\theta_i}^2} -   {\mathcal{O}( K\epsilon)}}
\end{align}
which goes to $0$ as $\epsilon \to 0$ since $K < K_{exit} = \mathcal{O}(\log(\epsilon^{-1}))$ provided $ \sqrt{\sum\limits_{i \in \mathcal{N}_{US}} \abs{\theta_i}^2} >   {\mathcal{O}( \epsilon\log(\epsilon^{-1})}) $. But this condition is automatically satisfied by the fact that $\u_0 \in \mathcal{K}_{\sigma} \cap \mathcal{B}_{\epsilon}(\boldsymbol{0})$ and so $\frac{\pi_{\mathcal{E}_{US}}(\u_0)}{\norm{\u_0}}  = \sqrt{\sum\limits_{i \in \mathcal{N}_{US}} \abs{\theta_i}^2} >    {\mathcal{O}( \epsilon\log(\epsilon^{-1})})  $ for any sufficiently small $\epsilon$.  

Since the relative error condition \eqref{relerror} for the two cases is satisfied, the conclusion that exit times of the approximate and the exact trajectories are approximately equal then follows directly from Theorem \ref{exittimethm1}.

\end{proof}

\subsection{Conditions on initial projections for the linear exit time bound from Theorem \ref{exittimethm1}}
\subsubsection{ Lemma \ref{lemmaprojectioncondition}}

\begin{proof}
\textbf{Case 1 of ${\inf_{i  \in \mathcal{N}_{US}}\abs{z_i}} < {\norm{\Lambda}_2} $:\\}
Observe that in \eqref{exittime7}, the left-hand side is a decreasing function of $K$ while the right-hand side is increasing with $K$. Therefore for the particular upper bound \eqref{exittime8} to exist, we must necessarily have that:
\begin{align}
    \frac{\inf_{i  \in \mathcal{N}_{US}}\abs{z_i}^{2K}}{\norm{\Lambda}_2^{2K}}\sum\limits_{i \in \mathcal{N}_{US}} \abs{\theta_i}^2 \bigg\vert_{K=1}  &>    K \epsilon\norm{\Lambda}^{-1}_2 \Gamma\bigg\vert_{K=1} +\mathcal{O}(\epsilon^{b_1}) \\
    \implies \frac{\inf_{i  \in \mathcal{N}_{US}}\abs{z_i}}{\norm{\Lambda}_2}\sum\limits_{i \in \mathcal{N}_{US}} \abs{\theta_i}^2 &>   \epsilon\norm{\Lambda}^{-1}_2 \Gamma +\mathcal{O}(\epsilon^{b_1}) > \epsilon\norm{\Lambda}^{-1}_2 \Gamma \\
    \implies \sum\limits_{i \in \mathcal{N}_{US}} \abs{\theta_i}^2 &> \frac{\epsilon \Gamma}{\inf_{i  \in \mathcal{N}_{US}}\abs{z_i}}. \label{necessary_case1}
\end{align}
Note that the condition \eqref{necessary_case1} is only a necessary condition for the existence of linear exit time bound \eqref{exittime8} and not sufficient. Obtaining a minimal sufficient condition, on the other hand, is much more harder for this case (see remark \ref{remarkcrucial}).
\\
\textbf{Case 2 of ${\inf_{i  \in \mathcal{N}_{US}}\abs{z_i}} = {\norm{\Lambda}_2} $:\\}
In this case we can derive the minimal sufficient condition (minimal in the sense of present proof technique) for which the exit time bound \eqref{exittime8a} will hold. Recall from \eqref{exittime7abc} that we had the following inequality:
\begin{align}
    \sum\limits_{i \in \mathcal{N}_{US}} \abs{\theta_i}^2  -  K \epsilon\norm{\Lambda}^{-1}_2 \Gamma  &> \frac{1}{\norm{\Lambda}_2^{2K}} \label{exittimeabc}
\end{align}
where we can substitute $g_1(x) = \sum\limits_{i \in \mathcal{N}_{US}} \abs{\theta_i}^2  -  x \epsilon\norm{\Lambda}^{-1}_2 \Gamma  $, $g_2(x) = \frac{1}{\norm{\Lambda}_2^{2x}}$ for $x \in \mathbb{R}$ and by restricting $x$ to $\mathbb{Z}$ we get $g_1(x)\vert_{x=K} = \sum\limits_{i \in \mathcal{N}_{US}} \abs{\theta_i}^2  -  K \epsilon\norm{\Lambda}^{-1}_2 \Gamma$, $ g_2(x)\vert_{x=K} =\frac{1}{\norm{\Lambda}_2^{2K}}$ as functions of integer argument $K$. Now $g_1(x)$ is a line with negative slope and positive intercept whereas $g_2(x)$ is a decaying exponential. Hence there can be three cases:
\begin{enumerate}
   \item Graphs of $g_1$ and $g_2$ intersect at two points in the first quadrant which will yield two exit time solutions for \eqref{exittimeabc}. 
    \item  Graph of $g_1$ always stays below the graph of $g_2$ and so there is no exit time solution for \eqref{exittimeabc}.
    \item Graph of $g_1$ just touches the graph of $g_2$ at a single point in first quadrant and so there is just one exit time solution for \eqref{exittimeabc}.
\end{enumerate}
Clearly, the third case is the minimal requirement (minimal in the sense of our proof technique) for the existence of solution for \eqref{exittimeabc}. Since $g_1$ just touches the graph of $g_2$ in this case, $g_1$ is tangent to $g_2$ and has a slope of $ - \epsilon\norm{\Lambda}^{-1}_2 \Gamma$. Computing the derivative of $g_2$ we get $\frac{\partial g_2}{\partial x} =   \frac{-2\log (\norm{\Lambda}_2) }{\norm{\Lambda}_2^{2x}} $. This slope must be equal to $  - \epsilon\norm{\Lambda}^{-1}_2 \Gamma$ which after solving yields:
\begin{align}
       \frac{-2\log (\norm{\Lambda}_2) }{\norm{\Lambda}_2^{2x}}  &= - \epsilon\norm{\Lambda}^{-1}_2 \Gamma \\
       \implies x &=  \frac{\log \bigg(\frac{2\log (\norm{\Lambda}_2)}{\epsilon\norm{\Lambda}^{-1}_2 \Gamma}\bigg)}{2\log\norm{\Lambda}_2 }. 
\end{align}
Now for this $x$ we have $g_2(x) = \frac{1}{\norm{\Lambda}_2^{2x}} = \frac{\epsilon\norm{\Lambda}^{-1}_2 \Gamma}{2\log (\norm{\Lambda}_2) }$ and so the equation of this tangent line to $g_2(x)$ at $x= \frac{\log \bigg(\frac{2\log (\norm{\Lambda}_2)}{\epsilon\norm{\Lambda}^{-1}_2 \Gamma}\bigg)}{2\log\norm{\Lambda}_2 }$ is $y - \frac{\epsilon\norm{\Lambda}^{-1}_2 \Gamma}{2\log (\norm{\Lambda}_2) } =  - \epsilon\norm{\Lambda}^{-1}_2 \Gamma \bigg(x -  \frac{\log \bigg(\frac{2\log (\norm{\Lambda}_2)}{\epsilon\norm{\Lambda}^{-1}_2 \Gamma}\bigg)}{2\log\norm{\Lambda}_2 }\bigg) $. At $x= 0$ the intercept is $y = \frac{\epsilon\norm{\Lambda}^{-1}_2 \Gamma}{2\log (\norm{\Lambda}_2) } \bigg(1 +  \log \bigg(\frac{2\log (\norm{\Lambda}_2)}{\epsilon\norm{\Lambda}^{-1}_2 \Gamma}\bigg) \bigg) $ which must be equal to the intercept of line $g_1$ which is $y= \sum\limits_{i \in \mathcal{N}_{US}} \abs{\theta_i}^2  $. Then the minimal sufficient condition so that \eqref{exittimeabc} has at least one solution is:
\begin{align}
    \sum\limits_{i \in \mathcal{N}_{US}} \abs{\theta_i}^2 \geq \frac{\epsilon\norm{\Lambda}^{-1}_2 \Gamma}{2\log (\norm{\Lambda}_2) } \bigg(1 +  \log \bigg(\frac{2\log (\norm{\Lambda}_2)}{\epsilon\norm{\Lambda}^{-1}_2 \Gamma}\bigg) \bigg) = \Theta \bigg(  \frac{\epsilon}{\norm{\Lambda}_2\log (\norm{\Lambda}_2)}\log \bigg( \frac{\norm{\Lambda}_2\log (\norm{\Lambda}_2)}{\epsilon}\bigg)\bigg). 
\end{align}
It is easy to check that $\Theta \bigg(  \frac{\epsilon}{\norm{\Lambda}_2\log (\norm{\Lambda}_2)}\log \bigg( \frac{\norm{\Lambda}_2\log (\norm{\Lambda}_2)}{\epsilon}\bigg)\bigg) \xrightarrow{\epsilon \to 0} 0$.
\end{proof}  

\begin{rema}\label{remarkcrucial}
Note that in the case of ${\inf_{i  \in \mathcal{N}_{US}}\abs{z_i}} < {\norm{\Lambda}_2} $ we cannot easily evaluate the sufficient condition. The reason behind this is the fact that in the inequality \eqref{exittime7} given by $\frac{\inf_{i  \in \mathcal{N}_{US}}\abs{z_i}^{2K}}{\norm{\Lambda}_2^{2K}}\sum\limits_{i \in \mathcal{N}_{US}} \abs{\theta_i}^2   >    K \epsilon\norm{\Lambda}^{-1}_2 \Gamma +\mathcal{O}(\epsilon^{b_1})$, the left hand side is a decaying exponential in $K$ while the right hand side is a straight line in $K$ with positive slope. The two curves can only intersect one another and so there always exists a solution $K_{exit}$. However obtaining a lower bound on $\sum\limits_{i \in \mathcal{N}_{US}} \abs{\theta_i}^2 $ so that the solution is linear time, i.e., $K_{exit} = \mathcal{O}(\log (\epsilon^{-1}))$, is not straightforward. In the case of ${\inf_{i  \in \mathcal{N}_{US}}\abs{z_i}} = {\norm{\Lambda}_2} $ we could easily evaluate the sufficient condition due to the fact that decaying exponential and the straight line could possibly touch each other at a point $ K = \mathcal{O}(\log (\epsilon^{-1}))$ which was enough to obtain a lower bound on $\sum\limits_{i \in \mathcal{N}_{US}} \abs{\theta_i}^2 $.
\end{rema}

\subsection{Upper bound on the radius $\xi$ for monotonic trajectories}
\subsubsection{Lemma \ref{xilemma}}
\begin{proof}
From \eqref{exittime1} we have that:
\begin{align}
     \u_{k+1} &= \Lambda \u_k + \B_k \u_k + \norm{\u_k} \mathbf{P}(\u_k) \u_k . \label{fr_1a}
\end{align}
Let us assume that $ \norm{\u_{k+1}} \geq \norm{\u_k}$ for some $k=K$ and that $ \sup_{k\geq K}\norm{\B_k} = o(\norm{\u_K})$ where $K$ is sufficiently large. Since $ \norm{\u_{K+1}} \geq \norm{\u_K}$ we also have $  \sup_{k\geq K}\norm{\B_k} = o(\norm{\u_{K+1}})$ for $k=K+1$. Then we need to find a radius $\xi$ for which $ \norm{\u_{k+1}} > \norm{\u_k} $ holds for all $k>K$ as long as $\norm{\u_k}\leq \xi$. To do so first observe that for $k= K$, after taking norm squared on both sides of \eqref{fr_1a} we get:
\begin{align}
    \norm{\u_{K+1}}^2 & = \u_{K+1}^H \u_{K+1} = \u_K^H \Lambda^H\Lambda\u_K +  \varsigma_K \label{tempm1}
\end{align}
where\footnote{Note that `$Re$' operator returns the real part of the argument.} $ \varsigma_K = \underbrace{\norm{\u_K}}_{\leq \xi}\bigg( 2 Re(\u_K^H \Lambda^H(o(1)+\mathbf{P}(\u_K)) \u_K) + \underbrace{\norm{\u_K}}_{\leq \xi} \u_K^H  (o(1)+\mathbf{P}(\u_K))^H (o(1)+\mathbf{P}(\u_K))\u_K  \bigg)$ and we have that $ \abs{\varsigma_K} \leq (o(1)+\gamma)\xi\norm{\u_K}^2(2\norm{\Lambda}_2   + (o(1)+\gamma)\xi  )$ where $ \sup_{\norm{\u_{k}}\leq \xi}\norm{\mathbf{P}(\u_k)} \leq \gamma$. Here, the $o(1)$ terms appear as a result of substituting $\norm{\B_K} = o(\norm{\u_K})$. {Next, we can assume without loss of generality that the trajectory $\{\u_k\}_k$ remains bounded within some compact set $D \supseteq \{\u_k\}_k$. Therefore, we can use the upper bound $o(1) \leq \text{diam}(D)$.} Further simplifying the bound on $\abs{\varsigma_K}$, we can write \( \abs{\varsigma_K} \leq c_1 \xi^2 \norm{\u_K}^2 \) for some constant $c_1$ independent of $\norm{\u_K}${, where \( c_1 := (\text{diam}(D) + \gamma)\left( \frac{2\norm{\Lambda}_2}{\xi} + \text{diam}(D) + \gamma \right) \). The constant $c_1$ satisfies
\begin{align}
    c_1 \xi^2 = a_0 \xi + a_1 \xi^2, \label{revc1boundv}
\end{align}
where \( a_0 := 2(\text{diam}(D) + \gamma)\norm{\Lambda}_2 \) and \( a_1 := (\text{diam}(D) + \gamma)^2 \). Note that by choosing the compact set $D$ sufficiently large so that \( \text{diam}(D) \gg \norm{\Lambda}_2 \), we can ensure that \( a_1 \gg a_0 \).} 
Next, for \( k = K + 1 \), using \eqref{tempm1} and the bound on \( \abs{\varsigma_K} \), we have:
{
\begin{align*}
    \abs{\varsigma_{K+1}} &\leq c_1 \xi^2\norm{\u_{K+1}}^2 \leq  c_1 \xi^2  \norm{\u_K}^2 \bigg( \frac{\u_K^H \Lambda^H\Lambda\u_K}{\norm{\u_K}^2} + \frac{\abs{\varsigma_K}}{\norm{\u_K}^2} \bigg) \leq  c_1 \xi^2  \norm{\u_K}^2 \bigg( \norm{\Lambda}_2 ^2 + c_1 \xi^2 \bigg) = c_2 \xi^2  \norm{\u_K}^2
\end{align*}
for some constant $c_2 =  \bigg( \norm{\Lambda}_2 ^2 + c_1 \xi^2 \bigg)c_1$. Further, using the bound $\xi \leq \text{diam}(D)$ we can upper bound $c_2$ in terms of $c_1$ as follows:
\begin{align}
    c_2 &= \bigg( \norm{\Lambda}_2 ^2 + c_1 \xi^2 \bigg)c_1 \\
        &=  \bigg( \norm{\Lambda}_2 ^2 + \bigg((\text{diam}(D) + \gamma)\bigg( \frac{2\norm{\Lambda}_2}{\xi} + (\text{diam}(D) + \gamma)\bigg)\bigg) \xi^2 \bigg)c_1 \\
        &  = \bigg( \norm{\Lambda}_2 ^2 + \bigg((\text{diam}(D) + \gamma)\bigg( {2\norm{\Lambda}_2} + (\text{diam}(D) + \gamma)\xi\bigg)\bigg) \xi \bigg)c_1 \\
        & \leq \bigg( \norm{\Lambda}_2 ^2 + \bigg((\text{diam}(D) + \gamma)\bigg( {2\norm{\Lambda}_2} + (\text{diam}(D) + \gamma)\text{diam}(D) \bigg)\bigg) \text{diam}(D)  \bigg)c_1 \\
         & \leq \bigg( \norm{\Lambda}_2 ^2 + 2\norm{\Lambda}_2(\text{diam}(D) + \gamma)\text{diam}(D) + (\text{diam}(D)(\text{diam}(D) + \gamma))^2   \bigg)c_1 = a_2 c_1\label{revc2boundv}
\end{align}
where $a_2 := \bigg( \norm{\Lambda}_2  + (\text{diam}(D) + \gamma)\text{diam}(D) \bigg)^2$.}
  
Next, using the assumption $ \norm{\u_{K+1}} \geq \norm{\u_K}$, we also have that $\u_{K+1}^H \u_{K+1} = \u_K^H \Lambda^H\Lambda\u_K +  \varsigma_K \geq \u_{K}^H \u_{K}$, which simplifies to:
\begin{align}
    \u_K^H \bigg(\Lambda^H\Lambda - \mathbf{I}\bigg)\u_K +  \varsigma_K &\geq 0. \label{monotonic1}
\end{align}
In order to establish monotonicity of the sequence $\{\norm{\u_K}\}$ in the expansion phase it suffices to show that $\norm{\u_{K+2}} > \norm{\u_{K+1}} $ holds given $\norm{\u_{K+1}}\geq \norm{\u_{K}} $. Then using induction, monotonicity of the sequence $\{\norm{\u_k}\}_{k\geq K}$ in the expansion phase can be concluded.
Now simplifying the claim which needs to be proved, i.e., $\norm{\u_{K+2}} > \norm{\u_{K+1}} $ and using \eqref{tempm1} for $k=K+1$ we get:
\begin{align}
   \u_{K+2}^H \u_{K+2} &> \u_{K+1}^H \u_{K+1} \\
\iff    \u_{K+1}^H \bigg(\Lambda^H\Lambda - \mathbf{I}\bigg)\u_{K+1} +  \varsigma_{K+1} &> 0 \\
\iff \u_{K}^H \Lambda^H \bigg(\Lambda^H\Lambda - \mathbf{I}\bigg)\Lambda\u_{K} + \vartheta_K + \varsigma_{K+1} &> 0 \label{monotonic2}
\end{align}
where
\begin{align*}
    \vartheta_K =& \underbrace{\norm{\u_K}}_{\leq \xi}\bigg( 2 Re(\u_K^H \Lambda^H\bigg(\Lambda^H\Lambda - \mathbf{I}\bigg)(o(1)+\mathbf{P}(\u_K)) \u_K) + \\ 
    &\qquad \underbrace{\norm{\u_K}}_{\leq \xi}  \u_K^H  (o(1)+\mathbf{P}(\u_K))^H \bigg(\Lambda^H\Lambda - \mathbf{I}\bigg)(o(1)+\mathbf{P}(\u_K))\u_K  \bigg)
\end{align*}
and we have that
 \begin{align*}
    \abs{\vartheta_K} &\leq \norm{\Lambda^H\Lambda - \mathbf{I}}_2 (o(1)+\gamma)\xi\norm{\u_K}^2(2\norm{\Lambda}_2   + (o(1)+\gamma)\xi  ) \\ 
    & \leq   (o(1)+\gamma)\xi\norm{\u_K}^2(1+ \norm{\Lambda}_2 ^2)(2\norm{\Lambda}_2   + (o(1)+\gamma)\xi  ) \\ &\leq c_3 \xi^2 \norm{\u_K}^2 
 \end{align*}
for some constant {$c_3$ that is given by
 \begin{align}
     c_3  &:= (\text{diam}(D) + \gamma)(1+ \norm{\Lambda}_2 ^2)\bigg( \frac{2\norm{\Lambda}_2}{\xi} + (\text{diam}(D) + \gamma)\bigg) \\
     &=(1+ \norm{\Lambda}_2 ^2) c_1 = a_3 c_1 \label{revc3boundv}
 \end{align}
where \( a_3 := 1 + \norm{\Lambda}_2^2 \), and we used the bound \( o(1) \leq \text{diam}(D) \) along with the expression \( c_1 = (\text{diam}(D) + \gamma)\left( \frac{2\norm{\Lambda}_2}{\xi} + \text{diam}(D) + \gamma \right) \) in the last step.}

Further simplifying \eqref{monotonic2} by using the fact that $\Lambda$ is a diagonal matrix we get:
\begin{align}
     \u_{K}^H \Lambda^H\Lambda \bigg(\Lambda^H\Lambda - \mathbf{I}\bigg)\u_{K} + \vartheta_K + \varsigma_{K+1} &> 0 \\
     \iff \u_{K}^H  \bigg(\Lambda^H\Lambda - \mathbf{I}\bigg)^2\u_{K} + \u_{K}^H  \bigg(\Lambda^H\Lambda - \mathbf{I}\bigg)\u_{K}  &> -( \vartheta_K + \varsigma_{K+1}) \\
      \iff \underbrace{\u_{K}^H  \bigg(\Lambda^H\Lambda - \mathbf{I}\bigg)^2\u_{K}}_{>0 \text{ when } det(\Lambda^H\Lambda -\mathbf{I}) \neq 0 } + \underbrace{\u_{K}^H  \bigg(\Lambda^H\Lambda - \mathbf{I}\bigg)\u_{K} + \varsigma_K}_{\geq 0 \text{ from } \eqref{monotonic1}} &> -( \vartheta_K + \varsigma_{K+1} - \varsigma_K)
\end{align}
which will hold true if we have the condition:
{
\begin{align}
    \u_{K}^H  \bigg(\Lambda^H\Lambda - \mathbf{I}\bigg)^2\u_{K} &> (1+ a_2+a_3)c_1\xi^2  \norm{\u_K}^2 \!\!\!\!\!\!\!\! \underbrace{\geq}_{\text{using } \eqref{revc2boundv}, \eqref{revc3boundv}} \!\!\!\!\!\!\!\! (c_1+c_2 + c_3)\xi^2 \norm{\u_K}^2 \geq ( \abs{\vartheta_K} + \abs{\varsigma_{K+1}} + \abs{\varsigma_K}) \\
    \iff \frac{\u_{K}^H  \bigg(\Lambda^H\Lambda - \mathbf{I}\bigg)^2\u_{K}}{\norm{\u_K}^2} &> (1+ a_2+a_3)c_1\xi^2 \underbrace{=}_{\text{using } \eqref{revc1boundv}} 
 (1+ a_2+a_3) (a_0 \xi + a_1 \xi^2 ),\label{monotonic1a}
\end{align}}
where we used the bounds $\abs{\varsigma_K}  \leq c_1 \xi^2 \norm{\u_K}^2$, $\abs{\varsigma_{K+1}}  \leq c_2 \xi^2 \norm{\u_K}^2$ and $\abs{\vartheta_K}  \leq c_3 \xi^2 \norm{\u_K}^2$ in the first step.
Next, since $  \frac{\u_{K}^H  \bigg(\Lambda^H\Lambda - \mathbf{I}\bigg)^2\u_{K}}{\norm{\u_K}^2} \geq \norm{\bigg(\Lambda^H\Lambda - \mathbf{I}\bigg)^{-2}}_2^{-1}$, so \eqref{monotonic1a} will hold whenever:
{
\begin{align}
 \norm{\bigg(\Lambda^H\Lambda - \mathbf{I}\bigg)^{-2}}_2^{-1} &>  (1+ a_2+a_3) (a_0 \xi + a_1 \xi^2 ) \\
 \iff \frac{1}{(1+ a_2+a_3)}\norm{\bigg(\Lambda^H\Lambda - \mathbf{I}\bigg)^{-2}}_2^{-1} &> a_0 \xi + a_1 \xi^2
\end{align}
where the constants $a_0,a_1,a_2,a_3$ are positive, bounded and $a_1 \gg a_0$ for sufficiently large compact set ${D} \supseteq \{\u_k\}_k$. Now let \( (\Lambda^H \Lambda - \mathbf{I}) \) be invertible, and define \( a_4 := \left\| \left( \Lambda^H \Lambda - \mathbf{I} \right)^{-2} \right\|_2 \). Then, solving the above quadratic inequality gives the bound:
\begin{align}
  \xi &<  \frac{\sqrt{ a_0^2 (1+ a_2+a_3)a_4 + 4a_1} - a_0 \sqrt{(1+ a_2+a_3)a_4 }}{2a_1}   \norm{\bigg(\Lambda^H\Lambda - \mathbf{I}\bigg)^{-2}}_2^{-\frac{1}{2}} \\
  \iff  \xi &< { C\norm{\bigg(\Lambda^H\Lambda - \mathbf{I}\bigg)^{-2}}_2^{-\frac{1}{2}} }, 
    \label{monotonic3}
\end{align}
where we used the fact that for \( a, b, c > 0 \), the inequality \( ax^2 + bx < c \) is satisfied for all \( 0 < x < \frac{ \sqrt{\frac{b^2}{c} + 4a} - \frac{b}{\sqrt{c}} }{2a} \sqrt{c} = \frac{-b + \sqrt{b^2 + 4ac}}{2a}, \) and \( C := \frac{ \sqrt{ a_0^2 (1 + a_2 + a_3)a_4 + 4a_1 } - a_0 \sqrt{(1 + a_2 + a_3)a_4} }{2a_1} \) is a positive absolute constant, assuming \( (\Lambda^H \Lambda - \mathbf{I}) \) is invertible. Now, if the eigenvalue spectrum of \( \Lambda^H \Lambda \) does not contain \(1\), then the bound from \eqref{monotonic3} holds with \( \xi > 0 \). Hence, given that \( \norm{\u_{K+1}} \geq \norm{\u_K} \), the trajectory \( \{\u_k\} \) is monotonic for all \( k > K \), provided \( \norm{\u_k} < \xi \), where \( \xi \) is upper bounded as in \eqref{monotonic3}. Moreover, if the eigenvalues of \( \Lambda^H \Lambda \) are not too close to 1, i.e., if \( \min_j |\lambda_j(\Lambda^H \Lambda) - 1| \gg \epsilon \) for some \( 0 < \epsilon \ll \min\{1, C\} \), then the upper bound on \( \xi \) will be of constant order with respect to \( \epsilon \).}
\end{proof}  

\subsection{Upper bounds on the matrix norm $\norm{\Lambda}_2$ and the perturbation parameter~$\Gamma$}
\subsubsection{Lemma \ref{lambdaeig_lemma}}
\begin{proof}
Notice that $\norm{\Lambda}_2$ comes from the unstable subspace of $\Lambda$ hence we can restrict ourselves to only those eigenvalues of $\Lambda$ whose magnitude is greater than $1$. First using \eqref{realeigen} we show that the eigenvalues $z_i = \frac{(1+\beta)\lambda - \sqrt{(1+\beta)^2\lambda^2-4\beta \lambda }}{2}$ for the case $\lambda > 1  \geq \frac{4\beta}{(1+\beta)^2} $ are always smaller than $1$. For that to happen, we require:
\begin{align}
     \frac{(1+\beta)\lambda - \sqrt{(1+\beta)^2\lambda^2-4\beta \lambda }}{2} &< 1 \\
     \iff  (1+\beta)\lambda - 2 & < \sqrt{(1+\beta)^2\lambda^2-4\beta \lambda } \\
     \iff ((1+\beta)\lambda - 2)^2 & < {(1+\beta)^2\lambda^2-4\beta\lambda } \\
     \iff 4 - 4(1+\beta)\lambda & < -4 \beta \lambda \\
     \iff \lambda > 1 
\end{align}
which is the assumption we started with.

Moreover for the root $z_i = \frac{(1+\beta)\lambda + \sqrt{(1+\beta)^2\lambda^2-4\beta \lambda }}{2}$ from \eqref{realeigen} we have $ z_i > 1$ which is trivial for $\lambda > 1$ and $\beta \geq 1$. Hence all the roots from \eqref{realeigen} with plus sign contribute to the unstable subspace of $\Lambda$. 

Next, the case of $\lambda< 1$ can contribute to both the unstable and stable subspaces of $\Lambda$. In particular, the case when $ \lambda < \frac{4\beta}{(1+\beta)^2}$, the magnitude of complex eigenvalues obtained from \eqref{complexeigen} will always be less than the larger eigenvalues obtained from \eqref{realeigen}. This holds because:
\begin{align}
    \frac{(1+\beta)\lambda + \sqrt{(1+\beta)^2\lambda^2-4\beta \lambda }}{2} \bigg\vert_{\lambda>1} > \frac{(1+\beta)\lambda}{2} \bigg\vert_{\lambda>1} > \underbrace{\frac{(1+\beta\lambda)}{2} \bigg\vert_{\lambda>1}}_{\geq \sqrt{\beta \lambda}\bigg\vert_{\lambda>1} } > \sqrt{\beta \lambda} \bigg\vert_{\lambda< \frac{4\beta}{(1+\beta)^2} }.
\end{align}
Hence for $f(\cdot) \in \mathcal{C}^{2,1}_{\mu, L}(\mathbb{R}^n)  $, $\norm{\Lambda}_2$ will be bounded by:
\begin{align}
  \norm{\Lambda}_2 =& \sup_{\lambda>1}\frac{(1+\beta)\lambda + \sqrt{(1+\beta)^2\lambda^2-4\beta \lambda }}{2}\\  \leq& \frac{(1+\beta)(1+\mu h)}{2} \bigg(1 + \sqrt{1 - \frac{4 \beta}{(1+\beta)^2(1+\mu h)}} \bigg)  \label{topeigen}
\end{align}
where we used the facts that $\sup_{\lambda>1}\lambda  \leq 1+ \mu h$ for $h \leq \frac{1}{L}$ by recalling that $\lambda$ was the eigenvalue of matrix $\mathbf{I} - h \nabla^2 f(\x^*)$ and $ \sqrt{(1+\beta)^2\lambda^2-4\beta \lambda } $ is an increasing function of $\lambda$ for $\lambda>1$.
\end{proof}  

\subsubsection{Lemma \ref{lambdagamma_lemma}}
\begin{proof}
Recall from the steps following up to \eqref{dynamicsystem1} that $2 \epsilon \R(\u_K) = \V^{-1} \textbf{M}_k\V$ where we have:
\begin{align}
 \textbf{M}_k =   \begin{bmatrix}
(1+\beta_k)h(\nabla^2f(\x^*)-\D(\y_k))\hspace{0.5cm} - \beta_k h(\nabla^2f(\x^*)-\D(\y_k)) \\  \boldsymbol{0} \hspace{3.5cm} \boldsymbol{0}
\end{bmatrix}
\end{align}
and $\D(\y_k) = \int_{p=0}^{1} \nabla^2 f(\x^*+p(\y_k-\x^*))dp $. Now, using Hessian Lipschitz boundedness from Assumption \textbf{A1} we have:
\begin{align}
    \norm{\nabla^2f(\x^*)-\D(\y_k)}_ 2 & = \norm{\nabla^2f(\x^*)-\int_{p=0}^{1} \nabla^2 f(\x^*+p(\y_k-\x^*))dp}_2 \\
    & = \norm{\int_{p=0}^{1}\bigg( \nabla^2f(\x^*)-\nabla^2 f(\x^*+p(\y_k-\x^*))\bigg)dp}_2 \\
    & \leq \int_{p=0}^{1}\norm{\bigg( \nabla^2f(\x^*)-\nabla^2 f(\x^*+p(\y_k-\x^*))\bigg)}_2dp \leq \frac{ M (1+2\beta_k)\epsilon}{2} 
\end{align}
where we used the fact that $ \y_k -\x^* = (1+\beta_k)(\x_k -\x^*) -\beta_k(\x_{k-1}-\x^*)$. Hence $\norm{\textbf{M}_k}_2 \leq  \frac{ M (1+2\beta_k)^2h\epsilon}{2} $ which implies $ 2\norm{\R(\u_k)}_2 \leq \frac{ M (1+2\beta_k)^2h}{2}$. Therefore we can set:
\begin{align}
    \sup_{\norm{\u_k}\leq \epsilon} \norm{\R(\u_k)}_2 \leq \sup_k\frac{ M (1+2\beta_k)^2h}{4} =  \Gamma  = \frac{ M (1+2 \beta)^2h}{4}
\end{align}
where we used the fact that $ \beta_k $ is a non-decreasing function of $k$ and $\beta_k \to \beta$.
\end{proof}  

\section{Analysis of accelerated gradient methods in convex neighborhoods of nonconvex functions}\label{local minima analysis appendix}

\subsection{A $\mathcal{C}^{\infty}$ function that is not locally convex around its local minimum}\label{loc_counterex}
{Let $\phi(x)$ be a $\mathcal{C}^{\infty}$ bump function that is positive on $(0,1)$ and zero elsewhere. Then $f(x)=sin^2(\frac{1}{x})\phi(x)$ is a smooth nonnegative function (define $f(0)=0$) which has zeroes at $x=\frac{1}{k\pi}$ for positive integers $k$ but is positive between these zeroes. Note that $f(x)$ will still have a local minimum at zero because $f(0)=0$, but because of the infinitely many oscillations in $f(x)$ between $x=0$ and $x=\frac{1}{k\pi}$ for any $k$, $f(x)$ cannot be locally convex in any open interval around $x=0$. So $f(x)$ is not locally convex at $x=0$.}

\subsection{Lemma \ref{lemsupnew123}}
\begin{proof}
     Since $f$ is $\mathcal{C}^1$ and is locally convex in the ball $\mathcal{B}_{R}(\x^*)$, we can write for any $\x \in \mathcal{B}_{R}(\x^*) $ the following two inequalities:
     \begin{align}
         f(\x) &\geq f(\x^*) + \langle \nabla f(\x^*), \x -\x^* \rangle , \\
            f(\x^*) &\geq f(\x) - \langle \nabla f(\x), \x -\x^* \rangle ,
     \end{align}
     which after adding yields $ \langle \nabla f(\x), \x -\x^* \rangle \geq 0 $. Next, we evaluate $\frac{\partial f(\x)}{\partial \norm{\x-\x^*}}$ for any $\x \in \mathcal{B}_{R}(\x^*) \backslash \x^*$ along some fixed direction using the chain rule as follows: 
     \begin{align}
         \frac{\partial f(\x)}{\partial \norm{\x-\x^*}} & = \bigg\langle  \frac{\partial f(\x)}{\partial \x}, \frac{\partial \x}{\partial \norm{\x-\x^*}} \bigg \rangle \\
         & =  \bigg\langle  {\nabla f(\x)}, \frac{\partial (\x - \x^*)}{\partial \norm{\x-\x^*}} \bigg\rangle \\
         & =  \bigg\langle  {\nabla f(\x)}, \frac{ \x - \x^*}{ \norm{\x-\x^*}} \bigg\rangle
     \end{align}
     where in the last step we used that $ \x - \x^* =  \norm{\x-\x^*}\widehat{(\x - \x^*)} $ where the fixed unit direction vector $\widehat{(\x - \x^*)} $ is independent of $ \norm{\x-\x^*}$ {(in the sense that in polar coordinates radial vector $\r$ is perpendicular to the angular direction $\theta$)}. But $ \langle \nabla f(\x), \x -\x^* \rangle \geq 0 $ and so $\frac{\partial f(\x)}{\partial \norm{\x-\x^*}}  \geq 0 $ implying $f$ is locally non-decreasing radially outwards from $\x^*$ along any direction. Hence, for any fixed $ \widehat{(\x - \x^*)}$ 
     $$ f(\x)\bigg\vert_{{\norm{\x-\x^*}}=R_1 }  \leq f(\x)\bigg\vert_{{\norm{\x-\x^*}}=R_2 } $$ 
     for any $0<R_1 < R_2 <R$ and thus $\sup_{{\norm{\x-\x^*}}=R_1} f(\x) \leq \sup_{{\norm{\x-\x^*}} = R_2} f(\x)$, $\inf_{{\norm{\x-\x^*}}=R_1} f(\x) \leq \inf_{{\norm{\x-\x^*}} = R_2} f(\x)$. 

     Next, suppose $f$ is locally strictly convex in the ball $\mathcal{B}_{R}(\x^*)$. Then $ \langle \nabla f(\x), \x -\x^* \rangle > 0 $ for any $\x \in \mathcal{B}_{R}(\x^*) $ by which $\frac{\partial f(\x)}{\partial \norm{\x-\x^*}}  > 0 $ implying $f$ is locally increasing radially outwards from $\x^*$ along any direction. Then $\sup_{{\norm{\x-\x^*}}=R_1} f(\x) < \sup_{{\norm{\x-\x^*}} = R_2} f(\x)$, $\inf_{{\norm{\x-\x^*}}=R_1} f(\x) < \inf_{{\norm{\x-\x^*}} = R_2} f(\x)$ which completes the proof.
 \end{proof}  

\subsection{Lemma \ref{convexextensionlemma}}
\begin{proof}
 { We first show that the iterate sequence $\{\x_k\}_{k\geq K}$ will stay in some connected component of a sublevel set of $f$. Let $$ S = \bigg\{\x : f(\x) \leq \sup_{\norm{\z-\x^*} = 3\xi}f(\z)\bigg\} ,$$  be a sublevel set of $f$ which is compact by coercivity of $f$. Let $S = \bigcup_{i\in \mathcal{P}} S_i$ where $ S_i$ are the disjoint connected components of $S$. Then for some $i \in \mathcal{P}$ we have $S_i = S \bigcap \mathcal{B}_{C_1\xi}(\x^*)$ and also $\mathcal{B}_{3\xi}(\x^*) \subset S_i \subset \mathcal{B}_{C_1\xi}(\x^*)$ for some constant $C_1$ where $3<C_1 < C$ and $C_1$ is left unspecified at this stage. To see this first note that the balls $ \mathcal{B}_{3\xi}(\x^*), \mathcal{B}_{C_1\xi}(\x^*)$ are connected. Next, we get that if $\x \in \mathcal{B}_{ 3\xi}(\x^*) $ then $ f(\x) < \sup_{\norm{\z-\x^*} = 3\xi}f(\z)$. Since $f \vert_{\mathcal{B}_{C\xi}(\x^*)}$ is locally strictly convex for $C \gg 1$, for some $C_2 > 3$ we\footnote{Since $C$ is sufficiently large with $C \gg 1$ we can assume without loss of generality that $C_2 < C$.} will have that $$ \sup_{\norm{\z-\x^*} = 3\xi} f(\z) - \inf_{\norm{\z-\x^*} = C_2\xi} f(\z) \leq 0$$ by Lemma \ref{lemsupnew123} and the fact that $\sup_{\norm{\z-\x^*} = 3\xi} f(\z) - \inf_{\norm{\z-\x^*} =3\xi} f(\z) $ is bounded by continuity of $f$. Then for $C>C_1 >C_2 $ and $\x \in \mathcal{B}_{ 3\xi}(\x^*) $ we have $$ f(\x) < \sup_{\norm{\z-\x^*} =3 \xi} f(\z) < \inf_{\norm{\z-\x^*} = C_1\xi} f(\z)  $$ by Lemma \ref{lemsupnew123} and so $\x \in S \bigcap \mathcal{B}_{C_1\xi}(\x^*)$. Thus $$\mathcal{B}_{3\xi}(\x^*) \subset S \bigcap \mathcal{B}_{C_1\xi}(\x^*)\subset \mathcal{B}_{C_1\xi}(\x^*)$$ and it only remains to show that $S_i = S \bigcap \mathcal{B}_{C_1\xi}(\x^*)$. 

We now proceed to a proof by contradiction. Suppose there exists some $\x \in S \bigcap \mathcal{B}_{C_1\xi}(\x^*)$ and some $\y \in S \bigcap \bigg(\mathcal{B}_{C_1\xi}(\x^*)\bigg)^c$ with $ \norm{\y-\x^*} > C_1 \xi$ such that $\x,\y$ are connected by some path in $S$. Let $\epsilon>0$ such that $C_1 < C_1 + \epsilon< C$. Then there exists some $\v$ on the path joining $\x,\y$ such that $ \norm{\v-\x^*} = (C_1 +\epsilon)\xi$ and $\v \in S$. Hence, $f(\v) \leq  \sup_{\norm{\z-\x^*} = 3\xi} f(\z) $ but that is not possible since $f(\v) \geq  \inf_{\norm{\z-\x^*} = (C_1+\epsilon)\xi} f(\z) > \inf_{\norm{\z-\x^*} = C_1\xi} f(\z) > \sup_{\norm{\z-\x^*} = 3\xi} f(\z)$ from Lemma \ref{lemsupnew123}. Thus, $S_i = S \bigcap \mathcal{B}_{C_1\xi}(\x^*)$.

Next, since $\x_K, \x_{K-1} \in \mathcal{B}_{\xi}(\x^*)$ we have that $\x_K, \x_{K-1} \in S_i$. Also we have
\begin{align}
   \norm{\y_K -\x^*} &\leq  (1+\beta_K) \norm{ \x_K-\x^*} + \beta_K\norm{ \x_{K-1}-\x^*}  \\
   & \leq (1+  2\beta_K) \xi \leq 3 \xi \label{tempconvex1a}
\end{align}
where we used the fact that $\beta_K = \frac{K }{(K+3-r)} \leq 1$ for $r \in [0,3)$ by which $\y_K \in \mathcal{B}_{3\xi}(\x^*)$ and thus $\y_K \in S_i$. Then $\x_{K+1} \in S$ by gradient Lipschitz continuity of $f$ and hence $\x_{K+1} \in S_i$ by local convexity of $f$ on the set $S_i$. Note that $ \x_{K+1} = \y_K - h \nabla f(\y_K)$ from \eqref{familyof momentum} and so $ \norm{\x_{K+1} -\x^*} \leq \norm{\y_K -\x^*} $ by local convexity and local gradient Lipschitz continuity of $f$ on $\mathcal{B}_{C\xi}(\x^*) $ and thus $\x_{K+1}$ cannot belong to any other connected component of $S$. Inducting this for all $k>K$ we get that $\{\x_k\}_{k\geq K} \subset S_i$. Since $f$ is strictly convex on $S_i$, the rate from Theorem \ref{thmconvexrate} will hold for the sequence $\{\x_k\}_{k\geq K}$ in $S_i$. Repeating this entire argument for the gradient descent method, i.e. $\beta_k =0$ and constant $h$, we will also get that the sequence $\{f(\x_k)\}_{k \geq K}$ converges to $f(\x^*)$ with $\mathcal{O}(1/k)$ rate (see \cite{tibshirani2010proximal} for convergence rate of gradient descent on convex functions). This completes the proof.
}
\end{proof}  

{Before presenting the next section, we remark that the upcoming ODE analysis for accelerated methods is not novel and has been covered in great detail in \cite{su2014differential, attouch2019rate, vassilis2018differential} and many other works. We only re-derive the ODE here for sake of completeness and ease of the reader in drawing parallels between the discrete \eqref{familyof momentum} and its continuous time counterpart. }
\subsection{ODE limit of the family of accelerated gradient methods \eqref{familyof momentum}}\label{ODEsection}
 Consider the functional equation given by $F(\x(t),\nabla f(\x(t)),t)=\ddot{\x}+\frac{\zeta}{t}\dot{\x} +\nabla f(\x)$ where $f \in \mathcal{C}^1$. Then the curve equation in $t$ where this functional equals $\mathbf{0}$ is given by the following ODE:
\begin{align}
F(\x(t),\nabla f(\x(t)),t)&=\mathbf{0} \nonumber\\
\implies \ddot{\x}+\frac{\zeta}{t}\dot{\x} & = -\nabla f(\x). \label{accode1}
\end{align}
Note that this ODE will be a limit of some discrete step accelerated gradient descent method, where the limit is obtained when the step size goes to $0$. We are interested in identifying the family of accelerated gradient methods whose limiting ODE is given by \eqref{accode1}. \\
Using the updates \eqref{familyof momentum} for $r \in [0,3)$ and following the steps from \cite{su2014differential} yields:
\begin{align}
    \frac{\x_{k+1}-\x_k}{\sqrt{h}} = \frac{k }{(k+3-r)} \frac{\x_{k}-\x_{k-1}}{\sqrt{h}} - \sqrt{h}\nabla f(\y_k). \label{generaldsodea}
\end{align}
Let $\x_k \approx X(k\sqrt{h})$ for some smooth curve $X(t)$ for $t \geq 0$ and let $ k = \frac{t}{\sqrt{h}}$. Then as $h \to 0$, we have $ X(t) \approx x_{\frac{t}{\sqrt{h}}} = \x_k $, $  X(t + \sqrt{h})  \approx x_{\frac{t+ \sqrt{h}}{\sqrt{h}}} = \x_{k+1}$ and by Taylor expansion we get:
\begin{align}
    \frac{\x_{k+1} - \x_k}{\sqrt{h}} &= \dot{X}(t) + \frac{1}{2} \ddot{X}(t)\sqrt{h} + o(\sqrt{h}) \\
    \frac{\x_{k} - \x_{k-1}}{\sqrt{h}} &= \dot{X}(t) - \frac{1}{2} \ddot{X}(t)\sqrt{h} + o(\sqrt{h}) \\
    \sqrt{h} \nabla f(\y_k) &= \sqrt{h} \nabla f(X(t)) + o(\sqrt{h}).
\end{align}
Using these substitutions \eqref{generaldsodea} can be written as:
\begin{align}
   \dot{X}(t) + \frac{1}{2} \ddot{X}(t)\sqrt{h} + o(\sqrt{h}) &= \bigg(1 - \frac{(3-r)\sqrt{h}}{(t + \sqrt{h}(2-r))} \bigg) \bigg( \dot{X}(t) - \frac{1}{2} \ddot{X}(t)\sqrt{h} + o(\sqrt{h}) \bigg) - \nonumber\\ & \hspace{1cm} \sqrt{h} \nabla f(X(t)) + o(\sqrt{h}). \label{generaldsode1}
\end{align}
By comparing coefficients of $\sqrt{h}$ in \eqref{generaldsode1} and then taking $h \to 0$ we get the following ODE:
\begin{align}
     \ddot{X}(t) + \frac{(3-r)}{t}\dot{X}(t) + \nabla f(X(t)) = \boldsymbol{0}. \label{generaldsode10}
\end{align}
Remarkably, this ODE has the same form as the ODE in \eqref{accode1} with parameter $\zeta = 3-r$ and $\zeta \in (1, 3]$. The parameter $ \zeta $ corresponds to the friction term of the ODE. Therefore Nesterov ODE which corresponds to $\zeta = 3$ has higher friction as compared to the regime where $ \zeta < 3$. Lesser friction implies larger momentum which results in better escape from strict saddle neighborhood however that comes at the cost of poor convergence rate to a local minimum. The work \cite{su2014differential} explores the regime of $\zeta >3 $, i.e., larger friction which causes damping while converging to local minimum neighborhood. The works \cite{attouch2018fast,attouch2019rate,aujol2017optimal, vassilis2018differential} studied $\zeta<3$ regime and showed convergence rates for the ODE to local minima. However, existing works including \cite{su2014differential,attouch2018fast} do not consider the saddle escape behavior. In particular, using Theorem \ref{exittimethm1} from our work, the exit time bound \eqref{exittimetradeoffbound} establishes that larger momentum or smaller friction results in smaller exit time. Moreover, our numerical section concretely establishes that the discrete algorithms generated from the ODE \eqref{accode1} for $\zeta <3 $ will escape strict saddle points faster due to larger momentum. 

\begin{rema}
    Note that \cite{gadat2018stochastic} derives a similar ODE to the one given by \eqref{generaldsode10} in stochastic setting, however the discretization of their ODE results into a totally different momentum based method termed as the stochastic heavy ball. Also, in this work we first developed the class of general accelerated methods \eqref{familyof momentum} and then obtained the limiting ODE whereas \cite{gadat2018stochastic} first develop continuous dynamics and then obtain the discretized momentum algorithm. Since it is well known that both \eqref{originalnesterov} and the heavy ball method have the same limiting ODE \cite{su2014differential}, this difference in the discretized schemes is not completely unexpected. 
\end{rema}
\subsubsection{Note on the Lyapunov function and dissipative property of the ODE \eqref{generaldsode10}}
From standard theory of dynamical systems \cite{shub2013global,ott2002chaos} recall that for $X(t) \in \mathcal{C}^1([0, \infty) \times \mathbb{R}^n) $ the gradient flow equation given by $$ \dot{X}(t) = -\nabla f(X(t)),$$ has a Lyapunov function $V(X(t))$ that satisfies $V \equiv f$ and the following:
\begin{itemize}
    \item[(i)]  $V(X(t))$ is a decreasing function of $t$ for any $X \not\in \mathcal{X}$ where $\mathcal{X}$ is the set of critical points of $f$,
    \item[(ii)] $V(X(t))$ is a non-decreasing function of $t$ for any $X \in \mathcal{X}$ where $\mathcal{X}$ is the set of critical points of $f$.
\end{itemize}
 Now unlike the gradient flow equation where Lyapunov function is simply $f$, the ODE \eqref{generaldsode10} has a slightly different\footnote{It is easy to check that $V \equiv f$ will not satisfy the Lyapunov function property of $ \frac{d}{dt}V(X(t)) \leq 0$ for the ODE \eqref{generaldsode10}.} Lyapunov function given by $$V(X(t)) = f({X}(t)) + \frac{1}{2}\norm{\dot{X}(t)}^2,$$ which can be viewed as the Hamiltonian, i.e., total energy functional where $f$ corresponds to the potential energy and $ \frac{1}{2}\norm{\dot{X}(t)}^2$ represents the kinetic energy of the particle. To see that this modified $V(X(t))$ is a Lyapunov function for the dynamical system governed by the ODE \eqref{generaldsode10} we just need to show that $ \frac{d}{dt}V(X(t)) \leq 0 $ for any $X(t) $. This is in fact already established in the literature \cite[Lemma 2.1]{attouch2018fast} and follows from the following steps in a straightforward manner:
 \begin{align}
      \frac{d}{dt}V(X(t)) &= \frac{d}{dt}f(X(t)) + \frac{d}{dt}\frac{1}{2}\norm{\dot{X}(t)}^2 \\
      & = \langle \nabla f({X}(t)),  \dot{X}(t)\rangle + \langle \ddot{X}(t), \dot{X}(t) \rangle \\
      & = - \frac{(3-r)}{t} \norm{\dot{X}(t)}^2 ,\label{lyapunovnewbound1}
  \end{align}
where in the last step we used the identity $$ \bigg\langle \dot{X}(t),  \ddot{X}(t) + \frac{(3-r)}{t}\dot{X}(t) + \nabla f(X(t))\bigg\rangle = 0 $$ by taking dot product with $\dot{X}(t) $ in \eqref{generaldsode10}. Then for any $r \leq 3$ from \eqref{lyapunovnewbound1} we have that $ \frac{d}{dt}V(X(t)) = - \frac{(3-r)}{t} \norm{\dot{X}(t)}^2 $ which is non-positive for any $X(t)$. Hence the ODE \eqref{generaldsode10} is a dissipative system with respect to the Lyapunov function $V(X(t)) = f({X}(t)) + \frac{1}{2}\norm{\dot{X}(t)}^2$ for any $r \leq 3$. 

\begin{rema}
    It should be noted that one of the possible discretizations of the ODE \eqref{generaldsode10} for $r \in [0,3]$ will yield the class of accelerated gradient methods \eqref{familyof momentum} with the parameter $r$ in the expanded interval of $ [0,3]$. Though this class of accelerated methods only offers convergence guarantees in convex neighborhoods when $ r \in [0,3)$ from Lemma \ref{convexextensionlemma}, yet this class has a Lyapunov function in the continuous time setting for any $r \in [0,3]$ since it was derived from a dissipative ODE system \eqref{generaldsode10}. 
\end{rema}

\section{Convergence rates to second order stationary points of smooth nonconvex functions}\label{local analysis appendix1}

\subsection{Lemma \ref{lemmalyapunov}}
\begin{proof}
From gradient Lipschitz continuity of $f$, the update \eqref{generalds} and the fact that $Lh<1$ we have the following two inequalities:
\begin{align}
    f(\x_{k+1}) & \leq f(\y_k) + \langle \nabla f(\y_k) , \x_{k+1}-\y_k \rangle + \frac{L}{2} \norm{\x_{k+1} - \y_k}^2 \nonumber\\
   & = f(\y_k) - h \norm{\nabla f(\y_k)}^2 + \frac{L h^2}{2}\norm{\nabla f(\y_k)}^2 ,\label{temp2a} \\
    f(\y_k) &\leq f(\x_k) - \langle \nabla f(\y_k) , \x_{k}-\y_k \rangle + \frac{L}{2} \norm{\x_{k} - \y_k}^2  \nonumber\\
    & = f(\x_k) - \frac{\norm{\x_k -\y_k + h \nabla f(\y_k)}^2}{2h} + \frac{\norm{\x_k -\y_k}^2}{2h} + \frac{h}{2}\norm{\nabla f(\y_k)}^2  + \frac{L}{2} \norm{\x_{k} - \y_k}^2 \nonumber\\
    & \leq f(\x_k) - \frac{\norm{\x_k -\x_{k+1} }^2}{2h} + \frac{\norm{\x_k -\y_k}^2}{h} + \frac{h}{2}\norm{\nabla f(\y_k)}^2. \label{temp2b} 
\end{align}
Adding these inequalities \eqref{temp2a} and \eqref{temp2b} we get:
\begin{align}
    f(\x_{k+1}) & \leq f(\x_k) - \frac{\norm{\x_k -\x_{k+1} }^2}{2h} + \frac{\norm{\x_k -\y_k}^2}{h}  - \frac{h}{2}\norm{\nabla f(\y_k)}^2   + \frac{L h^2}{2}\norm{\nabla f(\y_k)}^2\\
    \implies f(\x_k) - f(\x_{k+1}) & \geq    \bigg(\frac{h}{2}-\frac{L h^2}{2} \bigg)\norm{ \nabla f(\y_k)}^2 + \frac{\norm{\x_k -\x_{k+1} }^2}{2h} - \frac{\norm{\x_k -\y_k}^2}{h}. 
      \label{gas3}
\end{align}
Now adding $ \frac{\norm{\x_{k-1} -\x_{k} }^2}{2h}$ to both sides of \eqref{gas3} we get:
\begin{align}
    \implies f(\x_k)+ \frac{\norm{\x_{k-1} -\x_{k} }^2}{2h} - f(\x_{k+1}) - \frac{\norm{\x_k -\x_{k+1} }^2}{2h} & \geq    \bigg(\frac{h}{2}-\frac{L h^2}{2} \bigg)\norm{ \nabla f(\y_k)}^2 + \nonumber \\ & \frac{\norm{\x_{k-1} -\x_{k} }^2}{2h}- \frac{\norm{\x_k -\y_k}^2}{h} \\
    \implies f(\x_k)+ \frac{\norm{\x_{k-1} -\x_{k} }^2}{2h} - f(\x_{k+1}) - \frac{\norm{\x_k -\x_{k+1} }^2}{2h} & \geq    \bigg(\frac{h}{2}-\frac{L h^2}{2} \bigg)\norm{ \nabla f(\y_k)}^2 + \nonumber \\ & \frac{\norm{\x_k -\y_k}^2}{2h\beta_k^2}- \frac{\norm{\x_k -\y_k}^2}{h}  \label{gas3a}
\end{align}
where in the last step we substituted $\x_{k-1} -\x_{k} = \frac{1}{\beta_k}(\x_k - \y_k) $ from \eqref{gas1} for $\beta_k >0$\footnote{When $\beta_k = 0$, \eqref{generalds} is just the gradient descent method with Lyapunov function as $f$.}. Let $ \hat{f}(\x_k) = f(\x_k)+ \frac{\norm{\x_{k-1} -\x_{k} }^2}{2h}$ for all $k \geq 0$. Now if $\beta_k \leq \frac{1}{\sqrt{2}}$ for all $k \geq 0$ then $ \frac{\norm{\x_k -\y_k}^2}{2h\beta_k^2}- \frac{\norm{\x_k -\y_k}^2}{h} \geq 0$ for all $k \geq 0$. Hence, for the class of algorithms \eqref{generalds} with $\beta_k \leq \frac{1}{\sqrt{2}}$ for all $k \geq 0$ and $h < \frac{1}{L}$ we get a Lyapunov function $ \hat{f}$ which decreases monotonically from \eqref{gas3a}. Substituting $ \hat{f}$ into \eqref{gas3a}, using the relation $\x_{k-1} -\x_{k} = \frac{1}{\beta_k}(\x_k - \y_k) $ and telescoping from $k=0$ to $K-1$ where $\beta_k \leq \frac{1}{\sqrt{2}}$ for all $k \geq 0$ and $h < \frac{1}{L}$ we get:
\begin{align}
    \hat{f}(\x_k) - \hat{f}(\x_{k+1}) & \geq \bigg(\frac{h}{2}-\frac{L h^2}{2} \bigg)\norm{ \nabla f(\y_k)}^2 +  \underbrace{\beta_k^2\bigg(\frac{1}{2h\beta_k^2}- \frac{1}{h}\bigg)}_{\geq 0 \hspace{0.1cm} \text{for } \beta_k \leq \frac{1}{\sqrt{2}}} \norm{\x_k -\x_{k-1}}^2 \label{liapunovdecrease}\\
    \implies \sum\limits_{k=0}^{K-1} \hat{f}(\x_k) - \hat{f}(\x_{k+1}) & \geq \sum\limits_{k=0}^{K-1}\bigg(\frac{h}{2}-\frac{L h^2}{2} \bigg)\norm{ \nabla f(\y_k)}^2 \\
    \implies \hat{f}(\x_0) - \hat{f}(\x_{K}) & \geq K\bigg(\frac{h}{2}-\frac{L h^2}{2} \bigg)\inf_{0\leq k \leq K-1}\norm{ \nabla f(\y_k)}^2 \\
    \implies \inf_{0\leq k \leq K-1}\norm{ \nabla f(\y_k)}^2 & \leq \frac{\hat{f}(\x_0) - \hat{f}(\x_{K})}{K\bigg(\frac{h}{2}-\frac{L h^2}{2} \bigg)} = \frac{{f}(\x_0) + \frac{\norm{\x_{0} -\x_{-1} }^2}{2h}- {f}(\x_{K})-\frac{\norm{\x_{K-1} -\x_{K} }^2}{2h}}{K\bigg(\frac{h}{2}-\frac{L h^2}{2} \bigg)} \\
    \implies \inf_{0\leq k \leq K-1}\norm{ \nabla f(\y_k)}^2 & \leq \frac{{f}(\x_0) - {f}(\x_{K})}{K\bigg(\frac{h}{2}-\frac{L h^2}{2} \bigg)} \label{gas4abc}
\end{align}
where in the last step we substituted $\x_{0} =\x_{-1}  $. Note that if $f$ was convex then inequality \eqref{temp2b} becomes 
\begin{align}
    f(\y_k) &\leq f(\x_k) - \langle \nabla f(\y_k) , \x_{k}-\y_k \rangle \nonumber\\
    & \leq f(\x_k) - \frac{\norm{\x_k -\x_{k+1} }^2}{2h} + \frac{\norm{\x_k -\y_k}^2}{2h} + \frac{h}{2}\norm{\nabla f(\y_k)}^2
\end{align}
and then instead of \eqref{liapunovdecrease} we will get the following inequality
\begin{align}
     \hat{f}(\x_k) - \hat{f}(\x_{k+1}) & \geq \bigg(\frac{h}{2}-\frac{L h^2}{2} \bigg)\norm{ \nabla f(\y_k)}^2 +  \underbrace{\beta_k^2\bigg(\frac{1}{h\beta_k^2}- \frac{1}{h}\bigg)}_{\geq 0 \hspace{0.1cm} \text{for } \beta_k \leq {1}} \norm{\x_k -\x_{k-1}}^2. \label{liapuovdecrease1a}
\end{align}

\end{proof}

\subsection{Theorem \ref{supptheoremnew}}
\begin{proof}
Since $f(\cdot)$ is a Morse function it has isolated critical points \cite{matsumoto2002introduction} and by the isolation property these critical points are countable. Next suppose $\y_k = \x^* \in \mathcal{T}$ for some finite $k$ $\probP_1$-almost surely, then $\x_{k+1} = \y_k - h \nabla f(\y_k) = \x^*$ $\probP_1$-almost surely which is not possible
from Lemma \ref{lemmasupport} under the given initialization scheme. Hence $\probP_1(\{\y_k=\x^* \hspace{0.1cm} \vert \hspace{0.1cm} \x^* \in \mathcal{T}; \hspace{0.1cm} k < \infty\}) =0 $ which implies $\probP_1(\{\y_k \in \mathcal{T}; \hspace{0.1cm} k < \infty\}) =   \sum_{\x^* \in \mathcal{T}}\probP_1(\{\y_k=\x^* \hspace{0.1cm} \vert \hspace{0.1cm} \x^* \in \mathcal{T}; \hspace{0.1cm} k < \infty\}) =0 $ where we took countable union over all points in the set $\mathcal{T}$. This implies $ \norm{ \nabla f(\y_k)}^2$ is positive $\probP_1$-almost surely for all $k \geq 0$ and can only go to $0$ as $k \to \infty$. 

Since $f$ is coercive i.e. $ \lim_{\norm{\x}\to \infty} f(\x) = \infty$ and $f$ is continuous (and hence lower semi-continuous), we will have $f(\x) \geq \inf_{\x} f(\x)> - \infty$ i.e. the infimum of the function {values} exists \cite{kinderlehrer2000introduction}. From \eqref{liapunovdecrease} for any $k \geq \tilde{K}$ we observe that the sequence $\{\hat{f}(\x_k)\}_{k \geq \tilde{K}}$ is monotonically decreasing since $\beta_k \leq \frac{1}{\sqrt{2}} $ for $k \geq \tilde{K}$ (Lyapunov function from Lemma \ref{lemmalyapunov} only works for $\beta_k \leq \frac{1}{\sqrt{2}}$) and we also have that $ \hat{f}(\x_k) \geq f(\x_k)\geq \inf_{\x} f(\x)> - \infty$. Then by monotone convergence theorem, the sequence $\{\hat{f}(\x_k)\}_{k \geq \tilde{K}}$ converges. Hence taking $\limsup_{k \to \infty}$ on both sides of \eqref{liapunovdecrease} yields:
\begin{align}
 \limsup_{k \to \infty} \bigg(\hat{f}(\x_k) - \hat{f}(\x_{k+1})\bigg) & \geq \limsup_{k \to \infty}\bigg(\frac{h}{2}-\frac{L h^2}{2} \bigg)\norm{ \nabla f(\y_k)}^2 + \liminf_{k \to \infty} {\beta_k^2\bigg(\frac{1}{2h\beta_k^2}- \frac{1}{h}\bigg)} \norm{\x_k -\x_{k-1}}^2   \\
 \implies \lim_{k \to \infty} \norm{ \nabla f(\y_k)}^2 &= 0. \label{gas4b}
\end{align}

Since $f(\cdot)$ is coercive, i.e., $ \lim_{\norm{\x}\to \infty} f(\x)  = \infty$, all its sublevel sets are compact \cite{kinderlehrer2000introduction}. Hence the tail sequence $\{\x_k\}_{k \geq \tilde{K}}$ stays within the compact sublevel set $ \bigg\{\x \hspace{0.1cm} \bigg\vert \hspace{0.1cm} f(\x) \leq f(\x_{\tilde{K}}) + \frac{\norm{\x_{\tilde{K}}-\x_{\tilde{K}-1}}^2}{2h}\bigg\}$ by monotonicity of the sequence $\{\hat{f}(\x_k)\}_{k \geq \tilde{K}}$ from \eqref{liapunovdecrease}. Hence the tail sequence $\{\y_k\}_{k \geq \tilde{K}}$ also stays within some compact set $\mathcal{U} \supset \bigg\{\x \hspace{0.1cm} \bigg\vert \hspace{0.1cm} f(\x) \leq f(\x_{\tilde{K}}) + \frac{\norm{\x_{\tilde{K}}-\x_{\tilde{K}-1}}^2}{2h}\bigg\}$ for all $k\geq 0$ since $\y_k = \x_k + \beta_k(\x_k - \x_{k-1})$ and $\beta_k$ is bounded. Therefore we get a convergent subsequence $\{\y_{k_m}\}$, of the sequence $\{\y_k\}$, which converges to some $\y^*$ in the compact domain $\mathcal{U}$. This implies $\lim_{m \to \infty} \norm{\nabla f(\y_{k_m})}^2 = \norm{\nabla f(\lim\y_{k_m})}^2 =  \norm{\nabla f(\y^*)}^2 =0$ by the facts that $\nabla f(\cdot)$ is continuous and $ \lim_{k \to \infty} \norm{ \nabla f(\y_k)}^2 = 0$. Hence $\y^* \in \mathcal{T}$ and so $[\y^*;\y^*]$ is a fixed point of the algorithm $[\x_{k+1}; \x_k] = P_k([\x_k;\x_{k-1}])$ from Lemma~\ref{lemma_pk}. Thus we have shown that any convergent subsequence of $\{\y_k\}$ converges to some point in $\mathcal{T}$. Now if any convergent subsequence $\{\y_{k_m}\}$ converges to $\y^*$, then the subsequence $\{\x_{k_m+1}\}$ given by $\x_{k_m+1} = \y_{k_m} - h \nabla f(\y_{k_m})  $ also converges to $\y^*$. Since the set of accumulation points (subsequential limit points) of any sequence $\{\y_k\}$ belong to $\mathcal{T}$, these accumulation points are isolated from the fact that $f$ is Morse. Also, since the map $\mathrm{id}  -h \nabla f$ is a homeomorphism on any compact set for $h < \frac{1}{L}$ from the proof of Theorem \ref{diffeomorphthm}, we get that for any convergent subsequence $\{\x_{k_m}\}$ converging to some $\x^*$, the subsequence $\{\y_{k_m-1}\}$ converges to $ (\mathrm{id}  -h \nabla f)^{-1}(\x^*)$ which will belong to $\mathcal{T}$ and thus $ \x^* \in \mathcal{T}$ since $ \mathcal{T}$ is the fixed point set for the map $ \mathrm{id}  -h \nabla f$.

Then recalling the recursions from Theorem \ref{measuretheorem3} and defining the following recursions for the given momentum sequence $\{\beta_k\}$ 
\begin{align}
    \w^r_{k+1} =  \begin{cases} 
      P_k(\w^r_k) & 0 \leq k\leq r \\
     P(\w^r_k) & k > r 
   \end{cases} \label{recuro1**}
\end{align} for any $r\geq \Tilde{K}$ and
\begin{align}
    \w_{k+1} = P_k(\w_k) \hspace{0.1cm} \forall \hspace{0.1cm} k \geq 0, \label{recuro2**}
\end{align}
 we have effectively shown that under compact initialization the accumulation points for any sequence generated from either \eqref{recuro1**} or \eqref{recuro2**} are isolated and belong to the set $\mathcal{I} = \{[\y^*; \y^*] : \y^* \in \mathcal{T}\}$. Then following the proof of Theorem \ref{measurethm7} from the symbol `$\bLozenge$' to the symbol `$ \clubsuit$', we get that the sequence $\{[\x_k;\x_{k-1}]\}$ can only converge to $[\x^*;\x^*]$ $\probP$-almost surely where $\x^*$ is some local minimum of $f(\cdot)$.
\end{proof}

\subsection{Lemma \ref{lemsublevelanalytic}}
\begin{proof}
 Since $f$ is coercive, its sublevel sets are compact. Let $\mathcal{D} = \{\z \hspace{0.1cm} : \hspace{0.1cm} f(\z) \leq c\}$ be any sublevel set of $f$ for any real $c$. Since $f$ is gradient Lipschitz continuous in every compact set, let $L_c$ be the local gradient Lipschitz constant of $f$ on some sufficiently large compact set $\mathcal{D}_c  \Supset \mathcal{D}$ ($\Supset$ implies $ \mathcal{D}$ is compactly contained in $\mathcal{D}_c$). Since $\beta_k \leq \frac{1}{\sqrt{2}}$ for all $k$ for \eqref{generalds}, we can use Lemma \ref{lemmalyapunov} on the compact set $\mathcal{D}$ and a basic induction argument\footnote{To avoid repetition of Lemma \ref{lemmalyapunov}'s proof, we do not prove the simple induction argument and directly use it here.} on the compact set $\mathcal{D}_c$ to get that $\hat{f}(\x_k) = f(\x_k)+ \frac{\norm{\x_{k-1} -\x_{k} }^2}{2h} $ decreases monotonically if the sequence $\{\x_k\}$ is initialized in $\mathcal{D} $ and $h < \frac{1}{L_c}$. Then for all $k \geq 0$ we have:
    \begin{align}
        f(\x_{k+1})+ \frac{\norm{\x_{k+1} -\x_{k} }^2}{2h} &\leq f(\x_k)+ \frac{\norm{\x_{k-1} -\x_{k} }^2}{2h} \\
        \implies f(\x_{k+1})+ \frac{\norm{\x_{k+1} -\x_{k} }^2}{2h} &\leq f(\x_0)
    \end{align}
    for the initialization scheme of $\x_0= \x_{-1} \in \mathcal{D}$.
Since $\x_{k+1} = N_k(\x_k)$ and $\x_{k-1}=N_{k-1}^{-1}(\x_k) $ for all $k$ from Theorem \ref{diffeomorphthm}, we can rewrite the above inequality for any $k>0$ as:
\begin{align}
   f(N_k(\x_k))+ \frac{\norm{N_k(\x_k) -\x_{k} }^2}{2h} &\leq f(N_{0}^{-1} \circ \dots \circ N_{k-1}^{-1}(\x_k))  \\
   \implies f(N_k(\x))+ \frac{\norm{N_k(\x) -\x }^2}{2h} &\leq f(N_{0}^{-1} \circ \dots \circ N_{k-1}^{-1}(\x)) , \label{sublevelanalytic1}
\end{align}
where in the last step we replaced $\x_k$ with any $\x \in \mathcal{D}$ by using the maps $N_k, N_{k-1}$ defined in Theorem~\ref{diffeomorphthm} on $\mathbb{R}^n$. Since $N_{-1} \equiv \mathrm{id}$ for the given initialization scheme of $\x_0= \x_{-1}$, we have that $N_{-1} : \mathcal{D} \rightarrow \mathcal{D}$. Since $N_{-1} \equiv \mathrm{id}$, from Theorem \ref{diffeomorphthm} we get that $N_0 \equiv \mathrm{id} - h \nabla f$. Then for any $\x \in \mathcal{D}$ and $\mathcal{D}_c  \Supset \mathcal{D}$ by local gradient Lipschitz continuity of $f$ we have 
\begin{align}
    f(N_0 (\x)) \leq  f(\x) - \frac{h}{2}(1- L_c h)\norm{\nabla f(\x)}^2 < f(\x) \leq c
\end{align}
and so $ N_0 : \mathcal{D} \rightarrow \mathcal{D} $ for $h< \frac{1}{L_c}$. We now proceed by induction. Suppose for some $K$, where $K>0$ we have that for any $\x \in \mathcal{D}$, $N_{j} : \mathcal{D} \rightarrow \mathcal{D}$ for all $-1 \leq j < K$. Then we have $ N_{0}^{-1} \circ \dots \circ N_{K-1}^{-1}(\x) \in \mathcal{D}$ since $\x \in \mathcal{D}$ and so from \eqref{sublevelanalytic1}, for $k= K$ we get $f(N_K(\x))+ \frac{\norm{N_K(\x) -\x }^2}{2h} \leq f(N_{0}^{-1} \circ \dots \circ N_{K-1}^{-1}(\x)) \leq c$ implying $ N_K(\x)\in \mathcal{D}$ or $N_{K} : \mathcal{D} \rightarrow \mathcal{D} $. Since $ N_0 : \mathcal{D} \rightarrow \mathcal{D} $, by induction we get that $ N_k : \mathcal{D} \rightarrow \mathcal{D}$ for all $k$. 
\end{proof}

\subsection{Theorem \ref{kirszthmadapted}}
\begin{proof}
Since $f(\cdot)$ is twice continuously differentiable, it is locally gradient Lipschitz continuous in every compact set. Hence $f(\cdot)$ is gradient and Hessian Lipschitz continuous in the compact set $\mathcal{K} \supset \{\x \hspace{0.1cm} \vert \hspace{0.1cm} f(\x) \leq f(\x_0)\} $ where the compact set $ \mathcal{K}$ {will be specified later.} Suppose $L$ is the gradient Lipschitz constant of $f(\cdot)$ in this compact set or equivalently $L$ is the local Lipschitz constant of the function $\nabla f : \mathbb{R}^n \rightarrow \mathbb{R}^n$ {when restricted to $\mathcal{K}$}. Next, by the  Kirszbraun Theorem (Theorem \ref{kirszbraunthm}), there exists an extension $G : \mathbb{R}^n \rightarrow \mathbb{R}^n$ of the function $g \equiv \nabla f$ on the entire Euclidean space $\mathbb{R}^n$ such that $G \equiv \nabla f$ on the compact set $\mathcal{K} \supset \{\x \hspace{0.1cm} \vert \hspace{0.1cm} f(\x) \leq f(\x_0)\} $, $G$ is globally Lipschitz continuous with a Lipschitz constant of $L$ and $G$ is a gradient vector field on $\mathbb{R}^n$ (see \cite{azagra2021kirszbraun}). Next suppose $F : \mathbb{R}^n \rightarrow \mathbb{R}$ is the primitive of $G$ on $ \mathbb{R}^n$ given by the line integral $ F(\x) = \int_{\gamma_0} \langle G(\v), d \v \rangle $ where $\gamma_0$ is any smooth curve from $\mathbf{0}$ to $\x$. Then taking directional derivative of $F$ with respect to $\v$ we get $ \frac{\partial F(\x)}{\partial \v} = \langle G(\x), \v \rangle$ and so $G(\x) = \nabla F(\x)$ for all $\x \in \mathbb{R}^n$. Now $\nabla F \equiv G \equiv \nabla f$ on the compact set $\mathcal{K}$ so $F \equiv f + c$ on this set for any constant $c$. Without loss of generality we can take $c=0$ so that $F \equiv f $ on the set $\mathcal{K} \supset \{\x \hspace{0.1cm} \vert \hspace{0.1cm} f(\x) \leq f(\x_0)\} $. Since $G$ is $L$ Lipschitz continuous and so is $\nabla F$, we can write:
\begin{align}
       F(\x) -  F(\y)    &\geq \langle \nabla F(\y), \x -\y \rangle  - \frac{L}{2}\norm{\x - \y}^2 \\
       \implies F(\x) + \frac{L+ \epsilon}{2}\norm{\x - \y}^2 &\geq \langle \nabla F(\y), \x -\y \rangle  + \frac{\epsilon}{2}\norm{\x - \y}^2. \label{kirszeq1}
\end{align}
Next recall that since $f$ is coercive and continuous it has a global minimum \cite{kinderlehrer2000introduction} which will belong to the compact sublevel set $\{\x \hspace{0.1cm} \vert \hspace{0.1cm} f(\x) \leq f(\x_0)\} \subset \mathcal{K}  $. Since $f(\cdot) \in \mathcal{C}^2$, this global minimum, say $\x^*$, will be a critical point of $f$ and therefore a critical point of $F$ because $ F \equiv f$ on the set $\{\x \hspace{0.1cm} \vert \hspace{0.1cm} f(\x) \leq f(\x_0)\} \subset \mathcal{K}$. Then setting $\y = \x^*$ in \eqref{kirszeq1} we get the following for $\epsilon>0$:
\begin{align*}
    F(\x) + \frac{L+ \epsilon}{2}\norm{\x - \x^*}^2 &\geq \langle \nabla F(\x^*), \x -\x^* \rangle  + \frac{\epsilon}{2}\norm{\x - \x^*}^2 \\
    \implies  F(\x) + \frac{L+ \epsilon}{2}\norm{\x - \x^*}^2 &\geq \frac{\epsilon}{2}\norm{\x - \x^*}^2 {\geq} \frac{\epsilon}{4}\norm{\x}^2 - \frac{\epsilon}{2}\norm{\x^*}^2 \label{kirszeq2}
\end{align*}
where in the last step we used the inequality $\norm{\x - \x^*}^2 {\geq }  \frac{1}{2}\norm{\x}^2 - \norm{\x^*}^2$. Since $\x^*$ belongs to a compact set, $\norm{\x^*}$ is bounded and so the function $F(\x) + \frac{L+ \epsilon}{2}\norm{\x - \x^*}^2 $ is coercive by the following {argument}:
\begin{align}
\limsup_{\norm{\x} \to \infty} \bigg( F(\x) + \frac{L+ \epsilon}{2}\norm{\x - \x^*}^2 \bigg)\geq\liminf_{\norm{\x} \to \infty} \bigg( F(\x) + \frac{L+ \epsilon}{2}\norm{\x - \x^*}^2 \bigg)& {\geq}  
\liminf_{\norm{\x} \to \infty}\bigg(\frac{\epsilon}{4}\norm{\x}^2 - \frac{\epsilon}{2}\norm{\x^*}^2\bigg) \\
  \implies \lim_{\norm{\x} \to \infty} \bigg( F(\x) + \frac{L+ \epsilon}{2}\norm{\x - \x^*}^2 \bigg) &= \infty. \label{kirszeq3}
\end{align}
Since $F(\x) + \frac{L+ \epsilon}{2}\norm{\x - \x^*}^2 $ is coercive, $ \inf_{\x \in \mathbb{R}^n}\bigg( F(\x)+ \frac{L+ \epsilon}{2}\norm{\x - \x^*}^2\bigg) > -\infty$.\\
Next, consider the function $ \tilde{F} : \mathbb{R}^n \rightarrow \mathbb{R}$ given by:
\begin{align}
   \tilde{F}(\x) &= F(\x) + \frac{L+ \epsilon}{2}\norm{\x - \x^*}^2 + (1-\Phi_{\mathcal{K}}(\x) )\bigg(f(\x_0) - \inf_{\x \in \mathbb{R}^n}\bigg( F(\x)+ \frac{L+ \epsilon}{2}\norm{\x - \x^*}^2\bigg)\bigg) \nonumber \\ &  \hspace{1cm} - \frac{L+ \epsilon}{2}\norm{\x - \x^*}^2 \Phi_{\mathcal{K}}(\x) \label{kirszfunc}
\end{align}
 where $ \Phi_{\mathcal{K}}$ is a $\mathcal{C}^{\infty}$ smooth bump function, {where we choose} the compact set $\mathcal{K}$ {to satisfy} the condition $\mathcal{K} \supset \{\x \hspace{0.1cm} \vert \hspace{0.1cm} f(\x) \leq f(\x_0)\} + \mathcal{B}_{\epsilon}(\mathbf{0}) $ \footnote{{Here the operator $+$ defines the Minkowski sum operation between sets. The openness of the set $ \{\x \hspace{0.1cm} \vert \hspace{0.1cm} f(\x) \leq f(\x_0)\} + \mathcal{B}_{\epsilon}(\mathbf{0}) $ follows from the fact that for two sets $A,B$ their Minkowski sum $A+B$ is open even if only one of the sets (say, $B$) is open \cite{rudin1976principles}.}}and we have $0 \leq \Phi_{\mathcal{K}} \leq 1  $, $\Phi_{\mathcal{K}} \equiv 1 $ on the open set $  \{\x \hspace{0.1cm} \vert \hspace{0.1cm} f(\x) \leq f(\x_0)\} + \mathcal{B}_{\epsilon}(\mathbf{0})$ and $\Phi_{\mathcal{K}} \equiv 0 $ on $\mathbb{R}^n \backslash \mathcal{K}$. Such a smooth bump function exists by Proposition 2.25 in \cite{lee2013smooth}. Clearly $ \tilde{F} \equiv f $ on the compact set $\{\x \hspace{0.1cm} \vert \hspace{0.1cm} f(\x) \leq f(\x_0)\} $ and $\tilde{F}(\x) = F(\x) + \frac{L+ \epsilon}{2}\norm{\x - \x^*}^2 + f(\x_0) - \inf_{\x \in \mathbb{R}^n}\bigg( F(\x)+ \frac{L+ \epsilon}{2}\norm{\x - \x^*}^2\bigg)$ for $\x \in \mathbb{R}^n \backslash \mathcal{K}$ and so $\tilde{F}$ is coercive by \eqref{kirszeq3}. Moreover we have that $\tilde{F}(\x) \geq f(\x_0) $ for $\x \in \mathbb{R}^n \backslash \mathcal{K}$. Since $ \x^* \in \{\x \hspace{0.1cm} \vert \hspace{0.1cm} f(\x) \leq f(\x_0)\} \subset \mathcal{K}$ and $ F \equiv f$ on $\mathcal{K}$ we will have $\inf_{\x \in \mathbb{R}^n}\bigg( F(\x)+ \frac{L+ \epsilon}{2}\norm{\x - \x^*}^2\bigg) \leq F(\x^*) = f(\x^*) \leq f(\x_0) $ which implies $f(\x_0) - \inf_{\x \in \mathbb{R}^n}\bigg( F(\x)+ \frac{L+ \epsilon}{2}\norm{\x - \x^*}^2\bigg)\geq 0 $. Therefore from \eqref{kirszfunc} and the facts that $ F \equiv f$ on $\mathcal{K}$, $0 \leq \Phi_{\mathcal{K}} \leq 1  $ we also have that $\tilde{F}(\x) \geq f(\x) $ for $\x \in \mathcal{K} \backslash \{\x \hspace{0.1cm} \vert \hspace{0.1cm} f(\x) \leq f(\x_0)\}$. Now $f(\x)>f(\x_0)$ on the complement of the set $ \{\x \hspace{0.1cm} \vert \hspace{0.1cm} f(\x) \leq f(\x_0)\}$ so $ \tilde{F}(\x) > \tilde{F}(\x_0)= f(\x_0)$ on the set $\mathbb{R}^n \backslash \{\x \hspace{0.1cm} \vert \hspace{0.1cm} f(\x) \leq f(\x_0)\} $ where we used the fact that $f$ is coercive. Therefore we have $\{\x \hspace{0.1cm} \vert \hspace{0.1cm} f(\x) \leq f(\x_0)\} = \{\x \hspace{0.1cm} \vert \hspace{0.1cm} \tilde{F}(\x) \leq \tilde{F}(\x_0)\} $. 
\\ 
Since $\tilde{F}$ is $\mathcal{C}^2$ smooth on $\mathcal{K}$, it will be gradient Lipschitz continuous on the compact set $\mathcal{K} \supset \{\x \hspace{0.1cm} \vert \hspace{0.1cm} f(\x) \leq f(\x_0)\} + \mathcal{B}_{\epsilon}(\mathbf{0}) $ with some gradient Lipschitz constant $\tilde{L}>L>0$, where $ \tilde{L}$ is also the global gradient Lipschitz constant for $\tilde{F} $. Next, the iterate sequence $\{\x_k\}$ generated by the general accelerated method \eqref{generalds} on the function $\tilde{F}$ with the initialization scheme of $\x_{0} =\x_{-1}$, $\x_0 \notin \mathcal{T}$, $\beta_k \leq \frac{1}{\sqrt{2}}$ for all $k \geq 0$, $\beta_k \to \beta$ and $h < \frac{1}{\tilde{L}}$ will always stay within the compact set $\{\x \hspace{0.1cm} \vert \hspace{0.1cm} f(\x) \leq f(\x_0)\} = \{\x \hspace{0.1cm} \vert \hspace{0.1cm} \tilde{F}(\x) \leq \tilde{F}(\x_0)\} $ by coercivity of $\tilde{F} $, $ \tilde{F} \equiv f$ on $ \{\x \hspace{0.1cm} \vert \hspace{0.1cm} f(\x) \leq f(\x_0)\} $ and the fact that the sequence $\{\hat{f}(\x_k)\}$ decreases monotonically from Lemma \ref{lemmalyapunov} for $h < \frac{1}{\tilde{L}}$. Hence the iterate sequence $\{\x_k\}$ generated on the function $\tilde{F} $ in the compact set $\{\x \hspace{0.1cm} \vert \hspace{0.1cm} f(\x) \leq f(\x_0)\}  $ is exactly the same as the iterate sequence generated on the function $f $  with the initialization scheme of $\x_{0} =\x_{-1}$, $\x_0 \notin \mathcal{T}$, $\beta_k \leq \frac{1}{\sqrt{2}}$ for all $k \geq 0$, $\beta_k \to \beta$ and $h < \frac{1}{\tilde{L}}$. Since $f$ is Morse, the function  $\tilde{F} $ in the compact set $\{\x \hspace{0.1cm} \vert \hspace{0.1cm} f(\x) \leq f(\x_0)\}  $ will be Morse. Furthermore, from Lemma \ref{lemsublevelanalytic}, the sequence of maps $\{N_k\}$ from Theorem \ref{diffeomorphthm} for the given \eqref{generalds} method will map the compact sublevel set $\{\x \hspace{0.1cm} \vert \hspace{0.1cm} f(\x) \leq f(\x_0)\}  $ to itself. Thus, the sequence of maps $\{P_k\}$ and the map $P$ from Lemma~\ref{lemma_pk} will map the compact set $\{\x \hspace{0.1cm} \vert \hspace{0.1cm} f(\x) \leq f(\x_0)\} \times \{\x \hspace{0.1cm} \vert \hspace{0.1cm} f(\x) \leq f(\x_0)\} $ to itself. Then from Corollary \ref{corsup2}, the sequence of maps $\{P_k\}$ and the map $P$ from Lemma \ref{lemma_pk} will be $\probP_1$-almost sure diffeomorphisms on the compact set $ \{\x \hspace{0.1cm} \vert \hspace{0.1cm} \tilde{F}(\x) \leq \tilde{F}(\x_0)\} \times \{\x \hspace{0.1cm} \vert \hspace{0.1cm} \tilde{F}(\x) \leq \tilde{F}(\x_0)\}  $ by the fact that $ \tilde{F} \in \mathcal{C}^{2,1}_{\tilde{L}}(\mathbb{R}^n)$ is coercive and is Hessian Lipschitz continuous on this compact set.\footnote{Since $\tilde{F} \equiv f $ on $\{\x \hspace{0.1cm} \vert \hspace{0.1cm} \tilde{F}(\x) \leq \tilde{F}(\x_0)\} $ so $ \tilde{F}$ is Hessian Lipschitz continuous on this compact set. Outside this compact set, Hessian Lipschitz continuity is not required since the sequence $\{\x_k\}$ from \eqref{generalds} never leaves this compact set for $\beta_k \leq \frac{1}{\sqrt{2}}$.} Hence, the $\probP$-almost sure convergence result from Theorem \ref{measurethm7} will hold in the given compact set. Then the statements from Theorem \ref{supptheoremnew} for $h < \frac{1}{\tilde{L}}$ and Theorem \ref{thmlipschitzrate} for $ {\tilde{L}}h \ll 1$ follow directly where now $f$ is not required to be globally gradient Lipschitz continuous. This completes the proof.

Note that in the above extension argument we implicitly assumed a conservative vector field extension on a sufficiently large compact set $\mathcal{K}$ containing the relevant part of the trajectory. The initialization $\x_0$ and the reference point $\x^*$ lie in a smaller compact set $B \subset \mathcal{K}$, so when the iterates remain in $\mathcal{K}$, path integrals reduce to differences of a potential. For non-conservative vector field extensions, we can define the potential $F$ via integrals along straight line paths from $\x^*$ and refer to the directional path gradients of $F$ as $\nabla F \equiv G$. By simple calculus and geometry, it can be shown that $F$ will satisfy a gradient (directional) Lipschitz inequality of the form $$F(\y) + \langle G(\y) , \x -\y \rangle - \frac{L'}{2} \norm{\x -\y}^2 \le F( \x) \le F(\y) + \langle G(\y) , \x -\y \rangle + \frac{L'}{2} \norm{\x -\y}^2$$ with a larger global Lipschitz constant $L'>L$ when at least one of $\x$ or $\y$ is in $B$. If an iterate $\y$ exits $\mathcal{K}$ but $\x \in B$, a uniform lower bound on $\|\y - \x\|$ in terms of $\|\y - \x^*\| + \|\x - \x^*\|$ yields directional gradient Lipschitz continuity with a slightly larger constant $L'>L$ (one can take $\mathcal{K} = \mathcal{B}_{100R}(\mathbf{0})$, $B = \mathcal{B}_{R}(\mathbf{0})$ with $R \gg 1$, for example, and use separation of $\y$ from $B$), ensuring descent for a suitable step size and preventing escape to infinity. More precisely, for $\mathcal{K} = \mathcal{B}_{100R}(\mathbf{0})$, $B = \mathcal{B}_{R}(\mathbf{0})$, using the fundamental theorem of calculus for integrals along straight line paths from $\x^*$ and the separation of $\y$ from $B$, one gets that $L' := L\left(\frac{99}{97}\right)^2\!\left(1 + \frac{2}{99} + \frac{2}{99^2}\right)$ is a universal constant independent of $R$. Compactness of sub-level sets then guarantees that the trajectory eventually remains in $\mathcal{K}$. The above steps can be expanded more rigorously; however, this is not the main focus of the work, and so we did not present this argument in the paper for brevity. \looseness=-1
\end{proof}   

\subsection{On the equivalence of initial unstable subspace projections}\label{equivprojsec}
Consider the matrix 
\begin{align}
  \D =  \begin{bmatrix}
(1+\beta)(\mathbf{I}-h\nabla^2f(\x^*))\hspace{0.5cm} - \beta(\mathbf{I}-h\nabla^2f(\x^*)) \\  \mathbf{I} \hspace{3.5cm} \boldsymbol{0}
\end{bmatrix} \label{projeqrev1}
\end{align}
with eigendecomposition $\D = \V \Lambda \V^{-1} $, the initial augmented radial vector {$ \z_0 = [\x_0-\x^*;\x_{-1} - \x^*]$}, and the unstable subspace of $\Lambda$ given by $$\tilde{\mathcal{E}}_{US}  = \text{span}\bigg\{\v \hspace{0.1cm}: \hspace{0.1cm} \Lambda\v = \lambda \v;\hspace{0.1cm} \abs{\lambda} >1 \bigg\}.$$ 
{For simplicity of analysis, we assume that the unstable and stable subspace directions are fixed, thereby fixing $\V$, and only the diagonal entries of $\Lambda$ are allowed to vary. This assumption is intended solely to motivate the experimental setup; we did not generate our numerical experiments based on this assumption, nor do we claim that the equivalence of projections holds in a more general setting.} Let {$ \u_0= \V^{-1}\z_0 = \sum\limits_{j \in \tilde{\mathcal{N}}_{US}}\tilde{\theta}_j^{us} \tilde{\e}_j + \sum\limits_{i \in \tilde{\mathcal{N}}_{S}}\tilde{\theta}_i^{s}\tilde{\e}_i$}, where $\tilde{\mathcal{N}}_{US} $ corresponds to the index set for the eigenvalues of $\Lambda$ outside unit circle, $ \tilde{\e}_j,  \tilde{\e}_i $ are the orthonormal eigenvectors of the diagonal matrix $\Lambda$,\footnote{{From Theorem~\ref{generalacclimiteigen}, it can be checked that the $i$-th eigenvalue pair of $\D$ is $0$ iff $\lambda_i(\M) =0 $, where $\M = \mathbf{I}-h\nabla^2f(\x^*) $. Then for $h < \frac{1}{L}$, we get that $\lambda_i(\M) >0 $ for all $i$, which gives the invertibility of $\Lambda$.}} and $ \sqrt{\sum_{j \in \tilde{\mathcal{N}}_{US}} (\tilde{\theta}_j^{us})^2}$ is the value of projection of {$\V^{-1} \z_0$} on $ \tilde{\mathcal{E}}_{US}$. Furthermore, let $ \x_0 - \x^* = \sum\limits_{j \in \mathcal{N}_{US}}\theta_j^{us} \e_j + \sum\limits_{i \in \mathcal{N}_{S}}\theta_i^{s}\e_i$, where $\mathcal{N}_{US}$ and $\mathcal{N}_{S}$ correspond to the index sets for the negative and positive eigenvalues, respectively, of \( \nabla^2 f(\x^*) \), $\e_j,  \e_i $ are the orthonormal eigenvectors of $\nabla^2 f({\x^*})$, and  $ \sqrt{\sum_{j \in \mathcal{N}_{US}} (\theta_j^{us})^2}$ is the value of projection of the initial radial vector $ \x_0-\x^*$ on $ {\mathcal{E}}_{US}$ where $${\mathcal{E}}_{US}  = \text{span}\bigg\{\v \hspace{0.1cm}: \hspace{0.1cm} (\mathbf{I} - h \nabla^2 f({\x^*}) )\v = \lambda \v;\hspace{0.1cm} \abs{\lambda} >1 \bigg\}.$$

Now, for the initialization scheme of $\x_{0} = \x_{-1}$, we get that
 \begin{align}
      \V \Lambda \V^{-1} \z_0 = \D \z_0 &=  \begin{bmatrix}
(1+\beta)(\mathbf{I}-h\nabla^2f(\x^*))\hspace{0.5cm} - \beta(\mathbf{I}-h\nabla^2f(\x^*)) \\  \mathbf{I} \hspace{3.5cm} \boldsymbol{0}
\end{bmatrix}  \begin{bmatrix}
\x_0 - \x^* \\ \x_0 - \x^*
\end{bmatrix} \\
\implies \V \Lambda \V^{-1} \z_0  &= \begin{bmatrix}
\bigg(\mathbf{I}-h\nabla^2f(\x^*)\bigg)(\x_0 - \x^*) \\ \bigg(\mathbf{I}-h\nabla^2f(\x^*)\bigg)(\x_0 - \x^*)
\end{bmatrix} .
 \end{align}
On substituting $ \V^{-1}\z_0 = \sum\limits_{j \in \tilde{\mathcal{N}}_{US}}\tilde{\theta}_j^{us} \tilde{\e}_j + \sum\limits_{i \in \tilde{\mathcal{N}}_{S}}\tilde{\theta}_i^{s}\tilde{\e}_i$ and $ \x_0 - \x^* = \sum\limits_{j \in \mathcal{N}_{US}}\theta_j^{us} \e_j + \sum\limits_{i \in \mathcal{N}_{S}}\theta_i^{s}\e_i$ in the last equation we get:
\begin{align}
      \bigg(\underbrace{\Lambda \sum\limits_{j \in \tilde{\mathcal{N}}_{US}}\tilde{\theta}_j^{us} \tilde{\e}_j}_{{T'_1}} + \Lambda\sum\limits_{i \in \tilde{\mathcal{N}}_{S}}\tilde{\theta}_i^{s}\tilde{\e}_i\bigg) &= \V^{-1} \begin{bmatrix}
\bigg(\mathbf{I}-h\nabla^2f(\x^*)\bigg)\bigg(\sum\limits_{j \in \mathcal{N}_{US}}\theta_j^{us} \e_j + \sum\limits_{i \in \mathcal{N}_{S}}\theta_i^{s}\e_i\bigg) \\ \bigg(\mathbf{I}-h\nabla^2f(\x^*)\bigg)\bigg(\sum\limits_{j \in \mathcal{N}_{US}}\theta_j^{us} \e_j + \sum\limits_{i \in \mathcal{N}_{S}}\theta_i^{s}\e_i\bigg)
\end{bmatrix} \\
& \hspace{-2cm} = \underbrace{\V^{-1} \begin{bmatrix}
\bigg(\mathbf{I}-h\nabla^2f(\x^*)\bigg)\bigg(\sum\limits_{j \in \mathcal{N}_{US}}\theta_j^{us} \e_j \bigg) \\ \bigg(\mathbf{I}-h\nabla^2f(\x^*)\bigg)\bigg(\sum\limits_{j \in \mathcal{N}_{US}}\theta_j^{us} \e_j \bigg)
\end{bmatrix}}_{T_1} +\V^{-1} \begin{bmatrix}
\bigg(\mathbf{I}-h\nabla^2f(\x^*)\bigg)\bigg(\sum\limits_{i \in \mathcal{N}_{S}}\theta_i^{s}\e_i\bigg) \\ \bigg(\mathbf{I}-h\nabla^2f(\x^*)\bigg)\bigg( \sum\limits_{i \in \mathcal{N}_{S}}\theta_i^{s}\e_i\bigg)
\end{bmatrix}. \label{eqvalence1}
\end{align}
{Next, observe that the matrix $\D$ from \eqref{projeqrev1} is the same as the Jacobian $DP([\x^*;\x^*])$ from Theorem~\ref{generalacclimiteigen}. Further, recall from Theorem~\ref{generalacclimiteigen} that the complex eigenvalues of $DP([\x^*;\x^*])$ have magnitude $\sqrt{\beta \lambda_i(\M)}$, and these eigenvalues arise when $\lambda_i(\M) \in \left(0, \frac{4\beta}{(1+\beta)^2} \right]$, where $\M = \mathbf{I} - h\nabla^2f(\x^*)$. Since $\beta \geq 0$, it follows that $\frac{4\beta}{(1+\beta)^2} \leq 1$, and hence if $\lambda_i(\M) \in \left(0, \frac{4\beta}{(1+\beta)^2} \right]$, then $\lambda_i(\M) \in (0, 1]$. Given that $h < \frac{1}{L}$, we have $\lambda_i(\M) \in (0, 1]$ if and only if $i \in \mathcal{N}_S$. Therefore, the complex eigenvalues of $DP([\x^*;\x^*])$ can be generated only by the eigenvalues associated with the stable subspace of $\M$.} {For $\beta \leq 1$, the magnitude of these complex eigenvalues is at most 1, and thus we can conclude that the eigenvalues of the matrix $\D$ corresponding to the index set $\tilde{\mathcal{N}}_{US}$ must be real-valued.}

{We next recall the following from Theorem~\ref{generalacclimiteigen} for the real-valued eigenvalues of \( DP[\x^*;\x^*] \):
\begin{align}
    \lambda_i(DP[\x^*;\x^*]) = \frac{1}{2}\left((1+\beta)\lambda_i(\M) \pm \sqrt{(1+\beta)^2\lambda_i(\M)^2 - 4\beta\lambda_i(\M)}\right), \quad \text{for } \lambda_i(\M) > \frac{4\beta}{(1+\beta)^2}. \label{eqvalence1d}
\end{align}
We now seek to show that the real-valued eigenvalues of the matrix \( \D \) corresponding to the index set \( \tilde{\mathcal{N}}_{US} \) cannot be generated by the eigenvalues of \( \mathbf{I} - h \nabla^2 f(\x^*) \) associated with the index set \( \mathcal{N}_S \), and to characterize their form as a function of the eigenvalues of the matrix \( \M = \mathbf{I} - h \nabla^2 f(\x^*) \). To this end, a simple calculation for the case with the \( + \) sign in \eqref{eqvalence1d} shows that when $\lambda_i(\M) \in \left(\frac{4\beta}{(1+\beta)^2}, \frac{2}{1+\beta}\right]$ the following holds:
\begin{align}
    \frac{1}{2}\left((1+\beta)\lambda_i(\M) + \sqrt{(1+\beta)^2\lambda_i(\M)^2 - 4\beta\lambda_i(\M)}\right)  &> 1  \label{eqvalence1ab} \\
    \iff \sqrt{(1+\beta)^2\lambda_i(\M)^2 - 4\beta\lambda_i(\M)} &> 2 - (1+\beta)\lambda_i(\M) \\
    \iff \-4\beta\lambda_i(\M) &> 4 - 4(1+\beta)\lambda_i(\M) 
    \\
    \iff \lambda_i(\M) &> 1. 
\end{align}
For \( \lambda_i(\M) > \frac{2}{1+\beta} \), \eqref{eqvalence1ab} holds trivially, and thus \eqref{eqvalence1ab} can hold only when \( \lambda_i(\M) > 1 \), or equivalently \( i \not\in \mathcal{N}_S \).
A similar calculation for the case with the \( - \) sign in \eqref{eqvalence1d} shows that for $\lambda_i(\M) \geq \frac{2}{1+\beta}$:
\begin{align}
    0 \leq \frac{1}{2}\left((1+\beta)\lambda_i(\M) - \sqrt{(1+\beta)^2\lambda_i(\M)^2 - 4\beta\lambda_i(\M)}\right) &< 1 \label{eqvalence1abc} \\
    \iff \sqrt{(1+\beta)^2\lambda_i(\M)^2 - 4\beta\lambda_i(\M)} &> (1+\beta)\lambda_i(\M) - 2 \\
    \iff -4\beta\lambda_i(\M) &> 4 - 4(1+\beta)\lambda_i(\M) 
    \\
  \iff \lambda_i(\M) &> 1. 
\end{align}
For \( \lambda_i(\M) \in \left(\frac{4\beta}{(1+\beta)^2}, \frac{2}{1+\beta}\right) \), \eqref{eqvalence1abc} holds trivially, and thus \eqref{eqvalence1abc} can also hold only when \( \lambda_i(\M) > 1 \), or equivalently \( i \not\in \mathcal{N}_S \)}.

{Now, let \( \tilde{\lambda}_j \) denote the \( j \)-th eigenvalue of \( \Lambda \), and let \( \lambda_j \) denote the \( j \)-th eigenvalue of \( \M = \mathbf{I} - h \nabla^2 f(\x^*) \). It then follows from the preceding discussion that for any \( j \in \tilde{\mathcal{N}}_{US} \), we can express \( \tilde{\lambda}_j \) as
\begin{align}
    \tilde{\lambda}_j = \frac{1}{2}\left((1+\beta)\lambda_{q(j)} + \sqrt{(1+\beta)^2\lambda_{q(j)}^2 - 4\beta\lambda_{q(j)}}\right),
\end{align}
where \( q(j) \in \mathcal{N}_{US} \) and \( q : \tilde{\mathcal{N}}_{US} \to \mathcal{N}_{US} \). Using the above substitution and further simplifying \eqref{eqvalence1}, we get:\looseness=-1
\begin{align}
   & \hspace{-8cm} \bigg( \underbrace{\sum\limits_{j \in \tilde{\mathcal{N}}_{US}}\tilde{\theta}_j^{us} \tilde{\lambda}_j  \tilde{\e}_j}_{{T'_1}} + \sum\limits_{i \in \tilde{\mathcal{N}}_{S}}\tilde{\theta}_i^{s} \tilde{\lambda}_i\tilde{\e}_i\bigg) = 
 \underbrace{\V^{-1} \begin{bmatrix}
\sum\limits_{j \in \mathcal{N}_{US}}\theta_j^{us} {\lambda}_j \e_j  \\ \sum\limits_{j \in \mathcal{N}_{US}}\theta_j^{us} {\lambda}_j \e_j 
\end{bmatrix}}_{T_1} +\V^{-1} \begin{bmatrix}
\sum\limits_{i \in \mathcal{N}_{S}}\theta_i^{s} {\lambda}_i\e_i \\  \sum\limits_{i \in \mathcal{N}_{S}}\theta_i^{s} {\lambda}_i\e_i
\end{bmatrix} \\ 
 {\bigg( \underbrace{\sum\limits_{q(j) \in {\mathcal{N}}_{US}} \frac{1}{2}\tilde{\theta}_j^{us}  \bigg((1+\beta)\lambda_{q(j)} + \sqrt{(1+\beta)^2\lambda_{q(j)}^2 - 4\beta\lambda_{q(j)}}\bigg)  \tilde{\e}_j}_{T'_1} + \sum\limits_{i \in \tilde{\mathcal{N}}_{S}}\tilde{\theta}_i^{s} \tilde{\lambda}_i\tilde{\e}_i\bigg) }&= 
 \underbrace{\V^{-1} \begin{bmatrix}
\sum\limits_{j \in \mathcal{N}_{US}}\theta_j^{us} {\lambda}_j \e_j  \\ \sum\limits_{j \in \mathcal{N}_{US}}\theta_j^{us} {\lambda}_j \e_j 
\end{bmatrix}}_{T_1} + \nonumber \\ & \V^{-1} \begin{bmatrix}
\sum\limits_{i \in \mathcal{N}_{S}}\theta_i^{s} {\lambda}_i\e_i \\  \sum\limits_{i \in \mathcal{N}_{S}}\theta_i^{s} {\lambda}_i\e_i
\end{bmatrix} .\label{eqvalence1a}
\end{align}}
{We further let the eigenvalues \( {\lambda}_j \) of the matrix \( \M \) be treated as unknown variables and assume that the matrix \( \V \) is fixed, as discussed earlier. Then the only unknown variables in the matrix \( \D \) from \eqref{projeqrev1} are its eigenvalues \( \tilde{\lambda}_j \), which are functions of \( {\lambda}_j \).} Since \( \tilde{\e}_j, \tilde{\e}_i \) are the canonical basis vectors of \( \mathbb{R}^{2n} \), \( {\e}_j, {\e}_i \) are the orthonormal basis vectors of \( \mathbb{R}^n \), and the index sets \( \mathcal{N}_S \) and \( \mathcal{N}_{US} \) are disjoint, it follows by comparing the index sets for the unknown variables \( {\lambda}_j \) on both sides of \eqref{eqvalence1a} that the vector \( T'_1 \) on the left-hand side can only be generated by the vector \( T_1 \) on the right-hand side of \eqref{eqvalence1a}, for any given set of coefficients \( \{{\theta}_j^{us}\} \) from the term \( T_1 \).} %\textcolor{blue}{Furthermore, for any set of eigenvalues \( \{\lambda_j\} \), \eqref{eqvalence1a} describes an analytic implicit equation between the vectors $ T_1, T_1'$ of the form $ g_1(T_1, T_1') = \mathbf{0}$ where $g_1 : \mathbb{R}^{2n} \times \mathbb{R}^{2n} \to \mathbb{R}^{2n}$ is analytic and the vectors $ T_1, T_1'$ depend analytically on the coefficients ${\theta}_j^{us} , \tilde{\theta}_j^{us} $ respectively. Then from $ g_1(T_1, T_1') = \mathbf{0}$ we can get an analytic implicit equation $ g_2([{\theta}_j^{us}] ,[\tilde{\theta}_j^{us}]) = \mathbf{0}$ where $g_2 : \mathbb{R}^{|\mathcal{N}_{US}|} \times \mathbb{R}^{|\mathcal{N}_{US}|} \to \mathbb{R}^{|\mathcal{N}_{US}|}$ is analytic. Then by implicit function theorem, provided the Jacobian matrix $ \bigg[\frac{\partial g_2}{\partial \tilde{\theta}_j^{us}}\bigg]$ is invertible at $(\mathbf{0},\mathbf{0}) $, we have $[\tilde{\theta}_j^{us}] = g_3([{\theta}_j^{us}]) $ for some $\mathcal{C}^1$ smooth $g_3$ and hence $ T'_1 = g_4(T_1)$ for some $\mathcal{C}^1$ smooth $g_4$. The invertibility of the Jacobian matrix $ \bigg[\frac{\partial g_2}{\partial \tilde{\theta}_j^{us}}\bigg]$ at $(\mathbf{0},\mathbf{0}) \in  \mathbb{R}^{2|\mathcal{N}_{US}|} $ may be hard to verify but since the Jacobian matrix $ \bigg[\frac{\partial g_1}{\partial [T'_1]_j }\bigg]$ at $(\mathbf{0},\mathbf{0}) \in \mathbb{R}^{2n} $ has a rank $|\mathcal{N}_{US}| $ by a simple computation provided the eigenvalue $\lambda_j$ is non-zero for all $j$, we will have full rank for the Jacobian matrix $ \bigg[\frac{\partial g_2}{\partial \tilde{\theta}_j^{us}}\bigg]$ at $(\mathbf{0},\mathbf{0}) \in  \mathbb{R}^{2|\mathcal{N}_{US}|} $. Then locally linearizing $ T'_1 = g_4(T_1)$ about $T_1 = \mathbf{0}$ via Taylor's theorem with integral remainder will yield the $T'_1 = \C T_1 $ for some matrix $\C$.} 
Hence, {it can be shown that} \eqref{eqvalence1} can be decomposed into the following equations  
\begin{align}
 \Lambda \bigg(\sum\limits_{i \in \tilde{\mathcal{N}}_{S}}\tilde{\theta}_i^{s} \tilde{\e}_i \bigg) & =  \bigg(\mathbf{I}-{\C} \bigg) \V^{-1} \begin{bmatrix}
\bigg(\mathbf{I}-h\nabla^2f(\x^*)\bigg)\bigg(\sum\limits_{j \in \mathcal{N}_{US}}\theta_j^{us} \e_j \bigg) \\ \bigg(\mathbf{I}-h\nabla^2f(\x^*)\bigg)\bigg(\sum\limits_{j \in \mathcal{N}_{US}}\theta_j^{us} \e_j \bigg)
\end{bmatrix}  + \V^{-1}\begin{bmatrix}
\bigg(\mathbf{I}-h\nabla^2f(\x^*)\bigg)\bigg( \sum\limits_{i \in \mathcal{N}_{S}}\theta_i^{s}\e_i\bigg) \\ \bigg(\mathbf{I}-h\nabla^2f(\x^*)\bigg)\bigg( \sum\limits_{i \in \mathcal{N}_{S}}\theta_i^{s}\e_i\bigg)
\end{bmatrix}, \\
     \Lambda \bigg(\sum\limits_{j \in \tilde{\mathcal{N}}_{US}}\tilde{\theta}_j^{us} \tilde{\e}_j \bigg) & = {\C} \V^{-1} \begin{bmatrix}
\bigg(\mathbf{I}-h\nabla^2f(\x^*)\bigg)\bigg(\sum\limits_{j \in \mathcal{N}_{US}}\theta_j^{us} \e_j \bigg) \\ \bigg(\mathbf{I}-h\nabla^2f(\x^*)\bigg)\bigg(\sum\limits_{j \in \mathcal{N}_{US}}\theta_j^{us} \e_j \bigg)
\end{bmatrix}  \iff { T'_1 = \C T_1,} \label{eqvalence2}
\end{align}
for some invertible, bounded matrix\footnote{The matrix $\C$ will depend on the value of $\beta$, $n$ and the eigenvalues in the ${\mathcal{E}}_{US} $ subspace of the matrix $ \bigg(\mathbf{I}-h\nabla^2f(\x^*)\bigg)$.} $\C$ for any coefficients ${\theta}_j^{us}, {\theta}_i^{s} $. {From \eqref{eqvalence1a}, for arbitrary coefficients $\tilde{\theta}_j^{us}, \tilde{\theta}_i^{s} $, the vector $T'_1$ is zero iff $ {\lambda}_{q(j)} = 0$ for all $q(j) \in  \mathcal{N}_{US}$ and similarly, for arbitrary coefficients ${\theta}_j^{us}, {\theta}_i^{s} $, the vector $T_1$ is zero iff ${\lambda}_{j} = 0$ for all $j \in  \mathcal{N}_{US}$. Both these cases can only be possible if the index set $ \mathcal{N}_{US}$ is empty. Then, for arbitrary coefficients ${\theta}_j^{us}, {\theta}_i^{s} $, if $T_1 \neq \mathbf{0}$, it must be that $T'_1 \neq \mathbf{0}$ since $\mathcal{N}_{US}$ is non-empty. Hence, the matrix $\C$ cannot have a {non-trivial} null space for any set of coefficients ${\theta}_j^{us}, {\theta}_i^{s} $ and so the invertibility of $\C$ follows. The equivalence of the vectors $T_1, T'_1$ in norm follows immediately by invertibility of $\C$ and from \eqref{eqvalence2} we have the bound $$\norm{\C^{-1}}_2^{-1}\norm{T_1} \leq \norm{T_1'} \leq \norm{\C}_2 \norm{T_1}  .$$}  Finally, after appropriate rearrangements in \eqref{eqvalence2} followed by taking operator norm on both sides and using the submultiplicativity of operator norm {along with the invertibility of $\C, \Lambda$}, we will get the equivalence of the initial unstable projections, i.e., $\sum\limits_{j \in \mathcal{N}_{US}}(\theta_j^{us})^2 \approx \sum\limits_{j \in \tilde{\mathcal{N}}_{US}}{(\tilde{\theta}_j^{us} )}^2$, where `$\approx$' means equality up to some multiplicative constants. Hence, if $\sum\limits_{j \in \mathcal{N}_{US}}(\theta_j^{us})^2 \ll 1$ then we also have $  \sum\limits_{j \in \tilde{\mathcal{N}}_{US}}{(\tilde{\theta}_j^{us} )}^2 \ll 1$.

\end{document}